\documentclass[aos]{imsart}

\RequirePackage{amsthm,amsmath,amsfonts,amssymb}
\RequirePackage[authoryear]{natbib}%% uncomment this for author-year citations
\RequirePackage[colorlinks,citecolor=blue,urlcolor=blue]{hyperref}
\RequirePackage{graphicx}
\usepackage{stmaryrd}
\usepackage[ruled,vlined]{algorithm2e}
\usepackage{enumerate}

\usepackage{tikz}
\usetikzlibrary{shapes ,arrows ,positioning, automata}

\startlocaldefs
\theoremstyle{plain}

\newtheorem{theorem}{Theorem}[section]
\newtheorem{corollary}[theorem]{Corollary}
\newtheorem{lemma}[theorem]{Lemma}

\newtheorem{proposition}[theorem]{Proposition}
\theoremstyle{remark}
\newtheorem{remark}[theorem]{Remark}

\def \E{I\!\!E}
\def \P{I\!\!P}

\def \Z{\mathbb Z}

\newcommand{\indiq}{{{\mathbf 1}}}
\def \equaldistrib{\,{\buildrel d \over =}\,}

\newcommand{\R}{\mathbb {R}}

\newcommand{\vp}{\!\dagger\!}
\newcommand{\tn}{\vert\kern-0.25ex\vert\kern-0.25ex\vert}

\DeclareMathOperator{\var}{\mathop{Var}}
\DeclareMathOperator{\cov}{\mathop{Cov}}

\newcommand{\changes}[1]{#1}
\endlocaldefs
\begin{document}

\begin{frontmatter}
\title{Inferring the dependence graph density of binary graphical models in high dimension}
\runtitle{Dependence graph density inference in high dimension}

\begin{aug}
%%%%%%%%%%%%%%%%%%%%%%%%%%%%%%%%%%%%%%%%%%%%%%%
%% Only one address is permitted per author. %%
%% Only division, organization and e-mail is %%
%% included in the address.                  %%
%% Additional information can be included in %%
%% the Acknowledgments section if necessary. %%
%% ORCID can be inserted by command:         %%
%% \orcid{0000-0000-0000-0000}               %%
%%%%%%%%%%%%%%%%%%%%%%%%%%%%%%%%%%%%%%%%%%%%%%%
\author[A]{\fnms{Julien} \snm{Chevallier}\ead[label=e1]{julien.chevallier1@univ-grenoble-alpes.fr}\orcid{0000-0002-0736-8487}},
\author[B]{\fnms{Eva} \snm{L\"ocherbach}\ead[label=e2]{Eva.Locherbach@univ-paris1.fr}\orcid{0000-0003-4436-2532}}
\and
\author[C]{\fnms{Guilherme} \snm{Ost}\ead[label=e3]{guilhermeost@im.ufrj.br}\orcid{0000-0003-0887-9390}}
%%%%%%%%%%%%%%%%%%%%%%%%%%%%%%%%%%%%%%%%%%%%%%
%% Addresses                                %%
%%%%%%%%%%%%%%%%%%%%%%%%%%%%%%%%%%%%%%%%%%%%%%
\address[A]{Univ. Grenoble Alpes, CNRS, Inria, Grenoble INP\footnote{Institute of Engineering Univ. Grenoble Alpes}, LJK, 38000 Grenoble, France,
\printead{e1}
}

\address[B]{Université Paris 1 Panthéon Sorbonne,
\printead{e2}
}

\address[C]{Institute of Mathematics, Federal University of Rio de Janeiro,
\printead{e3}}

\end{aug}

\begin{abstract}
We consider a system of binary interacting chains describing the dynamics of a group of $N$ components that, at each time unit, either send some signal to the others or remain silent otherwise. 
The interactions among the chains are encoded by a directed Erd\"os-R\'enyi random graph with unknown parameter $ p \in  (0, 1) .$
Moreover, the system is structured within two populations (excitatory chains versus inhibitory ones) which are coupled
via a mean field interaction on the underlying Erd\"os-R\'enyi graph.
In this paper, we address the question of inferring the connectivity parameter $p$ based only on the observation of the interacting chains over $T$ time units.
In our main result, we show that the connectivity parameter $p$ can be estimated with rate $N^{-1/2}+N^{1/2}/T+(\log(T)/T)^{1/2}$ through an easy-to-compute estimator. 
Our analysis relies on a precise study of the spatio-temporal decay of correlations of the interacting chains. This is done through the study of coalescing random walks defining a backward regeneration representation of the system. 
Interestingly, we also show that this backward regeneration representation allows us to perfectly sample the system of interacting chains (conditionally on each realization of the underlying Erd\"os-R\'enyi graph) from its stationary distribution. These probabilistic results have an interest in its own. 
\end{abstract}



\end{frontmatter}

\tableofcontents

\section{Introduction}

Understanding how to infer and interpret interactions between measured components of large scale systems is a central question in many scientific fields, including Neuroscience, Statistical Physics, and Social Networks \citep{strogatz2001exploring}. This has been done traditionally in the framework of probabilistic graphical models \citep{lauritzen1996graphical}, in which the interaction among the components of the system is encoded through a graph, sometimes referred to as the graph of conditional dependencies. For such models, a fundamental question is that of estimating the underlying graph of conditional dependencies or some function of it from data. 
Over the past three decades, this question has been extensively investigated not only under the assumption that the data is a set of independent and identically distributed observations and the underlying graphical model has pairwise interactions \citep{bresler2015efficiently,Montanari2009which,ravikumar2010highdimensional}, but also for time-dependent data \citep{Eichler2012graphical,duarte2019estimating,ReynaudBouret2013inference} or graphical models with high-order interactions \citep{Basu2015regularized,lerasle2016sharp}.
In recent years, understating the theoretical guarantees of the proposed methods (e.g., consistency, computational complexity) in the high-dimension setting became a major issue, as now the simultaneous activity of many components can be routinely recorded.
In high-dimension, most of the works assume that the underlying graph of dependencies is sparse in a suitable sense, showing that under this assumption the proposed method performs reasonably well. The cases in which the graphs of dependencies are dense, on the other hand, are much less studied. In this setting we are only aware of the works of \citep{kim2021twosample,delattre2016statistical} and  \citep{liu2020statistical}. 
In these works the focus is not on the estimation of the graph of dependencies of the corresponding graphical model, as we detail in what follows. 

In \citep{kim2021twosample}, particularly motivated by applications in Genetics and Neuroscience, the authors propose a density ratio approach to estimate the ``difference network'', the difference between two graphs of dependencies. The key assumption in their approach is that this difference network is sparse and, in particular, their method can also be applied when each individual graph is dense as long as their difference is sparse. \citep{delattre2016statistical} consider a graphical model with pairwise interactions for which the graph of dependencies is a realization of a weighted Erd\H{o}s-Rényi random graph with edge probability $p$. Specifically, \citep{delattre2016statistical} work with $N$ Hawkes point processes on $[0,T]$ - a class of point processes in which the intensity at any given time is a linear function of the past events - where each pair of these point processes are independently coupled with probability $p$. 
Moreover, the existing interactions are excitatory and of mean-field type (i.e., the occurrence of an event increases the chance of a new event to occur and the increment scales as $1/N$). 
The main goal of the paper is not to estimate the random graph of dependencies but rather the density of connections of this graph (the parameter $p$) based only on the information present in the $N$ point processes on $[0,T]$.    
Under some few extra assumptions, the authors show that this can be done with a precision of the order $N^{-1/2}+N^{1/2}/m_T$ (up to some correction factor) where $m_T$ denotes the mean number of events per point process.
In \citep{liu2020statistical}, the author complements the analysis started in \citep{delattre2016statistical} by investigating the problem of estimating the density of connections $p$ in the same setting with the difference being that now one has access only to the information present in $K$ Hawkes point processes for some large $K\leq N$. Under this constraint,  \citep{liu2020statistical} proposes an estimator of the parameter $p$ with rate of convergence $K^{-1/2}+N/(K^{1/2}m_T)+N/(Km^{1/2}_T).$   

In the present paper, we consider a graphical model which is defined in terms of a system of $N$ interacting $\{0,1\}$-valued chains $X=\{X_{i,t},t\in\Z, 1\leq i\leq N\}$ and a random matrix $\theta=(\theta_{ij})_{1\leq i,j\leq N}$ of i.i.d entries distributed as $\text{Ber}(p)$. The event $X_{i,t}=1$ indicates that the $i$-th chain sends some signal at time $t$, and $X_{i,t}=0$ otherwise. Let $X_{t}=(X_{1,t},\ldots, X_{N,t})$ be the configuration of the system $X$ at time $t$. The model is defined as follows. Conditionally on $\theta$, the system $X$ evolves as a stationary Markov chain on the state space $\{0,1\}^N$ in which the conditional distribution of $X_t$ given that $X_{t-1}=x$ is that of $N$ independent Bernoulli random variables with parameter $p_{\theta,i}(x)$, $i=1,\ldots,N$, where
\begin{equation}
\label{def:transition_prob_2}
p_{\theta,i}(x)= \mu +(1-\lambda)\left(\frac{1}{N} \sum_{j\in\mathcal{P}_+} \theta_{i j}x_{j}+\frac{1}{N}\sum_{j\in\mathcal{P}_-} \theta_{i j}(1-x_{j})\right), \ x=(x_1,\ldots,x_N).
\end{equation}
Here, $0 <  \lambda < 1 $ and $0\leq \mu \leq \lambda$  are parameters, and $\mathcal{P}_+$, $\mathcal{P}_-$ are subsets of $[N]=\{1,\ldots, N\}$ forming a partition. Neither the parameters $\mu$ and $\lambda$ nor the subsets $\mathcal{P}_+$ and $\mathcal{P}_-$ depend on the random matrix $\theta$.

The function $p_{\theta,i}(x)$ appearing in \eqref{def:transition_prob_2} models the probability of observing a new signal for the $i$-th chain given that the configuration of the system at the precedent time is $x$.  
The parameter $\mu$ models the baseline activity of each interacting chain. Note that the probability $p_{\theta,i}(x)$ depends on the past $x$ only through the values $x_j$ for which $\theta_{ij}=1$, so that the random matrix $\theta$ encodes the interaction among the chains of the system. In light of this, $\theta$ can be considered as the proxy for the graph of conditional dependencies of the model. 
Note also that each chain is either excitatory or inhibitory in the model; the set $\cal{P}_+$ denotes the set of excitatory chains and $\cal{P}_-$ the set of inhibitory ones. An excitatory chain increases the probability of future signals to occur whenever it sends any signal, whereas the signals sent by any inhibitory chain reduce the probability of occurrence of new signals. The increment of this changing is $(1-\lambda)/N$, implying that the interactions are of mean-field type and that $1-\lambda$ parametrizes the strength of the interactions. 

In what follows, we denote $r^N_+=|\mathcal{P}_+|/N$ and $r^N_-=|\mathcal{P}_-|/N$ the fraction of excitatory and inhibitory components of the model, respectively, and suppose that there are values $0<r_+,r_-<1$ satisfying $r_++r_-=1$ such that 
$|r^N_+-r_+|\vee |r^N_--r_-|\leq KN^{-1}$ for some universal constant $K$.  
One can easily check that a choice satisfying this assumption with $K=1$ is $|\mathcal{P}_+|=\lceil r_+N\rceil$ and $|\mathcal{P}_-|=N-|\mathcal{P}_+|$.

\paragraph*{Problem formulation} 
The goal of this paper is to address the following statistical question.
By observing a sample $X_{1},\ldots, X_{T}$ of the system $X$, can we estimate the asymptotic density of connections $p$ of the random matrix $\theta$, knowing only the size of the system $N$ and the asymptotic fraction of excitatory and inhibitory components $r_+$ and $r_-$?
This means that we do not assume any prior knowledge on $\mu, \lambda, p, \theta, \mathcal{P}_+, \mathcal{P}_-$, i.e., they are all unknown.

Although the estimation problem investigated here is similar to the ones considered in \citep{delattre2016statistical, liu2020statistical}, our work is different from theirs in at least two substantial ways. First, in their model the interaction between each pair of components can only be excitatory, whereas it can be either excitatory or inhibitory in ours, making the analysis of the estimation problem (especially the study of the random environment) more challenging. Second, a crucial step in the analysis is to study the decay of correlation between $X_{i,t}$ and $X_{j,s}$ (and of products of these), for each fixed realization of the random matrix $\theta$. In our case, this is done through the study of coalescing random walks defining a backward regeneration representation of the system $X$.    
Interestingly, we show that this backward regeneration representation allows us to 
perform perfect simulation of the system $X$. That is, one can
exactly sample the system $X$ (conditionally on $\theta$) from its unique stationary distribution. These probabilistic results have an interest in its own. 
In \citep{delattre2016statistical} and \citep{liu2020statistical}, on the other hand, the analysis of the model for each realization of the graph of conditional dependence follows a rather different approach, fundamentally based on martingale arguments. Moreover, the question of how to perfectly sample the underlying process is not discussed. 

As will become apparent in our results, we can not only estimate the parameter $p$ but also the parameters $\mu$ and $\lambda$. In our main result, we show that the parameters $(\mu,\lambda,p)$ can be simultaneously estimated with rate $N^{-1/2}+N^{1/2}/T +(\log(T)/T)^{1/2}$, under some extra conditions on the asymptotic fraction $r_+$ of excitatory components of the model. 
\changes{
This convergence rate is consistent with the one found in \citep{delattre2016statistical} (up to a $\log(T)$ factor), because the quantity $m_T$ (adapted to our setting) increases linearly in $T$. %Moreover, let us mention that in Appendix \ref{app:lower:bound}, %we establish a lower bound of the order $N^{-1/2}+N^{1/2}/T$ in a simplified statistical setting where the observations follow a binomial mixture model. This lower bound suggests that our estimation rate might be near optimal. Establishing a matching lower bound in the statistical setting considered in the paper is an interesting open problem.
In Appendix \ref{app:lower:bound}, we briefly discuss the optimality of our estimation rate by analyzing two related statistical settings in which the observations follow a binomial mixture model.
Taken together, the results in Appendix \ref{app:lower:bound} suggest that our estimation rate may be near optimal. Establishing rigorously such a result in the statistical setting considered in this paper remains an interesting open problem.
}

We are particularly motivated by the statistical analysis of neuronal networks. Neurons communicate with each other by firing short electrical pulses, often called spikes.
The spiking times of a network of neurons depend on its graph of interaction, a combinatorial structure (typically unknown) encoding the type of interaction (excitatory or inhibitory) between each pair of neurons in the network. An important question in Neuroscience is to understand which features of this graph can be inferred from the spiking times of the recorded neurons. Our results suggest that one such feature is the density of connections in the graph, as long as we know the true fraction $r_+$ of excitatory neurons in the network. In many situations, this turns out to be the case. For example, in humans and many other mammalian species the value $r_+$ is known and ranges from $0.7$ to $0.8$.

%\changes{
%    The present paper opens new avenues for research in the field of graphical models and high-dimensional inference. Here are some examples: 1) the detection of the two communities $\mathcal{P}_+$ and $\mathcal{P}_-$, 2) the sparse cases $p \sim c/n$ or $p\sim \log(n)/n$. Note that example 1 above is the subject of the preprint \citep{chevallier2024community}. We believe that example 2 needs a new approach: first and foremost, the interaction normalization by $N^{-1}$ is not relevant and most of our results do not apply to this case.
%}

\changes{
Analogously to \cite{delattre2016statistical}, our results also hold in the symmetric interaction case, namely when $(\theta_{ij})_{1\leq i\leq j\leq N}$ are i.i.d. and $\theta_{ij} = \theta_{ji}$ for all $1\leq i,j \leq N$. Only a few steps in the proofs need to be adapted when working in this new framework, and all these modifications are located in Appendix \ref{app:annealed:rates}. We have gathered the details at the beginning of that section.
}

\changes{Throughout the paper, the parameter $p$ does not scale with the number of components $N$ of the system $X$, ensuring that the random matrix $\theta$ remains  typically dense. An interesting open question is whether our results extend to sparse regimes, where $p=c/N$ for some positive constant $c>0$ or $p\sim \log(N)/N$.
In these sparse regimes, the model must be suitably modified, as the normalization factor $N^{-1}$ is no longer appropriate. Furthermore, many of the results concerning the coalescing random walks associated with the backward regeneration representation of the system $X$ 
(e.g. Proposition \ref{prop:1}) do not directly apply and require substantial extensions. Such extensions are beyond the scope of this work.}    

%Finally, let us highlight the recent work \citep{chevallier2024community}, which addresses the complementary problem of estimating the sets ${\cal P_+}$ and ${\cal P_-}$ without prior knowledge of the other model parameters. We believe these results could enable us to estimate edge probability $p$ even without knowing the asymptotic fraction of excitatory and inhibitory components $r_+$ and $r_-$ of the system $X$. For sake of simplicity, we do not pursue this extension in the present work.  

\changes{
Finally, let us highlight two recent works \citep{gaitonde2024bypassing} and \citep{chevallier2024community}, which appeared during the revision of this paper.
 The work of \citep{gaitonde2024bypassing} deals with the problem of estimating order $k$ graphical models with $N$ sites and bounded degree from dependent trajectories generated by Glauber dynamics. Among their results, the authors show that given $O(N \log N)$ total site updates from the Glauber trajectory, it is possible to correctly output the conditional dependency structure of the underlying graphical model in time $O(N^2 \log N)$, overcoming the known $N^{\Theta(k)}$ computational barrier of the i.i.d. setting. 
 Our framework differs from theirs in two key aspects. First, the dynamics of their model is reversible, whereas ours is not. Second,
 their analysis assumes that the chosen site to be updated is always observed at each time step, even when the configuration of the system remains unchanged, an assumption that may be difficult to verify in some real datasets.  
The work of \citep{chevallier2024community} complements the analysis done in the present paper by
 addressing the %complementary 
 problem of estimating the sets ${\cal P_+}$ and ${\cal P_-}$ without prior knowledge of the other model parameters. We believe these results could enable us to estimate edge probability $p$ even without knowing the asymptotic fraction of excitatory and inhibitory components $r_+$ and $r_-$ of the system $X$. For sake of simplicity, we do not pursue this extension in the present work.   
}

In the next section we define our model rigorously,  introduce some general notation used throughout the paper, state our main results and give some heuristics behind them.
At the end of that section, we also provide the organization of the rest of the paper.

\section{Model definition, notation and main results}
\label{sec:results}

\paragraph*{Model definition}
We consider a system of $N$ interacting chains $X=\{X_{i, t} , t \in \Z , 1 \le i \le N\} $ taking values in $ \{0, 1 \} $ denoting the presence or the absence of a signal at a given time. This system evolves in a random environment which is given by the realization of a directed Erd\"os-R\'enyi random graph via the selection of $N^2 $ i.i.d.\! Bernoulli random variables $ \theta_{ ij } \sim \text{Ber}(p), 1 \le i, j \le N$ with $0\leq p\leq 1$.
Conditionally on the realization of $\theta=(\theta_{ij})_{1\leq i,j\leq N}$, the evolution of the system $X$ is that of a stationary Markov chain on the state space $\{0,1\}^N$ with transition probabilities (which depend on $\theta$) given as follows.
Writing $X_t=(X_{1,t},\ldots, X_{N,t})$ and $x,y\in \{0,1\}^{N}$, we have, for all $t\in\Z$,
\begin{equation}
\label{def:transition_prob_1}
\P_{\theta}(X_{t}=y  | X_{t-1}=x) =\prod_{i=1}^{N}(p_{\theta,i}(x))^{y_i}(1-p_{\theta,i}(x))^{(1-y_i)}, 
\end{equation}
where $p_{\theta,i}(x)=\P_{\theta}( X_{i,t}=1| X_{t-1}=x)$ is defined in \eqref{def:transition_prob_2}.
The existence and uniqueness of a stationary version of the Markov chain having transition probabilities as defined in \eqref{def:transition_prob_1} and \eqref{def:transition_prob_2} is granted by Theorem \ref{thm:perfect_sampling} presented in Section \ref{sec:3}.

\paragraph*{Notation}
Throughout the paper, the letters $t,s$ denote some time values whereas the letters $i,j$ denote some spatial values, that is the index of the corresponding component. For a vector $v$ in $\mathbb{R}^N$, the notation $\overline{v}$ denotes the spatial mean $\overline{v} = N^{-1} \sum_{i=1}^N v_i$. In agreement with the left-hand side of \eqref{def:transition_prob_1}, we write $\P_{\theta}$ to denote the probability measure under which the environment $\theta$ is kept fixed and the process $X$ is distributed as the unique stationary version of the Markov chain having transition probabilities as defined in \eqref{def:transition_prob_1} and  \eqref{def:transition_prob_2}. The expectation taken with respect to $\P_{\theta}$ is denoted by $\E_{\theta}$. The variance and covariance computed from $\E_{\theta}$ are denoted $\var_{\theta}$ and $\cov_{\theta}$ respectively.   
Moreover, we write $\P$ to denote the probability measure under which the random environment $\theta=(\theta_{ij})_{1\leq i,j\leq N}$ is distributed as a collection of i.i.d. random variables with distribution $\text{Ber}(p)$ and
the conditional distribution of the process $X$ given $\theta$ is that of the process $X$ under $\P_{\theta}$, i.e., the following identity holds $\P(X\in \cdot|\theta)=\P_{\theta}(X\in \cdot)$.  Finally, we denote $\E$ the expectation taken with respect to the probability measure $\P$, and $\var$ and $\cov$ the variance and covariance, respectively, computed from $\E$.

\paragraph*{Main results}

We are given a sample $X_{1},\ldots, X_{T}$ of a stationary Markov chain with transition probabilities  given by \eqref{def:transition_prob_1} and \eqref{def:transition_prob_2}, associated to some realization of the random environment $\theta$ which is not observed.
The goal is to estimate the parameter $p$, the asymptotic density of connections in the underlying random environment. We consider an estimator of $p$ which is a function of three other estimators defined below. In what follows, let $\overline{X}_t=N^{-1}\sum_{i=1}^{N}X_{i,t}$ denote the average number of signals emitted by the system at time $t$; $Z_{i,t}=\sum_{s=1}^t X_{i,s}$ denote the number of signals emitted by the $i$-th component of the system in the discrete interval $\{1,\ldots, t\}$ and $\overline{Z}_{t}=N^{-1}\sum_{i=1}^{N}Z_{i,t}$ denote the average number of signals emitted by the system in the interval $\{1,\ldots, t\}$. For notational convenience, we set $\overline{Z}_{0}:=0$.    
With this notation, our three main estimators are defined as follows:
\begin{equation}\label{eq:def:m:hat:v:hat}
\hat{m}=\frac{\overline{Z}_{T}}{T}, \ \ \
 \hat{v}=\frac{(T+1)N}{T^3}\left[\frac{1}{N}\sum_{i=1}^{N}\left(Z_{i,T}\right)^2-\frac{T}{(T+1)}\left(\overline{Z}_{T}+\left(\overline{Z}_{T}\right)^2\right)\right] \ \text{and} \ 
\end{equation}
\begin{equation}\label{eq:def:w:hat}
\hat{w}=2\mathcal{W}_{2\Delta}-\mathcal{W}_{\Delta}, \ \text{with} \  
\mathcal{W}_{\Delta}=\frac{N}{T}\sum_{k=1}^{\lfloor T/\Delta \rfloor}\left(\overline{Z}_{k\Delta}-\overline{Z}_{(k-1)\Delta}-\Delta\hat{m}\right)^2,
\end{equation}
where $\Delta\in \{1,\ldots, \lfloor T/2 \rfloor \}$ is a tuning parameter. The estimator $\hat{m}$ is the \emph{spatio-temporal mean} of signals emitted by the system in the observed interval $\{1,\ldots, T\}$.
As we explain in Section \ref{subsection:spatial_var}, 
the estimator $\hat{v}$ is related to the empirical variance of the random variables $Z_{1,T},\ldots, Z_{N,T}$, i.e., a variance over the spatial components of the system. 
The estimator $\hat{w}$ is computed from the empirical variance of the random variables $\overline{Z}_{\ell\Delta}, \overline{Z}_{2\ell\Delta}-\overline{Z}_{\ell\Delta}, \ldots, \overline{Z}_{T}-\overline{Z}_{T-\ell\Delta}$, i.e., the empirical variance of the mean number of signals emitted by the system over different time intervals. For these reasons, the estimators $\hat{v}$ and $\hat{w}$ are called \emph{spatial variance} and \emph{temporal variance} respectively.

Obviously, the three estimators $\hat{m}, \hat{v},\hat{w}$ depend on $N$ and $T$ (as well as the tuning parameter $\Delta$ for $\hat{w}$). We chose not to specify this dependence in the notation to keep notation as light as possible.

\begin{theorem}\label{theo:1}
There exists a constant $K>0$ depending only on $\lambda$ such that for all $\varepsilon\in (0,1)$, $N\geq 1$, $T\geq 2$ and $1\leq \Delta\leq \lfloor T/2\rfloor$,
\begin{gather}
\P(|\hat{m}-m|\geq \varepsilon)\leq \frac{K}{\varepsilon}\left(\frac{1}{\sqrt{TN}}+\frac{1}{N}\right) \label{conver_spatio_temporal_mean} ,\\
\P(|\hat{v}-v|\geq \varepsilon)\leq \frac{K}{\varepsilon}\left(\frac{\sqrt{N}}{T} + \frac{1}{\sqrt{N} } \right) \label{conver_spatial_variance},\\
\P(|\hat{w}-w|\geq \varepsilon)\leq \frac{K}{\varepsilon}\left(\frac{1}{N}+\frac{(1-\lambda)^{\Delta}}{\Delta}+\sqrt{\frac{\Delta}{T}}\right) \label{conver_temporal_variance},
\end{gather}
where the vector $(m,v,w)$ is given by
\begin{equation}\label{eq:definition:m:v:w}
\begin{cases}
    m=\frac{\mu+(1-\lambda)pr_-}{1-p(1-\lambda)(r_+-r_-)},\\
    v=(1-\lambda)^2 p(1-p)((m-r_-)^2+r_+r_-) ,\\
    w= m(1-m)\frac{1+4(1-\lambda)^2p^2r_+r_-}{(1-p(1-\lambda)(r_+-r_-))^2}.
\end{cases}
\end{equation}
\end{theorem}
The proof of Theorem \ref{theo:1} is given at the end of Section \ref{sec:key:steps:and:proof}.
The constant $K$ depends only on the parameter $\lambda$. Let us mention that this constant diverges when $\lambda \to 0$ (because of the inversion of the linear system, see for instance the bound of Lemma \ref{lem:bound:infty:ell:c}) and $\lambda \to 1$ (because the dependence between the components of the model becomes too weak, see for instance the bounds of Proposition \ref{prop:1}).

Once the parameters $(m,v,w)$ are estimated from the sample, one wants to deduce some estimators of the parameters $(\mu,\lambda,p)$. Here is how one can proceed.

Let $\Lambda = \{(\mu, \lambda, p) \in (0,1)^3: 0< \mu < \lambda\}$ be the open set of admissible parameters. For all $(\mu, \lambda, p) \in \Lambda$, let us define 
\begin{equation}
\label{def_denominator_as_function_of_lambda_and_p}
    D(\lambda, p) = 1 - (1-\lambda)p(r_+ - r_-) >0,
\end{equation}
the denominator that appears in the expressions of $m$ and $w$. Then, for $k=1,2,3$, let $\Psi_k : \Lambda \to \mathbb{R}$ be defined by
\begin{equation}\label{eq:definition:Psi}
    \begin{cases}
        \Psi_1(\mu, \lambda, p) = (\mu + (1 - \lambda)pr_-) / D(\lambda,p)\\
        \Psi_2(\mu, \lambda, p) = (1-\lambda)^2 p(1-p)[(\Psi_1(\mu, \lambda, p)-r_-)^2+r_+r_-]\\
        \Psi_3(\mu, \lambda, p) = \Psi_1(\mu, \lambda, p)[1 - \Psi_1(\mu, \lambda, p)] [1+4(1-\lambda)^2p^2r_+r_-] / D(\lambda,p)^2.
    \end{cases}
\end{equation}
Finally, let $\Psi : \Lambda \to \mathbb{R}^3$ be defined by the three coordinate functions above so that Equation \eqref{eq:definition:m:v:w} rewrites as $(m,v,w) = \Psi(\mu,\lambda,p)$. Finally, let us remark that, whatever the value of $r_+$, the image $\Psi(\Lambda)$ is included in $(0,1)\times (0,\infty)^2$ (see Proposition \ref{prop:preliminary:D:Psi} stated and proved in Appendix \ref{app:inversion}).

The following proposition gives some information about the inversion of the function $\Psi$, by means of an auxiliary function $\kappa : (0,1)\times (0,\infty) \to (0,\infty)$ which is defined by 
\begin{equation}\label{eq:definition:kappa}
    \kappa(m,w) = (r_+ - r_-)^2 \frac{w}{m(1-m)}.
\end{equation}
\begin{proposition}\label{prop:there:exists:inversion}
    Whatever the value of $r_+$, there exist two explicit functions $\Phi^{(+)}$ and $\Phi^{(-)}$ such that the following results hold.
    \begin{enumerate}
        \item For all $(\mu, \lambda, p)\in \Lambda$, $(\mu, \lambda, p) \in \{ \Phi^{(+)}\circ \Psi(\mu, \lambda, p), \Phi^{(-)}\circ \Psi(\mu, \lambda, p)\}$,
        \item Moreover, if $r_+ \geq 1/2$ or 
        $$
        \kappa(\Psi_1(\mu, \lambda, p), \Psi_3(\mu, \lambda, p)) \geq 4r_+r_-,
        $$
        then $(\mu, \lambda, p) = \Phi^{(-)}\circ \Psi(\mu, \lambda, p)$.
    \end{enumerate}
\end{proposition}
Proposition \ref{prop:there:exists:inversion} is deduced as a corollary of Proposition \ref{prop:inversion} stated and proved in the Appendix \ref{app:inversion}. Furthermore, the expressions of the functions $\Phi^{(+)}$ and $\Phi^{(-)}$ are also given in Appendix \ref{app:inversion}.

\changes{
\begin{remark}
    
    Note that the system of equations defining $X$ is not symmetric with respect to the sets $\mathcal{P}_+$ and $\mathcal{P}_-$. One source of asymmetry is the term $N^{-1} \sum_{j\in \mathcal{P}_-} \theta_{ij}$ which appears in \eqref{def:transition_prob_2} and is denoted by $L^{N, \bullet -}_i$ below. This asymmetry partly explains why the inversion of \eqref{eq:definition:m:v:w} is always feasible when $r_+\geq 1/2$, but not when $r_+ <1/2$.

\end{remark}
}

We are now in position to deduce the following corollary of Theorem \ref{theo:1}.

\begin{corollary}
\label{cor:consistency_of_the_estimator}
    Assume that the condition 2 of Proposition \ref{prop:there:exists:inversion} is satisfied. Then there exists a constant $K>0$ depending only on $\lambda$ such that for all $\varepsilon\in (0,1)$, $N\geq 1$, $T\geq 2$ and $1\leq \Delta\leq \lfloor T/2\rfloor$,
    \begin{equation}
    \label{infor_statemente_of_the_convergece2}
    \P(\| \Phi^{(-)}(\hat{m}, \hat{v}, \hat{w}) - (\mu, \lambda, p)\|_{\infty} \geq \varepsilon)
    \leq \frac{K}{\varepsilon}\left(\frac{\sqrt{N}}{T}+\frac{1}{\sqrt{N}}
    +\frac{(1-\lambda)^{\Delta}}{\Delta}+\sqrt{\frac{\Delta}{T}}\right).
    \end{equation}
\end{corollary}
\begin{proof}
    By definition of $\Psi$, remark that Theorem \ref{theo:1} gives a control of $\| (\hat{m}, \hat{v}, \hat{w}) - \Psi(\mu, \lambda, p)\|_{\infty}$. Hence, the result follows from the facts that $\Phi^{(-)}$ is locally Lipschitz (see Proposition \ref{prop:Phi:smooth}) and $\Phi^{(-)}\circ \Psi(\mu, \lambda, p) = (\mu, \lambda, p)$.
\end{proof}

In practice, a value must be chosen for the tuning parameter $\Delta$. The question of how to choose $\Delta$ with respect to $T$ is discussed in Section \ref{sec:simulation}.
In what follows we present some heuristic arguments leading to the convergence stated in
Theorem \ref{theo:1}.

\subsection{Heuristics for the spatio-temporal mean}
\label{subsec_heuristics_first_estimator}

\changes{
Let us remind that $X$ follows the stationary version of a Markov chain having transition probabilities as defined in \eqref{def:transition_prob_1} and \eqref{def:transition_prob_2}, conditionally on the environment $\theta$.
}
Denote $m^N:=(m^N_{1},\ldots, m^N_{N})= \E_{\theta}\left[X_0\right]$. Notice that $m^N$ depends on the realization of $\theta$, but we omit to specify this dependence to keep concise notation. Throughout this section, the notation $a \approx b$ is used to express the fact that $a$ and $b$ are expected to be close to each other as either $N$ or $T$ are large enough.
By ergodicity, we must have as $T\to\infty$, 
%\comJulien{Do we need to precise that temporal covariances must be weak ?}
$$
\frac{1}{T}\sum_{t=1}^TX_{t}\to m^N, \ \ \P_{\theta}-a.s.,
$$
so that we expect for $T$ large enough that 
\begin{equation}
\label{eq_1_heuristics}
\hat{m}=\frac{1}{N}\sum_{i=1}^{N}\frac{1}{T}\sum_{t=1}^TX_{i,t}\approx \frac{1}{N}\sum_{i=1}^{N}m^N_{i}:=\overline{m^N}.
\end{equation}
Hence, to find the limit of $\hat{m}$ one needs to understand the asymptotic behavior of $\overline{m^N}.$ To that end, observe that
by taking first the expectation in both sides of \eqref{def:transition_prob_2} and then using the stationarity, one can check that, for all $1\leq i\leq N$,
\begin{equation}
\label{equation_for_m_theta_i}
m^N_{i}=\mu+(1-\lambda)\left(\frac{1}{N}\sum_{j\in\mathcal{P}_+}\theta_{ij}m^N_{j}+\frac{1}{N}\sum_{j\in\mathcal{P}_-}\theta_{ij}(1-m^N_{j})\right).
\end{equation}
Denoting $A^N(i,j)=N^{-1}\theta_{ij}$ for all $1\leq i\leq N$ and $j\in\mathcal{P}_+$ and $A^N(i,j)=-N^{-1}\theta_{ij}$ for all $1\leq i\leq N$ and $j\in\mathcal{P}_-$, we can then write the above system of equations in matrix form as follows:
$$
m^N=\mu 1_{N}+(1-\lambda)\left(A^Nm^N-L^{N,\bullet -}\right),
$$ 
where $1_{N}$ denotes the $N$-dimensional vector with all entries equal to 1 and $L^{N,\bullet -}=(L^{N,\bullet -}_{1},\ldots, L^{N,\bullet -}_{N})$ is given by $L^{N,\bullet -}_{i}=\sum_{j\in\mathcal{P}_-}A^N(i,j).$ Therefore, we deduce that
$$
\left(I_{N}-(1-\lambda)A^N\right)m^N=\mu 1_{N}-(1-\lambda)L^{N,\bullet -},
$$
where $I_{N}$ denotes the identify matrix of size $N$.
Let us denote $Q^N=\left(I_{N}-(1-\lambda)A^N\right)^{-1}$.  The random matrix $Q^N$ is well-defined whenever $\lambda>0$ (see Appendix \ref{app:annealed:rates}). Granted that $Q^N$ is well-defined, we then have that
\begin{equation}\label{eq:expression:m:theta}
    m^N=\mu Q^N1_{N} - (1-\lambda) Q^NL^{N,\bullet -}.
\end{equation}
Then, remark that $L^{N,\bullet -} = A^N 1_{\mathcal{P}_-}$, where $1_{\mathcal{P}_-}$ is the $N$-dimensional vector with value 1 for the coordinates in $\mathcal{P}_-$ and value 0 otherwise. Using the series expansion $Q^N = \sum_{k=0}^\infty (1-\lambda)^k(A^N)^k$, one can deduce that $Q^N A^N = (Q^N - I_N)/(1-\lambda)$. Hence, $(1-\lambda) Q^NL^{N,\bullet -} = \ell^{N, \bullet -} - 1_{\mathcal{P}_-}$ with $\ell^{N, \bullet -} = Q^N 1_{\mathcal{P}_-}$. Similarly, we denote $\ell^{N} = Q^N 1_{N}$, so that Equation \eqref{eq:expression:m:theta} rewrites as 
\begin{equation*}
    m^N=\mu \ell^N + 1_{\mathcal{P}_-} - \ell^{N, \bullet -} .
\end{equation*}
Now, let us give some heuristics on the asymptotic behavior of $\ell^N$ and $\ell^{N, \bullet -}$.
To that end, denote $L^N=A^N1_{N}$ and observe that by the definition of $\ell^N$, using once more the series expansion of $ Q^N$,
$$
\ell^N=1_{N}+(1-\lambda)L^N+\sum_{k=2}^{\infty}(1-\lambda)^k(A^N)^k1_{N}.
$$
%Now, whenever $N$ is large enough, we expect that 
\changes{
Now, by law of large numbers, we expect that, whenever $N$ is large enough,}
$L^N_{i} = \sum_{j=1}^{N}A^N(i,j)\approx p(r_+-r_-)$ 
for each $1\leq i\leq N$, and by induction, we expect that
$$
\left((A^N)^k1_{N}\right)_i \approx \left[p(r_+-r_-)\right]^{k-1} L^N_{i},
$$
for each $1\leq i\leq N$. Hence, for $N$ sufficiently large
$$
\ell^N \approx 1_{N} + \frac{1}{D(\lambda,p)} (1-\lambda) L^N,
$$
where $D(\lambda,p)$ is defined in \eqref{def_denominator_as_function_of_lambda_and_p}.
Similarly, by definition of $\ell^{N, \bullet -}$, we have
$$
\ell^{N,\bullet -} = 1_{\mathcal{P}_-} + (1-\lambda)L^{N,\bullet -} + \sum_{k=2}^{\infty}(1-\lambda)^k(A^N)^k1_{\mathcal{P}_-}.
$$
\changes{Using the law of large numbers once more, we expect that $L^{N,\bullet -} \approx -pr_- 1_N$ whenever $N$ is large enough, so that}
 %whenever $N$ is large enough, we expect that $L^{N,\bullet -} \approx -pr_- 1_N$,
by the previous induction, we expect that
$$
(A^N)^k1_{\mathcal{P}_-} \approx -pr_-\left[p(r_+-r_-)\right]^{k-2} L^N.
$$ 
Hence,
$$
\ell^{N,\bullet -} \approx 1_{\mathcal{P}_-} + (1-\lambda)L^{N,\bullet -} - \frac{(1-\lambda)^2 pr_-}{D(\lambda,p)} L^N.
$$
All in all, we expect that
\begin{equation}\label{eq_2_heuristics}
m^N \approx \mu \left( 1_N + \frac{1-\lambda}{D(\lambda,p)}L^N \right) - (1-\lambda)L^{N,\bullet -} + \frac{(1-\lambda)^2 pr_-}{D(\lambda,p)} L^N.
\end{equation}
Now, remind that $L^N \approx p(r_+ - r_-) 1_N$ and $L^{N,\bullet -} \approx -pr_- 1_N$, so that 
\begin{equation}\label{eq_3_heuristics}
    \overline{m^N} \approx \mu \left(1 + \frac{(1-\lambda)p(r_+ - r_-)}{D(\lambda,p)}\right) + (1-\lambda) pr_- + \frac{(1-\lambda)^2p^2(r_+ - r_-)r_-}{D(\lambda,p)} = m,
\end{equation}
where $m$ is defined in \eqref{eq:definition:m:v:w}.
Combining \eqref{eq_1_heuristics} and \eqref{eq_3_heuristics}, we expect that for $T$ and $N$ large enough, 
$$
\hat{m}\approx \frac{\mu+(1-\lambda)pr_-}{D(\lambda,p)}=m.
$$
This heuristics is made precise in Equations \eqref{eq:quenched:m} and \eqref{eq:convergence:m}.

\subsection{Heuristics for the spatial variance}
\label{subsection:spatial_var}
First note that we can always write that
$$
\var(Z_{i,T}) - \E\left[\var_{\theta}(Z_{i,T})\right] = \var\left(\E_{\theta}\left[Z_{i,T}\right]\right).
$$
The three terms above are considered from left to right.

First, \changes{since the spatial covariance of the model vanishes sufficiently fast (see Lemma \ref{lem:covariance:Y} for a precise statement)},
the term $\var(Z_{i,T}) $ can be estimated from the data via an empirical variance, that is,
\begin{equation}\label{eq:heuristics:v:1}
 \var(Z_{i,T}) \approx \frac{1}{N}\sum_{i=1}^{N}(Z_{i,T}-\overline{Z}_T)^2.   
\end{equation}

Second, the term $\E\left[\var_{\theta}(Z_{i,T})\right]$ can also be estimated from the data. Indeed, we expect that the temporal covariance of the model vanishes sufficiently fast (see Lemma \ref{lem:covariance:Y} for a precise statement) so that, for $T$ large enough,
\begin{equation}\label{eq:heuristics:v:2}
\E\left[\var_{\theta}(Z_{i,T})\right] \approx T\E\left[\var_{\theta}(X_{i,0})\right] .
\end{equation}
Given the random matrix $\theta$, we know that $X_{i,0}$ is a Bernoulli variable with parameter $m_i^N$ by stationarity of the process, so that
$$
\var_{\theta}(X_{i,0}) = m_i^N - (m_i^N)^2.
$$
Yet, we expect that $T^{-1} Z_{i,T} \approx m_i^N$ which implies that we can estimate $\E\left[\var_{\theta}(X_{i,0})\right]$ and get
\begin{equation}\label{eq:heuristics:v:3}
\frac{1}{N}\sum_{i=1}^{N} \frac{Z_{i,T}}{T}-\left(\frac{Z_{i,T}}{T}\right)^2 = \frac{\overline{Z}_{T}}{T} - \frac{1}{N}\sum_{i=1}^{N} \left(\frac{Z_{i,T}}{T}\right)^2 \approx \E\left[\var_{\theta}(X_{i,0})\right].
\end{equation}

Finally, we may express the limit of the last term. Since $\E_{\theta}\left[Z_{i,T}\right]=Tm^N_{i}$, we can start from the heuristics \eqref{eq_2_heuristics} and look at its variance. It is easy to check that $\var(L^N_{i}) = N^{-1}p(1-p)$ and $\var(L^{N,\bullet -}_{i}) = \cov(L^N_{i}, L^{N,\bullet -}_{i}) = N^{-1}p(1-p) r_-$. Hence, we expect that
\begin{eqnarray*}
    \var(m^N_{i}) &\approx& N^{-1}p(1-p) \left[ \left( \frac{\mu(1-\lambda) + (1-\lambda)^2 p r_-}{D(\lambda,p)} \right)^2 \right. \\
    && \left. \quad - 2 \left( \frac{\mu(1-\lambda) + (1-\lambda)^2 p r_-}{D(\lambda,p)} \right) (1-\lambda) r_- + (1-\lambda)^2 r_-\right]\\
    &=& N^{-1}(1-\lambda)^2 p(1-p) (m^2 - 2 m r_- + r_-) = N^{-1}v,
\end{eqnarray*}
where $m$ and $v$ are defined in \eqref{eq:definition:m:v:w} and $D(\lambda,p)$ is defined in \eqref{def_denominator_as_function_of_lambda_and_p}. Therefore,
\begin{equation}\label{eq:heuristics:v:4}
\frac{N}{T^2}\var\left(\E_{\theta}\left[Z_{i,T}\right]\right)\approx v.
\end{equation}
Combining Equations \eqref{eq:heuristics:v:1}-\eqref{eq:heuristics:v:4}, we have
\begin{equation*}
\hat{v} = \frac{1}{T^2}\sum_{i=1}^{N}(Z_{i,T}-\overline{Z}_T)^2-\frac{N}{T^2}\overline{Z}_T+\frac{1}{T^3}\sum_{i=1}^{N}(Z_{i,T})^2 \approx v.
\end{equation*}
This heuristics is made precise in Equations \eqref{eq:quenched:v} and \eqref{eq:convergence:v}.

\subsection{Heuristics for the temporal variance}
\label{subsection:temporal_var}
Let us remark that 
$$
\frac{\Delta}{N} {\cal W}_{\Delta} = \frac{\Delta}{T} \sum_{k=1}^{T/\Delta}\left[(\overline{Z}_{k\Delta}-\overline{Z}_{(k-1)\Delta}) - \Delta T^{-1}\overline{Z}_{T}\right]^2.
$$ 
Because we expect that the temporal covariance of the model vanishes sufficiently fast, we should have (supposing $\Delta\ll T$) that for $T$ large, 
\begin{equation}
\label{eq_4_heuristics}
\frac{\Delta}{N} {\cal W}_{\Delta} \approx \var_{\theta}\left[\overline{Z}_{\Delta}\right]=\var_{\theta}\left[\overline{U}_{\Delta}\right],
\end{equation}
where ${U}_{t}=({U}_{1,t},\ldots, {U}_{N,t})$ for each $t\geq 1$, with 
\begin{equation}\label{eq:defu}
{U}_{i,t}=Z_{i,t}-\E_{\theta}[Z_{i,t}] \ \text{and} \ \overline{U}_{t}=N^{-1}\sum_{i=1}^{N}{U}_{i,t}.
\end{equation}  
Next, \eqref{def:transition_prob_2} implies that for all $1\leq i\leq N$,
$$
U_{i,t}=M_{i,t}+(1-\lambda)\sum_{j=1}^{N}A^N(i,j)\sum_{s=1}^t\left(X_{j,s-1}-\E_{\theta}\left[X_{j,s-1}\right]\right),
$$
where $M_{i,t}=\sum_{s=1}^t(X_{i,s}-p_{\theta,i}(X_{s-1}))$ is a martingale. 
From the above identity, one can deduce that
\begin{equation}
\label{eq_for_U_t_1}
U_{t}=M_{t}+(1-\lambda)A^NU_{t}+R^N_{t},
\end{equation}
where $M_{t}=(M_{1,t},\ldots, M_{N,t})$ and $R^N_{t}=(1-\lambda)A^N(X_0-X_t)$. 
From our analysis, we have that the term $R^N_t$ is negligible (see Lemma \ref{Lemma_2_finer_control_of_the_first_term_of_exp_Ut_squared}), 
so that it follows from \eqref{eq_for_U_t_1} that
$
U_{t}\approx M_{t}+(1-\lambda)A^NU_{t},
$
or equivalently, 
$$
U_{t}\approx Q^N M_{t}.
$$ 
Hence, introducing the vector $(Q^N)^{\intercal}1_N=c^N=(c^N_{1},\ldots, c^N_{N})$ where $(Q^N)^{\intercal}$ denotes the transpose of $Q^N$,
we deduce that 
$$
\overline{U}_{t}\approx \overline{Q^NM_{t}} = \frac{1}{N}\sum_{i=1}^{N} \sum_{j=1}^N Q^N(i,j) M_{j,t} = \frac{1}{N}\sum_{j=1}^{N}c^N_{j}M_{j,t},
$$
which, in turn, implies that
$$
\var_{\theta}\left[\overline{U}_{t}\right]\approx \var_{\theta}\left[\overline{Q^N M_{t}}\right]=\frac{1}{N^2} \sum_{i=1}^{N}\sum_{j=1}^{N} c_{i}^Nc_{j}^N\E_{\theta}\left[M_{i,t}M_{j,t}\right].
$$
Since the martingales $(M_{i,t})_{t}$ and $(M_{j,t})_t$ are orthogonal for any $i\neq j$, i.e. $\E_{\theta}\left[M_{i,t}M_{j,t}\right]=0$, one obtains that
\begin{equation}
\label{eq_5_heuristics}
\var_{\theta}\left[\overline{U}_{t}\right]\approx \frac{1}{N^2}\sum_{i=1}^{N}(c^N_{i})^2\E_{\theta}\left[\left(M_{i,t}\right)^2\right].
\end{equation}
Then, by means of the quadratic variation and the stationarity of the process, we have
\begin{equation}
\label{eq_6_heuristics}
\E_{\theta}\left[\left(M_{i,t}\right)^2\right]=\E_{\theta}\left[ \sum_{s=1}^t \left(X_{i,s} - p_{\theta,i}(X_{s-1})\right)^2\right] = t\E_{\theta}\left[ \left(X_{i,1} - p_{\theta,i}(X_{0})\right)^2\right].
\end{equation}
%Yet, 
\changes{Observing that, conditionally on $\theta$, $X_{i,1}$ is distributed as a Bernoulli random variable with parameter $m_i^N$, and using the fact that  $p_{\theta,i}(X_{0}) \approx m_i^N,$ since the spatial covariances are weak,
we can deduce (see Lemma \ref{lemma:exp_martingale_differece}) that}
%$p_{\theta,i}(X_{0}) \approx m_i^N$ 
 \begin{equation}\label{eq_7_heuristics}
    \E_{\theta}\left[ \left(X_{i,1} - p_{\theta,i}(X_{0})\right)^2\right] \approx m_i^N - (m_i^N)^2.
\end{equation}
Combining \eqref{eq_4_heuristics} and \eqref{eq_5_heuristics}-\eqref{eq_7_heuristics}, 
we expect that 
$$
{\cal W}_{\Delta} \approx \frac{1}{N}\sum_{i=1}^{N}(c^N_{i})^2(m_i^N - (m_i^N)^2).
$$
Now, as $N$ is large, we have $m_i^N - (m_i^N)^2\approx m(1-m)$ for all $1\leq i\leq N$. Finally, to get the asymptotics of $N^{-1}\sum_{i=1}^{N}(c^N_{i})^2 $, one can follow the same lines as in Section \ref{subsec_heuristics_first_estimator} to check that 
$$
c^N \approx 1_{N} + \frac{(1-\lambda)}{D(\lambda,p)}C^N,
$$
with $C^N=(A^N)^{\intercal}1_N$ and $D(\lambda,p)$ as defined in \eqref{def_denominator_as_function_of_lambda_and_p}. Then, by observing that $C^N_{k}\approx p$ for $k\in \mathcal{P}_+$ and $C^N_{k}\approx -p$ for $k\in \mathcal{P}_-$, one can check that $(c^N_{k})^2 \approx (1+a)^2$ for $k\in \mathcal{P}_+$ and $(c^N_{k})^2 \approx (1-a)^2$ for $k\in \mathcal{P}_-$, with
$$
a = \frac{(1-\lambda)p}{D(\lambda,p)}.
$$
Hence,
\begin{eqnarray*}
    \frac{1}{N}\sum_{k=1}^{N}(c^N_{k})^2 
    &\approx& 1 + a^2+ 2a(r_+-r_-)\\
    &=& \frac{1+(1-\lambda)^2p^2(1-(r_+-r_-)^2)}{(D(\lambda,p))^2}
     = \frac{1+4(1-\lambda)^2p^2r_+r_-}{(D(\lambda,p))^2} .
\end{eqnarray*}
All in all, we find that
$$
{\cal W}_{\Delta} \approx m(1-m)\frac{1+4(1-\lambda)^2p^2r_+r_-}{(D(\lambda,p))^2}=w.
$$
This heuristics is made precise in Equations \eqref{eq:quenched:w} and \eqref{eq:convergence:w}.

By relying on the above heuristics now replacing $2\Delta$ by $\Delta$, we also deduce that ${\cal W}_{2\Delta}$ should converge to $w$ so that
$$
\hat{w}=2{\cal W}_{2\Delta}-{\cal W}_{\Delta} \approx w. 
$$
The reason of estimating $w$ using $\hat{w}$ instead of ${\cal W}_{\Delta}$ (or ${\cal W}_{2\Delta}$) alone is to improve the convergence rate (see the beginning of Appendix \ref{app:proof:quenched:w}).

\paragraph*{Organization of the rest of the paper}
\changes{
    The rest of this paper is devoted to the proof of Theorem \ref{theo:1}. Its proof is an immediate consequence of Propositions \ref{prop:control:quenched} and \ref{prop:control:N:infty} and is presented at the end of Section \ref{sec:key:steps:and:proof}. The proof of Proposition \ref{prop:control:quenched} is given in Appendix \ref{app:quenched:rates}. It relies on a precise study of the decay of correlation of $ X_{i, t } $ and $ X_{j, s } $ (and of products of these) conditionally on a fixed realization of the environment $\theta$. This is done via a backward regeneration scheme presented in Section \ref{sec:3}.
    The proof of Proposition \ref{prop:control:N:infty} is given in Appendix \ref{app:annealed:rates}. It relies on a study of the random matrices $A^N$ and $ Q^N$ and their associated row and column sums. In Section \ref{sec:simulation}, we illustrate the performance of our estimators through simulations. 
    Finally, many other auxiliary results used throughout the paper are proved in appendices.
}

\section{Backward regeneration scheme}\label{sec:3}
Recall formula \eqref{def:transition_prob_2}. First of all, since $ \mu \le \lambda, $ let us introduce $ \beta := \mu/ \lambda $ such that $ 0 \le \beta \le 1.$ Then we may interpret formula \eqref{def:transition_prob_2} in the following way.
At any given time $t$, the $i-$th component first decides to update independently of anything else with probability $ \lambda  . $ If it does so, then it decides to send a signal ($X_{i,t}=1$) with probability $ \beta , $ else it does not send a signal.  Moreover, if the component does not update independently of anything else, then it chooses one of its $N$ neighbors (including $i$ itself) randomly according to the uniform distribution. Suppose it has chosen $j$ as its neighbor. Then there are three possibilities:
\begin{itemize}
    \item if $\theta_{ij}=0$, then $X_{i,t} = 0$,
    \item if $\theta_{ij}=1$ and $j\in\cal{P}_+$, then $X_{i,t}$ copies the value of $X_{j,t-1}$,
    \item if $\theta_{ij}=1$ and $j\in\cal{P}_-$, then $X_{i,t}$ copies $1-X_{j,t-1}$.
\end{itemize}
This interpretation is formalized in Equation \eqref{eq:regeneration:representation} below, where the process is constructed via a backward regeneration scheme.

Models like this are called {\it imitation models} and have been studied in the one-dimensional frame in \citep{deSantis2015onedimensional}.
Since $ \lambda > 0, $ at each time step the probability of making an update which is independent of anything else is strictly positive, and this implies the existence of a unique invariant measure of the system as well as the possibility of perfectly sampling from it, which is formalized in the following section.

\subsection{Backward regeneration representation}
\label{sec:backward:regeneration}
\paragraph*{Notation}
In what follows, we denote by $z=(i,t)$ some space-time coordinate in $\mathcal{Z}=\{1,\dots, N\}\times \mathbb{Z}$ and we write $X_z$ instead of $X_{i, t}$ for all $ z = (i, t) \in \mathcal{Z}$.

To construct the Markov chain $X=\{X_{i, t} , t \in \Z , 1 \le i \le N\} $ via a backward regeneration representation, we introduce the following random variables. To each space-time coordinate $z\in \mathcal{Z}$, we associate a couple of independent random variables $(J_{z}, \xi_{z })  $ taking values in $  \{0,1, \ldots , N\} \times \{ 0, 1 \} $ such that 
$ \P (J_{z} = 0 ) = \lambda,  \P (J_{z} = k) = (1 - \lambda) \frac{1}{N}$   for $ 1 \le k \le N,$ and $ \P (\xi_{z} = 1 ) = 1 -\P (\xi_{z} = 0 ) = \beta $.  Moreover, we suppose that the couples $ ((J_{z}, \xi_{z }), z \in {\cal Z}) $ are independent. By convention, let us define $J_{(0, t )}=0$ for all $t\in \Z$.

Then, for any $z=( i, t)\in \mathcal{Z}$, let us define a backward random walk $I^z=(I^{z}_s)_{ s\in \Z}$ taking values in the state space $\{0,1,\dots,N,\infty\}$,  where $ I^{z}_s = \infty$ for all $s> t$, $I^{z}_{t} = i$ and, as $s$ decreases, $I^z$ follows the space coordinates given by the $J$ variables, that is
\begin{equation}\label{eq:def:backward:RW}
    I^{z}_{s-1} = J_{(I^{z}_{s}, s)}, \text{ for all $s\leq t$.}
\end{equation}
Because of the convention $J_{(0, t )}=0$, the state $0$ is a cemetery for the process $I^z$. Let us denote $\tau^R_z$ the last time that the random walk $I^z$ is not in the cemetery state $0$:
$$
\tau^R_z = \inf\{ s\in \Z :   I_s^{z} > 0\}.
$$
Reaching the cemetery happens after a finite time, almost surely for all the sites.
\begin{proposition}
    \label{prop:cemetery:as}
    For any $\theta$, it holds that $\P_\theta(\forall z\in \mathcal{Z}, \tau^R_z>-\infty)=1$.
\end{proposition}
\begin{proof}
Since $\mathcal{Z}$ is countable, it suffices to check that $\P_\theta(\tau^R_z>-\infty)=1$ for all $z\in \mathcal{Z}$.

For each $k\geq 1$,
$$
\P(\tau^R_z\leq t-k)=\P(I^z_{t-k}\neq 0,\ldots, I^z_{t-1}\neq 0)\leq (1-\lambda)^k, 
$$
where we used the independence between the random variables $J_{\tilde{z}}$ and the fact that $\P(J_{\tilde{z}}\neq 0)=1-\lambda. $ It follows that $\P(\tau^R_z=-\infty)=\lim_{k\to \infty}\P(\tau^R_z\leq t-k)=0$.
\end{proof}
The time $\tau^R_z$ is called a regeneration time for state $z$. Let us then define the set of regenerating sites by
$$ 
{\mathcal R} = \{ z\in {\cal Z} : J_{z} = 0 \}.
$$
\begin{remark} 
    The sites $z=(i,t)$ such that $\theta_{iJ_z}=0$ and $J_{z}\geq 1$ could also be considered as regenerating sites. Since we are interested in results valid for any realization of $\theta$, we did not include those sites in the definition of $\mathcal{R}$.
\end{remark}
Observe that since $I^z_{\tau^R_z-1}=J_{(I^z_{\tau^R_z},\tau^R_z)}=0$, we have that $(I^z_{\tau^R_z},\tau^R_z)\in \mathcal{R}$  is a regenerating site. 

In summary, the random walk $I^z$ equals $+\infty $ for all times $ u > t ,$ that is, before it ``starts to live''. Then, the random walk lives in $\{1,\dots,N\}$ until it reaches $0$ and thus regenerates. It then remains in state $0$ forever.

For all $z=(i,t)\in \mathcal{Z}$, we define $X_z$ as follows: 
\begin{equation}\label{eq:regeneration:representation}
    X_z = \xi_{z}\indiq_{z\in \mathcal{R}}+\theta_{i J_z} \left( X_{J_z,t-1} \mathbf{1}_{J_z\in\cal{P}_+} + (1-X_{J_z,t-1}) \mathbf{1}_{J_z\in\cal{P}_-} \right)\indiq_{z\notin \mathcal{R}}.
\end{equation}
Note that, if $z\in \mathcal{R}$, then $X_z$ is sampled according to the Bernoulli random variable $\xi_z$. When $z\notin \mathcal{R}$,   the random variables $X_z$ are defined recursively.  
On  $\Omega_0 = \left\{ \forall z\in \mathcal{Z}, \tau^R_z>-\infty \right\}$, this recursion ends in finite time for every site and so the process $X$ constructed via \eqref{eq:regeneration:representation} is well-defined. Proposition \ref{prop:cemetery:as} ensures that $\P_{\theta}(\Omega_0)=1$, implying that the process $X$ is well-defined almost surely.
%In the theorem below, we prove that the process $X$ is a stationary Markov chain with transition probability given by \eqref{def:transition_prob_1} and \eqref{def:transition_prob_2}. 
\changes{Since the process $X$ is defined in \eqref{eq:regeneration:representation} via the i.i.d. couples $(J_z,\xi_z)$, one can prove that it is a stationary version of the studied Markov chain.}
\begin{theorem}
\label{thm:perfect_sampling}
Let $X=(X_z)_{z\in\mathcal{Z}}$ be the process defined through \eqref{eq:regeneration:representation}. Then, conditionally on $\theta$, $X$ is a stationary Markov chain with transition probability given by \eqref{def:transition_prob_1} and \eqref{def:transition_prob_2}.     
\end{theorem} 

The proof of Theorem \ref{thm:perfect_sampling} is given in Appendix \ref{app:perfect:sampling}. 

\changes{
\begin{remark}
The representation \eqref{eq:regeneration:representation} is an example of a more general construction that allows to perfectly sample from a process using a {\it clan of ancestors method}. This method works under certain structural conditions on the transition probabilities of the process; that is, supposing that it is possible to decompose the transition probabilities into a mixture of more elementary transitions. These ideas go back at least to \cite{pablo}, followed by \cite{Ferrari} and \cite{Comets}. We also refer to \cite{Fernandez} for a comprehensive introduction to this approach.      
\end{remark}
}

The regeneration representation presented in \eqref{eq:regeneration:representation} above induces a natural (random) partitioning of the state space $\mathcal{Z}$ via the equivalence relation of coalescence defined by
$$ z_1\leftrightsquigarrow z_2  \mbox{ if and only if } \exists s\in \mathbb{Z}, I^{z_1}_s=I^{z_2}_s\notin \{0,\infty\}.$$ 

The next section is devoted to a study of this equivalence relation.

\subsection{Coalescence of the backward random walks}
For any $z_1,z_2\in \mathcal{Z}$, let us define the random time 
\begin{equation}\label{eq:coalescencetime}
\tau^c_{z_1,z_2} = \sup\left\{s\in\Z: I^{z_1}_s=I^{z_2}_s\notin \{0,\infty\}\right\},
\end{equation}
namely the time at which $z_1$ and $z_2$ coalesce (with the convention that $\tau^c_{z_1,z_2} = -\infty$ if they do not coalesce, denoted by $z_1 \not \leftrightsquigarrow z_2$). 
\changes{
The coalescence property is related to the dependence between the coordinates of the process $X$. Loosely speaking, if $z_1$ does not coalesce with $z_2$ then $X_{z_1}$ behaves as a copy of itself which is independent of $X_{z_2}$.
This idea is formalized via a coupling argument in Appendix \ref{app:coalescence:coupling}.
}

In the next result, we provide upper bounds for the probability of events involving the coalescence of two or more backward random walks.  

\begin{proposition}\label{prop:1}
Let $z_k=(i_k,t_k)\in \mathcal{Z}$, for $k=1,\dots, 4$, be four different sites.
There exists a constant $K$ only depending on $\lambda$ such that the following inequalities hold. 
\begin{enumerate}[(i)]
    \item $ \displaystyle \P_\theta \left( \{ z_1 \leftrightsquigarrow z_2 \}\right) 
\leq \frac{(1-\lambda)^{|t_1-t_2|\vee 1}}{1-(1-\lambda)^2}N^{-1}.$
  \smallskip
    \item $ \displaystyle \P_\theta ( \{ z_1 \leftrightsquigarrow z_2 \leftrightsquigarrow z_3  \} ) \le 
    K (1- \lambda)^{(\overline{t} - \underline{t})} N^{- 2},$
    where $\overline{t} = t_1 \vee t_2 \vee t_3$ and $\underline{t} = t_1 \wedge t_2 \wedge t_3$.
    \smallskip
    \item $ \displaystyle \P_\theta ( \{ z_1 \leftrightsquigarrow z_2  \} \cap \{z_3 \leftrightsquigarrow z_4 \} ) \le 
    K \left((1- \lambda)^{|t_1 - t_2| + |t_3- t_4|} N^{- 2} + (1-\lambda)^{(\overline{ t}  - \underline{ t})}N^{-3}\right),$\newline
    where $ \overline{t} = t_1 \vee t_2 \vee t_3 \vee t_4$ and $\underline{t} = t_1 \wedge t_2 \wedge t_3 \wedge t_4.$
    \item If $t_1 \wedge t_2 \geq t_3\vee t_4$, 
    $$\P_\theta ( \{ z_1 \leftrightsquigarrow z_3  \leftrightsquigarrow z_4 \not\leftrightsquigarrow z_2 \} \cap \{\tau^R_{z_2} < t_3 \vee t_4\} ) \le 
    K (1- \lambda)^{t_1 + t_2 - t_3- t_4} N^{- 2}.$$
\end{enumerate}
\end{proposition}

\changes{The proof of Proposition \ref{prop:1} is given in Appendix \ref{app:coalescence_of_2_or_more_BRW}.
}

\changes{
\begin{remark}
    Here are two simple remarks to intuitively understand the upper bounds of Proposition \ref{prop:1}.
    \begin{enumerate}
        \item The exponent of the factor $(1-\lambda)$ is the minimal waiting time before coalescence. For instance, in Item (ii), the first-born site (with birth time $\overline{t}$, remind that time runs backwards) must wait (hence survive) at least $\overline{t} - \underline{t}$ time steps before coalescing with the last-born site. 
        
        In that respect, Item (iv) is a bit different because of the event $\{\tau^R_{z_2} < t_3 \vee t_4\}$ which enforces the site $z_2$ to survive even if it does not coalesce.
        \item Each coalescence event adds a factor $N^{-1}$. For instance, in Items (ii)-(iv), at least two coalescences are needed.
    \end{enumerate}
\end{remark}
}

For all $z\in \mathcal{Z}$, let us denote $Y_{z} = X_{z}- \E_\theta ( X_{z})$ the centered version of $X_{z}$ and observe that $|Y_{z}|\leq 1$ almost surely. 
\changes{
    Since (the absence of) coalescence is related with independence, it is natural to expect that the rates of decay of coalescence probabilities in Proposition \ref{prop:1} can be transferred to rates of decay of covariances for the process.
Indeed, using Proposition \ref{prop:1} and the coupling introduced in Appendix \ref{app:coalescence:coupling}, we obtain several upper bounds on the covariance of the variables $Y_z$'s (Lemma \ref{lem:covariance:Y}) and the products of these variables (Lemma \ref{lem:covariance:produit}). 
This last result is the main result of this section. }
   % Notice that the rates below are closely related with the ones of Proposition \ref{prop:1}.

\begin{lemma}
    \label{lem:covariance:Y}
    Let $z_1 = (i_1,t_1)$ and $z_2=(i_2,t_2)$ two points in $\mathcal{Z}$. There exists a constant $K$ only depending on $\lambda$ such that: if $z_1\neq z_2$,
    $$
    \left| \cov_\theta \left[ X_{z_1}, X_{z_2} \right] \right| = \left| \cov_\theta \left[ Y_{z_1}, Y_{z_2} \right] \right| \leq K(1-\lambda)^{|t_1-t_2|\vee 1}N^{-1}.
    $$
    Otherwise, the quantity above is obviously bounded by $1$.
\end{lemma}
\begin{proof}
    The proof is the same as the one for item 3 of Lemma \ref{lem:covariance:produit} below.
\end{proof}

\begin{lemma}
    \label{lem:covariance:produit}
    Let $z_k = (i_k,t_k) \in \mathcal{Z},$ $k=1,\dots, 4$, and denote $B = \cov_\theta \left[ Y_{z_1}Y_{z_2}\,,\, Y_{z_3}Y_{z_4} \right]$ and $E = \{z_k: k=1,\dots, 4\}$.  There exists a constant $K$ only depending on $\lambda$ such that:
    \begin{enumerate}
       \item If $ \# E = 1,$ then
        $
        |B|\leq 1.
        $
        \item If $ \# E = 2$, $z_1\neq z_2$ and $z_3\neq z_4$, then
        $ 
        |B|\leq 1.
        $
         \item If $ \# E = 2$, $z_1=z_2$ and $z_3=z_4$, then
        $$ 
        |B|\leq K(1-\lambda)^{|t_1 - t_3|}N^{-1}.
        $$
        \item If $ \# E = 3$ and $z_1=z_2$ or $z_3=z_4$, 
        then
        $$ 
        |B|\leq K (1-\lambda)^{|s_3-s_2|+|s_2-s_1|}N^{-2},
        $$
        where $s_1\leq s_2\leq s_3$ is an ordering of $\{t_1,t_2,t_3,t_4\}$.

        \item If $ \# E = 3$, $z_1\neq z_2$ and $z_3\neq z_4$, for instance assume that $z_1=z_3$ (and so $z_2\neq z_4$), then
        $$ 
        |B|\leq K (1-\lambda)^{|t_2-t_4|}N^{-1}.
        $$
        \item If $ \# E = 4 $ and $t_1 \wedge t_2 \geq  t_3\vee t_4$, then
        $$
        |B| \le K (1-\lambda)^{t_1 \vee t_2 - t_3 \wedge t_4}N^{-2} .
        $$
        \item If $ \# E = 4 $, then
$$
    |B| \le K (1-\lambda)^{|s_2 - s_1| + |s_3 - s_4|}N^{-2},
$$
where $s_1\leq \dots \leq s_4$ is an ordering of $\{t_1,\dots,t_4\}$.
       \end{enumerate}
\end{lemma}
\begin{remark}
\label{rmk:case_4}
Observe that we cannot have $z_1=z_2$ and $z_3=z_4$ simultaneously in Item 4 of Lemma \ref{lem:covariance:produit}. In particular, in the case $z_1=z_2$ (resp. $z_3=z_4$), the set of time indices $\{t_1,t_2,t_3, t_4\}$ reduces to the set $\{t_1,t_3,t_4\}$ (resp. $\{t_1,t_2,t_3\}$).       
\end{remark}

\changes{
The proof of Lemma \ref{lem:covariance:produit} is given in Appendix \ref{app:covariance:decay}.
}

\section{Key steps and proof of Theorem \ref{theo:1}}
\label{sec:key:steps:and:proof}

In order to quantify the convergence stated in Theorem \ref{theo:1}, we chose first to study the convergence of our estimators as $T\to \infty$ for a fixed environment and then to study the convergence of these temporal limits with respect to the random environment as $N\to \infty$.

Hereafter, $1_N$ denotes the $N$ dimensional vector full of ones and, for any vector $v\in\R^{N}$, $\overline{v}=N^{-1}\sum_{i=1}^N v_i$ denotes the arithmetic mean of the coordinates of $v$. Remind our three estimators $\hat{m},\hat{v},\hat{w}$ defined in \eqref{eq:def:m:hat:v:hat}-\eqref{eq:def:w:hat}. Their respective limits when $T\to \infty$ while $\theta$ is fixed, as proved below, are:
\begin{equation}\label{eq:def:m:v:w:infty}
    \begin{cases}
        \displaystyle m_\infty^N := \overline{m^N} = \mu \overline{\ell^N} - (1-\lambda)\overline{Q^NL^{N,\bullet-}},\\
        \displaystyle v^N_\infty := \left\| m^N - m_\infty^N1_{N} \right\|_2^2,\\
        \displaystyle w^N_{\infty} := \frac{1}{N}\sum_{i=1}^{N}(c^N_{i})^2\left(m^N_{i}-(m^N_i)^2\right),
    \end{cases}
\end{equation}
where $m^N$ is defined in Section \ref{subsec_heuristics_first_estimator} and the second equality of the first line comes from Equation \eqref{eq:expression:m:theta}. The rates of these temporal convergences are given below.

\begin{proposition}
\label{prop:control:quenched}
There exists a constant $K$ depending only on $ \lambda$ such that for all $T\geq 1$, $N\geq 1$, 
\begin{equation}\label{eq:quenched:m}
    \E_\theta \left[ |\widehat{m} - m_\infty^N|^2 \right] = \var_\theta( \widehat{m} ) \leq K(TN)^{-1},
\end{equation}
\begin{equation}\label{eq:quenched:v} 
    \E_\theta \left[ | \hat v - v^N_\infty  | \right] \le K \left( 1 + \sqrt{v^N_\infty} \right) \left( \frac{N}{T^{2}} +\frac{1}{\sqrt{T}}  +  \frac{\sqrt{N}}{T}  \right). 
    \end{equation}
\begin{equation}\label{eq:quenched:w}
    \E_\theta \left[ \Big|\widehat{w} - w^N_{\infty}\Big| \right]\leq K\left( \frac{\Delta}{T} + \frac{1}{N} + \frac{(1-\lambda)^{\Delta}}  {\Delta} \right),
\end{equation}
\end{proposition}
\changes{The proof of Proposition \ref{prop:control:quenched} is given in Appendix \ref{app:quenched:rates}.}

The last ingredient needed to prove Theorem \ref{theo:1} are the rates of convergence of the temporal limits $m_\infty^N, v_\infty^N, w_\infty^N$ as $N\to \infty$.

\begin{proposition}\label{prop:control:N:infty}
    Assume that $1\geq \lambda>0$. There exists a constant $K>0$ which depends on $\lambda$ such that for all $N\geq 1$, it holds that 
    \begin{gather}
    \E\left[\left|m_\infty^N - m \right|^2\right] \leq \frac K {N^2}, \label{eq:convergence:m}\\
    \E\left[\left| v^N_\infty - v \right| \right] \leq \frac K {\sqrt{N}}, \label{eq:convergence:v}\\
    \E\left[\left|w^N_{\infty} - w\right|^2\right]\leq \frac{K}{N^2}
    \label{eq:convergence:w},
    \end{gather}
    where $m,v,w$ are defined in \eqref{eq:definition:m:v:w}
\end{proposition}
\changes{
The proof of Proposition \ref{prop:control:N:infty} is given in Appendix \ref{app:annealed:rates}.
}

\medskip
Finally, the proof of Theorem \ref{theo:1} is merely the combination of Propositions \ref{prop:control:quenched} and \ref{prop:control:N:infty}.

\begin{proof}[Proof of Theorem \ref{theo:1}]
Let us prove Inequality \eqref{conver_spatial_variance}. The proof of the two other inequalities are similar and even simpler (the only subtlety is to remark that $\Delta/T< \sqrt{\Delta/T}$ for Inequality \eqref{conver_temporal_variance}).

First, observe that \eqref{conver_spatial_variance} holds automatically when $N^{1/2}/T>1$. Suppose now that $N^{1/2}/T^2\leq 1$. By the triangle inequality, we have
\begin{equation*}
    \E\left[ |\hat{v} - v| \right] \leq \E\left[ |\hat{v} - v_\infty^N| \right] + \E\left[ |v_\infty^N - v| \right].
\end{equation*}
By the sub-additivity of $ x \mapsto \sqrt{x},$
\begin{multline}\label{eq:annealedcontrol}
    \E \left[ \sqrt{v_\infty^N} \right] \le \left( \E \left[ \left\| m^N - \overline{m^N}1_{N} \right\|^2_2 \right] \right)^{1/2}\le 
\\  \left(\E\left[\left| \left\| m^N - \overline{m^N}1_{N} \right\|_2^2 - v \right| \right]\right)^{1/2} + \sqrt{v} \le K.
\end{multline}
Thus, integrating \eqref{eq:quenched:v} against $ \P$ and using the upperbound above, we obtain 
$$ \E \left[ | \hat v - v^N_\infty  | \right] \le K  \left( \frac{N}{T^{2}} +\frac{1}{\sqrt{T}}  +  \frac{\sqrt{N}}{T}  \right). $$
In turn, thanks to \eqref{eq:convergence:v}, we have
\begin{equation*}
    \E \left[ | \hat v - v  | \right] \le K \left( \frac{N}{T^{2}} + \frac{1}{\sqrt{T}} +  \frac{\sqrt{N}}{T} + \frac{1}{\sqrt{N} } \right).
\end{equation*}
Combining this with Markov inequality and then using that $N/T^2\leq N^{1/2}/T$ and  $T^{-1/2}=N^{1/4}T^{-1/2}N^{-1/4}\leq N^{1/2}T^{-1}+N^{-1/2}$,
we obtain that for any $\epsilon\in (0,1)$, 
$$
\P(|\hat{v}-v|\geq \varepsilon)\leq \frac{K}{\varepsilon}\left(\frac{N^{1/2}}{T}+N^{-1/2}\right),
$$ 
proving \eqref{conver_spatial_variance}.
\end{proof}

\section{Simulation study}
\label{sec:simulation}

Before showing some numerical results, the practical implementation of the method (especially the inversion of the function $\Psi$ defined by \eqref{eq:definition:Psi}) and the choice of $\Delta$ are thoroughly discussed.

\subsection{Practical implementation}
\label{sec:practical:implementation}
The computation of the three estimators $\hat{m},\hat{v},\hat{w}$ is obvious. However, the choice and computation of the two inverse functions $\Phi^{(+)}$ and $\Phi^{(-)}$, defined in Appendix \ref{app:inversion}, need further explanations.

First, one has to compute $d^{(a)}$ defined in Equation \eqref{eq:d(a)}. To avoid numerical errors, we replaced the condition $\kappa(m,w) = 4r_+r_-$ by the condition $|\kappa - 4r_+r_-| < 10^{-4}$. Moreover, we replaced $(4r_+r_-)^2-4r_+r_-+\kappa(m,w)$ by $\max\{0, (4r_+r_-)^2-4r_+r_-+\kappa(m,w)\}$ to ensure that the square-root stays in $\mathbb{R}_+$. Remark that in those two cases, we have $d^{(+)}=d^{(-)}$. In particular, there is no issue regarding the choice of $\Phi^{(+)}$ or $\Phi^{(-)}$.
Out of those two cases, we use $\Phi^{(-)}$ whenever 
\begin{equation*}
    r_+ \geq 1/2 \quad\text{or}\quad \kappa(m,w) \geq 4r_+r_- \quad\text{or}\quad d^{(+)}(m,w) > 2r_-.
\end{equation*}
The first two conditions are chosen according to Proposition \ref{prop:inversion}. The last one is supported by the fact that, if $r_+ < 1/2$ and $\kappa(m,w) < 4r_+r_-$ then $D(\lambda,p) < 2r_-$ (see Proposition \ref{prop:preliminary:D:Psi}) and $d^{(-)}(m,w) < d^{(+)}(m,w)$.
Finally, if all the conditions above are not satisfied then the choice between $a=+$ and $a=-$ is arbitrarily made.

Second, one has to compute $\phi_1^{(a)}$. On the one hand, to avoid numerical errors and the choice between $a=+$ and $a=-$, we replaced the condition $r_+ = 1/2$ by $|r_+ - r_-| < 10^{-3}$. On the other hand, when $\phi_1^{(a)}$ is applied to the estimators $\hat{m}$ and $\hat{w}$, one may end up with a negative value. Yet, $\phi_1^{(a)}$ should be non negative. Hence, we replaced $\phi_1^{(a)}(m,w)$ by $|\phi_1^{(a)}(m,w)|$ (for instance in Equation \eqref{eq:definition:Phi}). 

Finally, when $\Phi^{(+)}$ or $\Phi^{(-)}$ are applied to the estimators $(\hat{m},\hat{v},\hat{w})$, one may end up with estimators $(\hat{\mu},\hat{\lambda},\hat{p})$ which do not belong to the admissible set $\Lambda$ defined above Equation \eqref{def_denominator_as_function_of_lambda_and_p}. In that case we chose to clip the values. For instance, clip the value of $\lambda$ to $0$ if negative, or to $1$ if strictly larger than $1$.

The numeric experiments were made using Julia programming language and a package should be available for the next revision of the paper.

\subsection{Choice of the tuning parameter}
From the bound obtained in Equation \eqref{infor_statemente_of_the_convergece2}, the optimal choice of $\Delta$ is $\Delta = - \log(T)/ (2\log(1-\lambda))$. In that case, 
\begin{equation*}
    \frac{(1-\lambda)^{\Delta}}{\Delta}+\sqrt{\frac{\Delta}{T}} = \frac{-2\log(1-\lambda)}{T^{1/2}\log(T)} + \frac{\sqrt{\log(T)}}{T^{1/2}}.
\end{equation*}
Since the value of $\lambda$ is not known, the choice $\Delta = \log(T)$ is made by default. However, the choice $\Delta=1$ seems to give the best results among several tests and across several choices of parameters (see Figure \ref{fig:choice:Delta}). This is the reason why most of the plots of the next section are made with $\Delta=1$.

\subsection{Numerical results}
\label{sec:numerical:results}

Let us first give the framework of our simulations. If not specified otherwise, the following values are used:
\begin{equation*}
    N = 500, \quad 
    r_+ = .5, \quad 
    \beta = .5, \quad 
    \lambda = .5, \quad 
    p = .5, \quad 
    \Delta = 1, \quad 
    N_{\rm simu} = 1000.
\end{equation*}
Furthermore, the fractions of excitatory and inhibitory components are chosen as $r_+^N = \lceil r_+N\rceil$ and $r_-^N=N-|\mathcal{P}_+|$.

In all the plots, the performance of one (or several) of our estimators is displayed. Let us denote $\hat{\vartheta}$ one of these estimators and $\vartheta$ its corresponding true parameter value. The solid lines correspond to the median of $|\hat{\vartheta} - \vartheta|$ computed over $N_{\rm simu}$ simulations and plotted as a function of the time horizon $T$. Furthermore, we add marks (circles or horizontal bars) on the right end of the plot. These marks correspond to the median of $|\hat{\vartheta}_{\infty} - \vartheta|$ computed over $N_{\rm simu}$ simulations, where $\hat{\vartheta}_{\infty}$ is the theoretical limit of $\hat{\vartheta}$ as $T\to \infty$ while the environment $\theta$ is fixed (remind Equation \eqref{eq:def:m:v:w:infty}).
Of course, these quantities are unknown in practice but we are able to compute them here since the parameters used for the simulation are known. Finally, the corresponding limit estimators of $\mu$, $\lambda$ and $p$ are defined by:
\begin{equation*}
    \left( \hat{\mu}_{\infty}, \hat{\lambda}_{\infty}, \hat{p}_{\infty} \right) = \Phi^{(a)}\left( \hat{m}_{\infty}, \hat{v}_{\infty}, \hat{w}_{\infty} \right),
\end{equation*}
where the choice of $a\in \{+,-\}$ is made according to Section \ref{sec:practical:implementation}.

\begin{figure}[t]
    \caption{\label{fig:choice:Delta}
    Absolute estimation error for the six estimators and their theoretical limits. The $y$-axis is in log-scale. Each line or mark correspond to a median computed over $N_{\rm simu}=100$ simulations. The panels correspond to the choices $\Delta = \log(T)$ and $\Delta = 1$ from left to right.
    }
    \includegraphics[width=.49\textwidth]{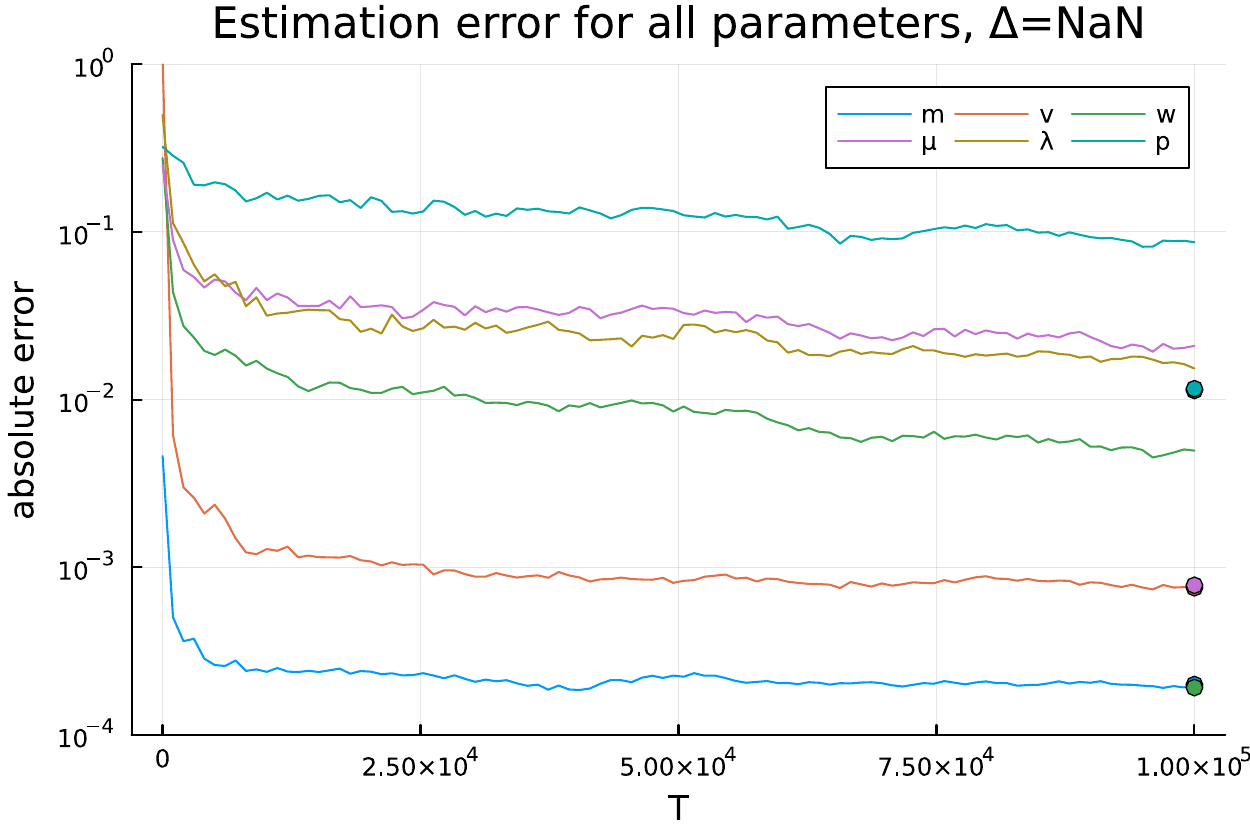}
    \includegraphics[width=.49\textwidth]{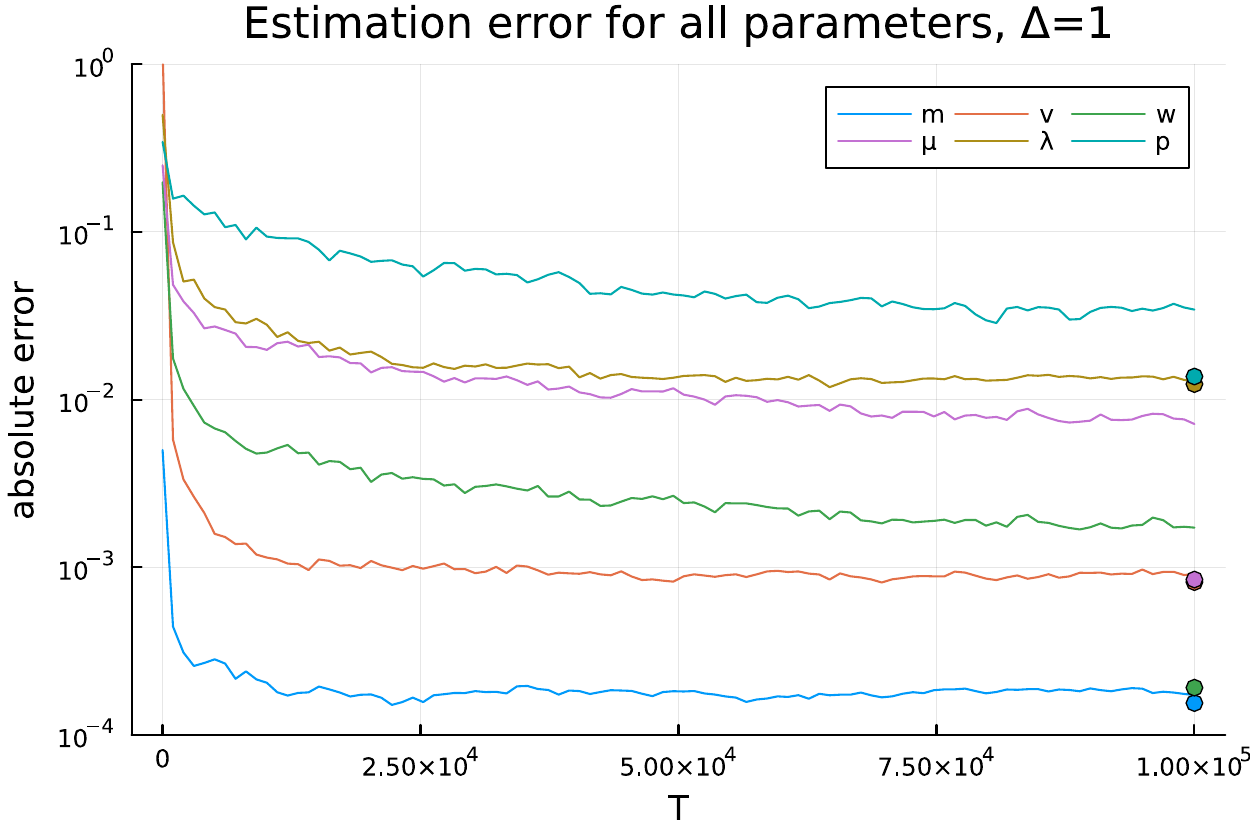}
\end{figure}

\begin{figure}[t]
    \caption{\label{fig:error:p}
    Absolute estimation error of $\hat{p}$ and its theoretical limit. Each line or mark correspond to a median computed over $N_{\rm simu}=1000$ simulations. The panels correspond to different choices of varying parameter (the non-varying parameters are chosen according to the default values given in Section \ref{sec:numerical:results}). The values of the varying parameter are given by the color legends.
    }
    \includegraphics[width=.49\textwidth]{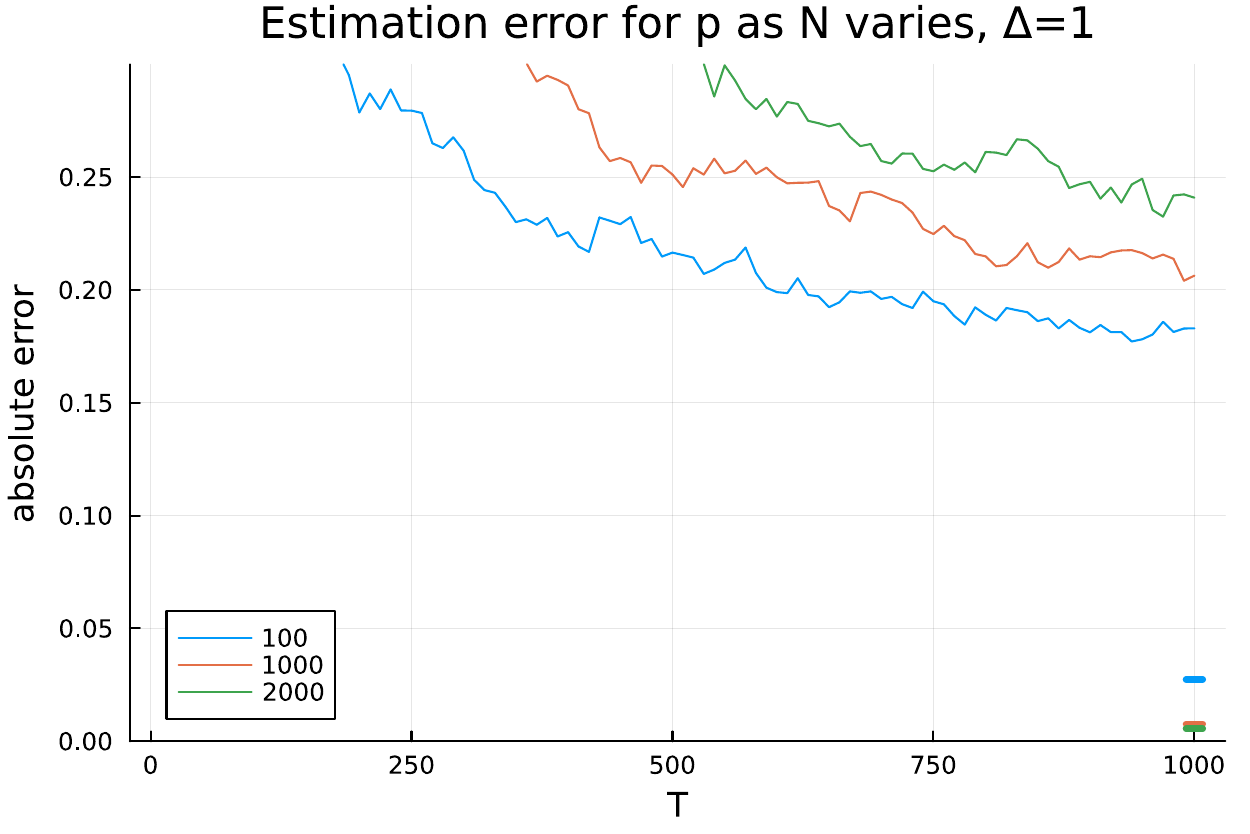}
    \includegraphics[width=.49\textwidth]{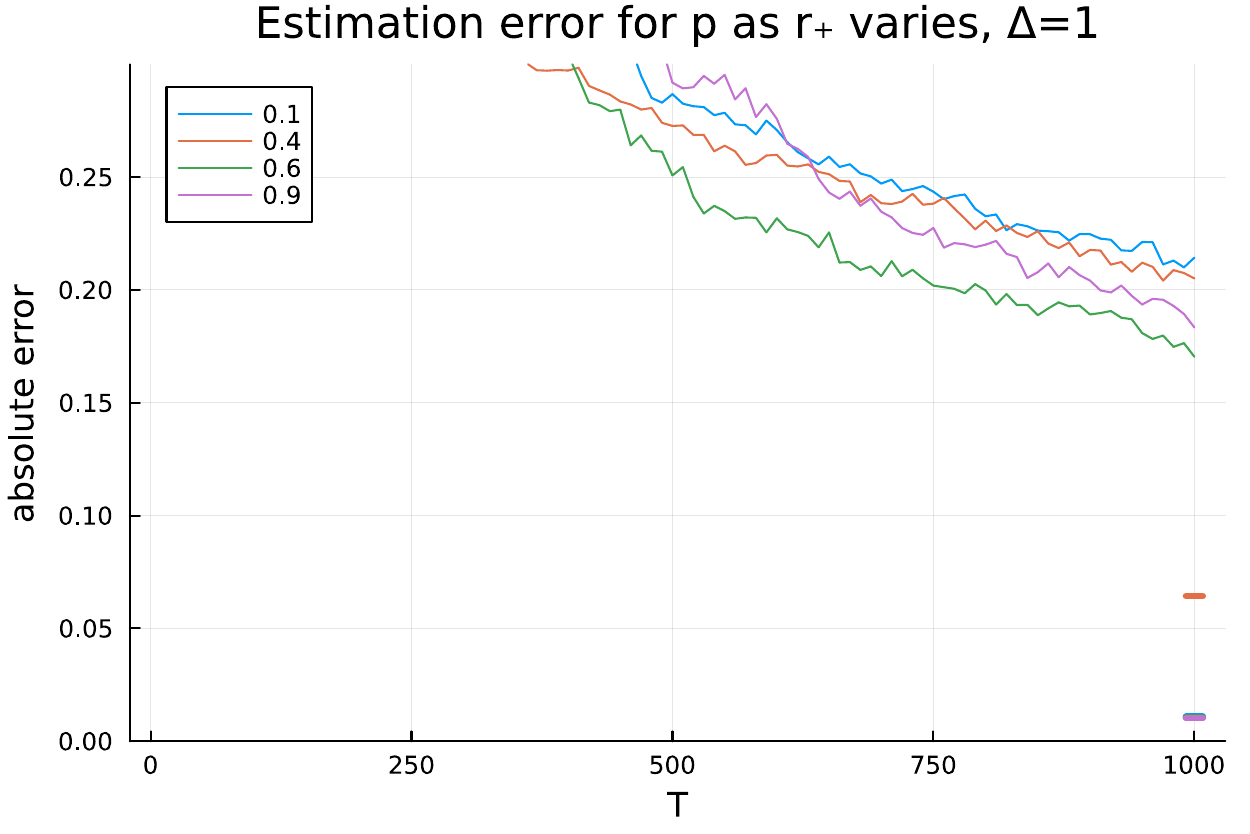}
    \includegraphics[width=.49\textwidth]{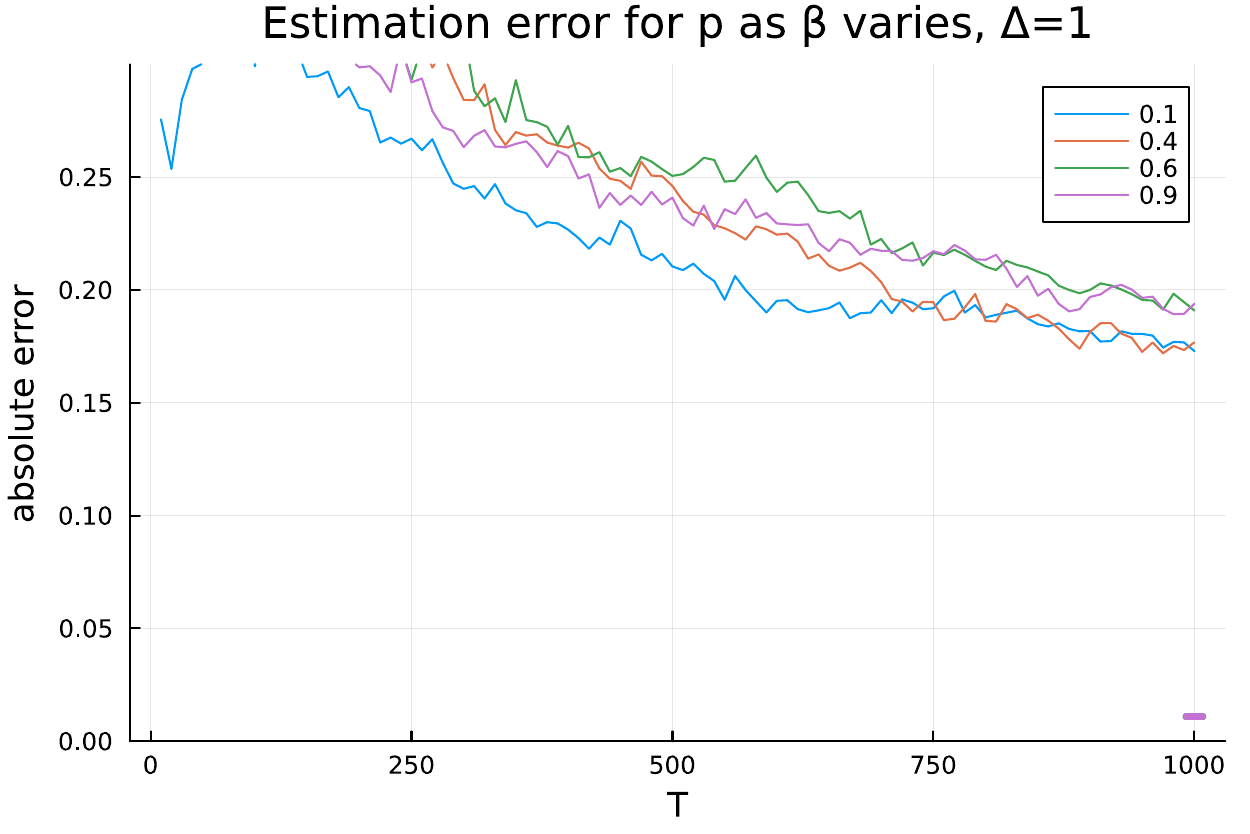}
    \includegraphics[width=.49\textwidth]{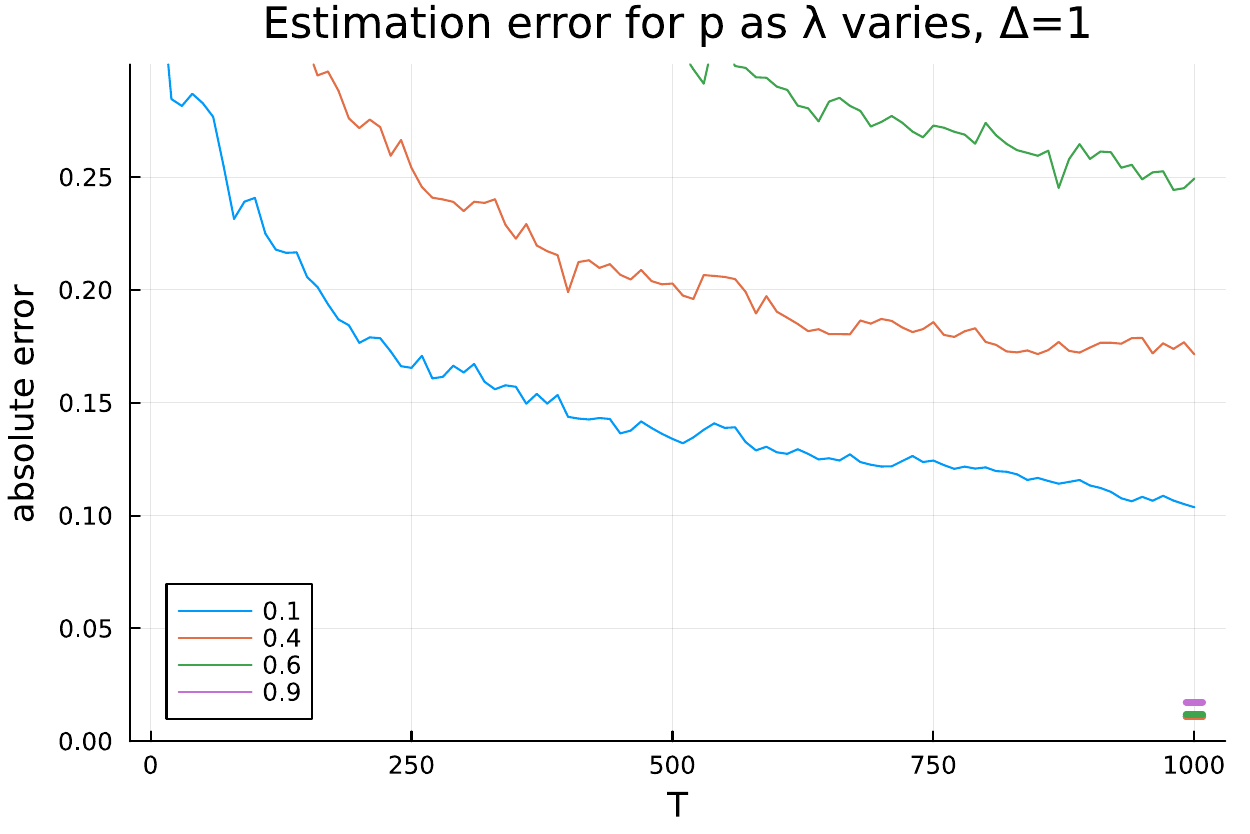}
    \includegraphics[width=.49\textwidth]{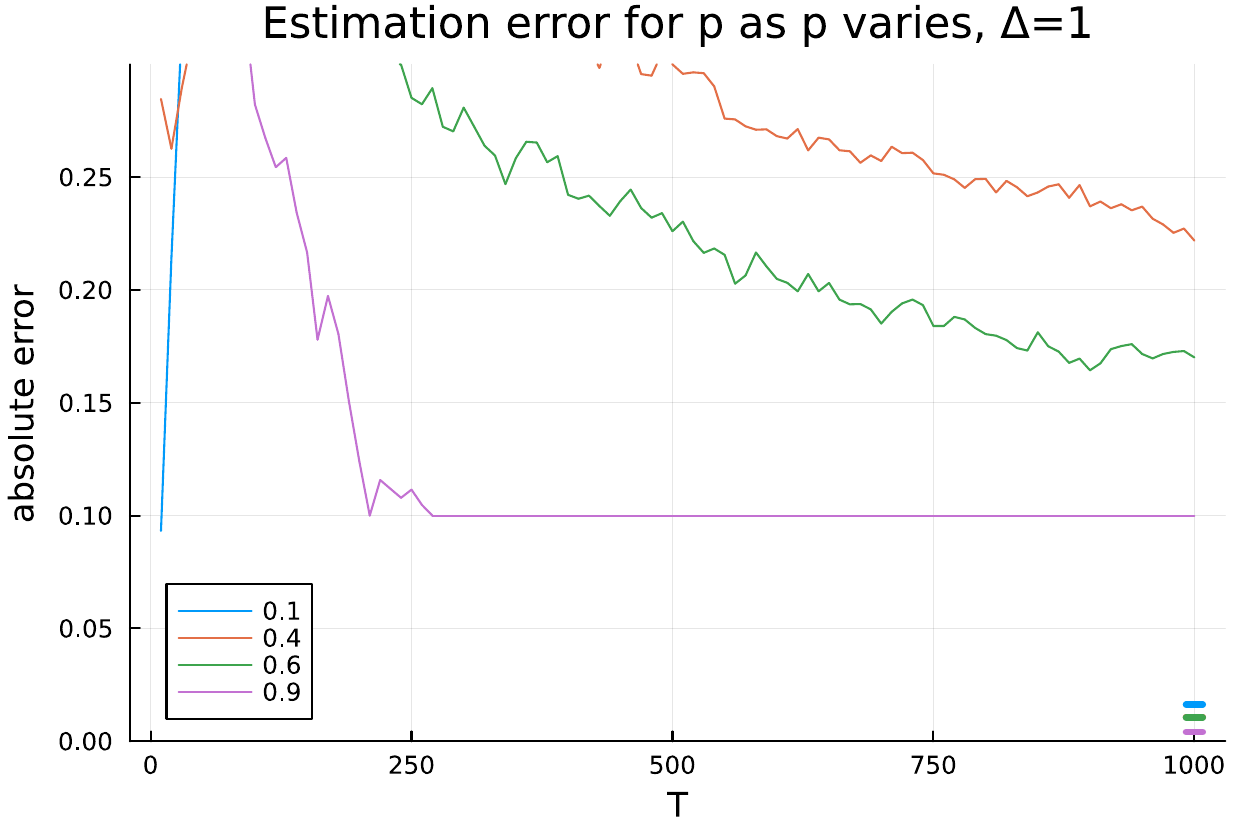}
\end{figure}

Figure \ref{fig:choice:Delta} compares the performance of all the estimators. The two panels correspond to two choices of $\Delta$ that will be discussed below. On the two panels, it is clear that the convergence of $\hat{m}$ is faster than all the others (which is expected from our analysis). Furthermore, for large $T$, $\hat{m}$, $\hat{v}$ and $\hat{\lambda}$ are really close to their theoretical limits $\hat{m}_\infty$, $\hat{v}_\infty$ and $\hat{\lambda}_\infty$ respectively. 
Note that this is not the case for $\hat{w}$, $\hat{\mu}$ and $\hat{p}$. 
In particular, it seems that $\hat{p}$ has the slowest convergence rate.

Figure \ref{fig:choice:Delta} gives also a comparison of the performance of our estimators with respect to the tuning parameter $\Delta$. Either $\Delta$ is chosen as a function of $T$ by $\Delta = \log(T)$ or $\Delta$ is fixed to the value $1$. Obviously, the estimators $\hat{m}, \hat{v}$ for finite $T$ and all the limit estimators $\hat{\vartheta}_{\infty}$ do not depend on $\Delta$ so the differences between the two plots are due to randomness only. However, it seems that the convergence of $\hat{w}$ is significantly faster when $\Delta=1$. In turn, it implies the same improvement for the triplet $(\hat{\mu},\hat{\lambda},\hat{p})$.

Figure \ref{fig:error:p} gives an overview of the performance of the estimator $\hat{p}$ as one of the parameters ($N$, $r_+$, $\beta$, $\lambda$ or $p$) varies. The choice to vary $\beta$ instead of $\mu = \lambda \beta$ is made because the set of admissible values of $\beta$ is independent of $\lambda$ (which is not the case for $\mu$). Remark that when $p=0.5$, which is the case for most of the panels in Figure \ref{fig:error:p}, the value $0.25$ correspond to the median of the absolute error of the most naive estimator: estimate $p$ by a random uniform value in $[0,1]$. Overall, the estimation error of the limit $\hat{p}_{\infty}$ is usually of the order of $0.01$ and that of the estimator $\hat{p}$ computed for $T=1000$ is approximately 20 times larger. 
As $N$ increases, the performance of $\hat{p}$ for $T$ fixed deteriorates. This phenomenon is encoded in the factor $\sqrt{N}/T$ appearing in Theorem \ref{theo:1} for example.
Also, the marks corresponding to the estimation errors of the theoretical limit $\hat{p}_\infty$ for different values of $N$ are ordered as expected: the estimation error goes to $0$ as $N$ goes to infinity.
As $r_+$ varies, the performance of $\hat{p}$ seems not to vary too much. However, it is important to note that the case $r_+ = 0.4$ corresponds to a case where the choice of $a\in\{+,-\}$ is not obvious and so it is arbitrary (see Section \ref{sec:practical:implementation}). Half of the time, the wrong $a$ is chosen. In particular, it is the reason why the mark corresponding to the limit estimator is around $0.06$ instead of $0.01$ for all the other cases. 
As $\beta$ varies, the performances of $\hat{p}$ and $\hat{p}_\infty$ seem not to vary too much. This is quite expected.
As $\lambda$ increases towards $1$, the performances of $\hat{p}$ and $\hat{p}_\infty$ decreases. This is expected since $\lambda \to 1$ implies that the strength of interactions in the system goes to $0$. In particular, it is more and more difficult to estimate $p$ which is closely related to the interactions. Nevertheless, contrarily to our upper-bounds which goes to infinity as $\lambda\to 0$ (see the remark below Theorem \ref{theo:1}), the performance is good for $\lambda=0.1$.
Finally, the performance of $\hat{p}$ as $p$ varies is more complex to analyze. For instance, for $p=0.9$, our method gives $\hat{p}=1$ most of the time whence the eventually constant violet curve at $y=0.1$. In the case $p=0.1$, the performance of $\hat{p}$ seems to be poor despite the fact that its limit $\hat{p}_\infty$ (blue mark) achieves good results.
This phenomenon is also expected because as $p\to 0$ the interactions in the system vanishes.

\begin{acks}[Acknowledgments]
This research has been conducted while J.C. was in Statify team at Centre Inria de l'Université Grenoble Alpes.
\end{acks}

\begin{funding}
G.O. was supported by the Serrapilheira Institute (grant number Serra – 2211-42049), FAPERJ (grants E-26/201.397/2021 and E-26/204.532/2024) and CNPq (grants 303166/2022-3). J.C. and E.L. were supported by ANR-19-CE40-0024 (CHAllenges in MAthematical NEuroscience).
\end{funding}

\bibliographystyle{imsart-nameyear}
\bibliography{Ref}

  \newpage

\begin{appendix}

\section{Proof of Theorem \ref{thm:perfect_sampling}}
\label{app:perfect:sampling}

The stationarity of the process $X$ follows from the fact that its construction is invariant under time shift. Hence, to conclude the proof, it remains to show that, conditionally on each realization of the random environment $\theta$, $X$ is Markovian and has the transition probabilities given by \eqref{def:transition_prob_1} and \eqref{def:transition_prob_2}. 
\changes{
The Markovianity of the process and the conditional independence of its coordinates follow from its representation as a function of an i.i.d. sequence. The compatibility of the transition probabilities follows from simple computations.
}

\paragraph*{Markovianity of the process}
Observe that \eqref{eq:regeneration:representation} implies that for each $i\in\{1,\ldots,N\}$, there exists a deterministic function $f_{\theta,i}:\{0,1\}^{N}\times \{0,1\}\times\{0,1,\ldots,N\}\to \{0,1\}$ such that for all $t\in\mathbb{Z}$,   
$$
X_{i,t}=f_{\theta,i}(X_{t-1},\xi_{i,t},J_{i,t}).
$$
Then defining the function   
$$
f_{\theta}(x,\xi,J)=(f_{\theta,1}(x,\xi_1,J_1),\ldots, f_{\theta,N}(x,\xi_N,J_N)), 
$$
where $x\in\{0,1\}^N$, $\xi=(\xi_1,\ldots,\xi_N)\in\{0,1\}^N$ and $J=(J_1,\ldots,J_N)\in\{0,1,\ldots, N\}^N$, we can write for all $t\in\mathbb{Z}$,   
$$
X_{t}=f_{\theta}(X_{t-1},\xi_{t},J_{t}),
$$
where $\xi_{t}=(\xi_{1,t},\ldots, \xi_{N,t})$ and $J_{t}=(J_{1,t},\ldots, J_{N,t})$. This representation together with the fact the sequence $(\xi_t,J_t)_{t\in\mathbb{Z}}$ is i.i.d ensures that $(X_{t})_{t\in\mathbb Z}$ is a Markov chain. Moreover, since the pairs $(\xi_{1,t},J_{1,t}),\ldots, (\xi_{N,t},J_{N,t})$ are independent, the random variables $X_{i,t}=f_{\theta,i}(X_{t-1},\xi_{i,t},J_{i,t})$, $i\in\{1,\ldots, N\}$, are clearly conditionally independent given $X_{t-1}$. 

\paragraph*{Transition probability}
It remains to show that $\P_{\theta}(X_{i,t}=1|X_{t-1}=x)=p_{\theta,i}(x)$ for all $x\in\{0,1\}^N$, where $p_{\theta,i}(x)$ is given by \eqref{def:transition_prob_2}.   
To see that, write $z=(i,t)$ and observe that 
\eqref{eq:regeneration:representation} allows us to write,
\begin{multline*}
\P_{\theta}(X_{z}=1|X_{t-1}=x)=\P_{\theta}(J_{z}=0,\xi_{z}=1)+\sum_{j\in\mathcal{P}_+}\theta_{ij}x_j\P_{\theta}(J_{z}=j|X_{t-1}=x)\\
+\sum_{j\in\mathcal{P}_-}\theta_{ij}(1-x_j)\P_{\theta}(J_{z}=j|X_{t-1}=x).
\end{multline*}
Using that $X_{t-1}$ is measurable with respect to $\sigma(\xi_{k,s},J_{k,s},s\leq t-1,k\in\{1,\ldots, N\})$ and the fact that both
$J_z$ and $\xi_z$ are independent of all random variables $J_{k,s}$, $\xi_{k,s}$ for $s\leq t-1$ and $k\in\{1,\ldots,N\}$, we have that 
$$
\P_{\theta}(J_{z}=0,\xi_{z}=1|X_{t-1}=x)=\P_{\theta}(J_{z}=0,\xi_{z}=1),
$$
and also that for each $j\in\{1,\ldots,N\},$ 
$$
\P_{\theta}(J_{z}=j|X_{t-1}=x)=\P_{\theta}(J_{z}=j)=\frac{(1-\lambda)}{N}.
$$
As a consequence, it follows that
\begin{multline*}
\P_{\theta}(X_{}=1|X_{t-1}=x)=\P_{\theta}(J_{z}=0,\xi_{z}=1)+\frac{(1-\lambda)}{N}\sum_{j\in\mathcal{P}_+}\theta_{ij}x_j
+\sum_{j\in\mathcal{P}_-}\theta_{ij}(1-x_j).
\end{multline*}
Since the random variables $J_{z}$ and $\xi_{z}$ are also independent, and  $\P_{\theta}(J_{z}=0)=\lambda,$ $\P_{\theta}(\xi_{z}=1)=\beta$ and $ \mu = \lambda \beta $, it follows that $\P_{\theta}(X_{i,t}=1|X_{t-1}=x)=p_{\theta,i}(x)$, concluding the proof.

\section{Coalescence couplings}
\label{app:coalescence:coupling}

\changes{
This section is devoted to the proof of two lemmas which relate independence and coalescence. It is structured as follows. First, the two lemmas are stated. Next, the construction of the coupling is detailed. After that, some properties of the coupling are provided. Finally, the proofs of the two lemmas are given.
}

\subsection{Two main lemmas}
\changes{
    Here are two lemmas that relate coalescence with independence via couplings. The tilde and hat versions of $X$ defined below respect some additional spatial independence properties. The objective of the two lemmas is to prove that these versions can be coupled in such a way that they are equal except on an event related with coalescence (which in turn we expect to be of small probability).
}

\begin{lemma}[All independent coupling]
    \label{lem:independent:construction:ksites}
    Let $k\geq 2$ and consider $k$ different points $z_1,\dots,z_k$ in $\mathcal{Z}$. There exist random variables $\tilde{X}_{z_1},\dots,\tilde{X}_{z_k}$ satisfying the following properties: for each $i\in \{1,\ldots, k\}$,
    \begin{enumerate}[(i)]
      \item $\tilde{X}_{z_i}\equaldistrib X_{z_i}$ (i.e., $\tilde{X}_{z_i}$ and $X_{z_i}$ have the same distribution);
        \item $\tilde{X}_{z_i}$ is independent of $(X_{z_j},\tilde{X}_{z_j})_{j\in \{1,\ldots, k\}\setminus\{i\}};$
        \item  $\{\tilde{X}_{z_i}\neq X_{z_i}\}\subset\cup_{j\in \{1,\ldots, k\}\setminus\{i\}}\{z_i\leftrightsquigarrow z_j\}.$
    \end{enumerate}
\end{lemma}

\begin{lemma}[Blockwise independence] 
\label{lem:final:independent}
Let $z_1,\dots,z_4$ be four different points in $\mathcal{Z}$. There exist random variables $\tilde{X}_{z_1},\dots,\tilde{X}_{z_4}$ and $\hat{X}_{z_1}, \hat{X}_{z_2}$ satisfying the following properties:
    \begin{enumerate}[(i)]
        \item $(\hat{X}_{z_1}, \hat{X}_{z_2}) \equaldistrib (X_{z_1}, X_{z_2})$;
        \item $(\hat{X}_{z_1}, \hat{X}_{z_2})$ is independent of $(X_{z_3}, X_{z_4}, \tilde{X}_{z_3}, \tilde{X}_{z_4})$;
        \item $\{\hat{X}_{z_1} \neq  X_{z_1} \} \subset \cup_{k=3}^4 \{ z_1  \leftrightsquigarrow z_k\} \cup \left( \{\tau^R_{z_1} < t_3 \vee t_4\} \cap ( \cup_{k=3}^4 \{ z_2  \leftrightsquigarrow z_k\})  \right)$, 
        \item $\{\hat{X}_{z_2} \neq  X_{z_2} \} \subset \cup_{k=3}^4 \{ z_2  \leftrightsquigarrow z_k\} \cup \left( \{\tau^R_{z_2} < t_3 \vee t_4\} \cap ( \cup_{k=3}^4 \{ z_1  \leftrightsquigarrow z_k\})  \right)$.
        \item $\hat{X}_{z_1}$ is independent of $\tilde{X}_{z_2}$ and $\hat{X}_{z_2}$ is independent of $\tilde{X}_{z_1}$.
       \end{enumerate}
Moreover, for each $i\in \{1,\ldots, 4\}$,
\begin{enumerate}
     \item[(iv)] $\tilde{X}_{z_i} \equaldistrib X_{z_i}$;
        \item[(v)] $\tilde{X}_{z_i}$ is independent of $(\tilde{X}_{z_j}, X_{z_j}, I^{z_j})_{j\in\{1,\ldots, 4\} \setminus\{ i\}}$;
        \item[(vi)] $\{\tilde{X}_{z_i}\neq X_{z_i}\}\subset \cup_{j\in\{1,\ldots, 4\} \setminus\{ i\}}\{z_i\leftrightsquigarrow z_j\}.$
\end{enumerate}
\end{lemma}

\begin{remark}
    Let us describe with some words Item (iii) above: if the hat version $\hat{X}_{z_1}$ is different from the standard version $X_{z_1}$, then we know that 
    \begin{itemize}
        \item $z_1$ coalesces with $z_3$ or $z_4$, or
        \item $z_2$ coalesces with $z_3$ or $z_4$ and the (backward) regeneration time of $z_1$ is less than the starting times of $z_3$ or $z_4$.
    \end{itemize}
\end{remark}

\subsection{Construction of the couplings}
\label{sec:coupling:construction}

Let us denote $(J,\xi)$ the sequence $(J_z,\xi_z)_{z}$ defined in the backward regeneration representation of the process $X$. Since we are working for a fixed matrix $\theta$ in this section, we omit its dependence in what follows. Due to the backward regeneration representation there exist measurable functions $f$, $g$, $r$ and $\tau$ (which depend on $\theta$) such that, for all $z\in \mathcal{Z}$,
\begin{equation*}
    I^z = (I^{z}_s)_{s\in \Z} = f(z,J) \text{ and } X_z = g(I^z,\xi_{r(I^z)}),
\end{equation*}
where $r(I^z)=(I^z_{\tau(I^z)},\tau(I^z))$ is the regenerating site associated with $z$ and $\tau(I^z) = \inf\{ s\in \Z :   I^z_s > 0\}$ is the last time that the random walk $I^{z}$ is not in the state $0$. In what follows, we denote $I=(I^{z})_{z\in\cal Z}$.

Now, consider $5$ independent sequences $$(J^{(1)},\xi^{(1)}),\ldots, (J^{(5)},\xi^{(5)}),$$ all distributed as the sequence $(J,\xi)$, and which are furthermore independent of $(J,\xi)$. By convenience of notation, let us denote $(J^{(0)},\xi^{(0)}) := (J,\xi)$. In other words,  the layer (0) corresponds to the original process, and we add five i.i.d. layers on top of that. We then define the backward random walks $I^{(i)}=(I^{(i),z})_{z\in\mathcal{Z}}$ by $I^{(i),z} = f(z,J^{(i)})$, for each $0\leq i\leq 5$. Clearly, the collections of random walks $I^{(0)}, \ldots, I^{(5)}$ are independent. We will need the layers $ (0) - (4) $ to construct the tilde versions, and the layers $ (0) $ and $(5) $ to construct the hat versions.

\paragraph*{Construction of the tilde versions}
For $i\in\{1,\ldots, 4\}$, let us denote 
\begin{equation*}
    {\tau}^c_i = \sup\left\{ \tau^c_{z_i,z_j}: j\in\{1,\ldots, 4\}\setminus\{i\} \right\},
\end{equation*}
the first time at which $I^{z_i}$ coalesces with (at least) one of the $3$ other random walks (observe that ${\tau}^c_i=-\infty$ if $I^{z_i}$ does not coalesce with any of the other $3$ random walks), and denote $z^c_i = (I^{z_i}_{{\tau}^c_i}, {\tau}^c_i)$ the site of coalescence in case ${\tau}^c_i>-\infty$. 
Then, for all $1\leq i\leq 4$ and $s\in \Z$, we define
\begin{equation}\label{eq:tilde:I}
    \tilde{I}^{z_i}_s = I^{(0),z_i}_s \mathbf{1}_{s >  \tau^c_i} + I^{(i),z^c_i}_s \mathbf{1}_{s \leq \tau^c_i}.
\end{equation}
In words, the random walk $\tilde{I}^{z_i}$ follows the random walk $I^{z_i}$ until it coalesces with one of the random walks $I^{z_j}$, $j\in \{1,\ldots, 4\}\setminus\{i\}$. After this time, the random walk $\tilde{I}^{z_i}$ follows the independent random walk $I^{(i),z_i^c}$ associated with the site of coalescence $z_i^c$. With the notation $\tilde{L}^{z_i}_s = i\cdot 1_{s\leq {\tau}^c_i} \in \{0,\dots, 4\}$, Equation \eqref{eq:tilde:I} rewrites as
\begin{equation*}
    \tilde{I}^{z_i}_s = I^{(0),z_i}_s \mathbf{1}_{\tilde{L}^{z_i}_s = 0} + I^{(i),z^c_i}_s \mathbf{1}_{\tilde{L}^{z_i}_s = i},
\end{equation*}
and some computations give the following easy recursion formula:
\begin{equation}\label{eq:recursion:I:tilde}
    \tilde{I}^{z_i}_{s-1} = J_{(\tilde{I}^{z_i}_{s}, s)}^{(\tilde{L}^{z_i}_s)}, \text{ for all $s\leq t_i$},
\end{equation}
which can be compared with Equation \eqref{eq:def:backward:RW}.
Moreover, remark that the layer processes $\tilde{L}$ are measurable functions of the state processes $\tilde{I}$ since
\begin{equation*}
    \tau^c_{z_i,z_j} = \sup\left\{s\in\Z: I^{z_i}_s=I^{z_j}_s\notin \{0,\infty\}\right\} = \sup\left\{s\in\Z: \tilde{I}^{z_i}_s=\tilde{I}^{z_j}_s\notin \{0,\infty\}\right\}.
\end{equation*}

\paragraph*{Construction of the hat versions}
Let us denote, for $i=1,2$,
\begin{equation*}
    \hat{\tau}_{i} = \sup\left\{ \tau^c_{z_i,z_j}: j\in \{3,4\}\right\},
\end{equation*}
and denote $\hat{z}^c_i = (I^{z_i}_{\hat{\tau}_i}, \hat{\tau}_i)$ the corresponding site of coalescence in case $\hat{\tau}_i>-\infty$. For instance, $\hat{\tau}_{1}$ corresponds to the first time at which $I^{z_1}$ coalesces with either $I^{z_3}$ or $I^{z_4}$. 
Then, for $i=1,2$, we define the auxiliary process, for all $s\in \Z$,
\begin{equation*}
    \hat{I}^{(5),z_i}_s = I^{z_i}_s \mathbf{1}_{s> \hat{\tau}_i} + I^{(5),\hat{z}^c_i}_s \mathbf{1}_{s \leq \hat{\tau}_i}.
\end{equation*}
The construction is not yet done since these two random walks may be missing some coalescence in order to mimic the joint distribution of $(X_{z_1}, X_{z_2}).$
To circumvent this problem, we define
\begin{equation*}
    \hat{\tau}_{1,2} = \sup\left\{s\in\Z: \hat{I}^{(5),z_1}_s=\hat{I}^{(5),z_2}_s\notin \{0,\infty\}\right\},
\end{equation*}
and 
$$ \hat{I}^{z_i} = \hat{I}^{(5),z_{i}}, \mbox{ if } \hat{\tau}_{1,2} \le \hat \tau_i ,$$
while we put, for all $s\in \Z$,
\begin{equation}\label{eq:I:hat}
\hat{I}^{z_i}_s =   I^{z_i}_s \mathbf{1}_{s>  \hat{\tau}_{1,2}} + \hat{I}^{(5),z_{3-i}}_s \mathbf{1}_{ s\leq   \hat{\tau}_{1,2}}, \mbox{ if } \hat{\tau}_{1,2} >  \hat \tau_i .
\end{equation}
For instance, if $ I^{z_1} $ first coalesces, say, with $ I^{z_3}, $ then it switches to its hat-version $ \hat{I}^{z_1}$  and remains stuck to this version forever. However, if $ I^{z_2} $ hits $ \hat{I}^{z_1} $ before hitting  $ I^{z_3}$ or $ I^{z_4}, $ then it coalesces with $\hat{I}^{z_1}$ and remains stuck to it forever. 
In words, we do not modify the random walk with the first (the largest in time) coalescence time $\hat{\tau}_i$, but we possibly modify the other one in between the times $\hat{\tau}_1$ and $\hat{\tau}_2$. Like Equation \eqref{eq:recursion:I:tilde}, we have the following recursion:
\begin{equation}\label{eq:recursion:I:hat}
    \hat{I}^{z_i}_{s-1} = J_{(\hat{I}^{z_i}_{s}, s)}^{(\hat{L}^{z_i}_s)}, \text{ for all $s\leq t_i$},
\end{equation}
where 
\begin{equation}\label{eq:hatL}
 \hat{L}^{z_1}_s =5 \cdot 1_{s \le \hat{\tau}_1 } +  5 \cdot 1_{\hat{\tau}_1 <  s  \le  \hat{\tau}_{1,2}}   1_{   \hat{\tau}_{1,2}  \leq \hat{\tau}_2 } 
\end{equation}
with a similar definition for $\hat{L}^{z_2}_s.$ In other words, in any case, process $1$ is in layer $5$ starting from time $\hat{\tau}_1$, that is $\hat{L}^{z_1}_s = 5$ for $s\leq \hat{\tau}_1$. But, if process $1$ coalesces with process $2$ (at time $\hat{\tau}_{1,2}$) and if process $2$ is already evolving on layer $5$ at that time, then process $1$ switches to layer $5$ at that coalescence time $\hat{\tau}_{1,2}$.

Once again, the layer processes $\hat{L}^{z_1}$ and $\hat{L}^{z_2}$ are measurable functions of the state processes $(\hat{I}^{z_1}, \hat{I}^{z_2}, I^{z_3}, I^{z_4})$. 

\subsection{Properties of the couplings}

First, let us formalize the fact that the processes $I$ do not depend on the whole sequences $J^{(0)},\dots, J^{(5)}$ but merely on a small subset of those. Let $\gamma\in \{0,1,\dots, N, \infty\}^\Z$ and $\ell \in \{0,\dots, 5\}^\Z$ be some generic trajectories of the processes $I$ and $L$. 
Let us then define $J_\gamma^\ell = (J_{(\gamma_s,s)}^{(\ell_s)})_{s\in \Z}\in \{0,1,\dots, N\}^\mathbb{Z}$ and, for all $z=(i,t)\in \mathcal{Z}$, $f^|_z : \{0,1,\dots, N\}^\mathbb{Z} \to \{0,1,\dots, N, \infty\}^\mathbb{Z}$ given by
\begin{equation*}
    (f^|_z(J_\gamma^\ell))_s = \begin{cases}
        \infty & \text{if $s>t$}\\
        i & \text{if $s=t$}\\
        J_{(\gamma_{s+1},s+1)}^{(\ell_{s+1})} & \text{else}.
    \end{cases}
\end{equation*}

Finally, let us remark that:
\begin{enumerate}
    \item the definition of $I^{z_i}$ implies that $\{I^{z_i} = \gamma\} = \{f^|_{z_i}(J_\gamma^{\ell^\emptyset}) = \gamma\}$ where $\ell^\emptyset_s=0$ for all $s\in \Z$;
    \item Equation \eqref{eq:recursion:I:tilde} and the measurability of the layer processes imply that for all $\tilde{\gamma}^1,\dots, \tilde{\gamma}^4 \in  \{0,1,\dots, N, \infty\}^\Z$, there exist $\tilde{\ell}^1,\dots, \tilde{\ell}^4 \in \{0,\dots, 4\}^\Z$ such that
    \begin{equation*}
        \cap_{i=1}^4 \{\tilde{I}^{z_i} = \tilde{\gamma}^i\} = \cap_{i=1}^4 \{\tilde{I}^{z_i} = \tilde{\gamma}^i, \tilde{L}^{z_i} = \tilde{\ell}^i\} 
        = \cap_{i=1}^4 \{f^|_{z_i}(J_{\tilde{\gamma}^i}^{\tilde{\ell}^i}) = \tilde{\gamma}^i\};
    \end{equation*}
    \item Equation \eqref{eq:recursion:I:hat} and the measurability of the layer processes imply that for all choices $\hat{\gamma}^1,\hat{\gamma}^2, \gamma^3,\gamma^4,\tilde{\gamma}^3,\tilde{\gamma}^4 \in  \{0,1,\dots, N, \infty\}^\Z$, there exist $\hat{\ell}^1,\hat{\ell}^2 \in \{ 0, 5 \}^\Z $ and $\tilde{\ell}^3,\tilde{\ell}^4 \in \{0,3, 4 \}^\Z$ such that
    \begin{eqnarray*}
        A 
        &:=& \left( \cap_{i=1}^2 \{\hat{I}^{z_i} = \hat{\gamma}^i\} \right) \cap \left( \cap_{j=3}^4 \{ I^{z_j} = \gamma^j, \tilde{I}^{z_j} = \tilde{\gamma}^j \} \right)\\
        &=& \left( \cap_{i=1}^2 \{\hat{I}^{z_i} = \hat{\gamma}^i, \hat{L}^{z_i} = \hat{\ell}^{i} \} \right) \cap  \left( \cap_{j=3}^4 \{ I^{z_j} = \gamma^j, \tilde{I}^{z_j} = \tilde{\gamma}^j, \tilde{L}^{z_j} = \tilde{\ell}^j \} \right)\\
        &=& \left( \cap_{i=1}^2 \{f^|_{z_i}(J_{\hat{\gamma}^i}^{\hat{\ell}^i}) = \hat{\gamma}^i\} \right) \cap \left( \cap_{j=3}^4 \{f^|_{z_j}(J_{\gamma^j}^{\ell^\emptyset}) = \gamma^j, f^|_{z_j}(J_{\tilde{\gamma}^j}^{\tilde{\ell}^j}) = \tilde{\gamma}^j\} \right).\\
    \end{eqnarray*}
\end{enumerate}

First, let us summarize the coupling properties of the processes $I$.
\begin{proposition}\label{prop:I:coupling:property}
    The processes $I$, $\tilde{I}$ and $\hat{I}$ satisfy the following coupling properties: 
    \begin{enumerate}[(i)]
        \item for all $i=1,\dots, 4$, $\{\tilde{I}^{z_i}\neq I^{z_i}\} = \cup_{j\in\{1,\ldots, 4\} \setminus\{ i\}}\{z_i\leftrightsquigarrow z_j\}$,
        \item we have that $$\{\hat{I}^{z_1} \neq  I^{z_1} \} \subset \bigcup_{k=3}^4 \{ z_1  \leftrightsquigarrow z_k\} \cup \left( \bigcup_{k=3}^4 \{ z_2 \leftrightsquigarrow z_k\} \cap \{ \tau^R_{z_1} <  t_3 \vee t_4 \} \right)  $$ 
        and 
        $$\{\hat{I}^{z_2} \neq  I^{z_2} \} \subset \bigcup_{k=3}^4 \{ z_2  \leftrightsquigarrow z_k\} \cup \left( \bigcup_{k=3}^4 \{ z_1 \leftrightsquigarrow z_k\} \cap \{ \tau^R_{z_2} <  t_3 \vee t_4 \} \right)  .$$ 
    \end{enumerate}
\end{proposition}
\begin{proof}
    This follows directly from the construction of the processes and the remark that 
    \begin{equation*}
        \left\{ \hat{\tau}_i < \hat{\tau}_{1,2} \leq \hat{\tau}_{3-i}  \right\} \subset 
        \bigcup_{k=3}^4 \{ z_{3-i} \leftrightsquigarrow z_k\} \cap \{ \tau^R_{z_i} <  t_3 \vee t_4 \}.
    \end{equation*}   
\end{proof}

We are now in position to prove the following independence properties of the processes $I$.
\begin{proposition}\label{prop:I:independence:property}
    The processes $I$, $\tilde{I}$ and $\hat{I}$ satisfy the following independence properties:
    \begin{enumerate}[(i)]
        \item for all $i$, $\tilde{I}^{z_i}$ is independent of $(I^{z_j}, \tilde{I}^{z_j})_{j\neq i}$;
        \item $(\hat{I}^{z_1}, \hat{I}^{z_2})$ is independent of $(I^{z_j}, \tilde{I}^{z_j})_{j=3,4}$.
        \item $\hat{I}^{z_1} $ is independent of $\tilde {I}^{z_2}, $ and $\hat{I}^{z_2} $ is independent of $\tilde {I}^{z_1}.$ 
    \end{enumerate}
\end{proposition}
\begin{proof}\ 
\paragraph*{Proof of (i)}
Assume without loss of generality that $i=1$. Let $\tilde{\gamma}^1,\dots, \tilde{\gamma}^4, \gamma^2,\dots, \gamma^4 \in  \{0,1,\dots, 2N, \infty\}^\Z,$ and let $ \tilde \ell_1, \ldots, \tilde \ell_4 $ be the associated layer processes, which are uniquely defined being measurable functions of $\tilde{\gamma}^1,\dots, \tilde{\gamma}^4, \gamma^2,\dots, \gamma^4.$ By the point 2 above, we have
\begin{eqnarray*}
    A
    &=& \left( \cap_{i=1,\dots,4} \{\tilde{I}^{z_1} = \tilde{\gamma}^1\} \right) \cap \left( \cap_{j=2,\dots,4}  \{I^{z_j} = \gamma^j\} \right)\\
    &=& \{f^|_{z_1}(J_{\tilde{\gamma}^1}^{\tilde{\ell}^1}) = \tilde{\gamma}^1\} \cap \left( \cap_{j=2}^4 \{f^|_{z_j}(J_{\tilde{\gamma}^j}^{\tilde{\ell}^j}) = \tilde{\gamma}^j, f^|_{z_j}(J_{\gamma^j}^{\ell^\emptyset}) = \gamma^j\} \right)
\end{eqnarray*}
By construction, if $\P(A)$ is non null, then the set $\{ (\tilde{\gamma}^1_s, \tilde{\ell}^1_s, s), s\in \Z, \tilde{\gamma}^1_s \notin \{0,\infty\} \}$ is disjoint from the sets $\{ (\tilde{\gamma}^j_s, \tilde{\ell}^j_s, s), s\in \Z, \tilde{\gamma}^j_s \notin \{0,\infty\} \}$ and $\{ (\gamma^j_s, 0, s), s\in \Z, \gamma^j_s \notin \{0,\infty\} \}$ for $j>1$ so that the two events in the final expression of $A$ are independent, and Item (i) follows.

\paragraph*{Proof of (ii)}
Let $\hat{\gamma}^1,\hat{\gamma}^2,\gamma^3,\gamma^4,\tilde{\gamma}^3,\tilde{\gamma}^4 \in  \{0,1,\dots, N, \infty\}^\Z,$ and let $ \hat l^1, \hat l^2 , \tilde l^3, \tilde l^4 $ be the associated layer processes. By the point 3 above, we have
\begin{eqnarray*}
        A 
        &=& \left( \cap_{i=1}^2 \{\hat{I}^{z_i} = \hat{\gamma}^i\} \right) \cap \left( \cap_{j=3}^4 \{ I^{z_j} = \gamma^j, \tilde{I}^{z_j} = \tilde{\gamma}^j \} \right)\\
        &=& \left( \cap_{i=1}^2 \{f^|_{z_i}(J_{\hat{\gamma}^i}^{\hat{\ell}^i}) = \hat{\gamma}^i\} \right) \cap \left( \cap_{j=3}^4 \{f^|_{z_j}(J_{\gamma^j}^{\ell^\emptyset}) = \gamma^j, f^|_{z_j}(J_{\tilde{\gamma}^j}^{\tilde{\ell}^j}) = \tilde{\gamma}^j\} \right).\\
\end{eqnarray*}
By construction, if $\P(A)$ is non null, then the sets $\{ (\hat{\gamma}^1_s, \hat{\ell}^1_s, s), s\in \Z, \hat{\gamma}^1_s \notin \{0,\infty\} \}$ and $\{ (\hat{\gamma}^2_s, \hat{\ell}^2_s, s), s\in \Z, \hat{\gamma}^2_s \notin \{0,\infty\} \}$ are disjoint from the sets $\{ (\tilde{\gamma}^j_s, \tilde{\ell}^j_s, s), s\in \Z, \tilde{\gamma}^j_s \notin \{0,\infty\} \}$ and $\{ (\gamma^j_s, 0, s), s\in \Z, \gamma^j_s \notin \{0,\infty\} \}$ for $j=3,4$ so that the two events in the final expression of $A$ are independent, and Item (ii) follows.

\paragraph*{Proof of (iii)}
Let $\hat{\gamma}^1 , \tilde{\gamma}^2 \in  \{0,1,\dots, N, \infty\}^\Z,$ and let $ \hat \ell^1 \in \{0,5\}^\Z$ and $ \tilde \ell^2 \in \{0,2\}^\Z $ be two fixed layer processes that are compatible with the event $ \{\hat{I}^{z_1} = \hat{\gamma}^1\} \cap \{ \tilde{I}^{z_2} = \tilde{\gamma}^2\}$. By compatible, we mean for example that supposing that $ \hat{\gamma}^1 $ and $ \tilde{\gamma}^2 $ meet at some time $ s \le t_1 \vee t_2,$ it is not possible to have $ \hat l^1_u = \tilde l^2_u$ for $ u \le s.$

Then, analogously to the proof of point (ii) above, the sets $\{ (\hat{\gamma}^1_s, \hat{\ell}^1_s, s), s\in \Z, \hat \gamma_s^1 \notin \{0, \infty \} \} $  and $\{ (\tilde{\gamma}^2_s, \tilde{\ell}^2_s, s), s\in \Z, \tilde \gamma_s^2 \notin \{0, \infty \} \} $ are disjoint sets implying the independence.  
\end{proof}

\begin{proposition}\label{prop:I:distribution:property}
    For all $i$, $\tilde{I}^{z_i}$ has the same distribution as $I^{z_i}$ and $(\hat{I}^{z_1}, \hat{I}^{z_2})$ has the same distribution as $(I^{z_1}, I^{z_2})$.
\end{proposition}

\begin{proof}
    Remind that the layers $J^{(0)},\dots,J^{(5)}$ are i.i.d. so that it is easy to prove that $\tilde{I}^{z_i}$ is a backward random walk with the same transitions as $I^{z_i}$ (compare Equation \eqref{eq:recursion:I:tilde} with Equation \eqref{eq:def:backward:RW}) which in turn implies that they share the same distribution.

    \medskip
    We now turn to the second part of the proof. For the marginals, the same argument as above applies (thanks to Equation \eqref{eq:recursion:I:hat}). Hence, $\hat{I}^{z_i}$ has the same distribution as $I^{z_i}$. Then, we show that the process $ (\hat{I}^{z_1}_{s}, \hat{I}^{z_2}_{s})_s$ is a two-dimensional Markov chain (backwards in time) which has the same transitions as $(I^{z_1}_{s}, I^{z_2}_{s})_s.$ Since both chains start from the same initial conditions, this implies the desired result. 
    
    Since transitions between $t_1 \vee t_2$ and $t_1 \wedge t_2$ only concern one of the two processes, evolving according to the right marginals, we do only need to 
    consider transitions $ s \to s-1$ for $ s \le t_1 \wedge t_2.$ Fix $i_1 , i_2, j_1 , j_2  \in \{0, \ldots, N \} . $ Let us first discuss the case $ i_1 \neq i_2.$ 
    
    Then for any $ l_1, l_2 \in \{0, 5 \}, $ writing 
    $$ A := \{ \hat{I}^{z_1}_s = i_1, \hat{I}^{z_2}_s = i_2, \hat{L}^{z_1}_s = l_1,\hat{L}^{z_2}_s = l_2\}, $$
    we have that 
    $$
    \P ( \hat{I}^{z_1}_{s-1} = j_1, \hat{I}^{z_2}_{s-1} = j_2|  A )  
    = \P ( J^{(l_1)}_{(i_1,s)} = j_1, J^{(l_2)}_{(i_2,s)} = j_2 |  A ).
    $$
    By construction, $J^{(l_1)}_{(i_1,s)} $ and $J^{(l_2)}_{(i_2,s)}$ are independent (since $i_1 \neq i_2),$ they have the same joint distribution as $(J^{(0)}_{(i_1,s)} , J^{(0)}_{(i_2,s)}),$ and they are independent of $A $ (since $A$ depends only on decisions strictly after time $s$). Thus 
    $$\P ( \hat{I}^{z_1}_{s-1} = j_1, \hat{I}^{z_2}_{s-1} = j_2|  A )= \P ( J^{(0)}_{(i_1,s)} = j_1, J^{(0)}_{(i_2,s)} =  j_2 ) . $$
    Summing over all possible choices of $l_1$ and $l_2,$ this implies that 
    \begin{equation}\label{eq:transition:hatI:1}
        \P ( \hat{I}^{z_1}_{s-1} = j_1, \hat{I}^{z_2}_{s-1} = j_2 | \hat{I}^{z_1}_s = i_1, \hat{I}^{z_2}_s = i_2) = \P ( J^{(0)}_{(i_1,s)} = j_1, J^{(0)}_{(i_2,s)} =  j_2 ).
    \end{equation}

    We now discuss the case $ i_1 = i_2.$ Then necessarily $l_1 = l_2 $ and thus, by our coalescence construction, $ j_1 = j_2.$ With the same notation for the set $A$ as above and using the same independence argument, it still holds that  
    \begin{equation*} 
    \P ( \hat{I}^{z_1}_{s-1} = j_1=  \hat{I}^{z_2}_{s-1} |  A )  
    = \P ( J^{(l_1)}_{(i_1,s)} = j_1|  A ) =\P ( J^{(l_1)}_{(i_1,s)} = j_1 ) = \P ( J^{(0 )}_{(i_1,s)} = j_1 ) . 
    \end{equation*}
    Summing over all possible values of $l_1 $ implies that 
    \begin{equation}\label{eq:transition:hatI:2}
        \P ( \hat{I}^{z_1}_{s-1} = j_1=  \hat{I}^{z_2}_{s-1} |\hat{I}^{z_1}_s = i_1=  \hat{I}^{z_2}_s ) =  \P ( J^{(0 )}_{(i_1,s)} = j_1 ).
    \end{equation}
    The transition probabilities given by Equations \eqref{eq:transition:hatI:1} and \eqref{eq:transition:hatI:2} correspond exactly to the transition probabilities of $(I^{z_1}_s, I^{z_2}_s)$ - see Equation \eqref{eq:def:backward:RW}.
\end{proof}

\subsection{Proof of the two lemmas}

\begin{proof}[Proof of Lemma \ref{lem:independent:construction:ksites}]
    Four tilde processes are constructed above so we write the proof for $k=4$. Nevertheless, the proof can easily be generalized to any $k\geq 2$.

    Let $z_1,\dots, z_4$ be four different sites in $\mathcal{Z}$. For all $i=1,\dots,4$, let $\tilde{I}^{z_i}$ be the backward random walks defined by \eqref{eq:tilde:I}, and $\tilde{L}^{z_i}$ be the associated layer process. We denote $\tilde{L}^{z_i}_{-\infty} = \lim_{s\to -\infty} \tilde{L}^{z_i}_s \in \{0,i\}$ the terminal layer. Then, remind the functions $g$ and $r$ defined in the beginning of Appendix \ref{sec:coupling:construction} and define, for all $i=1,\dots, 4$, 
    \begin{equation*}
        \tilde{X}_{z_i} = g\left( \tilde{I}^{z_i}, \xi^{(0)}_{r(\tilde{I}^{z_i})} \right) \mathbf{1}_{\tilde{L}^{z_i}_{-\infty} = 0} + g\left( \tilde{I}^{z_i}, \xi^{(i)}_{r(\tilde{I}^{z_i})} \right) \mathbf{1}_{\tilde{L}^{z_i}_{-\infty} = i}.
    \end{equation*}
    In comparison, remind that $X_{z_i} = g(I^{z_i}, \xi^{(0)}_{r(I^{z_i})})$. Using the fact that the $\xi^{(j)}$'s are i.i.d. with Propositions \ref{prop:I:independence:property} and \ref{prop:I:distribution:property}, one deduces Items (i) and (ii). Finally, Item (iii) follows from the fact that $\{\tilde{X}^{z_i}\neq X^{z_i}\} \subset \{\tilde{I}^{z_i}\neq I^{z_i}\}$ and Proposition \ref{prop:I:coupling:property}.
\end{proof}

\begin{proof}[Proof of Lemma \ref{lem:final:independent}]
    Let $z_1,\dots, z_4$ be four different sites in $\mathcal{Z}$. For all $i=1,\dots, 4$, let us define $\tilde{X}_{z_i}$ as in the proof of Lemma \ref{lem:independent:construction:ksites}. In particular, Items (iv) to (vi) follow from Lemma \ref{lem:independent:construction:ksites}.
    
    Then, for $i=1,2$, let $\hat{I}^{z_i}$ be the backward random walks defined by \eqref{eq:I:hat}, and $\hat{L}^{z_i}$ be the associated layer process. We denote $\hat{L}^{z_i}_{-\infty} = \lim_{s\to -\infty} \tilde{L}^{z_i}_s \in \{0,5\}$ the terminal layer. Then, remind the functions $g$ and $r$ defined in the beginning of Appendix \ref{sec:coupling:construction} and define, for all $i=1,\dots, 4$, 
    \begin{equation*}
        \hat{X}_{z_i} = g\left( \hat{I}^{z_i}, \xi^{(0)}_{r(\hat{I}^{z_i})} \right) \mathbf{1}_{\hat{L}^{z_i}_{-\infty} = 0} + g\left( \hat{I}^{z_i}, \xi^{(5)}_{r(\hat{I}^{z_i})} \right) \mathbf{1}_{\hat{L}^{z_i}_{-\infty} = 5}.
    \end{equation*}
    Using the fact that the $\xi^{(j)}$'s are i.i.d., Item (i) follows from Proposition \ref{prop:I:distribution:property}, Items (ii) and (v) follow respectively from Items (ii) and (iii) of Proposition \ref{prop:I:independence:property}. Finally, Items (ii) and (iii) follow from Item (ii) of Proposition \ref{prop:I:coupling:property}.
\end{proof}

\section{Coalescence of two or more backward random walks}
\label{app:coalescence_of_2_or_more_BRW}

The aim of this section is to prove Proposition \ref{prop:1}. Throughout the proof we will consider partitions of the sets of cardinal 2, 3 and 4. For a finite set $E$, we denote by $\mathcal{P}(E)$ the set of its partitions. For ease of notation, we consider a notation which we exemplify in the case of a set $\{a,b,c\}$ with three different elements. In this case, the set $\mathcal{P}(\{a,b,c\})$ has 5 elements which are :
\begin{itemize}
    \item $\{ \{a\}, \{b\}, \{c\} \}$ written as $a \vp b \vp c$,
    \item $\{\{a, b\},\{c\}\}$, written as $a b \vp c$,
    \item $\{\{a,c\},\{b\}\}$, written as $a c \vp b$,
    \item $\{\{a\},\{b,c\}\}$, written as $a \vp b c$,
    \item $\{\{a,b,c\}\}$, written as $abc$.
\end{itemize}

%\changes{
%    The proof relies on the Markovian property of the backward random walks $(I^z)_{z\in \mathcal{Z}}$. From these backward random walks (BRW), we define a ``status'' random walk, for instance $S^1_t$ below, which informs on whether the BRW has started, is currently running, or has reached the cemetery. In turn, a partition valued Markov chain is defined, denoted $P_t$, and using the notation introduced above, which informs on the coalescing status of the BRWs. The proof is then based on the computation of the transition probabilities of the partition Markov chain (some transition graphs are drawn for pedagogical purposes). The proof is then concluded by the computation of the coalescence probability.
%}

\changes{
The proof relies on the Markovian property of the backward random walks $(I^z)_{z\in \mathcal{Z}}$. From these backward random walks (BRW), we define a ``status'' random walk - such as $S^1_t$ below - which informs us whether the BRW has started, is currently running, or has already reached the cemetery. This, in turn, allows us to define a partition valued Markov chain, denoted $P_t$. To keep concise notation for the states of this partition valued Markov chain, we use the notation introduced above, which informs us on the coalescing status of the BRWs. 
    The proof is then based on the computation of the transition probabilities of the partition Markov chain (some transition graphs are drawn for pedagogical purposes). Finally, the proof is concluded by the computation of the coalescence probabilities.
}

\begin{proof}[Proof of Proposition \ref{prop:1}]
Without loss of generality, we assume in this proof that $t_1\geq t_2\geq t_3\geq t_4$.

\paragraph*{Proof of Item (i)} 
Consider the backward random walks $I^{z_1}$ and $I^{z_2}$ from time $t=t_1$ to time $t=-\infty$. Let us introduce the following notation: in what follows, the letter `$S$' stands for the state of a random walk; $ S^1_t $ will be the state at time $t$ of the random walk associated to $z_1,$ and $S^2_t $ the state at time $t$ of the random walk associated to $z_2$. We define $S^{1}_t = 1_{\infty} $ if $t > t_1,$ $S^{1}_t = 1_0 $ if $I^{z_1}_t=0$, $S^{1}_t = 1$ else, and by analogy, $ S^2_t= 2_{\infty} $ if $t > t_2,$ $S^2_t = 2_0 $ if $I^{z_2}_t = 0 ,$ $ S^2_t = 2 $ else. 

In what follows, we extend the notation introduced above for partitions of sets. Let us denote $P_t = S^{1}_t \vp S^{2}_t$ if $t>\tau^c_{z_1,z_2}$ and $P_t = S^{1}_tS^{2}_t$ else. Then, the backward process $(P_t)_{-\infty< t\leq t_1}$ is a time in-homogeneous backward Markov chain on the state space $\mathcal{P}_2 =\cup_{a\in \{1_\infty, 1, 1_0\}} \cup_{b\in \{2_\infty, 2, 2_0\}} \{ab, a \vp b\}$. At time $t=t_1$, the backward process starts from the initial condition $P_{t_1} = 1 \vp 2_\infty $ if $t_1 >   t_2 ,$ and $P_{t_1} = 1 \vp 2 $ if $t_1 = t_2 $ (note that in that case $z_1\neq z_2$ implies that $i_1 \neq i_2$ so that $P_{t_1} = 12$ is not possible). 
By Proposition \ref{prop:cemetery:as}, the random variable $\min\{\tau^R_{z_1}, \tau^R_{z_2}\}>-\infty$ almost surely, implying that $P_{-\infty}:=\lim_{t\to-\infty}P_t$ exists almost surely and $P_{-\infty}\in\{1_{0}2_{0}, 1_{0}\vp 2_{0}\}$, that is the set of absorbing states. 
One can check that
\begin{equation}
\label{eq:coalesce:equiv:markov}
    \P_\theta ( \{ z_1 \leftrightsquigarrow z_2 \}) = \P_\theta \left( P_{-\infty} = 1_02_0 \,|\, P_{t_1} = 1 \vp 2_\infty \right),
\end{equation}
in case $ t_1 > t_2,$ and 
\begin{equation}
\label{eq:coalesce:equiv:markov2}
    \P_\theta ( \{ z_1 \leftrightsquigarrow z_2 \}) = \P_\theta ( P_{-\infty} = 1_02_0 \,|\, P_{t_1} = 1 \vp 2),
\end{equation}
in case $ t_1 = t_2.$

Now, for $t\leq t_1$, let us denote $Q_t(x,y)$ the transition probabilities of the time in-homogeneous backward Markov chain $(P_t)_{-\infty< t\leq t_1}$:
\begin{equation*}
    Q_t(x,y) = \P_\theta\left( P_{t-1} = y | P_t = x\right),\; x,y\in \mathcal{P}_2 .
\end{equation*}

Assume for now that $t_1>t_2$. Since we are only interested in computing the coalescence probability \eqref{eq:coalesce:equiv:markov}, we only need to give the transition probabilities encountered on the path from state $1 \vp 2_\infty$  to state  $1_{0}2_{0}$.
For $t>t_2+1$, the only relevant transition probability is $Q_t(1 \vp 2_\infty, 1 \vp 2_\infty)=1-\lambda$. At time $t=t_2+1$ there are two relevant transitions:  $Q_{t_2+1}(1 \vp 2_\infty$,$1 \vp 2)=(1-\lambda)(1-1/N)$ and $Q_{t_2+1}(1 \vp 2_\infty, 12)=(1-\lambda)/N$.

For $t\leq t_2$, there are four relevant transition probabilities: 
\begin{itemize}
    \item $Q_t(1 \vp 2, 1 \vp 2)=(1-\lambda)^2(1-1/2)$
    \item $Q_t(1 \vp 2, 12)=(1-\lambda)^2/N$,
    \item $Q_t(12, 1_02_0)=1- Q_t(12, 12)= \lambda$.
\end{itemize}
Figure \ref{fig:graph:1|2:to:12} gives a graphical representation of the relevant transitions just described in case $ t_2 < t_1$. 

\begin{figure}[ht] \centering
\begin{tikzpicture}[->,>=latex,shorten >=2pt, line width=0.5pt, node distance=2.5cm]
\newcommand\Y{3.5}
\newcommand\Xt{-1.3}

\node [circle, draw] (1_2-) {$1 \vp 2_\infty$};
\node [circle, draw] (1_2) [above left  = of 1_2-] {$1 \vp 2$};
\node [circle, draw] (12) [below left = of 1_2-] {$12$};
\node [circle, draw] (1+2+) [left = 1cm of 12] {$1_02_0$};
% t>t_2+1
\path (1_2-)  edge[loop right]  node[above=5pt]{$1-\lambda$}  (1_2-);
% t=t_2+1
\path (1_2-)  edge  node[above right]{$(1-\lambda)(1 - 1/N)$}   (1_2);
\path (1_2-)  edge  node[below right]{$\frac{1-\lambda}{N}$}  (12);
% t<t_2+1
\path (1_2)  edge[loop left]  node[left]{$(1-\lambda)^2(1-1/N)$} (1_2);
\path (1_2)  edge  node[left]{$\frac{(1-\lambda)^2}{N}$} (12);
\path (12)  edge[loop below]  node[right]{$1-\lambda$} (12);
\path (12)  edge  node[above]{$\lambda$} (1+2+);
\path (1+2+)  edge[loop left]  node[left]{$1$} (1+2+);

\draw[-, dashed, gray] (\Xt,\Y) -- (\Xt,-\Y); % vertical line at t=t_2+1
\end{tikzpicture} 
\caption{Graph of the relevant transitions of the backward process $P$ used to compute a bound for $\P_\theta ( \{ z_1 \leftrightsquigarrow z_2 \})$ when $t_1\geq t_2$. The starting node is $1 \vp 2_\infty$ if $t_1>t_2$ and $1 \vp 2$ else. On each edge, the corresponding transition probability is given. The gray vertical line separates two temporal zones: the right one corresponds to times $t>t_2+1$ and the left one corresponds to times $t\leq t_2+1$.} 
\label{fig:graph:1|2:to:12}
\end{figure}

Starting from Equation \eqref{eq:coalesce:equiv:markov}, one can check that  
\begin{multline*}
\P_\theta(\{z_1 \leftrightsquigarrow z_2\})=\P_\theta(P_{t_2}=12|P_{t_1}=1\vp 2_{\infty}) 
\P_\theta(P_{-\infty}=1_{0}2_{0}|P_{t_2}=12)\\
+\P_\theta(P_{t_2}=1\vp 2|P_{t_1}=1\vp 2_{\infty}) 
\P_\theta(P_{-\infty}=1_{0}2_{0}|P_{t_2}=1\vp 2).
\end{multline*}
Now, by the dynamics of the Markov chain $(P_t)_{t\leq t_1}$ we have that $\P_\theta(P_{t_2}=1\vp 2|P_{t_1}=1\vp 2_{\infty})=(1-\lambda)^{t_1-t_2}$ and $\P_\theta(P_{t_2}=12|P_{t_1}=1\vp 2_{\infty})=(1-\lambda)^{t_1-t_2}/N$. Moreover, $\P_\theta(P_{-\infty}=1_{0}2_{0}|P_{t_2}=12)=1$ by Proposition \ref{prop:cemetery:as}. By observing that
\begin{eqnarray*}
 \P_\theta(P_{-\infty}=1_{0}2_{0}|P_{t_2}=1\vp 2)
 & = &\frac{(1-\lambda)^2}{N} \sum_{k=0}^{\infty}(1-\lambda)^{2k} (1-1/N)^{k}\\
 & \leq &\frac{(1-\lambda)^2}{N}\frac{1}{1-(1-\lambda)^2},  
\end{eqnarray*}
and by putting all pieces together, we then obtain that
\begin{equation}
\label{prop_coalescence_ineq_1}
\displaystyle \P_\theta ( \{ z_1 \leftrightsquigarrow z_2 \}) \leq \frac{(1-\lambda)^{|t_1-t_2|} }{N}\left[1+\frac{(1-\lambda)^{2}}{{1-(1-\lambda)^2}}\right] = \frac{(1-\lambda)^{|t_1-t_2|} }{N(1-(1-\lambda)^2)}.
\end{equation}

In the case $t_2=t_1$, we have
\begin{multline}
\label{prop_coalescence_ineq_2}
\displaystyle \P_\theta ( \{ z_1 \leftrightsquigarrow z_2 \}) = \P_\theta(P_{-\infty}=1_{0}2_{0}|P_{t_2}=1\vp 2) \\
\leq \frac{(1-\lambda)^2}{N}\frac{1}{1-(1-\lambda)^2} \leq \frac{(1-\lambda) }{N(1-(1-\lambda)^2)},    
\end{multline}
so that Item (i) follows from inequalities \eqref{prop_coalescence_ineq_1} and \eqref{prop_coalescence_ineq_2}.

\paragraph*{Proof of Item (ii)}
By analogy to the case with two sites, we define $S^{3}_t = 3_{\infty} $ if $t > t_3,$ $S^{3}_t = 3_0 $ if $I^{z_3}_t=0$, $S^{3}_t = 3$ else.
Let us describe the backward process $(P_t)_{-\infty< t\leq t_1}$ in that case. First note that for all $t\leq t_1$, $P_t$ is a partition of $\{S^{1}_t,S^{2}_t,S^{3}_t\}$ and in particular $P_s \in \mathcal{P}_3 =\cup_{a\in \{1_\infty, 1, 1_0\}} \cup_{b\in \{2_\infty, 2, 2_0\}} \cup_{c\in \{3_\infty, 3, 3_0\}} \mathcal{P}(\{a,b,c\})$. The choice of the partition is induced by the equivalence relation defined by 
\begin{equation*}
    S^{k}_t \leftrightarrow_t S^{\ell}_t \mbox{ if and only if } t\leq \tau^c_{z_k,z_\ell}.
\end{equation*}
This equivalence relation naturally induces a partition of $\{S^{1}_t,S^{2}_t,S^{3}_t\}$ and we define $P_t$ as this partition. Similarly to the proof of Item (i), we have
\begin{equation}
    \label{eq:coalesce:equiv:markov3}
    \P_\theta ( \{ z_1 \leftrightsquigarrow z_2 \leftrightsquigarrow z_3  \} ) =  \P_\theta \left( P_{-\infty} = 1_02_03_0 \,|\, P_{t_1} = 1 \vp 2_\infty \vp 3_\infty \right),
\end{equation}
in case $t_1>t_2\geq t_3$ for instance, where
$P_{-\infty}=\lim_{t\to-\infty}P_t$ exists almost surely and $P_{-\infty}\in{\cal{P}}(\{1_0,2_0,3_0\}).$
Assume for now that $t_1>t_2> t_3$ which is the most general case. The other cases can be treated similarly.

For $t\leq t_1$, let us denote the transition probabilities of the chain 
\begin{equation*}
    Q_t(x,y) = \P_\theta\left( P_{t-1} = y | P_t = x\right),\; x,y\in \mathcal{P}_3 .
\end{equation*}
Since we are only interested in computing the coalescence probability \eqref{eq:coalesce:equiv:markov3}, we only need to give the transition probabilities encountered on the path from state $1 \vp 2_\infty \vp 3_\infty$  to state $1_{0}2_{0}3_{0}$.
For $t>t_2+1$, the only relevant transition probability is $Q_t(1 \vp 2_\infty \vp 3_\infty, 1 \vp 2_\infty \vp 3_\infty)=1-\lambda$. At time $t=t_2+1$, there are two relevant transition probabilities: $Q_{t_2+1}(1 \vp 2_\infty\vp 3_{\infty},1 \vp 2\vp 3_{\infty})=(1-\lambda)(1-1/N)$ and $Q_{t_2+1}(1 \vp 2_\infty\vp 3_{\infty}, 12\vp 3_{\infty})=(1-\lambda)/(N)$.

For $t_3+1< t\leq t_2$ the three relevant transition probabilities are :
\begin{itemize}
    \item $Q_t(1 \vp 2 \vp 3_\infty, 1 \vp 2 \vp 3_\infty) = (1-\lambda)^2(1-1/N) \leq (1-\lambda)^2 $,
    \item $Q_t(1 \vp 2 \vp 3_\infty, 1 2 \vp 3_\infty)=(1-\lambda)^2/N \leq (1-\lambda)N^{-1}$,
    \item $Q_t(1 2 \vp 3_\infty, 1 2 \vp 3_\infty)=(1-\lambda)$.
\end{itemize}
At time $t=t_3+1$, the seven relevant transition probabilities are : \begin{itemize}
    \item $Q_t(1 \vp 2 \vp 3_\infty, 1 \vp 2 \vp 3)=(1-\lambda)^2(1-1/N)(1-2/N)\leq (1-\lambda)^2$, 
    \item $Q_t(1 \vp 2 \vp 3_\infty, 1 \vp 23)= Q_t(1 \vp 2 \vp 3_\infty, 1 3 \vp 2) = Q_t(1 \vp 2 \vp 3_\infty, 1 2 \vp 3) \leq (1-\lambda)^2 N^{-1}$, 
    \item $Q_t(1 \vp 2 \vp 3_\infty, 1 2 3)=(1-\lambda)^2/N^2$.
    \item $Q_t(1 2 \vp 3_\infty, 1 2 \vp 3)=(1-\lambda)(1-/N)\leq (1-\lambda),$
    \item $Q_t(1 2 \vp 3_\infty, 1 2 3)=(1-\lambda)/N$.
\end{itemize}
Finally, for $t\leq t_3$, the thirteen relevant transition probabilities are :
\begin{itemize}
    \item $Q_t(1 \vp 2 \vp 3, 1 \vp 2 \vp 3) \leq (1-\lambda)^3$,
    \item $Q_t(1 \vp 2 \vp 3, 1 2 \vp 3)= Q_t(1 \vp 2 \vp 3, 1 \vp 2 3)=Q_t(1 \vp 2 \vp 3, 2 \vp 1 3) \leq (1-\lambda)^3 N^{-1}$,
    \item $Q_t(1 \vp 2 \vp 3, 1 2 3) \leq (1-\lambda)^3 N^{-2}$,
    \item $Q_t(1 2 \vp 3, 1 2 \vp 3)= Q_t(1 \vp 2 3, 1 \vp 2 3)=Q_t(2 \vp 1 3, 2 \vp 1 3)\leq (1-\lambda)^2$,
    \item $Q_t(1 2 \vp 3, 1 2 3)= Q_t(1 \vp 2 3, 1 2 3)=Q_t(2 \vp 1 3, 1 2  3) \leq (1-\lambda)^2 N^{-1}$,
    \item $Q_t(1 2 3, 1_0 2_0 3_0)= 1 - Q_t(1 2 3, 1 2 3)=\lambda$.
\end{itemize}

\begin{figure}[ht] \centering
\newcommand\Y{4}
\newcommand\Xt{-2.3}
\newcommand\Xtt{-6}
\begin{tikzpicture}[->,>=latex,shorten >=2pt, line width=0.5pt, node distance=2cm]
\node [circle, draw] (1_2-_3-) {$1 \vp 2_\infty \vp 3_\infty$};
\node [circle, draw] (1_2_3-) [above left = 1cm and 3cm of 1_2-_3-] {$1 \vp 2 \vp 3_\infty$};
\node [circle, draw] (12_3-) [below = 3cm of 1_2_3-] {$1 2 \vp 3_\infty$};
% t > t_2+1
\path (1_2-_3-)  edge[loop right]  node[right]{$1-\lambda$} (1_2-_3-);
% t = t_2+1
\path (1_2-_3-)  edge  node[above right]{$1-\lambda$}  (1_2_3-);
\path (1_2-_3-)  edge  node[below right]{$\frac{1-\lambda}{N}$}  (12_3-);
% t < t_2+1
\path (1_2_3-)  edge[loop above]  node[above]{$(1-\lambda)^2$} (1_2_3-);
\path (1_2_3-)  edge  node[right]{$\frac{(1-\lambda)^2}{N}$} (12_3-);
\path (12_3-)  edge[loop below]  node[below]{$1-\lambda$} (12_3-);

\draw[-, dashed, gray] (\Xt,\Y) -- (\Xt,-\Y); % vertical line at t=t_2+1
\node[draw, gray] at (-6,3) {\Huge II};
\node[draw, gray] at (2,3) {\Huge III};
\end{tikzpicture} 
\caption{Graph of the relevant transitions of the backward process $P$ (for times $t > t_3+1$) used to compute a bound for $\P_\theta ( \{ z_1 \leftrightsquigarrow z_2 \leftrightsquigarrow z_3 \})$ when $t_1\geq t_2 \geq t_3$. The starting node is $1 \vp 2_\infty \vp 2_\infty$ if $t_1>t_2>3$, $1 \vp 2 \vp 3_\infty$ if $t_1=t_2>t_3$ and $1 \vp 2 \vp 3$ (which appears in Figure \ref{fig:graph:1|2|3:to:123:part2}) else. On each edge, an upper bound of the corresponding transition probability is given. The gray vertical line separates two temporal zones: zone II corresponds to times $t_3+1<t\leq t_2+1$ and zone III corresponds to times $t>t_2+1$.}
\label{fig:graph:1|2|3:to:123:part1}
\end{figure}

\begin{figure}[ht] \centering
\newcommand\Y{5}
\newcommand\Xtt{-1.1}
\begin{tikzpicture}[->,>=latex,shorten >=2pt, line width=0.5pt, node distance=2cm]
\node [circle, draw] (1_2_3-) {$1 \vp 2 \vp 3_\infty$};
\node [circle, draw] (12_3-) [below = of 1_2_3-] {$1 2 \vp 3_\infty$};
\node [circle, draw] (13_2) [left = 3cm of 1_2_3-] {$1 3 \vp 2$};
\node [circle, draw] (1_23) [above = 0.5cm of 13_2] {$1 \vp 2 3$};
\node [circle, draw] (12_3) [below = 0.5cm of 13_2] {$1 2 \vp 3$};
\node [circle, draw] (1_2_3) [above right = of 1_23] {$1 \vp 2 \vp 3$};
\node [circle, draw] (123) [left = 5cm of 12_3-] {$1 2 3$};
\node [circle, draw] (1+2+3+) [left = of 123] {$1_0 2_0 3_0$};
% t > t_3+1
\path (1_2_3-)  edge[loop right]  node[below right]{$(1-\lambda)^2$} (1_2_3-);
\path (1_2_3-)  edge  node[right]{$\frac{(1-\lambda)^2}{N}$} (12_3-);
\path (12_3-)  edge[loop right]  node[above right]{$1-\lambda$} (12_3-);
% t = t_3+1
\path (1_2_3-)  edge  node[above right]{$(1-\lambda)^2$} (1_2_3);
\path (1_2_3-)  edge  node[above]{$\frac{(1-\lambda)^2}{N}$} (1_23);
\path (1_2_3-)  edge  (13_2);
\path (1_2_3-)  edge  (12_3);
\path (1_2_3-)  edge  node[right=.5cm]{$\frac{(1-\lambda)^2}{N^2}$} (123);
\path (12_3-)  edge  node[below left]{$1-\lambda$} (12_3);
\path (12_3-)  edge  node[below]{$\frac{1-\lambda}{N}$} (123);
% t < t_3+1
\path (1_23)  edge[loop above]  node[above left]{$(1-\lambda)^2$} (1_23);
\path (13_2)  edge[loop, in=45, out=90, looseness=4] (13_2);
\path (12_3)  edge[loop, in=45, out=90, looseness=4] (12_3);
\path (1_23)  edge  node[above=1.3cm]{$\frac{(1-\lambda)^2}{N}$} (123);
\path (13_2)  edge  (123);
\path (12_3)  edge  (123);
\path (1_2_3)  edge[loop above]  node[above]{$(1-\lambda)^3$} (1_2_3);
\path (1_2_3)  edge  node[above=.3cm]{$\frac{(1-\lambda)^3}{N}$} (1_23);
\path (1_2_3)  edge  (13_2);
\path (1_2_3)  edge  (12_3);
\path (1_2_3) edge [bend right=100,looseness=2,in=200] node[above left]{$\frac{(1-\lambda)^3}{N^2}$} (123);

\path (123)  edge[loop, out=100, in=170, looseness=4]  node[above left]{$1-\lambda$} (123);
\path (123)  edge  node[below right]{$\lambda$} (1+2+3+);
\path (1+2+3+)  edge[loop above]  node[above]{$1$} (1+2+3+);

\draw[-, dashed, gray] (\Xtt,\Y+.8) -- (\Xtt,-\Y+.5); % vertical line at t=t_3+1
\node[draw, gray] at (-9,3) {\Huge I};
\node[draw, gray] at (2,3) {\Huge II};
\end{tikzpicture} 
\caption{Graph of the relevant transitions of the backward process $P$ (for times $t \leq t_3+1$) used to compute a bound for $\P_\theta ( \{ z_1 \leftrightsquigarrow z_2 \leftrightsquigarrow z_3 \})$ when $t_1\geq t_2 \geq t_3$. The starting node is $1 \vp 2_\infty \vp 2_\infty$ (which appears in Figure \ref{fig:graph:1|2|3:to:123:part1}) if $t_1>t_2>3$, $1 \vp 2 \vp 3_\infty$ if $t_1=t_2>t_3$ and $1 \vp 2 \vp 3$ else. On some of the edges, an upper bound of the corresponding transition probability is given (the nodes $1\vp23$, $13\vp2$ and $12\vp3$ play almost the same role that is why we omit some the transition probabilities). The gray vertical line separates two temporal zones: zone I corresponds to times $t\leq t_3+1$ and zone II corresponds to times $t_3+1<t\leq t_2+1$.}
\label{fig:graph:1|2|3:to:123:part2}
\end{figure}

Looking at the Figures \ref{fig:graph:1|2|3:to:123:part1} and \ref{fig:graph:1|2|3:to:123:part2}, the Markov property of $(P_t)_{t\leq t_1}$ permits to compute the probabilities of ending in state $1_{0}2_{0}3_{0}$ from any of the other states of the graph. These two figures highlight three time zones denoted I, II and III, which are depicted from left to right corresponding to the forward evolution of time (hence the arrows go from right to left since they correspond to a backward dynamics). 

In the following, we decompose the computation of the probability in Equation \eqref{eq:coalesce:equiv:markov3} with respect to the three zones. To simplify the notation, let us denote for all $t\in\Z$ and $\Theta\in \mathcal{P}_3$, $p(t,\Theta) = \P_\theta \left( P_{-\infty} = 1_02_03_0 \,|\, P_{t} = \Theta \right)$.

For zone I, we have to consider initial conditions starting at time $t_3$ (see Figure \ref{fig:graph:1|2|3:to:123:part2}). That is,
\begin{equation*}
    p(t_3, 1 \vp 23) = p(t_3, 12 \vp 3)=p(t_3,13\vp 2)\leq K(1-\lambda)^2 N^{-1},
\end{equation*}
and 
\begin{equation*}
    p(t_3, 1 \vp 2 \vp 3 ) \leq K(1-\lambda)^3 \left[3p(t_3,1 \vp 23) N^{-1} + N^{-2} \right] \leq K (1-\lambda)^3 N^{-2}.
\end{equation*}
For zone II, we have for all $t_3< t \leq t_2$,
\begin{equation*}
    p(t, 12 \vp 3_\infty ) = (1-\lambda)^{t-t_3} [p(t_3, 12 \vp 3) + N^{-1}p(t_3, 123)] \leq K (1-\lambda)^{t-t_3} N^{-1},
\end{equation*}
since we have to wait in the middle part during $t-t_3-1$ steps and then take the edge from $12 \vp 3_\infty$ to $12\vp 3$ or the edge from $12 \vp 3_\infty$ to $123$. Moreover, we have
\begin{eqnarray*}
    p(t_2 ,1\vp 2 \vp 3_\infty ) &\leq& K \left\{ \sum_{k=1}^{t_2-t_3-1} (1-\lambda)^{2k} p(t_2-k, 12 \vp 3_\infty) N^{-1} \right. \\
    && \left. + (1-\lambda)^{2(t_2-t_3)} \left[ 3p(t_3, 1\vp23) N^{-1} + p(t_3, 1 \vp 2 \vp 3)  \right] \right\} \\
    &\leq & K \left\{ (1-\lambda)^{t_2-t_3} \sum_{k=1}^{t_2-t_3-1} (1-\lambda)^k N^{-2} + 4 (1-\lambda)^{2(t_2-t_3)} (1-\lambda)^2 N^{-2} \right\} \\
    & \leq & K (1-\lambda)^{t_2-t_3}N^{-2}.
\end{eqnarray*}
Finally, using zone III, we have 
\begin{eqnarray*} 
    \P_\theta ( \{ z_1 \leftrightsquigarrow z_2 \leftrightsquigarrow z_3  \} )&=& p(t_1,1\vp 2_\infty\vp 3_\infty) \\
     &=& (1-\lambda)^{t_1-t_2}[p(t_2 ,1\vp 2 \vp 3_\infty )(1-1/2N)+p(t_2 ,12 \vp 3_\infty )(2N)^{-1}],
\end{eqnarray*}
which ends the proof.

\paragraph*{Proof of Item (iii)}
The backward process $(P_t)_{-\infty<t\leq t_1}$ is naturally extended to the case of four sites and we use the decomposition
\begin{equation*}
\{ z_1 \leftrightsquigarrow z_2  \} \cap \{z_3 \leftrightsquigarrow z_4 \} 
=  \{ z_1 \leftrightsquigarrow z_2  \leftrightsquigarrow z_3 \leftrightsquigarrow z_4 \} \cup \{ z_1 \leftrightsquigarrow z_2  \not\leftrightsquigarrow z_3 \leftrightsquigarrow z_4 \}.
\end{equation*} 
Assume for now that $t_1 > t_2 \geq t_3 \geq t_4$ so that $P_{t_1} = 1 \vp 2_\infty \vp 3_\infty \vp 4_\infty.$ It suffices to control 
\begin{multline*}
    \P_\theta \left( P_{-\infty} = 1_02_03_04_0 \,|\, P_{t_1} = 1 \vp 2_\infty \vp 3_\infty \vp 4_\infty \right)\\ + \P_\theta \left( P_{-\infty} = 1_02_0 \vp 3_04_0 \,|\, P_{t_1} = 1 \vp 2_\infty \vp 3_\infty \vp 4_\infty \right).
\end{multline*}
The list of the relevant transitions and the corresponding graphs are not written here due to their huge size. However the computations can be made in the same manner.
For instance, one can prove that
\begin{equation*}
    \P_\theta \left( P_{-\infty} = 1_02_03_04_0 \,|\, P_{t_1} = 1 \vp 2_\infty \vp 3_\infty \vp 4_\infty \right) \leq K(1-\lambda)^{t_1-t_4}N^{-3},
\end{equation*}
and 
\begin{equation*}
    \P_\theta \left( P_{-\infty} = 1_02_0 \vp 3_04_0 \,|\, P_{t_1} = 1 \vp 2_\infty \vp 3_\infty \vp 4_\infty \right) \leq K(1-\lambda)^{(t_1-t_2)+(t_3-t_4)}N^{-2},
\end{equation*}
which in turn give the desired result.
Finally, the same upper bound holds if $t_1>t_2$ is not satisfied by replacing the initial condition (for instance $t_1 = t_2 > t_3 \geq t_4$ corresponds to the initial condition $P_{t_1} = 1 \vp 2 \vp 3_\infty \vp 4_\infty$ if $z_1 \neq z_2$).

\paragraph*{Proof of Item (iv)}
We will only prove the case $t_1\wedge t_2> t_3\vee t_4$. The case in which $t_1\wedge t_2=t_3\vee t_4$ follows along the same lines. In the sequel, as before,
we assume without loss of generality that $t_1\geq t_2>  t_3\geq t_4$ (the case where $t_1<t_2$ can be treated in the same way). In the following, we use the backward process $(P_t)_{-\infty<t\leq t_1}$ with four sites like in the proof of Item (iii). Then, since $t_3 + 1 \le t_2, $
\begin{multline*}
   \{ z_1 \leftrightsquigarrow z_3  \leftrightsquigarrow z_4 \not\leftrightsquigarrow z_2 \} \cap \{ \tau^R_{z_2} < t_3\} \subset 
   \{ z_1 \leftrightsquigarrow z_3  \leftrightsquigarrow z_4 \not\leftrightsquigarrow z_2 \} \cap \{\tau^R_{z_2} < t_3+1\} \\
=   \left\{ P_{-\infty} = 1_03_04_0 \vp 2_0 \right\} \cap \{ P_{t_3+1} = 1 \vp 2 \vp 3_\infty \vp 4_\infty \}.
\end{multline*}

Assume for now that $t_1 > t_2$. It follows from the same arguments as above that
\begin{equation*}
    \P_\theta \left( P_{-\infty} = 1_03_04_0 \vp 2_0  \,|\, P_{t_3+1} = 1 \vp 2 \vp 3_\infty \vp 4_\infty \right) \leq K(1-\lambda)^{t_3+1-t_4}N^{-2},
\end{equation*}
and
\begin{equation*}
    \P_\theta \left( P_{t_3+1} = 1 \vp 2 \vp 3_\infty \vp 4_\infty \,|\, P_{t_1} = 1 \vp 2_\infty \vp 3_\infty \vp 4_\infty \right) \leq (1-\lambda)^{t_1-(t_3+1)+t_2-(t_3+1)}.
\end{equation*}
which in turn give the desired result. 

Finally, if $t_1=t_2$, the same upper-bound holds with the initial condition $P_{t_1} = 1 \vp 2_\infty \vp 3_\infty \vp 4_\infty$ replaced by $P_{t_1} = 1 \vp 2 \vp 3_\infty \vp 4_\infty$.
\end{proof}

\section{Proof of Lemma \ref{lem:covariance:produit}}
\label{app:covariance:decay}

\changes{
The statements in both Items 1 and 2 are trivial, so that we only need to prove Items 3 through 7. Items 3, 4 and 5 are proved in a similar way: Lemma \ref{lem:independent:construction:ksites} is used to relate independence with the high probability event of no coalescence, and some easy computations are needed to bound the covariance.  Items 6 and 7 use the more involved Lemma \ref{lem:final:independent}.
}

\begin{proof}[Proof of Item 3]
In this case, $B=\cov_{\theta}[Y^2_{z_1},Y^2_{z_3}]$. Let $\tilde{X}_{z_1}$ be a random variable distributed as $X_{z_1}$, independent of $X_{z_3}$ 
and such that $\{X_{z_1}\neq \tilde{X}_{z_1}\}\subset \{z_1 \leftrightsquigarrow z_3\}$. 
The existence of such random variables is ensured by Lemma \ref{lem:independent:construction:ksites} with $k=2$.
In particular, we have that $\E_{\theta}[\tilde{X}_{z_1}]=\E_{\theta}[X_{z_1}]$. Denote $\tilde{Y}_{z_1}=\tilde{X}_{z_1}-\E_{\theta}[\tilde{X}_{z_1}]$ and observe that $\{\tilde{Y}_{z_1}\neq Y_{z_1}\}=\{\tilde{X}_{z_1}\neq X_{z_1}\}$. By using the properties satisfied by random variables $\tilde{X}_{z_1},$ 
one can check that $\tilde{Y}_{z_1}$ 
is independent of $Y_{z_3}$,
so that   
\begin{equation*}
B=\cov_{\theta}[Y^2_{z_1},Y^2_{z_3}]=\E_{\theta}[(Y^2_{z_1}-\tilde{Y}^2_{z_1})Y^2_{z_3}].
\end{equation*}
Since $\{\tilde{Y}_{z_1} \neq  Y_{z_1} \}
=\{\tilde{X}_{z_1} \neq  X_{z_1} \}
\subset \{ z_1 \leftrightsquigarrow z_3 \}$ and $\max\{|Y_{z_1}|,|\tilde{Y}_{z_1}|,|Y_{z_3}|\}\leq 1$ almost surely, it follows that 
$$
|B|\leq \P_{\theta}\left(z_1\leftrightsquigarrow z_3\right) ,
$$
so that the result follows Item (i) of Proposition \ref{prop:1}.
\end{proof}

\begin{proof}[Proof of Item 4]
Suppose that $ z_1 = z_2.$ Then $B=\cov_{\theta}[Y^2_{z_1},Y_{z_3}Y_{z_4}]$. 
By Lemma \ref{lem:independent:construction:ksites} with $k=3$, there exist random variables $\tilde{Y}_{z_1}, \tilde{Y}_{z_3}$ and $\tilde{Y}_{z_4}$ defined in such a way that the following properties hold for each $i\in \{1,3,4\}$: 1) $\tilde{Y}_{z_i}$ has the same law as $Y_{z_i}$; 2)   
$\tilde{Y}_{z_i}$ is independent of $\tilde{Y}_{z_j}$ and $Y_{z_j}$ for all $j\in \{1,3,4\}\setminus\{i\}$; 3) $\{\tilde{Y}_{z_i}\neq Y_{z_i}\} \subset \cup_{j\in \{1,3,4\}\setminus\{i\}}\{z_i\leftrightsquigarrow z_j\}.$
Using first that $Y^2_{z_1}$ has same law as $\tilde{Y}^2_{z_1}$ and then that  $\E_{\theta}[\tilde{Y}_{z_i}]=\E_{\theta}[Y_{z_i}]=0$, we can deduce from properties 1 and 2 above that
\begin{eqnarray*}
B &= &\E_{\theta}[(Y^2_{z_1}-\tilde{Y}^2_{z_1})Y_{z_3}Y_{z_4}]\\
&= & \E_{\theta}[(Y^2_{z_1}-\tilde{Y}^2_{z_1})(Y_{z_3}-\tilde{Y}_{z_3})(Y_{z_4}-\tilde{Y}_{z_4})].    
\end{eqnarray*}
By combining property 3, the above identity and the fact that $\leftrightsquigarrow$ is an equivalence relation, we can deduce that
$$
|B|\leq \E_{\theta}\left[ 1_{\{z_1\leftrightsquigarrow z_3 \leftrightsquigarrow z_4\}} \right] = \P_\theta ( \{ z_1 \leftrightsquigarrow z_2 \leftrightsquigarrow z_3  \} ),
$$
so that the result follows from Item (ii) of Proposition \ref{prop:1}. 
\end{proof}

\begin{proof}[Proof of Item 5]
Let us consider the case in which $z_1=z_3$. Let $z_1,z_2,z_4$ be three different sites. The goal is to bound the covariance
    \begin{equation}
    \label{prof_item_4_lemma_cov_id_1}
        B =\cov_\theta \left[ Y_{z_1}Y_{z_2}, Y_{z_1}Y_{z_4} \right]=\E_\theta \left[ Y^2_{z_1}Y_{z_2}Y_{z_4} \right]-\cov_\theta \left[ Y_{z_1},Y_{z_2}\right]\cov_\theta \left[ Y_{z_1},Y_{z_4} \right].
    \end{equation}
    Arguing exactly like in the proof of Item 3, one can show that
    $\cov_\theta \left[ Y_{z_1},Y_{z_2}\right]\leq K(1-\lambda)^{|t_1-t_2|}N^{-1}$ and $\cov_\theta \left[ Y_{z_1},Y_{z_4} \right]\leq K(1-\lambda)^{|t_1-t_4|}N^{-1}$, so that
    \begin{equation}
    \label{prof_item_4_lemma_cov_ineq_1}
\cov_\theta \left[ Y_{z_1},Y_{z_2}\right]\cov_\theta \left[ Y_{z_1},Y_{z_4} \right]\leq K(1-\lambda)^{|t_1-t_2|+|t_1-t_4|}N^{-2}.
    \end{equation}
    
We will now deal with the term $\E_\theta \left[ Y^2_{z_1}Y_{z_2}Y_{z_4} \right].$ For that end, let $\tilde{X}_{z_1}$, $\tilde{X}_{z_2}$ and $\tilde{X}_{z_4}$ be the random variables defined in Lemma \ref{lem:independent:construction:ksites} with $k=3$, and denote $\tilde{Y}_{z_i} = \tilde{X}_{z_i}- \E_\theta ( X_{z_i})$, $i\in\{1,2,4\}$ the corresponding centered versions. By using the independence properties of the random variables $\tilde{Y}_{z_i}$ (and the fact that they are all centered), one can check that 
$$
\E_\theta \left[ Y^2_{z_1}Y_{z_2}Y_{z_4} \right]=\E_\theta \left[ (Y^2_{z_1}-\tilde{Y}^2_{z_1})(Y_{z_2}-\tilde{Y}_{z_2})(Y_{z_4}-\tilde{Y}_{z_4})\right]+\E_{\theta}[\tilde{Y}^2_{z_1}]\cov_\theta \left[ Y_{z_2},Y_{z_4}\right].
$$
Arguing exactly like in the proof of Item 4, one can deduce that
\begin{equation}
\label{prof_item_4_lemma_cov_ineq_2}
\E_\theta \left[ (Y^2_{z_1}-\tilde{Y}^2_{z_1})(Y_{z_2}-\tilde{Y}_{z_2})(Y_{z_4}-\tilde{Y}_{z_4})\right]\leq K(1-\lambda)^{|s_3-s_2|+|s_2-s_1|}N^{-2},
\end{equation}
where $s_1\leq s_2\leq s_3$ is an ordering of $\{t_1,t_2,t_4\}$.
Finally, since $\cov_\theta \left[ Y_{z_2},Y_{z_4}\right]\leq K(1-\lambda)^{|t_2-t_4|}N^{-1}$, the random variables $\tilde{Y}_{z_k}$ are bounded by 1 almost surely and $|t_2-t_4|\leq \min\{|t_2-t_1|+|t_1-t_4|,|s_3-s_2|+|s_2-s_1|\}$, it follows from \eqref{prof_item_4_lemma_cov_id_1}, \eqref{prof_item_4_lemma_cov_ineq_1} and \eqref{prof_item_4_lemma_cov_ineq_2} that
$$
|B|\leq K(1-\lambda)^{|t_2-t_4|}N^{-1},
$$
concluding the proof.
\end{proof}

\begin{proof}[Proof of Item 6]
Let $z_1,\dots,z_4$ be four different points in $\cal Z$ and denote by $\hat{X}_{z_1}, \hat{X}_{z_2}$, $\tilde{X}_{z_1},\ldots, \tilde{X}_{z_4}$ the random variables defined in Lemma \ref{lem:final:independent}. Moreover, denote by $\tilde{Y}_{z_k} = \tilde{X}_{z_k}- \E_\theta ( \tilde{X}_{z_k})$, $1\leq k\leq 4$ and $\hat{Y}_{z_k} = \hat{X}_{z_k}- \E_\theta ( \hat{X}_{z_k})$, $1\leq k\leq 2$, their centered versions. First of all, observe that by using Items (i), (ii) and (iii) of Lemma \ref{lem:final:independent}, we can rewrite B as
    $$
B=\E_{\theta}\left[\left(Y_{z_1}Y_{z_2}-\hat{Y}_{z_1}\hat{Y}_{z_2}\right)Y_{z_3}Y_{z_4}1_{\{z_1 \leftrightsquigarrow z_3\}\cup\{z_1 \leftrightsquigarrow z_4\}\cup \{z_2 \leftrightsquigarrow z_3\}\cup \{z_2\leftrightsquigarrow z_4\}}\right].
    $$
In the remaining of the proof, we adopt the following notation. For $i,j,k,\ell \in\{1,\ldots, 4\}$, we write $\{ijk\ell\}$, $\{ijk|\ell\}$, $\{ij|k\ell\}$ and $\{ij|k|\ell\}$ to denote, respectively, the events $\{z_i\leftrightsquigarrow z_j\leftrightsquigarrow z_k\leftrightsquigarrow z_{\ell}\},$ 
$\{z_i\leftrightsquigarrow z_j\leftrightsquigarrow z_k\not\leftrightsquigarrow z_{\ell}\}$, $\{z_i\leftrightsquigarrow z_j \not\leftrightsquigarrow z_k\leftrightsquigarrow z_{\ell}\}$ and $\{z_k \not\leftrightsquigarrow z_i\leftrightsquigarrow z_j \not\leftrightsquigarrow z_\ell\}.$ 
With this notation, observe that we can write    
\begin{multline*}
1_{\{z_1 \leftrightsquigarrow z_3\}\cup\{z_1 \leftrightsquigarrow z_4\}\cup \{z_2 \leftrightsquigarrow z_3\}\cup \{z_2\leftrightsquigarrow z_4\}}=1_{\{1234\}}+1_{\{123|4\}}+1_{\{124|3\}}+1_{\{134|2\}}+1_{\{234|1\}}\\
+1_{\{13|24\}}+1_{\{14|23\}}+1_{\{13|2|4\}}+1_{\{14|2|3\}}+1_{\{23|1|4\}}+1_{\{24|1|4\}}.
\end{multline*}
Let us denote 
\begin{multline*}
M = \max\{ \P_{\theta}(1234), \P_{\theta}(13|24), \P_{\theta}(14|23), \\
\P_{\theta}(\{134|2\} \cap \{\tau^R_{z_2} < t_3 \vee t_4\}), \P_{\theta}(\{234|1\} \cap \{\tau^R_{z_1} < t_3 \vee t_4\})\}.
\end{multline*}

\paragraph*{Step 1}
Since the random variables $Y_{z_i}$'s and $\hat{Y}_{z_i}$'s are bounded by $1$ almost surely, we clearly have
\begin{equation*}
\left|\E_{\theta}\left[\left(Y_{z_1}Y_{z_2}-\hat{Y}_{z_1}\hat{Y}_{z_2}\right)Y_{z_3}Y_{z_4}\left(1_{\{1234\}}+1_{\{13|24\}}+1_{\{14|23\}}\right)\right]\right| \leq 6M.
\end{equation*}

\paragraph*{Step 2}
Let $E$ denote one of the following events $\{123|4\}$ or $\{124|3\}$. Then, we prove the following inequality
\begin{equation}\label{eq:step2}
\left|\E_{\theta}\left[\left(Y_{z_1}Y_{z_2}-\hat{Y}_{z_1}\hat{Y}_{z_2}\right)Y_{z_3}Y_{z_4}1_E\right]\right|
\leq 2M.
\end{equation}
To do so, let us assume that $E=\{123|4\}$. The other case is treated similarly. By Item (vi) of Lemma \ref{lem:final:independent}, we have that $Y_{z_4}=\tilde{Y}_{z_4}$ on $E$, so that
$$
\E_{\theta}\left[\left(Y_{z_1}Y_{z_2}-\hat{Y}_{z_1}\hat{Y}_{z_2}\right)Y_{z_3}Y_{z_4}1_E\right]=
\E_{\theta}\left[\left(Y_{z_1}Y_{z_2}-\hat{Y}_{z_1}\hat{Y}_{z_2}\right)Y_{z_3}\tilde{Y}_{z_4}1_E\right].
$$
By using that $1_{\{123|4\}}+1_{\{1234\}}=1_{\{123\}}$ is $\sigma(I^{z_1},I^{z_2},I^{z_3})$-measurable and the fact that $\tilde{Y}_{z_4}$ is a centered random variable which is independent of $\left(Y_{z_1}Y_{z_2}-\hat{Y}_{z_1}\hat{Y}_{z_2}\right)Y_{z_3}1_{\{123\}}$ (thanks to Items (ii) and (v) of Lemma \ref{lem:final:independent}), it follows that
$$
\E_{\theta}\left[\left(Y_{z_1}Y_{z_2}-\hat{Y}_{z_1}\hat{Y}_{z_2}\right)Y_{z_3}\tilde{Y}_{z_4}1_{E}\right]=-\E_{\theta}\left[\left(Y_{z_1}Y_{z_2}-\hat{Y}_{z_1}\hat{Y}_{z_2}\right)Y_{z_3}\tilde{Y}_{z_4}1_{\{1234\}}\right].
$$ 
Finally, we conclude like in Step 1.

\paragraph*{Step 3}
Let $E$ denote one of the following events $\{134|2\}$ or $\{234|1\}$. Then, we prove the following inequality
\begin{equation}\label{eq:step3}
\left|\E_{\theta}\left[\left(Y_{z_1}Y_{z_2}-\hat{Y}_{z_1}\hat{Y}_{z_2}\right)Y_{z_3}Y_{z_4}1_E\right]\right|
\leq 5M.
\end{equation}
To do so, let us assume that $E=\{134|2\}$. The other case is treated similarly. In the following, let us denote $E_1 = E \cap \{\tau^R_{z_2} \geq  t_3 \vee t_4\}$ and $E_2 = E \cap \{\tau^R_{z_2} < t_3 \vee t_4\}$. On the one hand, we use the fact that both 
$$
\E_{\theta}\left[\left(Y_{z_1}Y_{z_2}-\hat{Y}_{z_1}\hat{Y}_{z_2}\right)Y_{z_3}Y_{z_4}1_{E_2} \right] \text{ and } \E_{\theta}\left[\left(Y_{z_1}-\hat{Y}_{z_1}\right)\tilde{Y}_{z_2}Y_{z_3}Y_{z_4} 1_{E_2}\right]
$$
are upper bounded by $2\P_{\theta}\left( E_2 \right)$. On the other hand, we know by Items (iii) and (vi) of Lemma \ref{lem:final:independent}, that $\hat{Y}_{z_2}=Y_{z_2}=\tilde{Y}_{z_2}$ on $E_1$. Combining those two properties, we can get
\begin{equation*}
    \left|\E_{\theta}\left[\left(Y_{z_1}Y_{z_2}-\hat{Y}_{z_1}\hat{Y}_{z_2}\right)Y_{z_3}Y_{z_4}1_E\right]\right| \leq \left| \E_{\theta}\left[\left(Y_{z_1}-\hat{Y}_{z_1}\right)\tilde{Y}_{z_2}Y_{z_3}Y_{z_4} 1_{E}\right] \right| + 4\P_{\theta}\left( E_2 \right).
\end{equation*}
By using that $1_{\{134|2\}}+1_{\{1234\}}=1_{\{134\}}$ is $\sigma(I^{z_1},I^{z_3},I^{z_4})$-measurable and the fact that $\tilde{Y}_2$ is a centered random variable which is independent of $\left(Y_{z_1}-\hat{Y}_{z_1}\right)Y_{z_3}Y_{z_4}1_{\{134\}}$ (thanks to Item (v) of Lemma \ref{lem:final:independent}), it follows that
$$
\E_{\theta}\left[\left(Y_{z_1}-\hat{Y}_{z_1}\right)\tilde{Y}_{z_2}Y_{z_3}Y_{z_4}1_{\{134|2\}}\right]=-\E_{\theta}\left[\left(Y_{z_1}-\hat{Y}_{z_1}\right)\tilde{Y}_{z_2}Y_{z_3}Y_{z_4}1_{\{1234\}}\right].
$$ 
Finally, we conclude like in Step 1.

\paragraph*{Step 4}
Let $E$ denote one of the following events: $\{13|2|4\}$, $\{14|2|3\}$, $\{23|1|4\}$ or $\{24|1|3\}$. Then, we prove the following inequality
$$
\left|\E_{\theta}\left[\left(Y_{z_1}Y_{z_2}-\hat{Y}_{z_1}\hat{Y}_{z_2}\right)Y_{z_3}Y_{z_4}1_E\right]\right|
\leq 7M. 
$$
To do so, let us assume that $E=\{13|2|4\}$. The other cases are treated similarly. 
First, we combine the fact that $Y_{z_4}=\tilde{Y}_{z_4}$ 
on $E$ (this holds by Item (vi) of Lemma \ref{lem:final:independent}) together with the decomposition $1_{\{13|2|4\}}=1_{13|2}-1_{13|24}-1_{134|2}$ to deduce that
\begin{multline*}
\E_{\theta}\left[\left(Y_{z_1}Y_{z_2}-\hat{Y}_{z_1}\hat{Y}_{z_2}\right)Y_{z_3}Y_{z_4}1_{\{13|2|4\}}\right]=\E_{\theta}\left[\left(Y_{z_1}Y_{z_2}-\hat{Y}_{z_1}\hat{Y}_{z_2}\right)Y_{z_3}\tilde{Y}_{z_4}1_{\{13|2\}}\right] \\
-\E_{\theta}\left[\left(Y_{z_1}Y_{z_2}-\hat{Y}_{z_1}\hat{Y}_{z_2}\right)Y_{z_3}\tilde{Y}_{z_4}1_{\{13|24\}}\right]-
\E_{\theta}\left[\left(Y_{z_1}Y_{z_2}-\hat{Y}_{z_1}\hat{Y}_{z_2}\right)Y_{z_3}\tilde{Y}_{z_4}1_{\{134|2\}}\right].
\end{multline*}
Now, since $\tilde{Y}_{z_4}$ is centered and independent of $Y_{z_1},\hat{Y}_{z_1}, Y_{z_2}, \hat{Y}_{z_2}, Y_{z_3}$ and $1_{13|2}$ (because this random variable is $\sigma(I^{z_1},I^{z_2},I^{z_3})$-measurable), we have that
$$
\E_{\theta}\left[\left(Y_{z_1}Y_{z_2}-\hat{Y}_{z_1}\hat{Y}_{z_2}\right)Y_{z_3}\tilde{Y}_{z_4}1_{\{13|2\}}\right]=0,
$$
which implies that 
\begin{multline*}
\E_{\theta}\left[\left(Y_{z_1}Y_{z_2}-\hat{Y}_{z_1}\hat{Y}_{z_2}\right)Y_{z_3}Y_{z_4}1_{\{13|2|4\}}\right]=-\E_{\theta}\left[\left(Y_{z_1}Y_{z_2}-\hat{Y}_{z_1}\hat{Y}_{z_2}\right)Y_{z_3}\tilde{Y}_{z_4}1_{\{13|24\}}\right]\\
-
\E_{\theta}\left[\left(Y_{z_1}Y_{z_2}-\hat{Y}_{z_1}\hat{Y}_{z_2}\right)Y_{z_3}\tilde{Y}_{z_4}1_{\{134|2\}}\right].
\end{multline*}
The first term on the right-hand side of the above equality is in absolute value at most $2M$. The second one can be dealt with proceeding similarly as Step 3.  

Combining Steps 1 to 4, we get $|B|\leq 48M$, and Items (iii) and (iv) of Proposition \ref{prop:1} give the upper-bound  $M \le K (1 - \lambda)^{ t_1 \vee t_2 - t_3 \wedge t_4} N^{- 2 },$ which allows to conclude.
\end{proof}

\begin{proof}[Proof of Item 7] 
By symmetry, we only need to consider the case $ t_1 \le t_2 $ and $t_3 \le t_4.$ The case $t _1 \geq t_3 $ or $ t_2 \le t_3$ has already been treated in Item 6. So suppose that either $ t_1 \le t_3 \le t_2 \le t_4$ or that $ t_3 \le t_1 \le t_4 \le t _2.$ Since $\E_\theta[ Y_{z_i} ] = 0$, we can write
\begin{multline}\label{eq:control_cov_for_t1<t3<t2<t4bis}
B=\cov_{\theta}\left[Y_{z_1}Y_{z_3},Y_{z_2}Y_{z_4}\right]+\cov_{\theta}\left[Y_{z_1},Y_{z_3}\right]\cov_{\theta}\left[Y_{z_2},Y_{z_4}\right]\\
-\cov_{\theta}\left[Y_{z_1},Y_{z_2}\right]\cov_{\theta}\left[Y_{z_3},Y_{z_4}\right].
\end{multline}
By Lemma \ref{lem:covariance:Y}, it follows that 
\begin{multline}
\label{ineq:prod_cov_for_t1<t3<t2<t4bis}
|\cov_{\theta}\left[Y_{z_1},Y_{z_3}\right]\cov_{\theta}\left[Y_{z_2},Y_{z_4}\right]|
+|\cov_{\theta}\left[Y_{z_1},Y_{z_2}\right]\cov_{\theta}\left[Y_{z_3},Y_{z_4}\right]|\\
\leq K N^{-2}\left[(1-\lambda)^{|t_1-t_3|+|t_2-t_4|}+(1-\lambda)^{|t_1-t_2|+|t_3-t_4|}\right]  .  
\end{multline}
Also, from Item 6 above, we have that
\begin{equation}
\label{ineq:cov_for_the_product_of_Yt1t3_and_Yt2t4bis}
\cov_{\theta}\left[Y_{z_1}Y_{z_3},Y_{z_2}Y_{z_4}\right]\leq K (1-\lambda)^{\max t_i - \min t_i}N^{-2}.    
\end{equation}
Combing \eqref{eq:control_cov_for_t1<t3<t2<t4bis}, \eqref{ineq:prod_cov_for_t1<t3<t2<t4bis} and \eqref{ineq:cov_for_the_product_of_Yt1t3_and_Yt2t4bis}, we obtain the assertion. 
\end{proof}

\section{Temporal convergences}
\label{app:quenched:rates}

\subsection{Proof of Equation \eqref{eq:quenched:m}}
This one is simple. Since we start from the stationary distribution, we have
$$
\E_\theta [\widehat{m}] = T^{-1} N^{-1} \sum_{t=1}^T \sum_{i=1}^N \E_\theta [X_{i,t}] = N^{-1} \sum_{i=1}^N m^N_{i} = \overline{m^N}.
$$
Hence, the left-hand side of Equation \eqref{eq:quenched:m} is equal to $\var_\theta( \widehat{m} )$.

Then, note that $\widehat{m} = (TN)^{-1} \sum_{t=1}^T \sum_{i=1}^N X_{i,t}$ so that 
\begin{align*}
\var_\theta( \widehat{m} ) &\leq (TN)^{-2} \sum_{t_1,t_2=1}^T \sum_{i_1,i_2=1}^N |\cov_{\theta}(X_{i_1,t_1}, X_{i_2,t_2})|\\
& \leq (TN)^{-1} \left[\frac{2K}{\lambda} + K(1-\lambda) + 1\right],
\end{align*}
where the last inequality comes from Lemma \ref{lem:covariance:Y}.

\subsection{Proof of Equation \eqref{eq:quenched:w}}
\label{app:proof:quenched:w}

\changes{
The proof relies on a decomposition of the error $|\widehat{w}-w^N_{\infty}|$ (see Inequality \eqref{ineq:control_hatw_w_N}). Each term of the decomposition is then controlled by one of the three Propositions \ref{prop:control_DN1}, \ref{prop:control_DN3} or \ref{prop:control_DN2_squared}. Their proofs need intermediate results, which are provided as lemmas. They are linked with martingale properties of the process (Lemmas \ref{Lemma_2_finer_control_of_the_first_term_of_exp_Ut_squared}, \ref{lemma:exp_martingale_differece} and \ref{Lemma_1_finer_control_of_the_first_term_of_exp_Ut_squared}) or with the vanishing covariances (Lemma \ref{lemma:control_cov_barU_squared}).
}

Recall that
$$
\widehat{w} = 2\mathcal{W}_{2\Delta}-\mathcal{W}_{\Delta}, \ \text{with} \  
\mathcal{W}_{\Delta}=\frac{N}{T}\sum_{k=1}^{\lfloor T/\Delta \rfloor}\left(\overline{Z}_{k\Delta}-\overline{Z}_{(k-1)\Delta}-\Delta\hat{m}\right)^2.  
$$
\changes{
As it appears in the analysis below, $w^N_{\infty}$ is the limit of $\mathcal{W}_{\Delta}$ when $T,\Delta\to \infty$. Hence, the choice $\widehat{w} = \mathcal{W}_{\Delta}$ would give a consistent estimator. However, the two estimators $\mathcal{W}_{2\Delta}$ and $\mathcal{W}_{\Delta}$ share a common bias which can be eliminated by considering $\widehat{w} = 2\mathcal{W}_{2\Delta}-\mathcal{W}_{\Delta}$, which in turn drastically improves the rate of convergence. This common bias is related with the quantity $S^N_{\theta,t}$ defined in Lemma \ref{Lemma_2_finer_control_of_the_first_term_of_exp_Ut_squared}.
}

Note that by the triangle inequality, the following inequality holds: 
\begin{equation}
\label{ineq:control_hatw_w_N}
\Big|\widehat{w}-w^N_{\infty}\Big|\leq 2D^{N,1}_{2\Delta,T}+D^{N,1}_{\Delta,T}+2D^{N,2}_{2\Delta,T}+D^{N,2}_{\Delta,T}+D^{N,3}_T,
\end{equation}
where 
\begin{multline}
\label{def:DN1}
D^{N,1}_{\Delta,T}=\frac{N}{T}\left|\sum_{s=1}^{\lfloor T/\Delta \rfloor}\left(\overline{Z}_{s\Delta}
-\overline{Z}_{(s-1)\Delta}-\Delta\widehat{m}\right)^2 \right. \\ \left. 
-\sum_{s=1}^{\lfloor T/\Delta \rfloor}\left(\overline{Z}_{s\Delta}
-\overline{Z}_{(s-1)\Delta}-\E_{\theta}\left[ \overline{Z}_{s\Delta}-\overline{Z}_{(s-1)\Delta}\right]\right)^2\right|,  \end{multline}
\begin{multline}
\label{def:DN2}
D^{N,2}_T=\frac{N}{T}\left|\sum_{s=1}^{\lfloor T/\Delta \rfloor}\left(\overline{Z}_{s\Delta}
-\overline{Z}_{(s-1)\Delta}-\E_{\theta}\left[ \overline{Z}_{s\Delta}-\overline{Z}_{(s-1)\Delta}\right]\right)^2 \right. \\ \left. 
-\E_{\theta}\left[\sum_{s=1}^{\lfloor T/\Delta \rfloor}\left(\overline{Z}_{s\Delta}
-\overline{Z}_{(s-1)\Delta}-\E_{\theta}\left[ \overline{Z}_{s\Delta}-\overline{Z}_{(s-1)\Delta}\right]\right)^2\right]\right|,    
\end{multline}
and
\begin{multline} 
\label{def:DN3}
D^{N,3}_T=\left|\frac{2N}{T}\E_{\theta}\left[\sum_{s=1}^{\lfloor T/2\Delta \rfloor}\left(\overline{Z}_{2s\Delta}
-\overline{Z}_{2(s-1)\Delta}-\E_{\theta}\left[ \overline{Z}_{2s\Delta}-\overline{Z}_{2(s-1)\Delta}\right]\right)^2\right] \right. \\ \left. -\frac{N}{T}\E_{\theta}\left[\sum_{s=1}^{\lfloor T/\Delta \rfloor}\left(\overline{Z}_{s\Delta}
-\overline{Z}_{(s-1)\Delta}-\E_{\theta}\left[ \overline{Z}_{s\Delta}-\overline{Z}_{(s-1)\Delta}\right]\right)^2\right]-w^N_{\infty}\right |. 
\end{multline}

    In the rest of this section, propositions provide upper-bounds for the terms involved in the right-hand side of \eqref{ineq:control_hatw_w_N}, while lemmas provide intermediate results.

\begin{proposition}
\label{prop:control_DN1}
There exists a constant $K$ depending only on $ \lambda$ such that for all $N\geq 1$, $T\geq 2$ and $1\leq \Delta\leq \lfloor T/2\rfloor$, we have that 
\begin{equation*}
\E_{\theta}\left[D^{N,1}_{2\Delta,T}+D^{N,1}_{\Delta,T}\right]\leq K\frac{\Delta}{T}.  \end{equation*}
\end{proposition}
\begin{proof}
Using that $\sum_{s=1}^{\lfloor T/\Delta \rfloor}\left(\overline{Z}_{s\Delta}-\overline{Z}_{(s-1)\Delta}\right)=\lfloor T/\Delta \rfloor \Delta \widehat{m}$ and $\E_{\theta}[\overline{Z}_{\Delta}]=\Delta\E_{\theta}\left[\widehat{m}\right]$, one can check that
$$
D^{N,1}_{\Delta, T}=\frac{N}{T}\Delta^2 \left\lfloor \frac{T}{\Delta} \right\rfloor\Big|\widehat{m}-\E_{\theta}\left[\widehat{m}\right]\Big|^2\leq N\Delta \Big|\widehat{m}-\E_{\theta}\left[\widehat{m}\right]\Big|^2,
$$
which implies that $\E_{\theta}\left[D^{N,1}_{\Delta,T}\right] \le  N\Delta\var_{\theta}(\widehat{m})$. 
Replacing $\Delta$ by $2\Delta$ in the previous identity, we  obtain  $\E_{\theta}\left[D^{N,1}_{2\Delta,T}\right] \leq  2N\Delta\var_{\theta}(\widehat{m})$.
Hence, the result follows from Equation \eqref{eq:quenched:m}.
\end{proof}

To deal with the other two terms, we need to obtain a fine estimate on $\var_{\theta}\left( \overline{Z}_{\Delta} \right) = \E_{\theta}\left[\left(\overline{U}_{\Delta}\right)^2\right]$ (recall \eqref{eq:defu}).  This can be done as follows. Recall that for all $t\geq 1$ and $1\leq i\leq N$, 
$$
M_{i,t} = \sum_{s=1}^t(X_{i,s}-p_{\theta,i}(X_{s-1})), \text{ and } M_{t} = (M_{1,t},\ldots, M_{N,t}).
$$
Let us denote $\mathcal{F}_t=\sigma(X_0,\ldots, X_t)$. From the fact that the process $(X_t)_{t\in\Z}$ is Markovian,  it follows that $(M_{t})_{t\geq 1}$ is a martingale with respect to the filtration $(\mathcal{F}_t)_{t\geq 1}$.

Recall that $Q^N=(I_{N}-(1-\lambda)A^N)^{-1}.$
Then \eqref{eq_for_U_t_1} implies 
\begin{equation}
\label{eq_for_U_t_2}
U_t=Q^NM_t+(Q^N-I_{N})(X_0-X_t), \ t\geq 1.
\end{equation} 

\begin{lemma} 
\label{Lemma_2_finer_control_of_the_first_term_of_exp_Ut_squared}
There exists a constant $K$ depending only on $ \lambda$ such that for all $t\geq 1$ and $N\geq 1$,
$$
\E_{\theta}\left[\left(\overline{U}_t\right)^2\right]=\E_{\theta}\left[\left(\overline{Q^N M_t}\right)^2\right]+S^N_{\theta,t}+r^N_{\theta,t},
$$
where $|r^N_{\theta,t}|\leq K (1-\lambda)^{t}N^{-1}$ and $S^N_{\theta,t}$ has the following properties: $|S^N_{\theta,t}|\leq KN^{-1}$ and $|S^N_{\theta,2t}-S^N_{\theta,t}|\leq K(1-\lambda)^{t}N^{-1}$. 
\end{lemma}

\begin{proof}
Starting from Equation \eqref{eq_for_U_t_2}, one can check that for any $t\geq 1$, 
\begin{equation*}
\E_{\theta}\left[\left(\overline{U}_{t}\right)^2\right]=\E_{\theta}\left[\left(\overline{Q^N M_{t}}\right)^2\right]+ H^N_{\theta,t},
\end{equation*}
where
$$
H^N_{\theta,t}=2\E_{\theta}\left[\overline{Q^N M_{t}} \ \overline{(Q^N-I_{N})(X_0-X_t)}\right]+\E_{\theta}\left[\left(\overline{(Q^N-I_{N})(X_0-X_t)}\right)^2\right].
$$
Hence, to conclude the proof, it remains to show that we can write
$H^N_{\theta,t}=S^N_{\theta,t}+r^N_{\theta,t}$ where
$|r^N_{\theta,t}|\leq K(1-\lambda)^{t}N^{-1}$ and $S^N_{\theta,t}$ is such that $|S^N_{\theta,2t}-S^N_{\theta,t}|\leq K(1-\lambda)^{t}N^{-1}$, for all $t\geq 1$ and $N\geq 1$. To see that, first of all, observe that
\begin{align*}
\overline{(Q^N-I_{N})(X_0-X_t)}&=N^{-1}\sum_{i=1}^N\sum_{j=1}^N(Q^N(i,j)-I_N(i,j))(X_{j,0}-X_{j,t})\\
&=N^{-1}\sum_{j=1}^N(X_{j,0}-X_{j,t})(c^N_j-1),
\end{align*}
where $c^N=(Q^N)^{\intercal}1_N$, so that 
\begin{align*}
\left( \overline{(Q^N-I_{N})(X_0-X_t)} \right)^2&=N^{-2}\sum_{i=1}^N\sum_{j=1}^N(c^N_i-1)(c^N_j-1)(X_{i,0}-X_{i,t})(X_{j,0}-X_{j,t})\\
&=N^{-2}\sum_{i=1}^N\sum_{j=1}^N(c^N_i-1)(c^N_j-1)(Y_{i,0}-Y_{i,t})(Y_{j,0}-Y_{j,t}),
\end{align*}
where the second equality follows from the definition of $Y_{i,t}=X_{i,t}-\E_{\theta}[X_{i,t}]$ and the stationarity of the process $(X_t)_{t\in\Z}$. Using the stationarity once more and the above equation, we then deduce that 
\begin{multline*}
\E_\theta \left[ \left( \overline{(Q^N-I_{N})(X_0-X_t)}\right)^2 \right] = 2N^{-2}\sum_{i=1}^N\sum_{j=1}^N(c^N_i-1)(c^N_j-1)\left[\cov_{\theta}(Y_{i,0},Y_{j,0}) \right. \\ \left. -\cov_{\theta}(Y_{i,0},Y_{j,t})\right],
\end{multline*}
so that by applying Lemma \ref{lem:covariance:Y} and  using Inequality \eqref{omegancons1},  we obtain that
\begin{multline*}
 \E_{\theta}\left[\left(\overline{(Q^N-I_{N})(X_0-X_t)}\right)^2\right]= \frac{2}{N^2}\sum_{i=1}^N\sum_{j=1}^N(c^N_i-1)(c^N_j-1)\cov_{\theta}(Y_{i,0},Y_{j,0})+r^{N}_{\theta,t},
\end{multline*}
where $|r^{N}_{\theta,t}|\leq K(1-\lambda)^t N^{-1}$.

Now, proceeding similarly as above, one can check that  
\begin{align*}
\E_{\theta}\left[\overline{Q^N M_{t}} \ \overline{(Q^N-I_{N})(X_0-X_t)}\right]=N^{-2}\sum_{i=1}^N\sum_{j=1}^Nc^N_i(c^N_j-1)\E_{\theta}\left[M_{i,t}(X_{j,0}-X_{j,t})\right] .
\end{align*}
For each $1\leq i,j\leq N$ and $t\geq 1$, denote
$
S^N_{\theta,t}(i,j)=\E_{\theta}\left[M_{i,t}(X_{j,0}-X_{j,t})\right].
$
We claim that there exists a constant $K>0$ such that for all $1\leq i,j\leq N$ and $t\geq 1$, the following inequality holds: $|S^N_{\theta,t}(i,j)|\leq K(\delta_{ij}+N^{-1})$ and
$$
|S^N_{\theta,2t}(i,j)-S^N_{\theta,t}(i,j)|\leq K(1-\lambda)^tN^{-1}.
$$

Once this claim is proved, using Inequality \eqref{omegancons1} and  Lemma \ref{lem:covariance:Y} it is immediate to check  that $H^N_{\theta,t}=S^N_{\theta,t}+r^N_{\theta,t}$, where
$$
S^N_{\theta,t}=N^{-2}\sum_{i=1}^N\sum_{j=1}^N\left[c^N_i(c^N_j-1)S^N_{\theta,t}(i,j)+2(c^N_i-1)(c^N_j-1)\cov_{\theta}(X_{i,0},Y_{j,0})\right]
$$
satisfies $|S^N_{\theta,2t}-S^N_{\theta,t}|\leq K(1-\lambda)^{t}N^{-1}$ and $|S^N_{\theta,t}|\leq KN^{-1}$, implying the result.

To prove the claim above, we first write
$$
S^N_{\theta,t}(i,j)=\E_{\theta}\left[M_{i,t}(X_{j,0}-X_{j,t})\right]=\sum_{s=1}^t\E_{\theta}\left[(X_{i,s}-p_{\theta,i}(X_{s-1}))(X_{j,0}-X_{j,t})\right]
$$
and then use the fact that 
$$
\E_{\theta}\left[(X_{i,s}-p_{\theta,i}(X_{s-1}))X_{j,0}\right]=\E_{\theta}\left[X_{j,0}\E_{\theta}\left[(X_{i,s}-p_{\theta,i}(X_{s-1}))|\mathcal{F}_{s-1}\right]\right]=0
$$
to deduce that
$$
\E_{\theta}\left[M_{i,t}(X_{j,0}-X_{j,t})\right]=\sum_{s=1}^t\left(\E_{\theta}\left[p_{\theta,i}(X_{s-1})X_{j,t}\right]-\E_{\theta}\left[X_{i,s}X_{j,t}\right]\right).
$$
Next, we use \eqref{def:transition_prob_2} and the stationarity to obtain that  
\begin{multline*}
\sum_{s=1}^t\E_{\theta}\left[p_{\theta,i}(X_{s-1})X_{j,t}\right]=t m^N_j(\mu-(1-\lambda)L^{N,\bullet -}_i)\\
+(1-\lambda)\sum_{k=1}^NA^N(i,k)\sum_{s=1}^t\E_{\theta}\left[X_{k,s-1}X_{j,t}\right].
\end{multline*}
Hence, using that $\E_{\theta}\left[X_{k,s-1}X_{j,t}\right]=\cov_{\theta}(X_{k,s-1},X_{j,t})+m_k^Nm_j^N$ and the stationarity once more, one can check that
\begin{multline*}
\sum_{s=1}^t\E_{\theta}\left[p_{\theta,i}(X_{s-1})X_{j,t}\right]=t m^N_j(\mu-(1-\lambda)L^{N,\bullet -}_i+(1-\lambda)(A^Nm^N)_i)\\
+(1-\lambda)\sum_{k=1}^NA^N(i,k)\sum_{s=1}^t\cov_{\theta}\left(X_{k,0}, X_{j,s}\right) ,
\end{multline*}
and similarly that
$$
\sum_{s=1}^t\E_{\theta}\left[X_{i,s}X_{j,t}\right]=tm_i^Nm^N_j-\sum_{s=0}^{t-1}\cov_{\theta}\left[X_{i,0},X_{j,s}\right].
$$
Therefore, combining the last two equations and using \eqref{equation_for_m_theta_i}, we obtain that
$$
S^N_{\theta,t}(i,j)=(1-\lambda)\sum_{k=1}^NA^N(i,k)\sum_{s=1}^t \cov_{\theta}\left(X_{k,0}, X_{j,s}\right) +\sum_{s=0}^{t-1} \cov_{\theta}\left(X_{i,0}, X_{j,s}\right),
$$
so that Lemma \ref{lem:covariance:Y} allows us to conclude that
\begin{multline*}
\left|S^N_{\theta,2t}(i,j)-S^N_{\theta,t}(i,j)\right| \le \\
\leq (1-\lambda)\sum_{k=1}^N|A^N(i,k)|\sum_{s=t+1}^{2t} |\cov_{\theta}\left(X_{k,0}, X_{j,s}\right)|
+\sum_{s=t}^{2t-1} |\cov_{\theta}\left(X_{i,0}, X_{j,s}\right)|
\\
\leq \frac{K(1-\lambda)^{t}}{N}\left(\sum_{s=0}^{t-1}(1-\lambda)^s\right)
\left[(1-\lambda)^2 N \tn A^N\tn_{\infty}+1\right] \leq \frac{K(1-\lambda)^{t}}{N}.
\end{multline*}
Finally, using once more Lemma \ref{lem:covariance:Y}, one can easily check that $|S^N_{\theta,t}(i,j)|\leq K ( \delta_{ij} + N^{-1}) $ proving the claim.
\end{proof}

The next result will be important to find the leading term of $\E_{\theta}\left[\left(\overline{Q^N M_{t}}\right)^2\right]$.   
\begin{lemma}
\label{lemma:exp_martingale_differece}
For all $t\geq 1,$ the following equality holds.
 $$
\E_{\theta}\left[(X_{t}-p_{\theta}(X_{t-1}))^2\right]=m^N-(m^N)^2+r^N_{\theta},
 $$
 where $\|r^N_{\theta}\|_{\infty}\leq KN^{-1}$ for some constant $K>0$ depending only on $\lambda$.
\end{lemma}
\begin{remark}
Note that $X_{t}$, $p_{\theta}(X_{t-1})$ and $m^N$ are vectors so that the square terms in Lemma \ref{lemma:exp_martingale_differece} have to be interpreted in terms of the Hadamard product of vectors. Also, the remainder term $r^N_{\theta}$ does not depend on time $t$ because the process $(X_t)_{t\in\Z}$ is assumed to be stationary.  
\end{remark}

\begin{proof}
 First, observe that
 \begin{eqnarray*}
 \E_{\theta}\left[(X_{i,t}-p_{i,\theta}(X_{t-1}))^2\right]&=&\E_{\theta}\left[X_{i,t}\right]-\E_{\theta}\left[p^2_{i, \theta}(X_{t-1})\right]\\
 &=&m^N_{i}-\E_{\theta}\left[p^2_{i, \theta}(X_0)\right],
\end{eqnarray*}
 where the second equality follows by the stationarity. Now, from \eqref{def:transition_prob_2}, we have that
 $$
p_{i, \theta}(X_0)=\left(\mu-(1-\lambda)L^{N,\bullet-}_{i}\right)+(1-\lambda)(A^NX_0)_{i},
 $$
 so that 
 \begin{multline}
 \label{eq1_proof_Xi_minus_pi}
 \E_{\theta}\left[p^2_{i, \theta}(X_0)\right]=(\mu-(1-\lambda)L^{N,\bullet-}_{i})^2+2(\mu-(1-\lambda)L^{N,\bullet-}_{i})(1-\lambda)(A^N\E_{\theta}\left[X_0\right])_{i}\\+(1-\lambda)^2\E_{\theta}\left[((A^NX_0)_{i})^2\right].    
 \end{multline}
Next, notice that 
\begin{equation*}
\E_{\theta}\left[((A^N_{\theta}X_0)_{i})^2\right] = \sum_{j=1}^{N} \sum_{k=1}^{N} A^N(i,j)A^N(i,k)\E_{\theta}\left[X_{k,0}X_{j,0}\right],
\end{equation*}
so that using that $\E_{\theta}\left[X_{k,0}X_{j,0}\right]=m^N_jm^N_k+\cov_{\theta}(X_{k,0},X_{j,0})$, we deduce that 
\begin{equation}
\label{eq2_proof_Xi_minus_pi}
\E_{\theta}\left[((A^NX_0)_i)^2\right]=\left(\sum_{j=1}^{N}A^N(i,j)m^N_{j}\right)^2+\widetilde{r}^N_{\theta,i}    
\end{equation}
 where
 $$
\widetilde{r}^N_{\theta,i}= \sum_{j=1}^{N}\sum_{k=1}^{N} A^N(i,j)A^N(i,k)\cov_{\theta}(X_{k,0},X_{j,0}).
 $$
Using the fact that $|A^N(i,j)| \leq 1/N$ and Lemma \ref{lem:covariance:Y}, we have
\begin{equation*}
|\widetilde{r}^N_{\theta,i}| \leq  N^{-2} \sum_{j=1}^{N}\sum_{k=1}^{N} |\cov_{\theta}[X_{k,0},X_{j,0}]| \leq N^{-2} (N + N^2 K (1-\lambda) N^{-1}) \leq KN^{-1}.
 \end{equation*}
Finally, recalling that $m^N=\E_{\theta}\left[X_{0}\right]$, it follows from \eqref{eq1_proof_Xi_minus_pi} and \eqref{eq2_proof_Xi_minus_pi} that
$$
\E_{\theta}\left[p^2_{i, \theta}(X_0)\right]=\left[(\mu-(1-\lambda)L^{N,\bullet-}_{i})+(1-\lambda)(A^Nm^N)_i\right]^2 + (1-\lambda)^2\widetilde{r}^N_{\theta,i}.
$$
The result then follows from equation  \eqref{eq:expression:m:theta}.
\end{proof}

As a consequence of Lemma \ref{lemma:exp_martingale_differece}, we will now find the leading term of $\E_{\theta}\left[\left(\overline{Q^N M_{t}}\right)^2\right]$.

\begin{lemma} 
\label{Lemma_1_finer_control_of_the_first_term_of_exp_Ut_squared}
 For any $t\geq 1$ and $N\geq 1$,
 \begin{equation*}
\E_{\theta}\left[\Big|\overline{Q^NM_{t}}\Big|^2\right]=\frac{t}{N}w^N_{\infty}+r^N_{\theta,t} ,    
 \end{equation*}
where $w^N_{\infty}$ is defined in \eqref{eq:def:m:v:w:infty} and $r^N_{\theta,t}$ is such that $|r^N_{\theta,t}|\leq KtN^{-2}$ for some $K$ depending only on $\lambda$.
\end{lemma}
\begin{proof}
We start observing that 
$$
\overline{Q^NM_{t}}=N^{-1}\sum_{i=1}^{N}\sum_{j=1}^{N}Q^N(i,j)M_{j,t}=
N^{-1}\sum_{j=1}^{N}c^N_{j}M_{j,t},
$$
and then we use that for all $s\neq h$ and $1\leq j,k\leq N$,
$$
\E_{\theta}\left[(X_{i,s}-p_{i,\theta}(X_{s-1}))(X_{k,h}-p_{k, \theta}(X_{h-1}))\right]=0,
$$ 
to deduce that
\begin{eqnarray*}
\E_{\theta}\left[\Big|\overline{Q^NM_{t}}\Big|^2\right] &=&N^{-2}\sum_{j=1}^{N}\sum_{k=1}^{N}c^N_{j}c^N_{k}
\E_{\theta}\left[M_{j,t}M_{k,t}\right]\\
&=& N^{-2}\sum_{j=1}^{N}(c^N_{j})^2
\sum_{s=1}^t\E_{\theta}\left[\left(X_{j,s}-p_{j, \theta}(X_{s-1})\right)^2\right].
\end{eqnarray*}
Hence, it follows from Lemma \ref{lemma:exp_martingale_differece} that
 \begin{eqnarray*}
\E_{\theta}\left[\Big|\overline{Q^NM_{t}}\Big|^2\right]&=&N^{-2}\sum_{j=1}^{N}(c^N_{j})^2\sum_{s=1}^t\E_{\theta}\left[\left(X_{j,s}-p_{j, \theta}(X_{s-1})\right)^2\right]\\
&=& tN^{-2}\sum_{j=1}^{N}(c^N_{j})^2m^N_{j}(1-m^N_{j})+tN^{-2}\langle (c^N)^2,r^N_{\theta}\rangle\\
&=&tN^{-1}w^N_{\infty}+tN^{-2}\langle (c^N)^2,r^N_{\theta}\rangle.
 \end{eqnarray*}
Finally, by applying Hölder inequality and then using that $\|v\|_{1}\leq N\|v\|_{\infty}$ for $v\in\R^N$, we obtain
$$
\langle (c^N)^2,r^N_{\theta}\rangle\leq \|(c^N)^2\|_{\infty}\|r^N_{\theta}\|_{1}\leq \|(c^N)^2\|_{\infty}N\|r^N_{\theta}\|_{\infty},
$$
and the result follows from inequality \eqref{omegancons1} and Lemma  \ref{lemma:exp_martingale_differece}.   
\end{proof}

For later use, let us mention the following immediate corollary of Lemmas \ref{Lemma_2_finer_control_of_the_first_term_of_exp_Ut_squared} and \ref{Lemma_1_finer_control_of_the_first_term_of_exp_Ut_squared}. 
\begin{lemma}
\label{Lemma_bound_on_the_exp_of_Ut_squared}
There exists a constant $K>0$ depending only on $ \lambda$ such that for all $t\geq 1$ and $N\geq 1$,
$$
\E_{\theta}\left[\left(\overline{U}_t\right)^2\right]\leq KtN^{-1}.
$$
\end{lemma}
\begin{proof}
First, observe that $w^N_{\infty}$ defined in \eqref{eq:def:m:v:w:infty} satisfies $|w^N_{\infty}|\leq \lambda^{-2}/4$ by
Inequality \eqref{omegancons1}, and then combine this fact with Lemmas \ref{Lemma_2_finer_control_of_the_first_term_of_exp_Ut_squared} and \ref{Lemma_1_finer_control_of_the_first_term_of_exp_Ut_squared}.
\end{proof}

We are now in position to deal with the term $D^{N,3}_{T}$ defined in \eqref{def:DN3}.
\begin{proposition}
\label{prop:control_DN3}
There exists a constant $K>0$ depending only on $ \lambda$ such that for all $N\geq 1$, $T\geq 2$ and $1\leq \Delta\leq \lfloor T/2\rfloor$, 
\begin{equation*}
\E_{\theta}\left[D^{N,3}_{T}\right]\leq K\left(\frac{1}{N}+\frac{(1-\lambda)^{\Delta}}{\Delta}+\frac{\Delta}{T}\right).  
\end{equation*}
\end{proposition}
\begin{proof}
First, observe that by stationarity,
 $$
D^{N,3}_{T}=\Big|\frac{2N}{T} \lfloor T/{2\Delta} \rfloor \E_{\theta}\left[(\overline{U}_{2\Delta})^2\right]-\frac{N}{T} \lfloor T/{\Delta} \rfloor \E_{\theta}\left[(\overline{U}_{\Delta})^2\right]-w^N_{\infty}\Big|.
$$
Then, by using the triangle inequality, one can check that 
\begin{multline*}
\Big|\frac{2N}{T} \lfloor T/{2\Delta} \rfloor \E_{\theta}\left[(\overline{U}_{2\Delta})^2\right]-\frac{N}{T} \lfloor T/{\Delta} \rfloor \E_{\theta}\left[(\overline{U}_{\Delta})^2\right]-w^N_{\infty}\Big|\\
\leq \Big|\frac{N}{\Delta}\left(\E_{\theta}\left[(\overline{U}_{2\Delta})^2\right]-\E_{\theta}\left[(\overline{U}_{\Delta})^2\right]\right)-w^N_{\infty}\Big|+\frac{2N}{T}\left(\E_{\theta}\left[(\overline{U}_{2\Delta})^2\right]+\E_{\theta}\left[(\overline{U}_{\Delta})^2\right]\right). \end{multline*}
Then Lemmas \ref{Lemma_2_finer_control_of_the_first_term_of_exp_Ut_squared}, \ref{Lemma_1_finer_control_of_the_first_term_of_exp_Ut_squared} and \ref{Lemma_bound_on_the_exp_of_Ut_squared} imply the result.
\end{proof}

It remains to deal with $D^{N,2}_{\Delta,T}$ and $D^{N,2}_{2\Delta,T}$. To that end, we shall use two additional lemmas. The first one is the following.

\begin{lemma}
\label{lemma:control_cov_barU_squared}
There exists a constant $K>0$ depending only on $ \lambda$ such that for all $N\geq 1$, $\Delta\geq 1$ and $t\geq 2,$
\begin{equation*}
\cov_{\theta}\left[\left(\overline{U}_{\Delta}\right)^2,\left(\overline{U}_{t\Delta}-\overline{U}_{(t-1)\Delta}\right)^2\right]\leq K N^{-2} (1-\lambda)^{(t-2)\Delta}.  
\end{equation*}
\end{lemma}
\begin{proof}
First, use that $\overline{U}_t=N^{-1}\sum_{i=1}^{N}\sum_{s=1}^t Y_{{i,s}}$ to deduce that
$$
\cov_{\theta}\left[\left(\overline{U}_{\Delta}\right)^2,\left(\overline{U}_{t\Delta}-\overline{U}_{(t-1)\Delta}\right)^2\right]=N^{-4}\sum_{z_1\in F_1}\sum_{z_2\in F_1}\sum_{z_3\in F_t}\sum_{z_4\in F_t}\cov_{\theta}\left[Y_{z_1}Y_{z_2},Y_{z_3}Y_{z_4}\right],
$$
where $z_k=(i_k,t_k)$ for $k\in\{1,\ldots,4\}$, and $F_{s}:=\{ 1, \ldots, N \} \times \{(s-1)\Delta+1,\ldots, s\Delta\}$ for $s\geq 1$. In the rest of the proof, we shall denote ${\bf z}=(z_1,z_2,z_3,z_4)$ and $B({\bf z})=\cov_{\theta}\left[Y_{z_1}Y_{z_2},Y_{z_3}Y_{z_4}\right].$ 
Notice that the covariance function $B({\bf z})$ is symmetric with respect to $z_1$ and $z_2$, as well as with respect to $z_3$ and $z_4$. For each $p\in \{1,\ldots, 6\}$ and $s\geq 1$, we denote $\mathcal{C}_{p,s}$ the set of all vectors ${\bf z}\in F_1\times F_1\times F_s\times F_s$ satisfying the conditions of Item $p$ of Lemma \ref{lem:covariance:produit} and such that $t_1\leq t_2$ and $t_3\leq t_4$. 
With this notation, one can check using the symmetry of the covariance function that the following inequality holds: 
\begin{equation}
\label{eq_cov_U_in_terms_of_Ys}
\cov_{\theta}\left[\left(\overline{U}_{\Delta}\right)^2,\left(\overline{U}_{t\Delta}-\overline{U}_{(t-1)\Delta}\right)^2\right]\leq 4N^{-4}\sum_{p=1}^{6}\sum_{{\bf z}\in \mathcal{C}_{p,t}}|B({\bf z})|.
\end{equation}

Notice that $\mathcal{C}_{1,t}=\mathcal{C}_{2,t}=\mathcal{C}_{5,t}=\emptyset$ for $t\geq 2$, so that we need only to control the sum of covariances for the cases $p\in\{3,4,6\}$ . This is done in the 3 steps below. 

\medskip
{\it Step 1.} Here, we show that there exists a constant $K>0$ such that for all $N\geq 1$, $\Delta\geq 1$ and $t\geq 2$,  
$$
\sum_{{\bf z}\in \mathcal{C}_{3,t}} |B({\bf z})|\leq KN(1-\lambda)^{(t-2)\Delta+1}. 
$$
First, note that Item 3 of Lemma \ref{lem:covariance:produit} implies that
\begin{align*}
\sum_{{\bf z}\in \mathcal{C}_{3,t}} |B({\bf z})| &\leq K N^{2}\sum_{t_1=1}^{\Delta}\sum_{t_3=(t-1)\Delta+1}^{t\Delta}(1-\lambda)^{(t_3-t_1)}N^{-1}\\
&=KN(1-\lambda)^{(t-2)\Delta+1}\sum_{k=0}^{\Delta-1}(1-\lambda)^{k}\sum_{t_1=1}^{\Delta}(1-\lambda)^{\Delta-t_1}.
\end{align*}
Step 1 follows by noticing that $\sum_{t_1=1}^{\Delta}(1-\lambda)^{\Delta-t_1}=\sum_{k=0}^{\Delta-1}(1-\lambda)^k,$ and then by using that $\sum_{k=0}^{\Delta-1}(1-\lambda)^{k}\leq \sum_{k=0}^{\infty}(1-\lambda)^{k}\leq \lambda^{-1}$.

\medskip
{\it Step 2.} Here, we prove that there exists a constant $K>0$ such that for all $N\geq 1$, $\Delta\geq 1$ and $t\geq 2$,  
$$
\sum_{{\bf z}\in \mathcal{C}_{4,t}} |B({\bf z})|\leq KN(1-\lambda)^{(t-2)\Delta+1}. 
$$

First, by Remark \ref{rmk:case_4}, we can write
$$
\sum_{{\bf z}\in \mathcal{C}_{4,t}} |B({\bf z})|=\sum_{{\bf z}\in \mathcal{C}_{4,t}} \indiq_{\{z_1=z_2,z_3\neq z_4\}}|B({\bf z})|+\sum_{{\bf z}\in \mathcal{C}_{4,t}} \indiq_{\{z_1\neq z_2,z_3=z_4\}} |B({\bf z})|.
$$
To conclude, we apply Item 4 of Lemma \ref{lem:covariance:produit} to obtain an upper bound for each one of the terms on the right hand side of the above identity. Since both terms are treated very similarly, we explain how we handle the first term only. By applying Item 4 of Lemma \ref{lem:covariance:produit} to the first one, we obtain that
$$
\sum_{{\bf z}\in \mathcal{C}_{4,t}} \indiq_{\{z_1=z_2,z_3\neq z_4\}}|B({\bf z})|\leq KN^3\sum_{t_1=1}^{\Delta}\sum_{t_3=(t-1)\Delta+1}^{t\Delta}\sum_{t_4=t_3}^{t\Delta}(1-\lambda)^{(t_4-t_3)+(t_3-t_1)}N^{-2}.
$$
Now, since for any $(t-1)\Delta+1\leq t_3\leq t\Delta$,
\begin{equation}
\label{ineq_geom_sum_over_t4}
\sum_{t_4=t_3}^{t\Delta}(1-\lambda)^{(t_4-t_3)}=\sum_{k=0}^{t\Delta-t_3}(1-\lambda)^{k}\leq \sum_{k=0}^{\Delta-1}(1-\lambda)^{k}\leq \lambda^{-1}, 
\end{equation}
we can then proceed as in Step 1 to conclude that
$$
\sum_{{\bf z}\in \mathcal{C}_{4,t}} \indiq_{\{z_1=z_2,z_3\neq z_4\}}|B({\bf z})|\leq KN(1-\lambda)^{(t-2)\Delta+1}.
$$

\medskip
{\it Step 3.} Here, we prove that there exists a constant $K>0$ such that for all $N\geq 1$, $\Delta\geq 1$ and $t\geq 2$,  
$$
\sum_{{\bf z}\in \mathcal{C}_{6,t}} |B({\bf z})|\leq K N^2(1-\lambda)^{(t-2)\Delta+1}.
$$

By applying Item 6 of Lemma \ref{lem:covariance:produit} and inequality \eqref{ineq_geom_sum_over_t4}, one can check that
$$
\sum_{{\bf z}\in \mathcal{C}_{6,t}} |B({\bf z})| \leq KN^4\sum_{t_1=1}^{\Delta}\sum_{t_2=t_1}^{\Delta}\sum_{t_3=(t-1)\Delta+1}^{t\Delta} N^{-2}(1-\lambda)^{t_3-t_1}.
$$
Some algebraic computations imply that
\begin{multline*}
\sum_{t_1=1}^{\Delta}\sum_{t_2=t_1}^{\Delta}\sum_{t_3=(t-1)\Delta+1}^{t\Delta}(1-\lambda)^{t_3-t_1}\leq (1-\lambda)^{(t-2)\Delta+1}\left(\sum_{k=0}^{\Delta-1}(1-\lambda)^k\right)\\ 
\times \left(\sum_{k=0}^{\Delta-1}(k+1)(1-\lambda)^{k}\right).    
\end{multline*}
Putting together these estimates and using the fact that $\sum_{k=0}^{\infty}(k+1)(1-\lambda)^k<\infty$, we conclude the proof of Step 3.

Finally, combining Steps 1 through 3 and \eqref{eq_cov_U_in_terms_of_Ys}, the result follows.
\end{proof}

\begin{lemma}
\label{lemma:control_var_barU_squared}
There exists a constant $K$ depending only on $ \lambda, \mu $ and $ p,$ such that for all $N\geq 1$ and $\Delta\geq 1$,
\begin{equation*}
\var_{\theta}[(\overline{U}_{\Delta})^2]\leq K\Delta^2N^{-2}. 
\end{equation*}
\end{lemma}
\begin{proof}
We use the notation of the proof of Lemma \ref{lemma:control_cov_barU_squared}, except that we now denote $ {\mathcal{C}_{1,6}}$ the set of all vectors $ \bf z  \in F_1 \times F_1 \times F_1 \times F_1$ satisfying the conditions of Item 7. of Lemma \ref{lem:covariance:produit}, that is, without imposing the constraint that $ t_1 > t_4 $ or $ t_3 > t_2$. By taking $t=1$ in inequality \eqref{eq_cov_U_in_terms_of_Ys}, it follows that 
\begin{equation}
\label{eq_var_U_squared_in_terms_of_Ys}
\var_{\theta}\left[\left(\overline{U}_{\Delta}\right)^2\right]\leq 4N^{-4} \left( \sum_{p=1}^{6}\sum_{{\bf z}\in \mathcal{C}_{1,p}}|B({\bf z})|  \right) .
\end{equation}
So we need to show that the sum of covariances is at most $K\Delta^2N^2$. This is done in the 5 steps below.

\medskip
{\it Step 1.} First, we show that there exists a constant $K>0$ such that for all $N\geq 1$ and $\Delta\geq 1$,
$$
\sum_{p=1}^{2}\sum_{{\bf z}\in \mathcal{C}_{1,p}}|B({\bf z})|\leq K (\Delta N)^2.
$$
By Item 1 and Item 2 of Lemma \ref{lem:covariance:produit}, we have that
$$
\sum_{p=1}^{2}\sum_{{\bf z}\in \mathcal{C}_{1,p}}|B({\bf z})|\leq |\mathcal{C}_{1,1}|+|\mathcal{C}_{1,2}|\leq \Delta N+2\Delta N(\Delta N-1)\leq 3(\Delta N)^2,
$$
so that the result follows with $K=3$.

\medskip
{\it Step 2.} Next, we show that there exists a constant $K>0$ such that for all $N\geq 1$ and $\Delta\geq 1$,
$$
\sum_{{\bf z}\in \mathcal{C}_{1,3}}|B({\bf z})|\leq K N\Delta^2.
$$
Item 3 of Lemma \ref{lem:covariance:produit}, ensures that
$$
\sum_{{\bf z}\in \mathcal{C}_{1,3}}|B({\bf z})|\leq K N^2\sum_{t_1=1}^{\Delta}\sum_{t_3=1}^{\Delta}(1-\lambda)^{|t_3-t_1|}N^{-1}.
$$
Since $\sum_{t_1=1}^{\Delta}\sum_{t_3=1}^{\Delta}(1-\lambda)^{|t_3-t_1|}\leq \Delta^2$, the result follows from the previous inequality.

\medskip
{\it Step 3.} Here, we show that there exists a constant $K>0$ such that for all $N\geq 1$ and $\Delta\geq 1$,
$$
\sum_{{\bf z}\in \mathcal{C}_{1,4}}|B({\bf z})|\leq K N\Delta^2.
$$
We proceed very similarly as in the proof of Step 2 of Lemma \ref{lemma:control_cov_barU_squared}. First, observe that we can write (see Remark \ref{rmk:case_4}),
$$
\sum_{{\bf z}\in \mathcal{C}_{1,4}} |B({\bf z})|=\sum_{{\bf z}\in \mathcal{C}_{1,4}} \indiq_{\{z_1=z_2,z_3\neq z_4\}}|B({\bf z})|+\sum_{{\bf z}\in \mathcal{C}_{1,4}} \indiq_{\{z_1\neq z_2,z_3=z_4\}} |B({\bf z})|.
$$
To obtain the result, we apply Item 4 of Lemma \ref{lem:covariance:produit} to obtain an upper bound for each term on the right hand side of the above identity. Since both terms are treated very similarly, we explain how we deal with the first term only. By applying Point 4 of Lemma \ref{lem:covariance:produit} to the first one, we obtain that
$$
\sum_{{\bf z}\in \mathcal{C}_{1,4}} \indiq_{\{z_1=z_2,z_3\neq z_4\}}|B({\bf z})|\leq KN^3\sum_{t_1=1}^{\Delta}\sum_{t_3=1}^{\Delta}\sum_{t_4=1}^{\Delta}\indiq_{\{t_3\leq t_4\}}(1-\lambda)^{|s_3-s_2|+|s_2-s_1|}N^{-2},
$$
where $s_1\leq s_2\leq s_3$ denotes the ordering of the triple $(t_1,t_3,t_4)$. Next, we use that $|s_3-s_2|+|s_2-s_1|\geq t_4-t_3$ to deduce that   
\begin{align*}
\sum_{t_1=1}^{\Delta}\sum_{t_3=1}^{\Delta}\sum_{t_4=1}^{\Delta}\indiq_{\{t_3\leq t_4\}}(1-\lambda)^{|s_3-s_2|+|s_2-s_1|}&\leq \Delta \sum_{t_3=1}^{\Delta}\sum_{t_4=1}^{\Delta}\indiq_{\{t_3\leq t_4\}}(1-\lambda)^{t_4-t_3}\\
& \leq \Delta^2\sum_{k=0}^{\Delta -1}(1-\lambda)^k\leq \lambda^{-1}\Delta^2.
\end{align*}
As a consequence, we obtain that
$$
\sum_{{\bf z}\in \mathcal{C}_{1,4}} \indiq_{\{z_1=z_2,z_3\neq z_4\}}|B({\bf z})|\leq KN\lambda^{-1}\Delta^2,
$$
and the result follows.

\medskip
{\it Step 4.} In this step, we show that there exists a constant $K>0$ such that for all $N\geq 1$ and $\Delta\geq 1$,
$$
\sum_{{\bf z}\in \mathcal{C}_{1,5}}|B({\bf z})|\leq K N^2\Delta^2.
$$
We start observing that
$$
\sum_{{\bf z}\in \mathcal{C}_{1,5}} |B({\bf z})|=\sum_{{\bf z}\in \mathcal{C}_{1,5}} \indiq_{\{z_1=z_3,z_2\neq z_4\}}|B({\bf z})|+\sum_{{\bf z}\in \mathcal{C}_{1,5}} \indiq_{\{z_1\neq z_3,z_2=z_4\}} |B({\bf z})|.
$$
Hence, it suffices to provide an upper bound for each term on the right hand side of the above inequality. We will explain how we deal with the first one only. The second one can be treated similarly.
By Item 5 of Lemma \ref{lem:covariance:produit}, we have that 
$$
\sum_{{\bf z}\in \mathcal{C}_{1,5}} \indiq_{\{z_1=z_3,z_2\neq z_4\}}|B({\bf z})|\leq KN^3\sum_{t_1=1}^{\Delta}\sum_{t_2=1}^{\Delta}\sum_{t_4=1}^{\Delta}\indiq_{\{t_2\leq t_4\}}(1-\lambda)^{(t_4-t_2)}N^{-1}.
$$
Hence, by observing that
$$
\sum_{t_1=1}^{\Delta}\sum_{t_2=1}^{\Delta}\sum_{t_4=1}^{\Delta}\indiq_{\{t_2\leq t_4\}}(1-\lambda)^{(t_4-t_2)}\leq \Delta^2\sum_{k=0}^{\Delta -1}(1-\lambda)^{k}\leq \Delta^2\lambda^{-1},
$$
the result follows.

\medskip
{\it Step 5.} Finally, we show that there exists a constant $K>0$ such that for all $N\geq 1$ and $\Delta\geq 1$,
$$
\sum_{{\bf z}\in {\mathcal{C}_{1,6}}}|B({\bf z})|\leq K N^2\Delta^2.
$$
We start from
\begin{multline*}
\sum_{{\bf z}\in {\mathcal{C}_{1,6}}}|B({\bf z})| \leq 
\sum_{{\bf z}\in {\mathcal{C}_{1,6}}}\left[\indiq_{\{t_3\leq t_4\leq t_1\leq t_2\}}+\indiq_{\{t_1\leq t_2\leq t_3\leq t_4\}}\right]|B({\bf z})|
+\sum_{{\bf z}\in {\mathcal{C}_{1,6}}}\left[\indiq_{\{t_1\leq t_3\leq t_2\leq t_4\}}\right.\\ \left.+\indiq_{\{t_3\leq t_1\leq t_4\leq t_2\}} + \indiq_{\{t_1\leq t_3\leq t_4\leq t_2\}}+\indiq_{\{t_3\leq t_1\leq t_2\leq t_4\}}\right]|B({\bf z})|.     
\end{multline*}
In the sequel, we handle the two terms on the right-hand side of the above inequality separately.

\medskip
{\it Step 5.1} Here, we show that there exists a constant $K>0$ such that for all $N\geq 1$ and $\Delta\geq 1$,
$$
\sum_{{\bf z}\in {\mathcal{C}_{1,6}}}\left[\indiq_{\{t_3\leq t_4\leq t_1\leq t_2\}}+\indiq_{\{t_1\leq t_2\leq t_3\leq t_4\}}\right]|B({\bf z})|\leq K N^2\Delta.
$$
By Item 6 of Lemma \ref{lem:covariance:produit},
\begin{multline*}
\sum_{{\bf z}\in {\mathcal{C}_{1,6}}}\indiq_{\{t_3\leq t_4\leq t_1\leq t_2\}}|B({\bf z})|\leq KN^4\sum_{t_1=1}^{\Delta}\sum_{t_2=1}^{\Delta}\sum_{t_3=1}^{\Delta}\sum_{t_4=1}^{\Delta}\indiq_{\{t_3\leq t_4\leq t_1\leq t_2\}}\times \\
(1-\lambda)^{t_2-t_3}N^{-2} .
\end{multline*}
We upper bound
\begin{equation*}
\sum_{t_1=1}^{\Delta}\sum_{t_2=1}^{\Delta}\sum_{t_3=1}^{\Delta}\sum_{t_4=1}^{\Delta}\indiq_{\{t_3\leq t_4\leq t_1\leq t_2\}}(1-\lambda)^{t_2-t_3}\leq \Delta\left(\sum_{k=0}^{\Delta-1}(1-\lambda)^k\right)^3 ,    
\end{equation*}
such that 
$$
\sum_{{\bf z}\in {\mathcal{C}_{1,6}}}\indiq_{\{t_3\leq t_4\leq t_1\leq t_2\}}|B({\bf z})|\leq KN^2\Delta.
$$
Proceeding similarly, we can also check that $\sum_{{\bf z}\in {\mathcal{C}_{1,6}}}\indiq_{\{t_3\leq t_4\leq t_1\leq t_2\}}|B({\bf z})|\leq KN^2\Delta$. This concludes the proof of Step 5.1

\medskip
{\it Step 5.2} Here we show that there exists a constant $K>0$ such that for all $N\geq 1$ and $\Delta\geq 1$,
$$
\sum_{{\bf z}\in {\mathcal{C}_{1,6}}}\left[\indiq_{\{t_1\leq t_3\leq t_2\leq t_4\}}+\indiq_{\{t_3\leq t_1\leq t_4\leq t_2\}}+ \indiq_{\{t_1\leq t_3\leq t_4\leq t_2\}}+\indiq_{\{t_3\leq t_1\leq t_2\leq t_4\}}\right]|B({\bf z})|\leq K N^2\Delta^2.
$$
From Item 7 of Lemma \ref{lem:covariance:produit}, we deduce that
$$
\sum_{{\bf z}\in {\mathcal{C}_{1,6}}}\indiq_{\{t_1\leq t_3\leq t_2\leq t_4\}}|B({\bf z})|\leq K\sum_{{\bf z}\in {\mathcal{C}_{1,6}}}\indiq_{\{t_1\leq t_3\leq t_2\leq t_4\}}(1-\lambda)^{(t_4-t_2) + (t_3-t_1)}N^{-2} . 
$$
Since 
$$
\sum_{t_1=1}^{\Delta}\sum_{t_2=1}^{\Delta}\sum_{t_3=1}^{\Delta}\sum_{t_4=1}^{\Delta}\indiq_{\{t_1\leq t_3\leq t_2\leq t_4\}}(1-\lambda)^{(t_4-t_2)+(t_3-t_1)}\leq \Delta^2\left(\sum_{k=0}^{\Delta-1}(1-\lambda)^k\right)^2,
$$
it then follows that
$$
\sum_{{\bf z}\in {\mathcal{C}_{1,6}}}\indiq_{\{t_1\leq t_3\leq t_2\leq t_4\}}(1-\lambda)^{(t_1-t_3)+(t_2-t_4)}N^{-2}\leq N^2\Delta^2\lambda^{-2}
$$
such that
$$
\sum_{{\bf z}\in {\mathcal{C}_{1,6}}}\indiq_{\{t_1\leq t_3\leq t_2\leq t_4\}}|B({\bf z})|\leq KN^2  \Delta^2.
$$
Proceeding similarly, we can also check the three other cases, concluding the proof of Step 5.2.

Combining Steps 5.1, 5.2 and 5.3, we obtain 
$$
\sum_{{\bf z}\in {\mathcal{C}_{1,6}}}|B({\bf z})|\leq K N^2\Delta^2,
$$
concluding the proof of Step 5.

By summing all the upper bounds provided by the Steps 1 through 5, we conclude the proof of Lemma \ref{lemma:control_var_barU_squared}.

\end{proof}

We are now in position to handle the terms $D^{N,2}_{\Delta,T}$ and $D^{N,2}_{2\Delta,T}$.

\begin{proposition}
\label{prop:control_DN2_squared}
There exists a constant $K$ depending only on $ \lambda$ such that for all  $N\geq 1$, $T\geq 2$ and $1\leq \Delta\leq \lfloor T/2\rfloor$, we have that 
\begin{equation*}
\E_{\theta}\left[\left(D^{N,2}_{\Delta,T}\right)^2+\left(D^{N,2}_{2\Delta,T}\right)^2\right]\leq K\frac{\Delta}{T}.  
\end{equation*}
\end{proposition}
\begin{proof}
Clearly, the result will follow if we can show that $\E_{\theta}\left[(D^{N,2}_{\Delta,T})^2\right]\leq K\Delta T^{-1}$. To establish this inequality, first observe that
\begin{eqnarray*}
\E_{\theta}\left[\left(D^{N,2}_{\Delta,T}\right)^2\right]&=&\frac{N^2}{T^2}\var_{\theta}\left[\sum_{t=1}^{\lfloor T/\Delta \rfloor}\left(\overline{U}_{t\Delta}-\overline{U}_{(t-1)\Delta}\right)^2\right]\\
&=&\frac{N^2}{T^2}\sum_{t=1}^{\lfloor T/\Delta \rfloor}\sum_{s=1}^{\lfloor T/\Delta \rfloor}\cov_{\theta}\left[\left(\overline{U}_{t\Delta}-\overline{U}_{(t-1)\Delta}\right)^2,\left(\overline{U}_{s\Delta}-\overline{U}_{(s-1)\Delta}\right)^2\right].
\end{eqnarray*}
Next, we use the stationarity of the system to conclude that
\begin{multline*}
\E_{\theta}\left[\left(D^{N,2}_{\Delta,T}\right)^2\right]=\frac{N^2}{T^2} \lfloor T/\Delta \rfloor \var_{\theta}[(\overline{U}_{\Delta})^2]\\
+2\frac{N^2}{T^2}\sum_{t=2}^{\lfloor T/\Delta \rfloor}\left(\lfloor T/\Delta \rfloor-t+1\right)\cov_{\theta}\left[\left(\overline{U}_{\Delta}\right)^2,\left(\overline{U}_{t\Delta}-\overline{U}_{(t-1)\Delta}\right)^2\right].   
\end{multline*}

As a consequence of Lemma \ref{lemma:control_cov_barU_squared} and of the fact that
$\sum_{k=0}^{\infty} (1-\lambda)^{k\Delta} <\infty$, it then follows that  
$$
\E_{\theta}\left[\left(D^{N,2}_{\Delta,T}\right)^2\right]\leq \frac{N^2}{T\Delta}\var_{\theta}[(\overline{U}_{\Delta})^2]+K(\Delta T)^{-1}.
$$
Hence, using Lemma \ref{lemma:control_var_barU_squared}, we establish the desired inequality and conclude the proof. 
\end{proof}

\subsection{Proof of Equation \eqref{eq:quenched:v}}

\changes{
    The proof relies on a decomposition of the error $|\hat{v}-v^N_{\infty}|$ (see Inequality \eqref{eq:triangle:ineq:vhat:vinfty}). Each term of the decomposition is then controlled by one of the four Propositions \ref{prop:control:up2:up3}-\ref{prop:control:up121}. Their proofs are quite similar and less involved than the proofs given in the previous subsection.
}

Recall that 
$$
\hat{v}= \frac{1}{T^2}\sum_{i=1}^{N}(Z_{i,T}-\overline{Z}_T)^2-\frac{N}{T^2}\overline{Z}_T+\frac{1}{T^3}\sum_{i=1}^{N}(Z_{i,T})^2,
$$
and that $\hat{m} = \overline{Z}_T/T$. Therefore,
\begin{equation}\label{eq:triangle:ineq:vhat:vinfty}
    | \hat{v} - v^N_\infty| \le \Upsilon_T^{N, 1 } + \Upsilon_T^{N, 2} + \Upsilon_T^{N, 3 },
\end{equation}
where 
$$ \Upsilon_T^{N, 1 } = \left| \sum_{i=1}^{N}\left[ \left( \frac{Z_{i, T}}{T} - \frac{ \bar Z_T}{T} \right)^2 - (m_{i}^N - \overline{m^N})^2\right]- \frac{N}{T} \overline{m^N} + \frac1T \sum_{i=1}^{N} (m_{i}^N)^2  \right|, $$
$$ \Upsilon_T^{N, 2 } = \frac{N}{T} \left|\hat{m} -\overline{m^N}  \right| ,
\mbox{ and  } \Upsilon_T^{N, 3 }  = \frac1T \left| \sum_{i=1}^{N} \left(\frac{Z_{i, T}}{T}  \right)^2 -(m_{i}^N)^2   \right| .$$

\begin{proposition}\label{prop:control:up2:up3}
    There exists a constant $K$ depending only on $ \lambda$ such that for all $N\geq 1$ and $T\geq 1$, 
    \begin{equation*}
        \E_\theta \left[ \Upsilon_T^{N, 2 } + \Upsilon_T^{N, 3 } \right] \le K\left(\frac{N}{T^2}+\frac{ {N}^{1/2}}{T^{3/2}}\left(1+\sqrt{v_\infty^N}\right)\right).
    \end{equation*}
\end{proposition}
\begin{proof}
    First, remind that $\E_\theta[Z_{i, T}/T] = m_{i}^N$ and so $\E_\theta[\hat{m}] = \overline{m^N}$. On the one hand, by Cauchy-Schwarz, $\E_\theta [\Upsilon_T^{N, 2 }] \leq NT^{-1}\left( \var_\theta(\hat{m})\right)^{1/2} \leq KN^{1/2}T^{-3/2}$ by Equation \eqref{eq:quenched:m}.
    On the other hand, by first rewriting $\Upsilon_T^{N, 3 }$ as (recall that ${U}_{i,T}=Z_{i,T}-Tm^N_i$),
$$
\Upsilon_T^{N, 3 }=\frac{1}{T}\left|\sum_{i=1}^N\left(\frac{U_{i,T}}{T}\right)^2+2\sum_{i=1}^{N}(m^N_i-\overline{m^N})\frac{U_{i,T}}{T}+2\frac{N}{T}\overline{m^N} \ \overline{U}_T\right|
$$
and then using the triangle inequality, we obtain that    
   
\begin{equation*}  
    \E_\theta \left[ \Upsilon_T^{N, 3 } \right] \le  \frac{1}{T^3}\sum_{i=1}^N\E_{\theta}\left[\left(U_{i,T}\right)^2\right]+\frac{2}{T^2}\sum_{i=1}^{N}|m^N_i-\overline{m^N}| \E_{\theta}\left[\left|U_{i,T}\right|\right]+2\frac{N}{T^2}\overline{m^N} \E_{\theta}\left[\left|\overline{U}_T\right|\right].
    \end{equation*}
    Now, by stationarity and Lemma \ref{lem:covariance:Y},
    \begin{multline}\label{eq:controlvarz} 
    \var_\theta\left(Z_{i, T}\right) = T \var_\theta\left(X_{i, 0}\right) + 2 \sum_{t=1}^T (T-t) \cov_\theta (X_{i, 0}, X_{i, t }) \\
    \le T\var_\theta\left(X_{i, 0}\right)  + K T/N  \leq KT,  
    \end{multline}
    so that 
    $$
   \frac{1}{T^3}\sum_{i=1}^N\E_{\theta}\left[\left(U_{i,T}\right)^2\right]=
   \frac{1}{T^3}\sum_{i=1}^N\text{Var}_{\theta}\left(Z_{i,T}\right)\le K N T^{-2}.
    $$
Moreover, by using Cauchy-Schwarz inequality and then Jensen inequality, we deduce that 
$$
\sum_{i=1}^{N}|m^N_i-\overline{m^N}| \E_{\theta}\left[\left|U_{i,T}\right|\right]\leq \sqrt{\sum_{i=1}^N\left|m^N_i-\overline{m^N}\right|^2}\sqrt{\sum_{i=1}^N\E_{\theta}\left[\left(U_{i,T}\right)^2\right]}.
$$
Hence, using once more Inequality \eqref{eq:controlvarz} and reminding that $v_\infty^N = \sum_{i=1}^N\left|m^N_i-\overline{m^N}\right|^2$, we obtain that 
$$
\frac{2}{T^2}\sum_{i=1}^{N}|m^N_i-\overline{m^N}|\E_{\theta}\left[\left|U_{i,T}\right|\right]\leq K\frac{N^{1/2}}{T^{3/2}}\sqrt{v_\infty^N}. 
$$

Finally, by combining Jensen inequality and Lemma \ref{Lemma_bound_on_the_exp_of_Ut_squared}, it follows that $\E_{\theta}\left[\left|\overline{U}_T\right|\right]\leq \sqrt{\E_{\theta}\left[\left|\overline{U}_T\right|^2\right]}\leq KT^{1/2}N^{-1/2}$
so that
$$
2\frac{N}{T^2}\overline{m^N} \E_{\theta}\left[\left|\overline{U}_T\right|\right]\leq K N^{1/2}T^{-3/2},
$$
where we have also used that that $|\overline{m^N}|\leq 1$. Putting together the above estimates, it then follows that 
$$
\E_\theta \left[ \Upsilon_T^{N, 3 } \right] \le K\left(\frac{N}{T^2}+\frac{ {N}^{1/2}}{T^{3/2}}\sqrt{v_\infty^N}\right),
$$
and the result follows since we have already proved that $\E_\theta [\Upsilon_T^{N, 2 }] \leq KN^{1/2}T^{-3/2}$.

\end{proof}

We now turn to the study of $\Upsilon_T^{N, 1 }.$ Using that 
$$ (a-b)^2 - (\bar a - \bar b)^2 = ( a-b)^2 - ( a- \bar b)^2 + ( a- \bar a )^2 + 2 (a- \bar a ) (\bar a - \bar b) ,$$
we obtain that $ \Upsilon_T^{N, 1 } \le \Upsilon_T^{N, 1,1 } +\Upsilon_T^{N, 1,2 }+ \Upsilon_T^{N, 1,3 },$ where 
$$\Upsilon_T^{N, 1,1 } = \left| \sum_{i=1}^{N} \left( \frac{Z_{i, T}}{T} - \frac{ \bar Z_T}{T}\right)^2 - \left( \frac{Z_{i, T}}{T} - \overline{m^N} \right)^2 \right| ,$$
$$ \Upsilon_T^{N, 1,2 } = \left| \sum_{i=1}^{N} \left( \frac{Z_{i, T}}{T} - m_{i }^N\right)^2 - \frac{N}{T} \overline{m^N} + \frac1T \sum_{i=1}^{N} (m_{ i}^N)^2  \right|$$
and 
$$ \Upsilon_T^{N, 1,3 }= 2 \left| \sum_{i=1}^{N} ( \frac{Z_{i, T}}{T} - m_{i }^N) ( m_{ i }^N - \overline{m^N})  \right| .$$
Furthermore, remind that ${U}_{i,t}=Z_{i,t}-tm^N_i$ and so we write $\Upsilon_T^{N, 1,2 } = \Upsilon_T^{N, 1,2, 1 }+\Upsilon_T^{N, 1,2, 2 } ,$ with 
$$ \Upsilon_T^{N, 1,2 , 1 }= \left|\sum_{i=1}^{N}  \left( \frac{U_{i, T }}{T}\right)^2 - \E_\theta \left( \left( \frac{U_{i, T }}{T}\right)^2\right)  \right|$$
and 
$$ \Upsilon_T^{N, 1,2 , 2 }= \left| \sum_{i=1}^{N}\E_\theta \left( \left( \frac{U_{i, T }}{T}\right)^2\right) - \frac{N}{T} \overline{m^N} + \frac1T \sum_{i=1}^{N} (m_{ i}^N)^2  \right| . $$

\begin{proposition}
    For all $N\geq 1$ and $T\geq 1$, there exists a constant $K$ depending only on $ \lambda$ such that
    \begin{equation*}
        \E_\theta \left[ \Upsilon_T^{N, 1,1 } + \Upsilon_T^{N, 1,2 , 2 } \right] \le K \frac{1}{T}.
    \end{equation*}
\end{proposition}
\begin{proof}
    Adapting the argument used in the beginning of Proposition \ref{prop:control_DN1}, we have
    $$ \Upsilon_T^{N, 1,1 } = N \left( \frac{\bar Z_T}{T} - \overline{m^N} \right)^2. $$
    Hence, Equation \eqref{eq:quenched:m} implies that $\E_\theta \Upsilon_T^{N, 1,1 } \le K/ T$.

    Second, remark that $\Upsilon_T^{N, 1,2 , 2 }$ is not random and rewrites as
    $$
    \Upsilon_T^{N, 1,2 , 2 } = \left| \sum_{i=1}^{N} \var_\theta \left( \frac{Z_{i, T}}{T} \right) - \frac{N}{T} \overline{m^N} + \frac1T \sum_{i=1}^{N} (m_{ i}^N)^2  \right| . $$
    Using Equation \eqref{eq:controlvarz} and the fact that $ \var_\theta (X_{i, 0} ) = m_i^N - (m_i^N)^2 $ we conclude that $ \E_\theta \Upsilon_T^{N, 1,2 , 2 } \le K / T.$
\end{proof}

\begin{proposition}
For all $N\geq 1$ and $T\geq 1$, there exists a constant $K$ depending only on $ \lambda$ such that
$$ \E_\theta |\Upsilon_T^{N, 1, 3}|^2 \le \frac{K}{T} \sum_{i=1}^{N} | m_{ i }^N - \overline{m^N}|^2.$$
\end{proposition}

\begin{proof}
Clearly, 
$$ \E_\theta |\Upsilon_T^{N, 1, 3}|^2 \leq \frac{4}{T^2 } \sum_{i, j =1}^{N} |\cov_\theta (U_{i, T}, U_{j, T })| | m_{ i }^N - \overline{m^N}| | m_{ j }^N - \overline{m^N}| . $$
But, by stationarity and using Lemma \ref{lem:covariance:Y}, 
$$\left| \cov_\theta (U_{i, T}, U_{j, T }) \right| \leq \sum_{s, t= 1}^T \left| \cov_\theta ( X_{i, s}, X_{j, t}) \right| \le  T \left| \cov_\theta ( X_{i, 0}, X_{j, 0}) \right| + KT/N .  $$
Moreover, 
$$\left| \cov_\theta ( X_{i, 0}, X_{j, 0}) \right| \le  (m^N_{ i } - (m_{ i}^N)^2) \indiq_{\{ i=j\}} + K/N \indiq_{\{ i \neq j\}}, $$
such that all in all
\begin{multline*} \E_\theta |\Upsilon_T^{N, 1, 3}|^2 \\
\le \frac{K}{T} \left(\sum_{i=1}^N 
\left[ \frac{1}{N} + (m^N_{ i } - (m_{ i}^N)^2) \right]
( m_{ i }^N - \overline{m^N})^2 + \frac1N \sum_{i \neq j }  
| m_{ i }^N - \overline{m^N}| | m_{ j }^N - \overline{m^N}| \right) .
\end{multline*}
Since $m^N_{ i } - (m_{ i}^N)^2 \le 1  $  and $ | m_{ i }^N - \overline{m^N}| | m_{ j }^N - \overline{m^N}| \le ( m_{i }^N - \overline{m^N})^2 + ( m_{ j }^N - \overline{m^N})^2, $ the conclusion follows.   
\end{proof}

\begin{proposition}\label{prop:control:up121}
For all $N\geq 1$ and $T\geq 1$, there exists a constant $K$ depending only on $ \lambda$ such that
$$ \E_\theta |\Upsilon_T^{N, 1,2 , 1 }|^2  \le K N/ T^2 . $$
\end{proposition}
\begin{proof}
We have that 
\begin{multline}
   \E_\theta |\Upsilon_T^{N, 1,2 , 1 }|^2= \frac{1}{T^4 }\sum_{i, j=1}^N \cov_\theta ( (U_{i, T})^2, (U_{j, T })^2 )\\
   =  \frac{1}{T^4 }\sum_{i, j=1}^N \sum_{s=1}^T \sum_{t=1}^T \cov_\theta ( Y_{i,s} Y_{i, s}, Y_{j, t} Y_{j, t}) 
   +  \frac{1}{T^4 }\sum_{i, j=1}^N \sum_{s\neq s'} \sum_{t \neq t'  } \cov_\theta ( Y_{i,s} Y_{i, s '}, Y_{j, t} Y_{j, t'})
 \\
  =: S_1 + S_2.
\end{multline}
By  Lemma \ref{lem:covariance:produit}, Item 1 and Item 3,
$ S_1 \le K N / T^3  .$ To deal with $ S_2,$ we adapt the arguments of the proof of Lemma \ref{lemma:control_var_barU_squared}. As before, write ${\bf z}=(z_1,z_2,z_3,z_4), $ and $B({\bf z})=\cov_{\theta}\left[Y_{z_1}Y_{z_2},Y_{z_3}Y_{z_4}\right].$
This expression is symmetric with respect to $ t, t'$ and also with respect to $ s, s'.$ Now, introduce the set $F^{(2)} = \{(z_1,z_2): z_1=(i,s), z_2=(i,s'), i\in\{1,\dots, N\}, s,s'\in \{1,\dots, T\}, s\neq s' \}$ and define, for $p\in \{2,5,7\}$, $\mathcal P_{ p, T}$ as the set of all vectors $\mathbf{z}\in F^{(2)}\times F^{(2)}$ satisfying the conditions of Item $p$ of Lemma \ref{lem:covariance:produit}. Moreover, notice that $\mathbf{z} \in \mathcal P_{ 5, T} \cup \mathcal P_{ 2, T}$ implies that the first coordinates (denoted $i$ and $j$ above) of the four couples in $\mathbf{z}$ are equal.

Then 
$$ S_2 =\frac{1}{T^4 } [\sum_{{\bf z} \in  \mathcal P_{ 7, T} } B({\bf z}) +\sum_{{\bf z} \in  \mathcal P_{ 5, T} } B({\bf z})+ \sum_{{\bf z} \in  \mathcal P_{ 2, T} } B({\bf z}) ]. $$
Following Step 5 of the proof of Lemma \ref{lemma:control_var_barU_squared}, it is easy so see that 
$$ \sum_{{\bf z} \in  \mathcal P_{ 7, T} } B({\bf z}) \le K T^2 .$$
Moreover, Lemma \ref{lem:covariance:produit}, Item 5, implies that 
$$ \sum_{{\bf z} \in  \mathcal P_{ 5, T} } B({\bf z}) \le K T^2 .$$
Finally, 
$$ \sum_{{\bf z} \in  \mathcal P_{ 2, T} } B({\bf z})\le K | \mathcal P_{ 2, T} | = K N T^2.$$
All in all we therefore obtain that 
$S_2 \le KN / T^2 ,$ implying the assertion.
\end{proof}

\section{Proof of Proposition \ref{prop:control:N:infty}}
\label{app:annealed:rates}

\changes{
    This section contains a fine study of the random environment $\theta$. First, some matrix notation is introduced. Then, the asymptotics of the rows and columns of the rescaled random environment $A^N$ are stated in Lemma \ref{lemma_collection_of_bounds}. In turn, the asymptotics of the rows and columns of the inverse matrix $Q^N$ are stated in Lemmas \ref{lem:ell:and:c} and \ref{lem:control:ell-ellbar:elltilde}. Then, the proofs of these three lemmas are given in three disjoint subsections. Finally, the proof of Proposition \ref{prop:control:N:infty} is given in the last subsection.
    
    To readers interested in random environments that differ from the i.i.d. case, let us mention that the proof of Proposition \ref{prop:control:N:infty}  relies solely on the results of Lemmas \ref{lemma_collection_of_bounds}, \ref{lem:ell:and:c}, and \ref{lem:control:ell-ellbar:elltilde} in the sense that the same proof applies to any environment $\theta$ for which these lemmas hold. Similarly, the proofs of Lemmas \ref{lem:ell:and:c} and \ref{lem:control:ell-ellbar:elltilde} depend only on Lemma \ref{lemma_collection_of_bounds} Therefore, only the proof Lemma \ref{lemma_collection_of_bounds} must be adapted. In what follows, we specify the adaptation required for the symmetric interaction case.

The main modification concerns the bound of the variance terms for which we use Bienaymé's identity together with a control of the few non null covariance terms. The exhaustive list of modifications is: 1) Equation \eqref{eq:96}, 2) Equation \eqref{eq:;two:variance:terms} by proving that $\operatorname{Var}(\sum_{k\in{\cal P}_a} Z_k)$ and $\E\left[\left(\sum_{k\in{\cal P}_a} \theta_{1j}Z_j\right)^2\right]$ are smaller than some universal constant, 3) check that the "independence argument" below Equation \eqref{eq:control:YN:variance} still holds, 4) replace the "independence argument" by Cauchy-Schwarz inequality to bound $\mu_{1,1}$, 5) Equation \eqref{eq:99} is replaced by 
$$
\E\left[\left(Z_i(1-|\mathcal{P}_a|^{-1})-\frac{1}{|\mathcal{P}_a|}\sum_{j\in\mathcal{P}_a:j\neq i}Z_j\right)^2\right]\leq (1-|\mathcal{P}_a|^{-1}) \var(Z_i)+KN^{-2}
$$
for some universal constant $K$.

    Two technical results used throughout the proofs are stated as technical lemmas in Section \ref{app:auxiliary}.
}

\subsection{General notation}
Hereafter, for any subset $S$ of $[N]$, we write $1_S$ to indicate the $N$-dimensional vector having value $1$ in each coordinate belonging to $S$ and value $0$ in the remaining coordinates. To alleviate the notation, we will simply write $1_N$ and $0_N$ instead of $1_{[N]}$ and $1_{\emptyset}$, respectively. For any vector $v\in\R^{N}$, $\overline{v}=N^{-1}\sum_{i=1}^N v_i$ denotes the arithmetic mean of the coordinates of $v$, $\|v\|_{r}=\left(\sum_{i=1}^{N}|v_i|^r\right)^{1/r}$ where $r\in [1,\infty)$, denotes the $r$-norm of $v$ and  
$\|v\|_{\infty}=\max_{1\leq i\leq N}|v_i|$ its $\infty$-norm.
For vectors $v,u\in\R^N$, we write $v\odot u$ to denote the vector whose $i-$th coordinate  is $v_iu_i, $ for $1\leq i\leq N$. In other words, $v\odot u$ is the Hadamard product between the vectors $v$ and $u$.
To shorten the notation, we will simply note $v^2$ instead of $v\odot v$.
Observe that, with this notation, the orthogonal projection of a vector $v\in\R^N$ on the coordinates in $S\subseteq [N]$ can be written as $1_S\odot v$. In particular, all coordinates in $S^c$ of the vector $1_S\odot v$ are null. One can always write $v=1_S\odot v+1_{S^c}\odot v.$

For any $N$-by-$N$ matrix $B$ with real entries, we denote $B^\intercal$ its transpose. For all $r\in [1,\infty]$, we denote $\tn B\tn _r$ the operator norm of $B$ associated to the $r$-norm $\|\cdot\|_r:$ 
$$
\tn B\tn _r=\sup_{v\in \R^{N}:v\neq 0_{N}}\frac{\|Bv\|_{r}}{\|v\|_r}.
$$
It is well-known that $\tn B\tn _1$ and $\tn B\tn _{\infty}$ may be defined alternatively as
$$
\tn B\tn _1=\max_{1\leq j\leq N}\sum_{i=1}^{N}|B(i,j)| \ \text{and} \  \tn B\tn _{\infty}=\max_{1\leq i\leq N}\sum_{j=1}^{N}|B(i,j)|.
$$
The following fact will also be used in the sequel: for each $r \in (1,\infty)$, it holds that
\begin{equation}
\label{ineq_upper_bound_for_the_r_norm}
\tn B\tn _r\leq \tn B\tn ^{1/r}_1\tn B\tn ^{1-1/r}_{\infty}.
\end{equation}

\subsection{Study of the rescaled random environment}

%\comJulien{This is the only place where the assumption that the $\theta_{ij}'s$ are i.i.d. is used.}

Recall that $A^N=(A^N(i,j))_{1\leq i,j\leq N}$ is a rescaled version of the random environment $\theta$ defined as, for each $1\leq i\leq N$,
\begin{equation*}
A^N(i,j)=
\begin{cases}
N^{-1}\theta_{ij}, \ \text{if} \ j\in\cal{P}_+,\\    
-N^{-1}\theta_{ij}, \ \text{if} \ j\in\cal{P}_-.
\end{cases}
\end{equation*} 
One can check that $\max\{\tn A^N\tn _{1},\tn A^N\tn _{\infty}\}\leq 1$, so that \eqref{ineq_upper_bound_for_the_r_norm} implies that $\tn A^{N}\tn _r\leq 1$ for all $r\in (1,\infty)$ as well. 

\subsubsection{Notation}

In what follows, for $a,b\in\{-,+\}$, we denote $L^{N,ab}=1_{\mathcal{P}_a}\odot A^N1_{\mathcal{P}_b}$ and $C^{N,ab}=1_{\mathcal{P}_a}\odot (A^N)^{\intercal}1_{\mathcal{P}_b}$, where $(A^N)^{\intercal}$ denotes the transpose of the matrix $A^N$. 
On the one hand, both vectors have null coordinates outside of the set $\mathcal{P}_a$.
On the other hand, each coordinate $i\in\mathcal{P}_a$ of the random vector $L^{N,ab}$ (resp. $C^{N,ab}$) is obtained by summing the entries in ${\cal{P}}_b$ of the $i$-th row (resp. column) of the matrix $A^N$.
Alternatively, the random vectors $L^{N,ab}$ and $C^{N,ab}$ can be defined as follows:
\begin{equation}
\label{def:sum_of_the_rows}
L^{N,ab}_i=
\begin{cases}
\sum_{j\in{\cal{P}}_b}A^N(i,j), \ \text{if} \ i\in{\cal P}_a,\\
0, \ \text{otherwise,}
\end{cases}
\end{equation}
and 
\begin{equation}
\label{def:sum_of_the_cols}
C^{N,ab}_i=
\begin{cases}
\sum_{j\in{\cal{P}}_b}A^N(j,i), \ \text{if} \ i\in{\cal P}_a,\\
0, \ \text{otherwise. }
\end{cases}
\end{equation} 
For $a\in\{-,+\}$, we denote $L^{N,a\bullet }=1_{\mathcal{P}_a}\odot A^N 1_{N}$ and $C^{N,a \bullet }= 1_{\mathcal{P}_a}\odot (A^N)^{\intercal}1_{N}$. Note that the coordinates not belonging to $\mathcal{P}_a$ of these two vectors are also $0$.
Besides, each coordinate $i\in \mathcal{P}_a$ of the random vector $L^{N,a \bullet}$ (resp. $C^{N,a \bullet}$) is given by the sum over all entries of the $i$-th row (resp. column) of the matrix $A^N$.
One can easily check that 
\begin{equation}
\label{def:sum_of_La+_La-_and_sum_of_Ca+_La-}
L^{N,a\bullet }=L^{N,a+}+L^{N,a-} \ \text{and} \ C^{N,a\bullet }=C^{N,a+}+C^{N,a-},
\end{equation}
so that the vectors $L^{N,a\bullet}$ and $C^{N,a\bullet }$ could be defined alternatively through these identities. 

For $b\in\{-,+\}$, we denote $L^{N,\bullet b}=A^N1_{\mathcal{P}_b}$ and $C^{N,\bullet b}=(A^N)^{\intercal}1_{\mathcal{P}_b}$. Observe that each coordinate $1\leq i\leq N$ of the random vector $L^{N,\bullet b}$ (resp. $C^{N,\bullet b}$) is given by the sum over the entries in $\mathcal{P}_b$ of the $i$-th row (resp. column) of the matrix $A^N$.
One can also verify that 
\begin{equation}
\label{def:sum_of_L+b_L-b_and_sum_of_C+b_L-b}
L^{N,\bullet b}=L^{N,+b}+L^{N,-b} \ \text{and} \ C^{N,\bullet b}=C^{N,+b}+C^{N,-b}.
\end{equation}

In the sequel, for each $a,b\in\{-,+\}$, let 
\begin{equation}
\label{def:centered_versions_of_Lab_and_Cab}
\widetilde{L}^{N,ab}=L^{N,ab}-(r^N_a)^{-1}\overline{L^{N,ab}}1_{\mathcal{P}_a} \ \text{and} \ \widetilde{C}^{N,ab}=C^{N,ab}-(r^N_a)^{-1}\overline{C^{N,ab}}1_{\mathcal{P}_a}.
\end{equation}
The vectors $\widetilde{L}^{N,ab}$ and $\widetilde{C}^{N,ab}$ 
can be thought as the population-wise centered versions of the vectors $L^{N,ab}$ and  $C^{N,ab}$ respectively, in the sense that $\overline{\widetilde{L}^{N,ab}}=\overline{\widetilde{C}^{N,ab}}=0.$ Observe that these vectors are well-defined for all $N\geq 1$ such that $r^N_a>0$.  
In the next result, we collect some bounds which will be used throughout the section.

\subsubsection{Convergence rates}

In what follows, for $a,b\in\{-,+\}$, we write $\delta_{ab}$ to denote the Kronecker delta between $a$ and $b$.

\begin{lemma}\label{lemma_collection_of_bounds} 
Let $a,b,a_1,b_1\in\{-,+\}$. The following inequalities hold for all $N\geq 1$ such that $r^{N}_-\wedge r^{N}_+>0,$
\begin{gather}
\E\left[ \left| \left< \widetilde{L}^{N,a b}, \widetilde{L}^{N,a_1 b_1} \right>  - \delta_{a a_1}\delta_{b b_1} r_{a} r_{b} p(1-p)\right|^2\right] \leq K N^{-1},\label{ineq_l2_norm_of_tildeLab}
\\
\E\left[ \| A \widetilde{L}^{ab} \|_2^2 \right] \leq K N^{-1}
\label{ineq_l2_norm_of_AtildeLab},  
\end{gather}
\vspace{-0.6cm}
\begin{multline}
\label{ineq:max_variance_of_rarbbarLab_rarbbarCab}
\E\left[ \max_{a,b\in\{-,+\}}\left\{ \left|(r^N_ar^N_b)^{-1}\overline{L^{N,ab}} -(bp)\right|^2 \vee \left|(r^N_ar^N_b)^{-1}\overline{C^{N,ab}} -(ap)\right|^2 \right\} \right] \\
\leq K\max_{a,b}\left\{(r^{N}_ar^{N}_b)^{-1}\right\}N^{-2},\ \text{and}
\end{multline}
\vspace{-0.6cm}
\begin{equation}
\label{ineq:tildeL_Lcentered_wrt_environment}
\|\widetilde{L}^{N,ab}\|_2\leq \|L^{N,ab} - (bpr^N_b)1_{{\cal{P}}_a} \|_2 \ \text{and} \ \|\widetilde{C}^{N,ab}\|_2\leq \|C^{N,ab} - (apr^N_{b})1_{{\cal{P}}_a}\|_2,
\end{equation}
where $K$ is some universal constant. Moreover, for each $\alpha\geq 1$, there exists a constant $K_{\alpha}$ depending only on $\alpha$ such that, for all $N\geq 1$ satisfying $r^{N}_-\wedge r^{N}_+>0$,
\begin{multline}
\label{ineq:moments_l2norm_Lcentered_wrt_environment}
\E\left[\max_{a,b\in\{-,+\}}\left\{\|L^{N,ab} - (bpr^N_b)1_{{\cal{P}}_a} \|^{\alpha}_2 \vee\|C^{N,ab} - (apr^N_{b})1_{{\cal{P}}_a}\|^{\alpha}_2\right\}\right]\\
\leq K_{\alpha}\max_{a\in\{-,+\}}\left\{\left(r^N_a/r^N_b\right)^{\alpha/2}\right\}.
\end{multline}
\end{lemma}
The proof of Lemma \ref{lemma_collection_of_bounds} is postponed to Section \ref{sec:proof:lemma:bounds:L:C}.

\begin{remark}\label{rem:corollaries:convergence:AN}
    Let $a,b\in\{-,+\}$. Because of the relations \eqref{def:sum_of_La+_La-_and_sum_of_Ca+_La-} and \eqref{def:sum_of_L+b_L-b_and_sum_of_C+b_L-b}, some statements of Lemma \ref{lemma_collection_of_bounds} admit immediate corollaries. For instance, 
    \begin{itemize}
        \item $C^{N,\bullet b} = C^{N,+b} + C^{N,-b}$ and $\widetilde{C}^{N,\bullet b} = \widetilde{C}^{N,+b} + \widetilde{C}^{N,-b}$ are orthogonal sums, and in particular one can deduce from Equations \eqref{ineq:tildeL_Lcentered_wrt_environment} and \eqref{ineq:moments_l2norm_Lcentered_wrt_environment} that
        \begin{equation*}
            \E\left[ \|\widetilde{C}^{N,\bullet b}\|^{\alpha}_2\right] \leq \E\left[ \|C^{N,\bullet b} - pr^N_{b}(1_{{\cal{P}}_+}- 1_{{\cal{P}}_-})\|^{\alpha}_2\right] \leq 2 K_{\alpha}\max_{a\in\{-,+\}}\left\{\left(r^N_a/r^N_b\right)^{\alpha/2}\right\},
        \end{equation*}
        \item or, $L^{N,a\bullet} = L^{N,a+} + L^{N,a-}$ and $\widetilde{L}^{N,a\bullet } = \widetilde{L}^{N,a+} + \widetilde{L}^{N,a-}, $ and in particular one can deduce from Equations \eqref{ineq:tildeL_Lcentered_wrt_environment} and \eqref{ineq:moments_l2norm_Lcentered_wrt_environment} that
        \begin{equation*}
            \|\widetilde{L}^{N,a\bullet}\|_2\leq \|L^{N,a+} - pr^N_+1_{{\cal{P}}_a} \|_2 + \|L^{N,a-} + pr^N_-1_{{\cal{P}}_a} \|_2,
        \end{equation*}
        and
        \begin{equation*}
            \E\left[\|L^{N,a\bullet} - p(r^N_+ - r^N_-)1_{{\cal{P}}_a} \|^{\alpha}_2 \right]
            \leq K'_{\alpha}\max_{a\in\{-,+\}}\left\{\left(r^N_a/r^N_b\right)^{\alpha/2}\right\},
        \end{equation*}
        for some $K'_{\alpha}$ which may be different from $K_{\alpha}$.
    \end{itemize}
\end{remark}

\subsection{Results regarding the inverse matrix}

%\comJulien{The proofs of the results in this section do not use the assumption that the $\theta_{ij}$'s are i.i.d. More precisely, it uses only the results on the matrix $A^N$ of the preeceding section and "extend" them to the inverse matrix $Q^N$.}

Recall that $\lambda>0$. Under this condition, the random matrix $Q^N:=\left(I-(1-\lambda)A^N\right)^{-1} = \sum_{n=0}^{\infty} (1-\lambda)^n (A^N)^n$ is well-defined and satisfies $\tn Q^{N}\tn _r\leq \lambda^{-1}$ for all $r\in [1,\infty].$ For later use, let us observe that the easy-to-check properties $\tn (Q^N)^{\intercal}\tn _{1}=\tn Q^N\tn _{\infty}$ and $\tn (Q^N)^{\intercal}\tn _{\infty}=\tn Q^N\tn _{1}$ combined with inequality \eqref{ineq_upper_bound_for_the_r_norm} imply that $\tn (Q^N)^{\intercal}\tn _r\leq \lambda^{-1}$ for any $r\in[1,\infty].$

As suggested by the heuristics presented in Section  \ref{sec:results}, a crucial ingredient in our analysis is to study both the sum of rows and columns of the matrix $Q^N$. By definition of the matrix $Q^N$ itself, these quantities are related to the corresponding counterparts computed from the matrix $A^N$. 
\subsubsection{Notation}

Similarly as above, for $a,b\in\{-,+\}$, let us denote $\ell^{N,ab}$ (resp. $c^{N,ab}$) the random vector obtained by summing, for each row (resp. column) in ${\cal{P}}_a$ of the random matrix $Q^N$, the entries in ${\cal{P}}_b$. Also, we denote $\ell^{N,a \bullet}=\ell^{N,a+}+\ell^{N,a-}$ and $\ell^{N,\bullet b}=\ell^{N,+b}+\ell^{N,-b}$ for $a,b\in\{-,+\}$, and define $\ell^{N}=\ell^{N,+\bullet}+\ell^{N,-\bullet}$. Note that $\ell^{N}=Q^N1_N$, i.e., $\ell^{N}$ corresponds to the random vector obtained by summing the rows of $Q^N$. The vectors $c^{N,\bullet a}$, $c^{N,\bullet a}$ and $c^{N}$ are defined in a similar way. In particular, note that $c^{N}= (Q^{N})^{\intercal}1_{N}$ is given as the sum of the columns of $Q^N$. Here $(Q^{N})^{\intercal}$ denotes the transpose of the matrix $Q^N$. 
For later use, let us also observe that $\overline{\ell^{N}} = \overline{c^{N}}$.

Like we did for $L$ and $C$, we define, for each $a,b\in\{-,+\}$
\begin{equation}
\label{def:centered_versions_of_ellab_and_cab}
\widetilde{\ell}^{N,ab}=\ell^{N,ab}-(r^N_a)^{-1}\overline{\ell^{N,ab}}1_{\mathcal{P}_a} \ \text{and} \ \widetilde{c}^{N,ab}=c^{N,ab}-(r^N_a)^{-1}\overline{c^{N,ab}}1_{\mathcal{P}_a}.
\end{equation}
Moreover, we denote $\widetilde{\ell}^{N,a \bullet}=\widetilde{\ell}^{N,a+}+\widetilde{\ell}^{N,a-}$, $\widetilde{\ell}^{N,\bullet b}=\widetilde{\ell}^{N,+b}+\widetilde{\ell}^{N,-b}$ for $a,b\in\{-,+\}$, and $\widetilde{\ell}^{N}=\widetilde{\ell}^{N,+\bullet}+\widetilde{\ell}^{N,-\bullet}$.

    Here is an immediate result stating that the $\ell$ and $c$ vectors are uniformly bounded.

\begin{lemma}
    \label{lem:bound:infty:ell:c}
Assume that $1\geq \lambda>0$. Then,
\begin{gather}
\label{omegancons1}
\max\left\{ \max_{a,b\in\{-,+\}}\left\{\|\ell^{N,ab}\|_{\infty}, \|c^{N,ab}\|_{\infty}\right\},\|\ell^{N}\|_{\infty},\|c^{N}\|_{\infty}\right\} \leq \lambda^{-1} .  \ 
\end{gather}
\end{lemma}
\begin{proof}
To show \eqref{omegancons1}, first observe that $$\max_{1\leq i\leq N}\max\left\{ \max_{a,b\in\{-,+\}}\left\{|\ell^{N,ab}_{i}|, |c^{N,ab}_{i}|\right\},|\ell^{N}_{i}|,|c^{N}_{i}|\right\}\leq \max\{\tn Q^N\tn _1,\tn Q^N\tn _{\infty}\},$$
and then use the fact that $\max\{\tn Q^N\tn _1,\tn Q^N\tn _{\infty}\}\leq \lambda^{-1}$.
\end{proof}

\subsubsection{Convergence rates}
Recall that $\delta_{ab}$ denotes the Kronecker delta between $a$ and $b$.

\begin{lemma}\label{lem:ell:and:c}
Assume that $1\geq \lambda>0$. There exists a constant $K>0$ which depends on $\lambda$ such that for all $N\geq 1$ and $a,b\in\{-,+\}$, it holds that  
\begin{gather}
\E\left[\Big|\overline{\ell^{N,a b}}- \frac{r_a \delta_{ab} + b (1-\lambda)p r_+ r_- }{1-(1-\lambda)p(r_+-r_-)} \Big|^2 \right] \leq \frac K {N^{2}}, \label{eq:limit:ell:bar}\\
\E\left[\Big|\overline{c^{N,a \bullet}}- r_a\frac{\left[1+(2a)(1-\lambda)p(1-r_a)\right]}{1-(1-\lambda)p(r_+-r_-)} \Big|^2 \right] \leq \frac K {N^{2}}, \label{eq:limit:c:bar}\\
\E\left[ \| \widetilde{\ell}^{N,\bullet b} \|_2^4\right] \leq K \text{ and } \E\left[ \| \widetilde{c}^{N,\bullet b} \|_2^4\right] \leq K, \label{eq:bound:tilde}\\
\E\left[ \left| \| \widetilde{\ell}^{N,\bullet b} \|_2^2 -  \frac{(1-\lambda)^2 p(1-p) \left[ r_b + (1-\lambda)pr_+r_-(2b + (1-\lambda)p) \right]}{(1 - (1-\lambda)p(r_+ - r_-))^2}\right| \right] \leq \frac{K}{\sqrt{N}}, \label{eq:limit:ell.b:tilde}\\
\E\left[ \left| \| \widetilde{\ell}^{N} \|_2^2 -  \frac{(1-\lambda)^2 p(1-p)}{(1 - (1-\lambda)p(r_+ - r_-))^2}\right| \right] \leq \frac{K}{\sqrt{N}}. \label{eq:limit:ell:tilde}
\end{gather}
\end{lemma}

\begin{remark}\label{rem:limit:ell:bar:full}
    First, notice that Remark \ref{rem:corollaries:convergence:AN} also applies to Lemma \ref{lem:ell:and:c} (with $L$ and $C$ replaced by $\ell$ and $c$).

 For later use, let us observe that an immediate consequence of Equation \eqref{eq:limit:ell:bar} is that there exists a constant $K>0$ such that $N\geq 1$,
\begin{equation}\label{eq:limit:ell:bar:full}
    \E\left[\left|\overline{\ell^{N}} - \frac{1}{1-(1-\lambda)p(r_+-r_-)} \right|^2\right] \leq \frac K {N^2}.
\end{equation}   
\end{remark}
The proof of Lemma \ref{lem:ell:and:c} is postponed to Section \ref{sec:proof:lemma:ell:and:c}.

The following lemma gives a way to control the $l^2$ norms of the fully centered vectors by the population wise centered vectors.

\begin{lemma}\label{lem:control:ell-ellbar:elltilde}
    Assume that $1\geq \lambda>0$. There exists a constant $K>0$ which depends on $\lambda$ such that for all $N\geq 1$ such that $r^N_+\wedge r^N_->0$, it holds that
    \begin{gather}
        \E\left[ \left|  \left\| \ell^N - \overline{\ell^N}1_N  \right\|_2^2 -  \left\| \widetilde{\ell}^N \right\|_2^2  \right| \right] \leq \frac K {N}, \label{eq:control:ell-ellbar:elltilde}\\
       \E\left[\left|\langle \ell^N-\overline{\ell^N}1_N, (1-\lambda)^2(Q^NL^{N,\bullet-}-\overline{Q^NL^{N,\bullet-}})  \rangle -\langle \widetilde{\ell}^{N}
, \widetilde{\ell}^{N,\bullet -}\rangle \right|\right]\leq K N^{-1},\label{eq:control:product:ell-ellbar:elltilde}\\
\E\left[\left|(1-\lambda)\| QL^{N,\bullet} -\overline{Q^NL^{N,\bullet -}} \|^2_2-\|\widetilde{\ell}^{N,\bullet-}\|^2_2\right|\right]\leq KN^{-1}. \label{eq:control:QL-QL:elltilde}    \end{gather}
\end{lemma}

The proof of Lemma \ref{lem:control:ell-ellbar:elltilde} is postponed to Section \ref{sec:proof:lem:control:ell-ellbar:elltilde}.

\subsection{Proof of Lemma \ref{lemma_collection_of_bounds}}
\label{sec:proof:lemma:bounds:L:C}

Before the proof, we recall some classical results in the three lemmas below.

\begin{lemma}\label{lemma_pth_central_moments_empirical_mean_of_bernoulli}
Consider i.i.d. random variables $(B_{i})_{1\leq i\leq N}$ distributed as $\text{Ber}(p)$ with $p\in[0,1]$ and let $\bar{B}=N^{-1}\sum_{i=1}^NB_i$. Then, for each $\alpha\geq 1$ there exists a constant $C_{\alpha}>0$ depending only on $\alpha$ such that
\begin{equation*}
\E\left[|\bar{B}-p|^{\alpha}\right]\leq C_{\alpha}N^{-\alpha/2}.
\end{equation*}
\end{lemma}
\begin{proof}
Hoeffding's inequality implies that for any $x>0$,
$$
\P\left(|\bar{B}-p|\geq x\right)\leq 2e^{-2x^2N}.
$$
Now, using the tail sum formula for the expectation one can show that
$$
\E\left[|\bar{B}-p|^{\alpha}\right]\leq 2\alpha\int_{0}^{\infty} x^{\alpha-1}e^{-2x^2N}dx=\alpha\sqrt{\frac{2\pi}{N}}\E\left[Z^{\alpha-1}1_{Z>0}\right],
$$
where $Z\sim \mathcal{N}(0,(4N)^{-1})$. Finally, by observing that $Y=(4N)^{1/2}Z \sim \mathcal{N}(0,1)$, it follows that 
$$
\alpha\sqrt{\frac{2\pi}{N}}\E\left[Z^{\alpha-1}1_{Z>0}\right]=\alpha\frac{\sqrt{2\pi}}{2^{\alpha-1}}N^{-\alpha/2}\E\left[Y^{\alpha-1}1_{Y>0}\right],
$$
and the result follows.
\end{proof}

%\comJulien{The lemma below holds true as soon as for all $i$, the $(B_{ij})_{j}$ are i.i.d. (which is true in the symmetric case).}
\begin{lemma}\label{lemma_vector_concentration}
Let $(B_{ij})_{1\leq i,j\leq N}$ be i.i.d. random variables  distributed as $\text{Ber}(p)$ with $p\in[0,1]$. For $U,V\subset [N]$ and $\epsilon\in\{-1,1\}$, define $W_i=\epsilon N^{-1}\sum_{j\in V}B_{ij}$ for $i\in U$, and $W_i=0$, for $i\in U^c$. Denote $W=(W_1,\ldots, W_N)$. Then, for any $\alpha\geq 1$, there exists a constant $K_{\alpha}>0$ depending only on $\alpha$ such that
\begin{equation}
\E\left(\|W-p\epsilon|V|N^{-1}1_U\|^{\alpha}_2\right)\leq K_{\alpha}\left(\frac{|U|}{|V|}\right)^{\alpha/2}.
\end{equation} 
\end{lemma}
\begin{proof}
Jensen inequality combined with Lemma \ref{lemma_pth_central_moments_empirical_mean_of_bernoulli} and the fact that $|V|/N\leq 1$ implies that for any $\alpha\geq 1$,
$$
\E\left[\|W-p\epsilon|V|N^{-1}1_U\|^{2\alpha}_{2}\right]\leq |U|^{\alpha}\left(|V|N^{-1}\right)^{2\alpha} (C_{2\alpha}|V|^{-(2\alpha)/2})\le C_{2\alpha}\left(\frac{|U|}{|V|}\right)^{\alpha},
$$
where the constant $C_{2 \alpha}$ is the one of Lemma \ref{lemma_pth_central_moments_empirical_mean_of_bernoulli}. By using Jensen inequality once more, we deduce that 
$$
\E\left[\|W-p\epsilon|V|N^{-1}1_U\|^{\alpha}_{2}\right]\leq \left(\E\left[\|W-p\epsilon|V|N^{-1}1_U\|^{2\alpha}_{2}\right]\right)^{1/2},
$$
and the result follows, putting $ K_\alpha := C_{2\alpha}^{1/2}.$
\end{proof}

\begin{lemma}\label{lemma_variance_inequality_vectorial_form}
Consider a vector $v=(v_1,\ldots, v_N)\in\R^N$ supported on $S\subseteq\{1,\ldots,N\}$. For any $\xi\in\R$, the following inequality holds:
\begin{equation}
\left\|v-\frac{N}{|S|}\bar{v}1_{S}\right\|_{2}\leq \left\|v-\xi1_S\right\|_{2},
\end{equation} 
where $\bar{v}=N^{-1}\sum_{i=1}^Nv_i=N^{-1}\sum_{i\in S}v_i$.
\end{lemma}
\begin{proof}
Consider $U\sim\text{Unif}(\{v_i:i\in S\})$, and observe that $\E[U]=\bar{v}N|S|^{-1}$ and
$\var(U)=|S|^{-1}\|v-N|S|^{-1}\bar{v}1_S\|_2$. Since $\var(U)=\inf_{\xi\in\R}\E\left[(U-\xi)^2\right]$, it follows that for any $\xi\in\R$,
\begin{equation*}
|S|^{-1}\left\|v-\frac{N}{|S|}\bar{v}1_{S}\right\|_{2}\leq \E\left[(U-\xi)^2\right]=|S|^{-1}\left\|v-\xi1_S\right\|_{2},
\end{equation*}
implying the result.
\end{proof}

\begin{proof}[Proof of Lemma \ref{lemma_collection_of_bounds}]
Inequalities \eqref{ineq:tildeL_Lcentered_wrt_environment} and \eqref{ineq:moments_l2norm_Lcentered_wrt_environment} follow from Lemmas \ref{lemma_variance_inequality_vectorial_form} and  \ref{lemma_vector_concentration}  respectively. Next, we prove Inequality \eqref{ineq:max_variance_of_rarbbarLab_rarbbarCab}.
First, note that $(r^N_ar^N_b)^{-1}\overline{L^{N,ab}} -(bp) = b(|\mathcal{P}_a| |\mathcal{P}_b|)^{-1}\sum_{i\in\mathcal{P}_a}\sum_{j\in\mathcal{P}_b} (\theta_{ij}-p),$ so that  
\begin{align}
\label{eq:96}
\E\left[\left|(r^N_ar^N_b)^{-1}\overline{L^{N,ab}} -(bp)\right|^2\right]&=\var\left((|\mathcal{P}_a| |\mathcal{P}_b|)^{-1}\sum_{i\in\mathcal{P}_a}\sum_{j\in\mathcal{P}_b} \theta_{ij}\right) \nonumber\\
&=p(1-p)(|\mathcal{P}_a| |\mathcal{P}_b|)^{-1}=N^{-2}p(1-p)(r^N_ar^N_b)^{-1} .
\end{align}
% \comJulien{symmetric case: decompose the following variance thanks to covariances and use the fact that all covariances are null except $\cov(\theta_{ij},\theta_{ij}) = p(1-p)$ and $\cov(\theta_{ij},\theta_{ji}) = p(1-p)$, so that
% \begin{equation*}
%     \var\left(\sum_{i,j}\theta_{ij}\right) = \sum_{i,j} \var(\theta_{ij}) + \sum_{i\neq j} \cov(\theta_{ij},\theta_{ji}) \leq 2 |\mathcal{P}_a| |\mathcal{P}_b| p(1-p).
% \end{equation*}
% }
Since $p(1-p)\leq 1/4$, we conclude that
$$
\E\left[\left|(r^N_ar^N_b)^{-1}\overline{L^{N,ab}} -(bp)\right|^2\right]\leq (N^{-2}/4) \max_{c,d}\{(r^N_cr^N_d)^{-1}\}.
$$ 
Similar arguments can be applied to show that 
$$
\E\left[\left|(r^N_ar^N_b)^{-1}\overline{C^{N,ab}} -(ap)\right|^2\right]\leq (N^{-2}/4) \max_{c,d}\{(r^N_cr^N_d)^{-1}\}.
$$ 
Inequality \eqref{ineq:max_variance_of_rarbbarLab_rarbbarCab} follows then from the above inequalities and from the fact that $x\vee y\leq x+y$ for any $x,y\geq 0$.

We now establish Inequality \eqref{ineq_l2_norm_of_AtildeLab}. To that end, first observe that
$$
\E\left[\|A^N\widetilde{L}^{N,ab}\|_2^2\right]=N\E\left[\left(\sum_{j\in\mathcal{P}_a}A^N(1,j)\widetilde{L}^{N,ab}_j \right)^2\right].
$$
Next, observe that if we denote $Z_j=L^{N,ab}_j-(pb)r^N_b$ for $j\in\mathcal{P}_a$, then we can write (recall that $A^N(i,j)=bN^{-1}\theta_{ij}$ for $1\leq i\leq N$ and $j\in\mathcal{P}_b$), 
$$
\left[\sum_{j\in\mathcal{P}_a}A^N(1,j)\widetilde{L}^{N,ab}_j\right]^2 =N^{-2}\left[\sum_{j\in\mathcal{P}_a}\theta_{1j}Z_j
-\left(\sum_{j\in\mathcal{P}_a}\theta_{1j}\right)\left(|\mathcal{P}_a|^{-1}\sum_{k\in\mathcal{P}_a}Z_k\right)\right]^2,
$$
so that by applying Jensen's inequality we obtain that
\begin{multline}\label{eq:;two:variance:terms}
\E\left[\|A^N\widetilde{L}^{N,ab}\|_2^2\right]\leq \\
2 N^{-1} \left[\E\left[\left(\sum_{j\in\mathcal{P}_a}\theta_{1j}Z_j\right)^2\right]+
\E\left[\left(\sum_{j\in\mathcal{P}_a}\theta_{1j}\right)^2\left(|\mathcal{P}_a|^{-1}\sum_{k\in\mathcal{P}_a}Z_k\right)^2\right]
\right] .
\end{multline}
By observing that $\sum_{j\in\mathcal{P}_a}\theta_{1j}\leq |\mathcal{P}_a|$ and using that  $(Z_k)_{k\in\mathcal{P}_b}$ are i.i.d. centered random variables, one can check that    
\begin{align*}
\E\left[\left(\sum_{j\in\mathcal{P}_a}\theta_{1j}\right)^2\left(|\mathcal{P}_a|^{-1}\sum_{k\in\mathcal{P}_a}Z_k\right)^2\right]&\leq |\mathcal{P}_a|^2\var\left(|\mathcal{P}_a|^{-1}\sum_{k\in\mathcal{P}_a}Z_k\right)\\
&=|\mathcal{P}_a|\var\left(Z_1\right)=r^N_ar^N_bp(1-p) \leq 1/4,
\end{align*}
% \comJulien{symmetric case: remark that, for $k\neq k'$,
% \begin{equation*}
%     \cov(Z_k,Z_{k'}) = \cov(L^{N,ab}_k, L^{N,ab}_{k'}) = N^{-2} \cov(\sum_i \theta_{ki}, \sum_{i'} \theta_{k'i'}) = N^{-2} p(1-p).
% \end{equation*}
% Hence, 
% \begin{equation*}
%     \var\left(\sum_{k\in\mathcal{P}_a}Z_k\right) = |\mathcal{P}_a| \var(Z_1) + |\mathcal{P}_a|(|\mathcal{P}_a|-1) N^{-2} p(1-p) \leq  r^N_ar^N_bp(1-p) + (r^N_a)^2 p(1-p).
% \end{equation*}
% }
where in the last inequality we have used that $\var\left(Z_k\right)=\var\left(L^{N,ab}_k\right)=|\mathcal{P}_b|p(1-p)/N^2$.

Now, note that if $j\neq 1,$ then $\theta_{1j}$ is independent of $Z_j$ so that 
$\E[\theta_{1j}Z_j]=\E[\theta_{1j}]\E[Z_j]=0$, because $Z_j$ is centered. Moreover, one can check that $(\theta_{1j}Z_j)_{j\in \mathcal{P}_a:j\neq 1}$ are independent. Hence, by combining these facts with Jensen's inequality, we can deduce that
\begin{align*}
\E\left[\left(\sum_{j\in\mathcal{P}_a}\theta_{1j}Z_j\right)^2\right]& \leq 2\left[\E\left[\left(\sum_{j\in\mathcal{P}_a:j\neq 1}\theta_{1j}Z_j\right)^2\right]+\E\left[\left(\theta_{11}Z_1\right)^2\right]\right]\\
& \leq 2\left[\var\left(\sum_{j\in\mathcal{P}_a:j\neq 1}\theta_{1j}Z_j\right)+\var\left(Z_1\right)\right]\\
& \leq 2\left[(|\mathcal{P}_a|-1)p\var(Z_1)+\var\left(Z_1\right)\right]\leq 2r^N_ar^N_bp(1-p) \leq 1/2.
\end{align*}
Combining the last three inequalities we show that Inequality \eqref{ineq_l2_norm_of_AtildeLab} holds.

It remains to show Inequality \eqref{ineq_l2_norm_of_tildeLab}. In the case $a_1\neq a$, the two vectors $\widetilde{L}^{N,a b}$ and $\widetilde{L}^{N,a b_1}$ have disjoint supports so that $\left< \widetilde{L}^{N,a b}, \widetilde{L}^{N,a_1 b_1} \right>=0$ a.s. and \eqref{ineq_l2_norm_of_AtildeLab} trivially holds.

From now on, consider that $a_1=a$. Let us denote $Y^{N}_{a,bb_1} = \left< \widetilde{L}^{N,a b}, \widetilde{L}^{N,a b_1} \right>$ and prove that there exists a universal constant $K$ such that 
\begin{equation}\label{eq:control:YN:variance}
    \E\left[\left| Y^{N}_{a,bb_1} - \delta_{bb_1} r_ar_bp(1-p)\right|^2\right]\leq K\left[\E\left[\left|Y^{N}_{a,bb_1}-\E\left[Y^{N}_{a,bb_1}\right]\right|^2\right]+N^{-2}\right].
\end{equation}

First, consider the case $b_1\neq b$. Remind that $L^{N,a b}$ only depends on $\{\theta_{ij}, i\in \mathcal{P}_a, j\in \mathcal{P}_b\}$, and remark that $\E\left[ \widetilde{L}^{N,a b}_i \right] = 0$ for all $i=1,\dots,N$. In particular, the vectors $L^{N,a b}$ and $L^{N,a b_1}$ are independent and
%\comJulien{symmetric case: still independent because $(i,j) \in \mathcal{P}_a \times \mathcal{P}_b$ and $(j,i) \in \mathcal{P}_a \times \mathcal{P}_{b_1}$ is impossible.}
\begin{equation*}
    \E\left[ Y^{N}_{a,bb_1} \right] =  \left< \E\left[ \widetilde{L}^{N,a b}\right], \E\left[ \widetilde{L}^{N,a b_1} \right]\right>  =0,
\end{equation*}
so that
\begin{equation*}
    \E\left[\left| Y^{N}_{a,bb_1} - \delta_{bb_1} r_ar_bp(1-p)\right|^2\right] = \E\left[\left|Y^{N}_{a,bb_1}-\E\left[Y^{N}_{a,bb_1}\right]\right|^2\right].
\end{equation*}

Now, consider the case $b_1=b$. In that case, $Y^{N}_{a,bb_1}= \|\widetilde{L}^{N,ab}\|_2^2$, and observe that
$$
\E\left[\|\widetilde{L}^{N,ab}\|_2^2\right]=|\mathcal{P}_a|\E\left[\left(L_i^{N,ab}-\frac{1}{|\mathcal{P}_a|}\sum_{j\in\mathcal{P}_a}L_j^{N,ab}\right)^2\right],
$$
for any fixed $i\in\mathcal{P}_a$. 
In the rest of the proof, $i$ denotes an arbitrary index in $\mathcal{P}_a$.
Next, observe that if we denote $Z_j=L^{N,ab}_j-(pb)r^N_b$ for $j\in\mathcal{P}_a$, then we can write 
$$
\E\left[\|\widetilde{L}^{N,ab}\|_2^2\right]=|\mathcal{P}_a|\E\left[\left(Z_i(1-|\mathcal{P}_a|^{-1})-\frac{1}{|\mathcal{P}_a|}\sum_{j\in\mathcal{P}_a:j\neq i}Z_j\right)^2\right].
$$
Since $(Z_j)_{j\in\mathcal{P}_a}$ are i.i.d. centered random variables, one can check that    
\begin{equation}
\label{eq:99}
\E\left[\left(Z_i(1-|\mathcal{P}_a|^{-1})-\frac{1}{|\mathcal{P}_a|}\sum_{j\in\mathcal{P}_a:j\neq i}Z_j\right)^2\right]=(1-|\mathcal{P}_a|^{-1})\var(Z_i),
\end{equation}
% \comJulien{symmetric case: decompose the variance above via covariances. Use the fact that 
% $$
% \var\left(\sum_{j\in\mathcal{P}_a:j\neq i}Z_j\right) = (|\mathcal{P}_a|-1) \var(Z_1) + (|\mathcal{P}_a|-1)(|\mathcal{P}_a|-2) N^{-2} p(1-p),
% $$
% (remark the the second term is negligible) and
% $$
% \cov(Z_i, \sum_{j\in\mathcal{P}_a:j\neq i}Z_j) = (|\mathcal{P}_a|-1) N^{-2} p(1-p)
% $$
% (it is also negligible). Hence, the leading terms of $\E\left[\left(Z_i(1-|\mathcal{P}_a|^{-1})-\frac{1}{|\mathcal{P}_a|}\sum_{j\in\mathcal{P}_a:j\neq i}Z_j\right)^2\right]$ are $(1-|\mathcal{P}_a|^{-1})^2 \var(Z_i)$ and $|\mathcal{P}_a|^{-2}(|\mathcal{P}_a|-1) \var(Z_1)$. Their sum is $(1-|\mathcal{P}_a|^{-1}) \var(Z_i)$ as in the i.i.d. case.
% }
which, in turn, implies that
\begin{align*}
\E\left[\|\widetilde{L}^{N,ab}\|_2^2\right]&=|\mathcal{P}_a|(1-|\mathcal{P}_a|^{-1})\var(Z_i)\\
&=|\mathcal{P}_a|(1-|\mathcal{P}_a|^{-1})\var(L^{N,ab}_i)=r^N_ar^N_bp(1-p)-N^{-1}r^N_bp(1-p).    
\end{align*}
Combining the above identity with Jensen's inequality (and the assumption we made on the sequence of fractions $r_a^N$), we deduce that Equation \eqref{eq:control:YN:variance} is satisfied.

Next, notice that $(|\mathcal{P}_a|-1)^{-1}Y^{N}_{a,bb_1}$ can be seen as the empirical covariance of the random variables $L^{N,ab}_i$ and $L^{N,ab_1}_i$, $i\in\mathcal{P}_a$. In this perspective, $\E\left[\left|Y^{N}_{a,bb_1}-\E\left[Y^{N}_{a,bb_1}\right]\right|^2\right]$ corresponds to
$(|\mathcal{P}_a|-1)^{2}$ times the variance of the empirical covariance of the random variables $L^{N,ab}_i$ and $L^{N,ab_1}_i$, $i\in\mathcal{P}_a$. It is well-known (e.g., see \citep[page 363, exercise 7.45b]{casella2024statistical} for the case of the empirical variance) that the latter is equal to  
\begin{equation*}
    |\mathcal{P}_a|^{-1}\left(\mu_{2,2} + \frac{\mu_{2,0}\mu_{0,2} - (|\mathcal{P}_a| - 2)\mu_{1,1}^2}{|\mathcal{P}_a|-1} \right),
\end{equation*}
where $\mu_{r,s} = \E\left[ \left( L^{N,ab}_i -\E\left[ L^{N,ab}_i \right] \right)^r \left( L^{N,ab_1}_i -\E\left[ L^{N,ab_1}_i \right] \right)^s \right]$. We have $\mu_{2,2} = \E[(L^{N,ab}_i-bpr^N_b)^2(L^{N,ab_1}_i-b_1pr^N_{b_1})^2] \leq (\E[(L^{N,ab}_i-bpr^N_b)^4]\E[(L^{N,ab_1}_i-b_1pr^N_{b_1})^4])^{1/2}$ by Cauchy Schwarz inequality, $\mu_{2,0} = \E[(L^{N,ab}_i-bpr^N_b)^2]$, $\mu_{0,2} = \E[(L^{N,ab_1}_i-b_1pr^N_{b_1})^2]$ and $\mu_{1,1} = \delta_{b b_1} \E[(L^{N,ab}_i-bpr^N_b)^2]$ by independence. 
%\comJulien{symmetric case: the "independence argument" is replaced by Cauchy-Schwarz.}
By applying Lemma \ref{lemma_pth_central_moments_empirical_mean_of_bernoulli}, we know that, for some universal constant $K$, $\mu_{2,2} \leq K N^{-2}$ and $\mu_{2,0}+\mu_{0,2}+\mu_{1,1} \leq K N^{-1}$. Finally,
\begin{equation*}
 \E\left[\left|Y^{N}_{a,bb_1}-\E\left[Y^{N}_{a,bb_1}\right]\right|^2\right] \leq K \frac{(|\mathcal{P}_a|-1)^{2}}{|\mathcal{P}_a|} \left(N^{-2} + \frac{N^{-2} - (|\mathcal{P}_a| - 2)N^{-2}}{|\mathcal{P}_a|-1}\right) \leq KN^{-1},
\end{equation*}
which implies the result.
\end{proof}

\subsection{Proof of Lemma \ref{lem:ell:and:c}}
\label{sec:proof:lemma:ell:and:c}
It suffices to show that the inequalities hold for all $N$ sufficiently large. Recall that $r^N_a=|{\cal{P}}_a|/N$, for $a\in\{-,+\}$. Throughout the proof, we assume that $N$ is large enough ($N\geq N_0$) ensuring that $r^N_+\wedge r^N_-\geq r_{\rm min}>0$ for some $r_{\rm min}$ sufficiently small depending only on the choice of $r_+$ and $r_-$. 
In what follows, we shall denote $K$ a constant which may depend on $\lambda$ and which may change from one line to another. The proof is divided in the 12 steps below. 

\medskip
{\it Step 1.} 
Consider the event
\begin{equation*}
\mathcal{A}_N = \bigcap_{a\in\{-,+\}}\bigcap_{b\in\{-,+\}}\left\{\|L^{N,ab} - (bpr^N_b)1_{{\cal{P}}_a} \|_2 \vee\|C^{N,ab} - (apr^N_{b})1_{{\cal{P}}_a}\|_2 \leq N^{1/4} \right\}.
\end{equation*}
For later use, let us observe that Points (ii) and (iii) of Lemma \ref{lemma_collection_of_bounds} imply that $\mathcal{A}_N$ is included in the event
\begin{multline*}
\mathcal{G}_N=\bigcap_{b\in\{-,+\}}\left\{\|L^{N,\bullet b}-(bpr^N_b)1_{N}\|_2\vee \|L^{N,b\bullet }-p(r^N_+-r^N_-)1_{\mathcal{P}_b}\|_2 \right. \\ \left.
\vee
\|C^{N,\bullet b} - pr^N_{b}(1_{{\cal{P}}_+}-1_{{\cal{P}}_-})\|_2\vee \|C^{N,b\bullet } - (bp)1_{{\cal{P}}_b}\|_2 \leq 2N^{1/4} \right\}.
\end{multline*}
Combining Point (iv) of Lemma \ref{lemma_collection_of_bounds} with Markov's inequality, one can show that for any $\alpha\geq 1$, there exits a constant $K_{\alpha}>0$ such that
\begin{equation*}
    \P\left(\mathcal{G}_{N}\right) \geq \P\left(\mathcal{A}_{N}\right) \geq 1 - K_{\alpha}N^{-\alpha/4}.
\end{equation*}
Hence, $\mathcal{A}_{N}$ is a large probability event and we will first consider the expectations appearing in Lemma \ref{lem:ell:and:c} under the event $\mathcal{A}_{N}$ only. This is used in Step 4.

\medskip
{\it Step 2.} 
Let $a,b\in\{-,+\}$. Here, we prove that
$$
\text{(i)} \ \ \widetilde{\ell}^{N,ab} = (1-\lambda)\left[ 
    1_{\mathcal{P}_a} \odot A^{N}\widetilde{\ell}^{N,\bullet b}
    + \sum_{e\in\{-,+\}} (r^N_e)^{-1} \overline{\ell^{N,eb}} \widetilde{L}^{N,ae}
    + \epsilon^{N,ab} 1_{\mathcal{P}_a} \right],$$ 
where 
$$
\text{(ii)} \ \ |\epsilon^{N,ab}| \leq (r^N_a N)^{-1} \| C^{N,\bullet a} - pr_a(1_{\mathcal{P}_+} - 1_{\mathcal{P}_-}) \|_2  \| \widetilde{\ell}^{N,\bullet b} \|_2.
$$

Starting from $\ell^{N,\bullet b} = Q^{N}1_{\mathcal{P}_b} = (I_N - (1-\lambda)A^{N})^{-1} 1_{\mathcal{P}_b}$, we get $\ell^{N,\bullet b} =  1_{\mathcal{P}_b} + (1-\lambda) A^{N}\ell^{N,\bullet b} $. Hence, it implies that $\ell^{N,ab} = 1_{\mathcal{P}_a} \odot \left( 1_{\mathcal{P}_b} + (1-\lambda) A^{N}\ell^{N,\bullet b} \right)$ and $(r^N_a)^{-1} \overline{\ell^{N,ab}} = \delta_{ab} + (1-\lambda) (r^N_a N)^{-1} \left< A^{N}\ell^{N,\bullet b}, 1_{\mathcal{P}_a} \right>$, so that 
\begin{equation*}
    \widetilde{\ell}^{N,ab} = (1-\lambda) \left[ 1_{\mathcal{P}_a} \odot A^{N}\ell^{N,\bullet b} - (r^N_a N)^{-1} \left< A^{N}\ell^{N,\bullet b}, 1_{\mathcal{P}_a} \right> 1_{\mathcal{P}_a} \right].
\end{equation*}
Then, using the substitution $\ell^{N,\bullet b} = \widetilde{\ell}^{N,\bullet b} + \sum_{e\in\{-,+\}} (r^N_{e})^{-1} \overline{\ell^{N,eb}} 1_{\mathcal{P}_e}$ in the two terms above and the fact that $1_{\mathcal{P}_a} \odot A^{N} 1_{\mathcal{P}_e} = L^{N,a e}$, we get point (i) with 
\begin{equation*}
    \epsilon^{N,ab} = - (r^N_a N)^{-1} \left< A^{N}\widetilde{\ell}^{N,\bullet b}, 1_{\mathcal{P}_a} \right> = - (r^N_a N)^{-1} \left< \widetilde{\ell}^{N,\bullet b}, (A^{N})^{\intercal}1_{\mathcal{P}_a} \right>.
\end{equation*}
Yet, $(A^{N})^{\intercal}1_{\mathcal{P}_a} = C^{N,\bullet a}$ by definition and $\left<\tilde{\ell}^{N,\bullet b},1_{\mathcal{P}_+} \right> = \left<\tilde{\ell}^{N,\bullet b},1_{\mathcal{P}_-} \right> = 0$ by construction so that
\begin{equation*}
    \epsilon^{N,ab} = - (r^N_a N)^{-1} \left< \widetilde{\ell}^{N,\bullet b}, C^{N,\bullet a} - pr_a(1_{\mathcal{P}_+} - 1_{\mathcal{P}_-}) \right>,
\end{equation*}
and point (ii) follows from Cauchy-Schwarz inequality.

\medskip
{\it Step 3.} 
Proceeding as in Step 2, one can show that
$$
\text{(i)} \ \ \widetilde{c}^{N,ab} = (1-\lambda)\left[ 
    1_{\mathcal{P}_a} \odot A^{N}\widetilde{c}^{N,\bullet b}
    + \sum_{e\in\{-,+\}} (r^N_e)^{-1} \overline{c^{N,eb}} \widetilde{C}^{N,ae}
    + \gamma^{N,ab} 1_{\mathcal{P}_a} \right],$$ 
where 
$$
\text{(ii)} \ \ |\gamma^{N,ab}| \leq (r^N_a N)^{-1} \| L^{N,\bullet a} - (apr_a)1_{N} \|_2  \| \widetilde{c}^{N,\bullet b} \|_2.
$$

\medskip
{\it Step 4.} 
Here, we prove Equation \eqref{eq:bound:tilde}. 

From Step 2, one can sum for $a\in \{-,+\}$ to obtain
\begin{equation}\label{eq:tilde:ell:bullet:b}
    \widetilde{\ell}^{N,\bullet b} = (1-\lambda)\left[ 
    A^{N}\widetilde{\ell}^{N,\bullet b}
    + \sum_{a\in \{-,+\}} \sum_{e\in\{-,+\}} (r^N_e)^{-1} \overline{\ell^{N,eb}} \widetilde{L}^{N,a e}
    + \sum_{a\in \{-,+\}} \epsilon^{N,ab} 1_{\mathcal{P}_a} \right].
\end{equation}
Then, using the fact that $\|1_{\mathcal{P}_a}\|_2=N^{1/2}(r^N_a)^{1/2}\leq N^{1/2}$, we get
\begin{multline*}
    (1-\lambda)^{-1} \|\widetilde{\ell}^{N,\bullet b}\|_2 \leq  \tn A^N\tn _2\|\widetilde{\ell}^{N,\bullet b}\|_2 \\
    + 4 \max_{e\in\{-,+\}}\left\{ (r^N_e)^{-1} |\overline{\ell^{N,eb}}| \right\}\max_{a,e\in\{-,+\}}\left\{\|\widetilde{L}^{N,a e}\|_2\right\} 
    + N^{1/2} \sum_{a\in\{-,+\}} |\epsilon^{N,ab}|.
\end{multline*}
Recall that $\tn A^N\tn _2 \leq 1$ and $|\overline{\ell^{N,eb}}| \leq r^N_e\lambda^{-1}$ (see inequality \eqref{omegancons1}). On the event $\mathcal{A}_N \supset \mathcal{G}_N$, we have $\|C^{N,\bullet a} - pr^N_{a}(1_{{\cal{P}}_+}-1_{{\cal{P}}_-})\|_2 \leq 2N^{1/4}$, and using Point (ii) of Step 2, we have 
\begin{equation*}
 \indiq_{\mathcal{A}_N} (1-\lambda)^{-1} \|\widetilde{\ell}^{N,\bullet b}\|_2 \leq \|\widetilde{\ell}^{N,\bullet b}\|_2 \left[1 + \frac{4}{r^N_+\wedge r^N_-}N^{-1/4}\right]+ 4\lambda^{-1} \max_{a,e\in\{-,+\}}\left\{\|\widetilde{L}^{N,ae}\|_2\right\}. 
\end{equation*}
Now,  since $r^N_+\wedge r^N_-\geq r_{\rm min}$, it follows that $1 + \frac{4}{r^N_+\wedge r^N_-}N^{-1/4}\leq 1 + 4r_{\rm min}^{-1}N^{-1/4} < (1-\lambda)^{-1}$ for $N\geq N_1$ for some $N_1$ sufficiently large. As a consequence, there exists a constant $K$ such that
\begin{equation}\label{eq:bound:ell:tilde:L:tilde}
    \indiq_{\mathcal{A}_N} \|\widetilde{\ell}^{N,\bullet b}\|_2\leq K \max_{a,e\in\{-,+\}}\left\{\|\widetilde{L}^{N,ae}\|_2\right\},
\end{equation}
and Equations \eqref{ineq:tildeL_Lcentered_wrt_environment} and
\eqref{ineq:moments_l2norm_Lcentered_wrt_environment} of Lemma \ref{lemma_collection_of_bounds} implies that 
$$
 \E\left[\indiq_{\mathcal{A}_N}\|\widetilde{\ell}^{N,\bullet b}\|^4_2\right]\leq K \E\left[\max_{a,b\in\{-,+\}}\left\{\|\widetilde{L}^{N,ab}\|^4_2\right\}\right]\leq K,
$$
for all $N\geq N_1$. Taking the maximum value of  $\E\left[\indiq_{\mathcal{A}_N}\|\tilde{\ell}^{N,\bullet b}\|^4_2\right]$ over all $N \in [N_0,N_1]$, we get 
$$
 \E\left[\indiq_{\mathcal{A}_N}\|\widetilde{\ell}^{N,\bullet b}\|^4_2\right]\leq K,
$$
for all $N\geq N_0$. Yet, inequality \ref{omegancons1} implies that $|\widetilde{\ell}^{N,ab}_i|\leq 2\lambda^{-1}$ which in turn implies that $\|\widetilde{\ell}^{N,\bullet b}\|^4_2 \leq 4\lambda^{-2}N^{2}$. Hence, to get rid of the term $\indiq_{\mathcal{A}_N}$ it suffices to write
\begin{equation*}
    \E\left[\|\widetilde{\ell}^{N,\bullet b}\|^4_2\right]\leq 4\lambda^{-2}N^{2} \P\left( (\mathcal{A}_N)^{c} \right) + \E\left[\indiq_{\mathcal{A}_N}\|\widetilde{\ell}^{N,\bullet b}\|^4_2\right] \leq K,
\end{equation*}
where we used the last inequality of Step 1 with $\alpha=8$.

Finally, the proof of the inequality $\E\left[ \|\widetilde{c}^{N,\bullet b}\|_2^4 \right] \leq K$ follows the same line and is therefore omitted.

\medskip
{\it Step 5.} Here, we prove Equation \eqref{eq:limit:ell:bar}.

We have already used in Step 2 that $\ell^{N,ab} = 1_{\mathcal{P}_a} \odot \left( 1_{\mathcal{P}_b} + (1-\lambda) A^{N}\ell^{N,\bullet b} \right)$ Hence, we deduce that 
\begin{eqnarray*}
    \overline{\ell^{N,ab}} &= &r^N_a \delta_{ab} + (1-\lambda) N^{-1} \sum_{i\in\mathcal{P}_a}\sum_{j=1}^{N} A^N(i,j) \ell^{N,\bullet b}_j\\
    &=& r^N_a \delta_{ab} + (1-\lambda) N^{-1} \sum_{j=1}^{N} C_{j}^{N,\bullet a} \ell_{j}^{N, \bullet b}\\
    &=& r^N_a \delta_{ab} + (1-\lambda)p r^N_a \left(\overline{\ell^{N,+ b}}-\overline{\ell^{N,- b}} \right) + \eta^{N,ab},
\end{eqnarray*}
where $\eta^{N,ab} = (1-\lambda) N^{-1} \left< C^{N,\bullet a} - pr^N_a(1_{\mathcal{P}_+}-1_{\mathcal{P}_-}), \ell^{N, \bullet b} \right>$. In particular, we have
\begin{equation*}
    \begin{cases}
        \overline{\ell^{N,+ b}}-\overline{\ell^{N,- b}}=\frac{(br^N_b)}{1-(1-\lambda)p(r^N_+-r^N_-)}+\eta^{N,+b}-\eta^{N,-b},\\
        \overline{\ell^{N,+ b}}+\overline{\ell^{N,- b}} = \frac{r^N_b + 2b(1-\lambda)pr^N_+r^N_-}{1-(1-\lambda)p(r^N_+-r^N_-)} + (1-\lambda) p (\eta^{N,+b} - \eta^{N,-b}) + \eta^{N,+b} + \eta^{N,-b},
    \end{cases}
\end{equation*}
and so
\begin{equation*}
    \overline{\ell^{N,ab}} = \frac{r^N_a \delta_{ab} +  b(1-\lambda)pr^N_+r^N_-}{1-(1-\lambda)p(r^N_+-r^N_-)} + \frac{(1-\lambda) p}{2} (\eta^{N,+b} - \eta^{N,-b}) + \eta^{N,ab}.\\
\end{equation*}
Since by our assumption, $|r^N_+ - r_+| + |r^N_- - r_-|\leq K N^{-1}$, we have
$$
\left|\frac{r^N_a \delta_{ab} +  b(1-\lambda)pr^N_+r^N_-}{1-(1-\lambda)p(r^N_+-r^N_-)} - \frac{r_a \delta_{ab} +  b(1-\lambda)pr_+r_-}{1-(1-\lambda)p(r_+-r_-)}\right|\leq K (|r^N_+ - r_+| + |r^N_- - r_-|)\leq KN^{-1},
$$
for some constant $K$ that may depend on $\lambda$, and the result will follow once we check that $\E[ (\eta^{N,ab})^2] \leq K N^{-2}$.

To this end, we write $\eta^{N,ab}= (1-\lambda) N^{-1} (\eta^{N,ab}_+ + \eta^{N,ab}_-)$ where, for $e\in\{-,+\}$, $\eta^{N,ab}_e = \left< C^{N,e a} - (ep)r^N_a1_{\mathcal{P}_e}, \ell^{N, e b} \right>$. Remind that $\ell^{N, e b} = \widetilde{\ell}^{N, e b} + (r_e^N)^{-1} \overline{\ell^{N,eb}} 1_{\mathcal{P}_e}$. Then, Lemma \ref{lem:technical:V:v} can be applied with $V^N = C^{N,e a} - (ep)r^N_a1_{\mathcal{P}_e}$, $v^N_1 = 0$, $v_2^N = \widetilde{\ell}^{N,eb}$ and $v_3^N = (r_e^N)^{-1} \overline{\ell^{N,eb}}1_{\mathcal{P}_e}$: assumption (i) is satisfied thanks to Equation \eqref{ineq:moments_l2norm_Lcentered_wrt_environment} and the fact that $\| C^{N,e a} \|_{\infty} \leq 1$, assumption (ii) is satisfied thanks to Equation \eqref{eq:bound:tilde}, assumption (iii) is satisfied thanks to Equations \eqref{omegancons1} and \eqref{ineq:max_variance_of_rarbbarLab_rarbbarCab} because $\left< V^N, v_3^N \right> = (r_e^N)^{-1} \overline{\ell^{N,eb}}\,  \overline{V^N}$.

Hence, $\E[ (\eta^{N,ab}_e)^2] \leq K$ which in turn implies that $\E[ (\eta^{N,ab})^2] \leq K N^{-2}$.

\medskip
{\it Step 6.} Here we prove Equation \eqref{eq:limit:c:bar}.

Starting from 
\begin{equation*}
c^{N}=(Q^{N})^{\intercal}1_{N} = \left(I_{N}-(1-\lambda)\left(A^{N}\right)^{\intercal}\right)^{-1}1_{N} ,
\end{equation*}
we deduce that $c^{N,+\bullet}+c^{N,-\bullet}=c^N= 1_{N} + (1-\lambda) \left(A^N\right)^{\intercal} c^N$, so that
 \begin{eqnarray*}
\overline{c^{N,a\bullet}}&= & r^N_a + (1-\lambda) N^{-1} \sum_{i\in\mathcal{P}_a} \sum_{j=1}^{N} A^N(j,i) c_{j}^N\\
&=& r^N_a + (1-\lambda) N^{-1} \sum_{j=1}^{N} L_{j}^{N,\bullet a}c_{j}^N\\
&=& r^N_a + (1-\lambda)(a p r^N_a)\overline{c^N} + \xi^{N,a},
\end{eqnarray*}
where $\xi^{N,a} = (1-\lambda) N^{-1} \left< L^{N,\bullet a} - (apr^N_a)1_{N}, c^N \right>$.
Since $\overline{c^N}=\overline{\ell^N}$, to conclude the proof of this step it suffices to use Equation \eqref{eq:limit:ell:bar} and to show that $\E[(\xi^{N,a})^2] \leq K N^{-2}$.

Proceeding as before, we can write $\xi^{N,a}= (1-\lambda) N^{-1} (\xi^{N,a}_+ + \xi^{N,a}_-)$ where, for $e\in\{-,+\}$, $\xi^{N,a}_e = \left< L^{N,e a} - (ap)r^N_a1_{\mathcal{P}_e}, c^{N, e \bullet} \right>$. Remind that $c^{N, e b} = \widetilde{c}^{N, e b} + (r_e^N)^{-1} \overline{c^{N,eb}} 1_{\mathcal{P}_e}$. Then, Lemma \ref{lem:technical:V:v} can be applied with $V^N = L^{N,e a} - (ap)r^N_a1_{\mathcal{P}_e}$, $v^N_1 = 0$ , $v_2^N= \widetilde{c}^{N,e \bullet}$ and $v_3^N = (r_e^N)^{-1} \overline{c^{N,e \bullet}}1_{\mathcal{P}_e}$:  assumption (i) is satisfied thanks to Equation \eqref{ineq:moments_l2norm_Lcentered_wrt_environment} and the fact that $\| L^{N,e a} \|_{\infty} \leq 1$, assumption (ii) is satisfied thanks to Equation \eqref{eq:bound:tilde}, assumption (iii) is satisfied thanks to Equations \eqref{omegancons1} and \eqref{ineq:max_variance_of_rarbbarLab_rarbbarCab} because $\left< V^N, v_3^N \right> = (r_e^N)^{-1} \overline{c^{N,e \bullet}}\,  \overline{V^N}$.

Hence, $\E[ (\xi^{N,a}_e)^2] \leq K$ which in turn implies that $\E[ (\xi^{N,a})^2] \leq K N^{-2}$.

\medskip
{\it Step 7.} Here, we prove that $\E\left[ B^{N,\bullet b}  \right]\leq K N^{-1/2}$ where
\begin{equation*}
    B^{N,\bullet b} 
    = \left| \|\widetilde{\ell}^{N, \bullet b}\|_2^2 - \|(1-\lambda)\sum_{a\in \{-,+\}} (r^N_a)^{-1}\overline{\ell^{N,ab}}\widetilde{L}^{N,\bullet a}\|_2^2 \right|.
\end{equation*}

Denoting $y^{N,\bullet b} = \widetilde{\ell}^{N, \bullet b} - (1-\lambda)\sum_{a\in \{-,+\}} (r^N_a)^{-1}\overline{\ell^{N,ab}}\widetilde{L}^{N,\bullet a}$, one can use Equation \eqref{eq:tilde:ell:bullet:b} to get
\begin{equation*}
    y^{N,\bullet b} = (1-\lambda)\left[ 
        A^{N}y^{N,\bullet b} + 
        (1-\lambda)\sum_{a\in \{-,+\}} (r^N_a)^{-1}\overline{\ell^{N,ab}} (A^{N} \widetilde{L}^{N,\bullet a})
        + \sum_{a\in \{-,+\}} \epsilon^{N,ab} 1_{\mathcal{P}_a}
     \right].
\end{equation*}
Using the same kind of arguments as in the beginning of Step 4 (except the use of the event $\mathcal{A}_N$), we have
\begin{multline*}
    ((1-\lambda)^{-1}-1) \|y^{N,\bullet b}\|_2 \leq  
    4 (1-\lambda) \lambda^{-1} \max_{a,e\in\{-,+\}}\left\{\|A^{N} \widetilde{L}^{N,a e}\|_2\right\} \\
    + \frac{2}{r^N_+\wedge r^N_-}N^{-1/2} \| \widetilde{\ell}^{N,\bullet b} \|_2 \max_{a\in\{-,+\}} \| C^{N,\bullet a} - pr_a(1_{\mathcal{P}_+} - 1_{\mathcal{P}_-}) \|_2.
\end{multline*}
Furthermore, using Cauchy-Schwarz inequality, Step 4 of the current Lemma and Equation \eqref{ineq:moments_l2norm_Lcentered_wrt_environment} of Lemma \ref{lemma_collection_of_bounds}, we have
\begin{equation*}
    \E\left[ \| \widetilde{\ell}^{N,\bullet b} \|_2^2 \max_{a\in\{-,+\}} \| C^{N,\bullet a} - pr_a(1_{\mathcal{P}_+} - 1_{\mathcal{P}_-}) \|_2^2 \right] \leq K.
\end{equation*}
Combining the two equations above by convexity of the square function and then Equation \eqref{ineq_l2_norm_of_AtildeLab} of Lemma \ref{lemma_collection_of_bounds}, we prove that $y^{N,\bullet b}$ is negligible in the sense that
\begin{equation*}
    \E\left[ \|y^{N,\bullet b}\|_2^2 \right] \leq  
    K \left\{ \E\left[ \max_{a,e\in\{-,+\}}\left\{\|A^{N} \widetilde{L}^{N,a e}\|_2^2\right\} \right]     + N^{-1} \right\} \leq \frac{K}{N}.
\end{equation*}
In turn, we can prove that $B^{N,\bullet b} $ is negligible. More precisely, we factorize
\begin{eqnarray*}
    B^{N,\bullet b} 
    &=& \left| \|\widetilde{\ell}^{N, \bullet b}\|_2 - \|(1-\lambda)\sum_{a\in \{-,+\}} (r^N_a)^{-1}\overline{\ell^{N,ab}}\widetilde{L}^{N,\bullet a}\|_2 \right|\\
    && \times \left( \|\widetilde{\ell}^{N, \bullet b}\|_2 + \|(1-\lambda)\sum_{a\in \{-,+\}} (r^N_a)^{-1}\overline{\ell^{N,ab}}\widetilde{L}^{N,\bullet a}\|_2 \right)\\
    &\leq& \|y^{N,\bullet b}\|_2 \left( \|\widetilde{\ell}^{N, \bullet b}\|_2 + K \max_{a\in \{-,+\}} \|\widetilde{L}^{N,\bullet a}\|_2 \right),
\end{eqnarray*}
so that, by Cauchy-Schwarz inequality, and then using Step 4 of the current Lemma and Equations \eqref{ineq:tildeL_Lcentered_wrt_environment} and
\eqref{ineq:moments_l2norm_Lcentered_wrt_environment} of Lemma \ref{lemma_collection_of_bounds}, we get
\begin{equation*}
    \E\left[ B^{N,\bullet b}  \right] \leq K \sqrt{\E\left[ \|y^{N,\bullet b}\|_2^2 \right]} \sqrt{\E\left[ \|\widetilde{\ell}^{N, \bullet b}\|_2^2 +  \max_{a\in \{-,+\}} \|\widetilde{L}^{N,\bullet a}\|_2^2 \right]}\leq \frac{K}{\sqrt{N}}.
\end{equation*}

\medskip
{\it Step 8.} 
Using the same arguments as Step 7, one can prove that
\begin{equation*}
    \E\left[ \left| \|\widetilde{\ell}^{N}\|_2^2 - \|(1-\lambda)\sum_{a\in \{-,+\}} (r^N_a)^{-1}\overline{\ell^{N,a \bullet}}\widetilde{L}^{N,\bullet a}\|_2^2 \right| \right] \leq \frac{K}{\sqrt{N}}.
\end{equation*}

\medskip
{\it Step 9.} 
Here we prove, for all $a_1,a_2, b\in \{-,+\}$, that 
\begin{equation*}
    \E\left[\left| \frac{\overline{\ell^{N,a_1b}}}{r^N_{a_1}} \frac{\overline{\ell^{N,a_2b}}}{r^N_{a_2}} \left< \widetilde{L}^{N,\bullet a_1}, \widetilde{L}^{N,\bullet a_2} \right> 
    - \delta_{a_1a_2} \frac{\overline{\ell^{\infty,a_1b}}}{r_{a_1}} \frac{\overline{\ell^{\infty,a_2b}}}{r_{a_2}} r_{a_1}p(1-p) \right|\right] \leq \frac{K}{\sqrt{N}},
\end{equation*}
where $\overline{\ell^{\infty,ab}} := \frac{r_a\delta_{ab} + b (1-\lambda)pr_+r_- }{1-(1-\lambda)p(r_+-r_-)}$ is the limit appearing in Equation \eqref{eq:limit:ell:bar}.

Since the coordinates of $\ell^{N,ab}$ are bounded by inequality \eqref{omegancons1}, we know that there exists a constant $K$ such that
\begin{equation*}
    \left| \frac{\overline{\ell^{N,a_1b}}}{r^N_{a_1}} \frac{\overline{\ell^{N,a_2b}}}{r^N_{a_2}} 
    - \frac{\overline{\ell^{\infty,a_1b}}}{r_{a_1}} \frac{\overline{\ell^{\infty,a_2b}}}{r_{a_2}} \right| 
    \leq K \left( \left| \frac{\overline{\ell^{N,a_1b}}}{r^N_{a_1}} - \frac{\overline{\ell^{\infty,a_1b}}}{r_{a_1}} \right| 
    + \left| \frac{\overline{\ell^{N,a_2b}}}{r^N_{a_2}} - \frac{\overline{\ell^{\infty,a_2b}}}{r_{a_2}} \right| \right).
\end{equation*}
Hence, using Equation \eqref{eq:limit:ell:bar}, Cauchy-Schwarz inequality twice and finally Equations \eqref{ineq:tildeL_Lcentered_wrt_environment} and
\eqref{ineq:moments_l2norm_Lcentered_wrt_environment} of Lemma \ref{lemma_collection_of_bounds}, we have
\begin{multline*}
    \E\left[ \left| \frac{\overline{\ell^{N,a_1b}}}{r^N_{a_1}} \frac{\overline{\ell^{N,a_2b}}}{r^N_{a_2}} - \frac{\overline{\ell^{\infty,a_1b}}}{r_{a_1}} \frac{\overline{\ell^{\infty,a_2b}}}{r_{a_2}} \right| \left| \left< \widetilde{L}^{N,\bullet a_1}, \widetilde{L}^{N,\bullet a_2} \right> \right|   \right] \\
    \leq  \frac{K}{N} \E\left[ \left< \widetilde{L}^{N,\bullet a_1}, \widetilde{L}^{N,\bullet a_2} \right>^2   \right]^{1/2} \leq \frac{K}{N} \max_{a\in\{-,+\} }\E\left[ \|\widetilde{L}^{N,\bullet a}\|_2^2 \right]^{1/2} \leq \frac{K}{N}.
\end{multline*}
Finally, we conclude the step by combining the above equation with Equation \eqref{ineq_l2_norm_of_tildeLab} of Lemma \ref{lemma_collection_of_bounds}.

\medskip
{\it Step 10.} 
Using the same arguments as in Step 9, one can prove that, for all $a_1,a_2\in \{-,+\}$,
\begin{equation*}
    \E\left[\left| \frac{\overline{\ell^{N,a_1 \bullet}}}{r^N_{a_1}} \frac{\overline{\ell^{N,a_2 \bullet}}}{r^N_{a_2}} \left< \widetilde{L}^{N,\bullet a_1}, \widetilde{L}^{N,\bullet a_2} \right> 
    - \delta_{a_1a_2} \frac{\overline{\ell^{\infty,a_1 \bullet}}}{r_{a_1}} \frac{\overline{\ell^{\infty,a_2 \bullet}}}{r_{a_2}} r_{a_1}p(1-p) \right|\right] \leq \frac{K}{\sqrt{N}},
\end{equation*}
where $\overline{\ell^{\infty,a \bullet}} := \overline{\ell^{\infty,a +}} + \overline{\ell^{\infty,a -}} = \frac{r_a}{1-(1-\lambda)p(r_+-r_-)}$.

\medskip
{\it Step 11.} 
Here we prove Equation \eqref{eq:limit:ell.b:tilde}. According to Step 7, the limit of $\|\widetilde{\ell}^{N,\bullet b}\|_2^2$ is related to the limit of $\|(1-\lambda)\sum_{a\in \{-,+\}} (r^N_a)^{-1}\overline{\ell^{N,ab}}\widetilde{L}^{N,\bullet a}\|_2^2$ which can be expanded as
\begin{equation*}
    (1-\lambda)^2 \sum_{a_1,a_2\in \{-,+\}} \frac{\overline{\ell^{N,a_1b}}}{r^N_{a_1}} \frac{\overline{\ell^{N,a_2b}}}{r^N_{a_2}} \left< \widetilde{L}^{N,\bullet a_1}, \widetilde{L}^{N,\bullet a_2} \right>.
\end{equation*}
Hence, combining Steps 7 and 9, we have
\begin{equation*}
    \E\left[ \left| \|\widetilde{\ell}^{N, \bullet b}\|_2^2 
    - 
    (1-\lambda)^2 p(1-p) \sum_{a\in \{-,+\}} \left( \frac{\overline{\ell^{\infty,ab}}}{r_{a}} \right)^2\, r_{a} \right| \right] \leq \frac{K}{\sqrt{N}}.
\end{equation*}
In order to conclude this step, it suffices to simplify
\begin{eqnarray*}
    \sum_{a\in \{-,+\}} \left( \frac{\overline{\ell^{\infty,ab}}}{r_{a}} \right)^2\, r_{a}
    &=& \sum_{a\in \{-,+\}} \left( \frac{\delta_{ab} + b(1-\lambda)p(1-r_a)}{1-(1-\lambda)p(r_+-r_-)} \right)^2\, r_{a}\\ 
    &=& \frac{\left( 1 + b(1-\lambda)p(1-r_b) \right)^2\, r_{b} + \left( b(1-\lambda)pr_b \right)^2\, (1-r_{b})}{(1-(1-\lambda)p(r_+-r_-))^2}\\ 
    &=& \frac{r_b + 2b(1-\lambda)pr_b(1-r_b) + (1-\lambda)^2p^2r_b(1-r_b)}{(1-(1-\lambda)p(r_+-r_-))^2}\\ 
    &=& \frac{r_b + (1-\lambda)pr_+r_-(2b + (1-\lambda)p)}{(1-(1-\lambda)p(r_+-r_-))^2}.
\end{eqnarray*}

\medskip
{\it Step 12.} 
Following the lines of Step 11 (and using Steps 8 and 10), one can prove Equation \eqref{eq:limit:ell:tilde}. Let us mention that the final simplification here is
\begin{eqnarray*}
    \sum_{a\in \{-,+\}} \left( \frac{\overline{\ell^{\infty,a \bullet}}}{r_{a}} \right)^2\, r_{a}
    &=& \sum_{a\in \{-,+\}} \left( \frac{1}{1-(1-\lambda)p(r_+-r_-)} \right)^2\, r_{a}\\ 
    &=& (1-(1-\lambda)p(r_+-r_-))^{-2}.
\end{eqnarray*}

\subsection{Proof of Lemma \ref{lem:control:ell-ellbar:elltilde}}
\label{sec:proof:lem:control:ell-ellbar:elltilde}
It suffices to show that these inequalities hold for all $N$ sufficiently large. Recall that $r^N_a=|{\cal{P}}_a|/N$, for $a\in\{-,+\}$. As in the previous proofs, in what follows we assume that $N$ is large enough ($N\geq N_0$) ensuring that $r^N_+\wedge r^N_-\geq r_{\rm min}>0$ for some $r_{\rm min}$ sufficiently small depending only on the choice of $r_+$ and $r_-$. 
Also, we shall denote $K$ a constant which may depend on $\lambda$ and which may change from one line to another. The proof is divided in several steps. 

We will first prove \eqref{eq:control:ell-ellbar:elltilde}. To see that, we start by observing that
\begin{align*}
\|\ell^N-\overline{\ell^N}1_N\|^2_2-\|\widetilde{\ell}^N\|^2_2 &=\sum_{a\in\{-,+\}}\sum_{i\in\mathcal{P}_a}(\ell^{N,a\bullet}_i-\overline{\ell^N})^2-(\ell^{N,a\bullet}_i-(r^N_a)^{-1}\overline{\ell^{N,a\bullet}})^2\\
&= \sum_{a\in\{-,+\}} ((r^N_a)^{-1}\overline{\ell^{N,a\bullet}}-\overline{\ell^N})\sum_{i\in\mathcal{P}_a}(\ell^{N,a\bullet}_i-\overline{\ell^{N}}),\\
&=N\sum_{a\in\{-,+\}} r^N_a \left((r^N_a)^{-1}\overline{\ell^{N,a\bullet}}-\overline{\ell^N}\right)^2,
\end{align*} 
where in the second equality we have used that $x^2-y^2=(x-y)(x+y)$ for all $x,y\in\mathbb{R}$ and $\sum_{i\in\mathcal{P}_a}(\ell^{N,a\bullet}_i-(r^N_a)^{-1}\overline{\ell^{N,a\bullet}})=0$. 
Then, we combine \eqref{eq:limit:ell:bar}, \eqref{eq:limit:ell:bar:full} and \eqref{omegancons1} together with Jensen inequality to obtain that 
\begin{multline*}
N\E\left[\left((r^N_a)^{-1}\overline{\ell^{N,a\bullet}}-\overline{\ell^N}\right)^2\right]\leq KN\left(\E\left[\left((r_a)^{-1}\overline{\ell^{N,a\bullet}}-\frac{1}{1-(1-\lambda)p(r_+-r_-)} \right)^2\right] \right. \\ \left.
+\E\left[|r^N_a-r_a|^2\right] + \E\left[\left(\overline{\ell^{N}}-\frac{1}{1-(1-\lambda)p(r_+-r_-)} \right)^2\right]\right)  \leq KN^{-1},
\end{multline*}
so that 
$$
\E\left[\left|\|\ell^N-\overline{\ell^N}1_N\|^2_2-\|\widetilde{\ell}^N\|^2_2\right|\right]\leq KN^{-1}\sum_{a\in\{-,+\}}r^N_a=KN^{-1},
$$
which proves \eqref{eq:control:ell-ellbar:elltilde}. 

We will now prove \eqref{eq:control:product:ell-ellbar:elltilde}. To that end, first note that
\begin{equation}
\label{decomp:QL-QLbar}
(1-\lambda)(Q^NL^{N,\bullet-}-\overline{Q^NL^{N,\bullet-}})=\ell^{N,\bullet -}-\overline{\ell^{N,\bullet -}}1_N+r^N_-1_{\mathcal{P}_+}-r^N_+1_{\mathcal{P}_-}.
\end{equation}
Then, note that 
\begin{align*}
\langle \ell^N-\overline{\ell^N}1_N, r^N_-1_{\mathcal{P}_+}-r^N_+1_{\mathcal{P}_-}\rangle &=r^N_-\langle \ell^{N,+\bullet}-\overline{\ell^N}1_{\mathcal{P}_+},1_{\mathcal{P}_+} \rangle-r^N_+\langle \ell^{N,-\bullet}-\overline{\ell^N}1_{\mathcal{P}_-},1_{\mathcal{P}_-} \rangle\\
  &= Nr^N_{-}\left[\overline{\ell^{N,+\bullet}}-\overline{\ell^N}r^N_+\right]-Nr^N_{+}\left[\overline{\ell^{N,-\bullet}}-\overline{\ell^N}r^N_-\right].
\end{align*}
Next, by using that $\ell^{N}-\overline{\ell^N}1_N=\widetilde{\ell}^{N}+\sum_{a\in\{-,+\}}1_{\mathcal{P}_a}((r^N_a)^{-1}\overline{\ell^{N,a\bullet}}-\overline{\ell^N})$ and that 
$\ell^{N,\bullet -}-\overline{\ell^{N,\bullet -}}1_N=\widetilde{\ell}^{N,\bullet -}+\sum_{a\in\{-,+\}}1_{\mathcal{P}_a}((r^N_a)^{-1}\overline{\ell^{N,a - }}-\overline{\ell^{N,\bullet -}})$,
one can check that
\begin{multline*}
\langle \ell^{N}-\overline{\ell^N}1_N, \ell^{N,\bullet -}-\overline{\ell^{N,\bullet -}}1_N \rangle= \langle \widetilde{\ell}^N, \widetilde{\ell}^{N,\bullet -} \rangle\\
+N\sum_{a\in\{-,+\}}r^N_a\left[(r_a^N)^{-1}\overline{\ell^{N,a\bullet}}-\overline{\ell^{N}}\right]\left[(r_a^N)^{-1}\overline{\ell^{N,a -}}-\overline{\ell^{N,\bullet -}}\right].   
\end{multline*}
Now, observe that using the fact that 
$$\frac{\delta_{a-}-(1-\lambda)pr_+r_-(r_a)^{-1}}{1-(1-\lambda)p(r_+-r_-)}+ar_{-a}=\frac{r_--2(1-\lambda)pr_-r_+}{1-(1-\lambda)p(r_+-r_-)},
$$
we can write 
$$
\left[(r_a^N)^{-1}\overline{\ell^{N,a\bullet}}-\overline{\ell^{N}}\right]\left[(r_a^N)^{-1}\overline{\ell^{N,a -}}-\overline{\ell^{N,\bullet -}}\right]=-ar_{-a}\left[(r_a^N)^{-1}\overline{\ell^{N,a\bullet}}-\overline{\ell^{N}}\right]+\xi^N_a,
$$
where 
\begin{multline*}
\xi^N_a=\left[(r_a^N)^{-1}\overline{\ell^{N,a\bullet}}-\overline{\ell^{N}}\right]\left[\left((r^N_a)^{-1}\overline{\ell^{N,a-}}-\frac{(\delta_{a-}-(1-\lambda)pr_+r_-r^{-1}_{-a})}{1-(1-\lambda)p(r_+-r_-)}\right)\right.\\ \left.
+\left(\frac{(r_--2(1-\lambda)pr_+r_-)}{1-(1-\lambda)p(r_+-r_-)}\right)
\right].
\end{multline*}
Combining \eqref{eq:limit:ell:bar} and \eqref{eq:limit:ell:bar:full} with Jensen inequality, one can show that $\E\left[|\xi^N_a|\right]\leq KN^{-2}$,  
so that
\begin{multline*}
N\sum_{a\in\{-,+\}}r^N_a\left[(r_a^N)^{-1}\overline{\ell^{N,a\bullet}}-\overline{\ell^{N}}\right]\left[(r_a^N)^{-1}\overline{\ell^{N,a -}}-\overline{\ell^{N,\bullet -}}\right]=-Nr_{-}\left[\overline{\ell^{N,+\bullet}}-\overline{\ell^{N}}r_+^N\right] \\
+Nr_{+}\left[\overline{\ell^{N,-\bullet}}-\overline{\ell^{N}}r_-^N\right]+\xi^N
\end{multline*}
where $\xi^N=N\sum_{a\in\{-,+\}}r^N_a\xi^N_a$ satisfies $\E\left[|\xi^N|\right]\leq KN^{-1}.$
Therefore, putting together all previous identities, we deduce that 
\begin{multline*}
\langle \ell^N-\overline{\ell^N}1_N, (1-\lambda)(Q^NL^{N,\bullet-}-\overline{Q^NL^{N,\bullet-}})  \rangle -\langle \widetilde{\ell}^{N}
, \widetilde{\ell}^{N,\bullet -}\rangle=N(r^N_{-}-r_-)\left[\overline{\ell^{N,+\bullet}}-\overline{\ell^N}r^N_+\right]\\
N(r_+-r^N_{+})\left[\overline{\ell^{N,-\bullet}}-\overline{\ell^N}r^N_-\right]+\xi^N,
\end{multline*}
and the result follows from \eqref{eq:limit:ell:bar}, \eqref{eq:limit:ell:bar:full} and the assumption that $|r^N_a-r_a|\leq KN^{-1}$.

Hence, it remains to prove only \eqref{eq:control:QL-QL:elltilde}. Starting from \eqref{decomp:QL-QLbar}, one can check that (recall the definition of $\widetilde{\ell}^{N,ab}$ given in \eqref{def:centered_versions_of_ellab_and_cab})
$$
(1-\lambda)(Q^NL^{N,\bullet-}-\overline{Q^NL^{N,\bullet-}})=\widetilde{\ell}^{N,\bullet -}+\sum_{a\in\{-,+\}}1_{\mathcal{P}_a}\left[(r^N_a)^{-1}\overline{\ell^{N,a-}}-\overline{\ell^{N,\bullet-}}+ar^N_{-a}\right],
$$
which together with the fact that $\langle 1_{\mathcal{P}_+}, 1_{\mathcal{P}_-} \rangle=\langle 1_{\mathcal{P}_a}, \widetilde{\ell}^{N,\bullet -} \rangle=0$ implies that  
\begin{align*}
(1-\lambda)^2\|(Q^NL^{N,\bullet-}-\overline{Q^NL^{N,\bullet-}})\|^2_2-
\widetilde{\ell}^{N,\bullet -}=\sum_{a\in\{-,+\}}|\mathcal{P}_a|\left[(r^N_a)^{-1}\overline{\ell^{N,a-}}-\overline{\ell^{N,\bullet-}}+ar^N_{-a}\right]^2.
\end{align*}
Proceeding similarly as in the proof of \eqref{eq:control:product:ell-ellbar:elltilde}, one can show that
$$
\sum_{a\in\{-,+\}}|\mathcal{P}_a|\E\left[\left[(r^N_a)^{-1}\overline{\ell^{N,a-}}-\overline{\ell^{N,\bullet-}}+ar^N_{-a}\right]^2\right]\leq KN^{-1},
$$
concluding the proof of the lemma.

\subsection{Proof of Proposition \ref{prop:control:N:infty}}

We are now able to prove the three inequalities. 

\begin{proof}[Proof of Inequality \eqref{eq:convergence:m}]
First, remind that $\overline{m^N} = \mu \overline{\ell^N} - (1-\lambda)\overline{Q^NL^{N,\bullet-}}$.
Thanks to Remark \ref{rem:limit:ell:bar:full}, it suffices to prove that 
\begin{equation*}
    \E\left[\left|\overline{Q^NL^{N,\bullet-}}+\frac{pr_-}{1-(1-\lambda)p(r_+-r_-)}\right|^2\right]\leq \frac{K}{N^2}.
\end{equation*}
First, note that 
$$
Q^NL^{N,\bullet-}=-pr^N_-\ell^N+Q^N(L^{N,\bullet-}+pr^N_-1_{N}).
$$
Hence, it follows that
\begin{eqnarray*}
\overline{Q^NL^{N,\bullet-}}&=&-pr^N_-\overline{\ell^N}+N^{-1}\sum_{i=1}^{N}\sum_{j=1}^{N}Q^N(i,j)(L_{j}^{N,\bullet-}+pr^N_-)\\
&=&-pr^N_-\overline{\ell^N}+N^{-1}\sum_{j=1}^{N}(L_{j}^{N,\bullet-}+pr^N_-)c^N_{j}\\
&=&-pr^N_-\overline{\ell^N}+(1-\lambda)^{-1} \xi^{N,-},
\end{eqnarray*}
where $\xi^{N,-}=(1-\lambda)N^{-1} \left< L^{N,\bullet-} +pr^N_-1_{N}, c^N \right>$ is already defined in Step 6 of the proof of Lemma \ref{lem:ell:and:c}. From there, we know that $\E\left[|\xi^{N,-}|^2\right]\leq KN^{-2}$. As a consequence, it follows that 
\begin{multline*}
\E\left[\Big|\overline{Q^NL^{N,\bullet-}}+\frac{pr_-}{1-(1-\lambda)p(r_+-r_-)}\Big|^2\right]\leq\\
K\left(\E\left[\Big|\overline{\ell^N}-\frac{1}{1-(1-\lambda)p(r_+-r_-)}\Big|^2\right]+N^{-2}\right)\leq  KN^{-2},
\end{multline*}
where in the last inequality we used Equation \eqref{eq:limit:ell:bar:full}.

\end{proof}

\begin{proof}[Proof of Inequality \eqref{eq:convergence:v}]
Expanding the scalar product $v_\infty^N = \left\| m^N - \overline{m^N}1_{N} \right\|_2^2$ and using Lemma \ref{lem:control:ell-ellbar:elltilde}, we have
\begin{equation*}
    \E\left[\left| \left\| m^N - \overline{m^N}1_{N} \right\|_2^2 - \left( \mu^2 \|\widetilde{\ell}^N\|_2^2 -2\mu\left< \widetilde{\ell}^N, \widetilde{\ell}^{N,\bullet -} \right> + \|\widetilde{\ell}^{N,\bullet -}\|_2^2 \right) \right| \right] \leq \frac K {N}.
\end{equation*}
Yet, by the polarization identity, $-2\left< \widetilde{\ell}^N, \widetilde{\ell}^{N,\bullet -} \right> = \left( \|\widetilde{\ell}^{N,\bullet +}\|_2^2 - \|\widetilde{\ell}^{N}\|_2^2 - \|\widetilde{\ell}^{N,\bullet -}\|_2^2 \right)$, so that 
\begin{equation*}
    \E\left[\left| \left\| m^N - \overline{m^N}1_{N} \right\|_2^2 - \left( (\mu^2-\mu) \|\widetilde{\ell}^N\|_2^2 + \mu \|\widetilde{\ell}^{N,\bullet +}\|_2^2 + (1-\mu) \|\widetilde{\ell}^{N,\bullet -}\|_2^2 \right) \right| \right] \leq \frac K {N},
\end{equation*}
and we conclude this step thanks to Equations \eqref{eq:limit:ell.b:tilde} and \eqref{eq:limit:ell:tilde} as soon as we check the simplification:
\begin{eqnarray*}
    I &=& (\mu^2-\mu) + \mu \left[ r_+ + (1-\lambda)pr_+r_-(2 + (1-\lambda)p) \right] \\
    && +(1-\mu) \left[ r_- + (1-\lambda)pr_+r_-(-2 + (1-\lambda)p) \right]\\
    &=& \mu^2 + \mu(-1 + r_+ - r_-) + r_- \\
    && + (1-\lambda)pr_+r_- [\mu (2+(1-\lambda)p +2 -(1-\lambda)p) -2 + (1-\lambda)p]\\
    &=& \mu^2 - 2\mu r_- + r_- +(1-\lambda)pr_+r_- (4\mu - 2 +(1-\lambda)p)\\
    &=& (\mu+(1-\lambda)pr_-)^2+r_-(1-(1-\lambda)p(r_+-r_-))(1-2\mu-(1-\lambda)p)
\end{eqnarray*}
and the identity
$$
\frac{1-2\mu-(1-\lambda)p}{1-(1-\lambda)p(r_+-r_-)}=1-2m.
$$
\end{proof}

\begin{proof}[Proof of Inequality \eqref{eq:convergence:w}]
The proof is divided in the 5 steps below.

\medskip
{\it Step 1.} Here we prove that
\begin{equation}
    \E\left[\left|\frac{1}{N} \sum_{i=1}^{N} \ell^N_{i}(c^N_{i})^2  - \frac{1+4(1-\lambda)^2 p^2 r_-r_+}{(1-(1-\lambda)p(r_+-r_-))^3} \right|^2 \right]  \leq  \frac K  {N^2} \label{eq:product:c2:ell}.
\end{equation}

Notice that $\sum_{i=1}^{N} \ell_{i}^N(c_{i}^N)^2 = \sum_{a\in\{-,+\}} \left< \ell^{N,a\bullet}, (c^{N,a\bullet})^2 \right>$ and $(c^{N,a\bullet})^2=(\widetilde{c}^{N,a\bullet})^2+2(r^N_a)^{-1}\overline{c^{N,a\bullet}}\widetilde{c}^{N,a\bullet}+((r^N_a)^{-1}\overline{c^{N,a\bullet}})^2 1_{\mathcal{P}_a}$ for each $a\in\{-,+\}$. We argue that Lemma \ref{lem:technical:u:v} can be applied to
\begin{equation*}
    \begin{cases}
        u_1^N = 0,\, u_2^N = \widetilde{\ell}^{N,a\bullet},\, u_3^N = (r_a^N)^{-1}\overline{\ell^{N,a\bullet}},\\
        u_3^\infty = (1-(1-\lambda)p(r_+-r_-))^{-1},\\
        v_1^N = (\widetilde{c}^{N,a\bullet})^2,\, v_2^N = (r_a^N)^{-1}\overline{c^{N,a\bullet}}\widetilde{c}^{N,a\bullet},\, v_3^N = ((r_a^N)^{-1}\overline{c^{N,a\bullet}})^{2},\\
        v_3^\infty = \left( \frac{\left[1+(2a)(1-\lambda)p(1-r_a)\right]}{1-(1-\lambda)p(r_+-r_-)} \right)^2.
    \end{cases}
\end{equation*}
Indeed, the assumptions are satisfied thanks to Equations \eqref{omegancons1}, \eqref{eq:limit:ell:bar}, \eqref{eq:bound:tilde}, the fact that $\widetilde{\ell}^{N,a\bullet}$ and $\widetilde{c}^{N,a\bullet}$ are centered (for instance $\overline{\widetilde{\ell}^{N,a\bullet}}=0$) and the following argument: using the fact that $\overline{c^{N,a\bullet}} \leq 1$ and the previous step, we have
\begin{equation*}
    \E\left[ \left| v_3^N - v_3^\infty \right|^2 \right] \leq K \E\left[\left|\overline{c^{N,a\bullet}} - r_a\frac{\left[1+(2a)(1-\lambda)p(1-r_a)\right]}{1-(1-\lambda)p(r_+-r_-)}\right|^2  \right] \leq \frac K {N^{2}}.
\end{equation*}
The conclusion of Lemma \ref{lem:technical:u:v} is
\begin{equation*}
    \E\left[ \left| \frac{1}{N} \left< \ell^{N,a\bullet}, (c^{N,a\bullet})^2 \right> - r_a \frac{\left[1+(2a)(1-\lambda)p(1-r_a)\right]^2}{(1-(1-\lambda)p(r_+-r_-))^3} \right|^2 \right] \leq \frac{K}{N^2}.
\end{equation*}
Finally, the conclusion of this step follows from the fact that
\begin{equation}\label{eq:sum:c:bar}
    \sum_{a\in\{-,+\}} r_a \left[1+(2a)(1-\lambda)p(1-r_a)\right]^2 = 1+4(1-\lambda)^2 p^2 r_-r_+.
\end{equation}

\medskip
{\it Step 2.} Here we prove that
\begin{equation}
    \E\left[\left|\frac{1}{N} \sum_{i=1}^{N} (\ell^N_{i}c^N_{i})^2  - \frac{1+4(1-\lambda)^2 p^2 r_-r_+}{(1-(1-\lambda)p(r_+-r_-))^4}\right|^2 \right]  \leq  \frac K  {N^2}. \label{eq:product:c2:ell2}
\end{equation}

Notice that $\sum_{i=1}^{N} (\ell_{i}^N c_{i}^N)^2 = \sum_{a\in\{-,+\}} \left< (\ell^{N,a\bullet})^2, (c^{N,a\bullet})^2 \right>$. We argue that Lemma \ref{lem:technical:u:v} can be applied to
\begin{equation*}
    \begin{cases}
        u_1^N = (\widetilde{\ell}^{N,a\bullet})^2,\, u_2^N = (r_a^N)^{-1}\overline{\ell^{N,a\bullet}}\widetilde{\ell}^{N,a\bullet},\, u_3^N = ((r_a^N)^{-1}\overline{\ell^{N,a\bullet}})^{2},\\
        u_3^\infty = (1-(1-\lambda)p(r_+-r_-))^{-2},\\
        v_1^N = (\widetilde{c}^{N,a\bullet})^2,\, v_2^N = (r_a^N)^{-1}\overline{c^{N,a\bullet}}\widetilde{c}^{N,a\bullet},\, v_3^N = ((r_a^N)^{-1}\overline{c^{N,a\bullet}})^{2},\\
        v_3^\infty = \left( \frac{\left[1+(2a)(1-\lambda)p(1-r_a)\right]}{1-(1-\lambda)p(r_+-r_-)} \right)^2.
    \end{cases}
\end{equation*}
Indeed, the assumptions are satisfied thanks to Equations \eqref{omegancons1}, \eqref{eq:limit:ell:bar}, \eqref{eq:bound:tilde}, the fact that $\widetilde{\ell}^{N,a\bullet}$ and $\widetilde{c}^{N,a\bullet}$ are centered (for instance $\overline{\widetilde{\ell}^{N,a\bullet}}=0$) and the fact that $\E\left[ \left| u_3^N - u_3^\infty \right|^2 \right] \leq KN^{-2}$ can be proved like we did in the previous step for $v_3^N$

The conclusion of Lemma \ref{lem:technical:u:v} is
\begin{equation*}
    \E\left[ \left| \frac{1}{N} \left< (\ell^{N,a\bullet})^2, (c^{N,a\bullet})^2 \right> - r_a \frac{\left[1+(2a)(1-\lambda)p(1-r_a)\right]^2}{(1-(1-\lambda)p(r_+-r_-))^4} \right|^2 \right] \leq \frac{K}{N^2}.
\end{equation*}
Finally, the conclusion of this step follows one again from Equation \eqref{eq:sum:c:bar}.

\medskip

{\it Step 3.} 
In this step, we show the following result:
\begin{equation}\label{eq:product:c2:m}
    N^{-1} \langle (c^N)^2,m^N \rangle=(\mu+(1-\lambda)pr^N_-) N^{-1}\langle (c^N)^2,{\ell}^N \rangle + \xi^{N,(2),-},
\end{equation}
where $\E\left[(\xi^{N,(2),-})^2\right]\leq KN^{-2}$ for some constant $K>0$.

On the one hand, recall that $m^N=\mu\ell^N-(1-\lambda)Q^NL^{N,\bullet -}=(\mu+(1-\lambda)pr^N_-)\ell^N-(1-\lambda)Q^{N}\left(L^{N,\bullet -}+pr^N_-1_N\right)$. On the other hand, remark that $(c^{N})^2 = (c^{N,+\bullet})^2 + (c^{N,-\bullet})^2$. Hence, \eqref{eq:product:c2:m} is satisfied with $\xi^{N,(2),-} = -(1-\lambda)N^{-1}(\xi^{N,(2),-}_+ + \xi^{N,(2),-}_-)$ where for each $e\in\{-,+\}$,
\begin{equation*}
    \xi^{N,(2),-}_e = \langle (Q^N) [L^{N,\bullet -}+pr^N_-1_N],(c^{N,e \bullet})^2\rangle = \langle  L^{N,\bullet -}+pr^N_-1_N, (Q^N)^{\intercal}(c^{N,e \bullet})^2\rangle.
\end{equation*}

Therefore, it remains to prove that $\E\left[(\xi^{N,(2),-}_e)^2\right]\leq K$ for some constant $K>0$.
Remind that $(c^{N,e\bullet})^2=(\widetilde{c}^{N,e\bullet})^2+2(r^N_e)^{-1}\overline{c^{N,e\bullet}}\widetilde{c}^{N,e\bullet}+((r^N_e)^{-1}\overline{c^{N,e\bullet}})^2 1_{\mathcal{P}_e}$ for each $e\in\{-,+\}$.
Then, Lemma \ref{lem:technical:V:v} can be applied with $V^N = L^{N,\bullet -}+pr^N_-1_{N}$, $v^N_1 = (Q^N)^{\intercal} (\widetilde{c}^{N,e\bullet})^2$, $v_2^N = 2  (r_e^N)^{-1}\overline{c^{N,e\bullet}}  (Q^N)^{\intercal}\widetilde{c}^{N,e\bullet}$ and $v_3^N =  ((r_e^N)^{-1} \overline{c^{N,e\bullet}})^2  (Q^N)^{\intercal} 1_{\mathcal{P}_e} = ((r_e^N)^{-1} \overline{c^{N,e\bullet}})^2  c^{N,e\bullet}$. Assumption (i) is satisfied thanks to Equation \eqref{ineq:moments_l2norm_Lcentered_wrt_environment} and the fact that $\| L^{N,\bullet -} \|_{\infty} \leq 1$, assumption (ii) is satisfied thanks to Equations \eqref{eq:bound:tilde}, the facts that $\| (\widetilde{c}^{N,e\bullet})^2 \|_1 = \| \widetilde{c}^{N,e\bullet} \|_2^2$ and $\max\{\tn (Q^N)^{\intercal}\tn _2, \tn (Q^N)^{\intercal}\tn _1\} \leq \lambda^{-1}$, assumption (iii) is satisfied because $\left< V^N, v_3^N \right> = ((r_e^N)^{-1} \overline{c^{N,e\bullet}})^2 \left< L^{N,\bullet -}+pr^N_-1_{N}, c^{N,e\bullet} \right> = ((r_e^N)^{-1} \overline{c^{N,e\bullet}})^2 \xi^{N,-}_e$ which is defined and controlled in Step 7.

Hence, $\E[ (\xi^{N,(2),-}_e)^2] \leq K$ which in turn implies that $\E\left[(\xi^{N,(2),-})^2\right]\leq KN^{-2}$.

\medskip
{\it Step 4.} 
In this step, we show the following result:
\begin{equation}\label{eq:product:c2:m2}
N^{-1} \langle (c^N)^2,(m^N)^2 \rangle = (\mu+(1-\lambda)pr^N_-)^2 N^{-1}\langle (c^N)^2,({\ell}^N)^2 \rangle + \xi^{N,(3),-} + \xi^{N,(2,2),-},
\end{equation}
where $\xi^{N,(3),-}$ and $\xi^{N,(2,2),-}$ satisfy $\E\left[(\xi^{N,(3),-})^2 + (\xi^{N,(2,2),-})^2\right]\leq KN^{-2}$ for some constant $K>0$.

On the one hand, from the identity $m^N=(\mu+(1-\lambda)pr^N_-)\ell^N-(1-\lambda)Q^N(L^{N,\bullet -}+pr^N_-1_N)$, one can check that
\begin{multline*}
(m^N)^2=(\mu+(1-\lambda)pr^N_-)^2(\ell^N)^2\\
-2(\mu+(1-\lambda)pr^N_-)(1-\lambda)\ell^N\odot Q^N(L^{N,\bullet -}+pr^N_-1_N) + (1-\lambda)^2(Q^N(L^{N,\bullet -}+pr^N_-1_N))^2.
\end{multline*}
On the other hand, remind that $(c^{N})^2 = (c^{N,+\bullet})^2 + (c^{N,-\bullet})^2$ and $\ell^N = \ell^{N,+\bullet} + \ell^{N,-\bullet}$. Hence, Equation \eqref{eq:product:c2:m2} is satisfied with $\xi^{N,(3),-} = - 2(\mu+(1-\lambda)pr^N_-)(1-\lambda) N^{-1} (\xi^{N,(3),-}_+ + \xi^{N,(3),-}_-)$ where, for each $e\in\{-,+\}$,
\begin{equation*}
    \xi^{N,(3),-}_e = \left< (Q^N)^{\intercal} [(c^{N,e\bullet})^2 \odot \ell^{N,e\bullet}], L^{N,\bullet -}+pr^N_-1_N \right>,
\end{equation*}
and $\xi^{N,(2,2),-} = (1-\lambda)^2 N^{-1}(\xi^{N,(2,2),-}_+ + \xi^{N,(2,2),-}_-)$ where 
\begin{equation*}
    \xi^{N,(2,2),-}_e = \left< (Q^N)^{\intercal} [(c^{N,e\bullet})^2 \odot (Q^N(L^{N,\bullet -}+pr^N_-1_N))], L^{N,\bullet -}+pr^N_-1_N \right>.
\end{equation*}

Let us first check that $\E\left[ (\xi^{N,(3),-}_e)^2 \right] \leq K$.

Remind that $\ell^{N,e\bullet}=\widetilde{\ell}^{N,e\bullet} + (r^N_e)^{-1}\overline{\ell^{N,e\bullet}} 1_{\mathcal{P}_e}$ for each $e\in\{-,+\}$.
Then, Lemma \ref{lem:technical:V:v} can be applied with $V^N = L^{N,\bullet -}+pr^N_-1_{N}$, $v^N_1 = 0$, $v_2^N = (Q^N)^{\intercal} [(c^{N,e\bullet})^2 \odot \widetilde{\ell}^{N,e\bullet}]$ and $v_3^N =  (r_e^N)^{-1} \overline{\ell^{N,e\bullet}}  (Q^N)^{\intercal} (c^{N,e\bullet})^2$. Assumption (i) is satisfied thanks to Equation \eqref{ineq:moments_l2norm_Lcentered_wrt_environment} and the fact that $\| L^{N,\bullet -} \|_{\infty} \leq 1$, assumption (ii) is satisfied thanks to Equations \eqref{omegancons1}, \eqref{eq:bound:tilde}, the facts that $\tn (Q^N)^{\intercal}\tn _2 \leq \lambda^{-1}$ and $\|(c^{N,e\bullet})^2 \odot \widetilde{\ell}^{N,e\bullet}\|_2 \leq \|c^{N,e\bullet}\|_\infty^2 \|\widetilde{\ell}^{N,e\bullet}\|_2$, assumption (iii) is satisfied because $\left< V^N, v_3^N \right> = (r_e^N)^{-1} \overline{\ell^{N,e\bullet}} \left< L^{N,\bullet -}+pr^N_-1_{N}, (Q^N)^{\intercal}(c^{N,e\bullet})^2 \right> = (r_e^N)^{-1} \overline{\ell^{N,e\bullet}} \xi^{N,(2),-}_e$ which is defined and controlled in Step 3. Hence, $\E[ (\xi^{N,(3),-}_e)^2] \leq K$.

Now, let us check that $\E\left[ (\xi^{N,(2,2),-}_e)^2 \right] \leq K$.

Remind that $(c^{N,e\bullet})^2=(\widetilde{c}^{N,e\bullet})^2+2(r^N_e)^{-1}\overline{c^{N,e\bullet}}\widetilde{c}^{N,e\bullet}+((r^N_e)^{-1}\overline{c^{N,e\bullet}})^2 1_{\mathcal{P}_e}$ for each $e\in\{-,+\}$.
Then, Lemma \ref{lem:technical:V:v} can be applied with $V^N = L^{N,\bullet -}+pr^N_-1_{N}$, 
$$
v^N_1 = (Q^N)^{\intercal} [(\widetilde{c}^{N,e\bullet})^2 \odot (Q^N V^N)],
$$
$$
v_2^N = 2(r^N_e)^{-1}\overline{c^{N,e\bullet}} (Q^N)^{\intercal} [\widetilde{c}^{N,e\bullet} \odot (Q^N V^N)],
$$ and 
$$
v_3^N =  ((r^N_e)^{-1}\overline{c^{N,e\bullet}})^2 (Q^N)^{\intercal} [1_{\mathcal{P}_e} \odot (Q^N V^N)].
$$
Assumption (i) is satisfied thanks to Equation \eqref{ineq:moments_l2norm_Lcentered_wrt_environment} and the fact that $\| L^{N,\bullet -} \|_{\infty} \leq 1$, assumption (ii) is satisfied thanks to Equations \eqref{omegancons1}, \eqref{eq:bound:tilde}, the facts that $\| (\widetilde{c}^{N,e\bullet})^2 \|_1 = \| \widetilde{c}^{N,e\bullet} \|_2^2$,  $\max\{\tn (Q^N)^{\intercal}\tn _2, \tn (Q^N)^{\intercal}\tn _1, \tn Q^N\tn _\infty\} \leq \lambda^{-1}$ and $\|V^N\|_\infty \leq 2$, assumption (iii) is satisfied thanks to Equation \eqref{ineq:moments_l2norm_Lcentered_wrt_environment} because 
\begin{eqnarray*}
    \left| \left< V^N, v_3^N \right> \right| 
    & = & ((r^N_e)^{-1}\overline{c^{N,e\bullet}})^2 \left| \left< V^N, (Q^N)^{\intercal} [1_{\mathcal{P}_e} \odot (Q^N V^N)] \right> \right|\\
    &\leq& ((r^N_e)^{-1}\overline{c^{N,e\bullet}})^2 \left\| V^N \right\|_2 \tn (Q^N)^{\intercal}\tn _{2} \left\| 1_{\mathcal{P}_e} \odot (Q^N V^N) \right\|_2\\
    &\leq& ((r^N_e)^{-1}\overline{c^{N,e\bullet}})^2 \left\| V^N \right\|_2 \tn (Q^N)^{\intercal}\tn _{2} \left\| Q^N V^N \right\|_2\\
    &\leq& ((r^N_e)^{-1}\overline{c^{N,e\bullet}})^2 \tn (Q^N)^{\intercal}\tn _{2}\, \tn Q^N\tn _{2} \left\| V^N \right\|_2^2.
\end{eqnarray*}
Hence, $\E\left[ (\xi^{N,(2,2),-}_e)^2 \right] \leq K$ which ends the step.

\medskip
{\it Step 5.} Here we prove Equation \eqref{eq:convergence:w}.

First remark that $N^{-1}\sum_{i=1}^N\left(c^N_i\right)^2\left(m^N_i-\left(m^N_i\right)^2\right) = N^{-1}\left< (c^N)^2, m^N - (m^N)^2 \right>$. Then, combining Equations \eqref{eq:product:c2:m} and \eqref{eq:product:c2:ell}, we know that $N^{-1}\left< (c^N)^2, m^N \right>$ converges to (remind the definition of $m$ in \eqref{eq:definition:m:v:w})
\begin{equation*}
    (\mu+(1-\lambda)pr_-) \frac{1+4(1-\lambda)^2 p^2 r_+r_-}{(1-(1-\lambda)p(r_+-r_-))^3} = m \frac{1+4(1-\lambda)^2p^2r_+r_-}{(1-(1-\lambda)p(r_+-r_-))^2}.
\end{equation*}
Similarly, combining Equations \eqref{eq:product:c2:m2} and \eqref{eq:product:c2:ell2}, we know that $N^{-1}\left< (c^N)^2, (m^N)^2 \right>$ converges to
\begin{equation*}
    (\mu+(1-\lambda)pr_-)^2 \frac{1+4(1-\lambda)^2 p^2 r_+r_-}{(1-(1-\lambda)p(r_+-r_-))^4} = m^2 \frac{1+4(1-\lambda)^2p^2r_+r_-}{(1-(1-\lambda)p(r_+-r_-))^2}.
\end{equation*}
Finally, one ends up with the definition of $w$ by summing these two limits, which ends this step.
\end{proof}

\section{Inversion of $\Psi$}
\label{app:inversion}

\changes{
The main objective of this section is to prove Proposition \ref{prop:there:exists:inversion} and in particular provide the expressions of $\Phi^{(+)}$ and $\Phi^{(-)}$.
    As it appears below, the two functions $\Phi^{(+)}$ and $\Phi^{(-)}$ are related with the two roots $d^{(+)}$ and $d^{(-)}$ of a quadratic equation. In turn, these roots are related with the function $D$ defined in Equation \eqref{def_denominator_as_function_of_lambda_and_p}.
}

Here is a collection of preliminary results on the functions $D$ and $\Psi$.

\begin{proposition}\label{prop:preliminary:D:Psi}
    For all $(\mu,\lambda,p)\in \Lambda$,
    \begin{enumerate}
        \item if $r_+ <1/2$, then $1 < D(\lambda,p) < 2r_-$;
        \item if $r_+ = 1/2$, then $D(\lambda,p) = 1$;
        \item if $r_+ > 1/2$, then $2r_- < D(\lambda,p) < 1$.
    \end{enumerate}
    
    Whatever the value of $r_+$, the image $\Psi(\Lambda)$ is included in $(0,1)\times (0,\infty)^2$.

    Finally, if $r_+=1/2$, then $\Psi_3(\mu, \lambda, p) > \Psi_1(\mu, \lambda, p) [1 - \Psi_1(\mu, \lambda, p)]$, for all $(\mu,\lambda,p)\in \Lambda$.
\end{proposition}
\begin{proof}
    The statements regarding the function $D$ are obvious.

    Let $r_+\in [0,1]$ and $(\mu,\lambda,p)\in \Lambda$. For all $k=1,2,3$, the fact that $\Psi_k(\mu,\lambda,p)\in (0,\infty)$ is obvious. It only remains to prove that $\Psi_1(\mu, \lambda, p)<1$. This follows from 
    \begin{equation*}
        D(\lambda, p) - (\mu + (1 - \lambda)pr_-) = 1 - \mu - (1-\lambda)pr_+ > 1 - \lambda - (1-\lambda)r_+ = (1-\lambda)pr_- \geq 0.
    \end{equation*}

    Finally, if $r_+=1/2$, we have $D(\lambda,p) = 1$ and the statement follows from the fact that $1+4(1-\lambda)^2p^2r_+r_- > 1$.
\end{proof}

From now on, the objective is to invert the function $\Psi$ on the set of admissible parameters $\Lambda$. In view of Proposition \ref{prop:preliminary:D:Psi}, it suffices to find this inverse function on the set 
\begin{equation}
    \mathcal{M} = \begin{cases}
        (0,1)\times (0,\infty)^2, & \text{ if } r_+\neq 1/2;\\
        \left\{ (m,v,w)\in (0,1)\times (0,\infty)^2 : w > m(1-m) \right\}, & \text{ if } r_+ = 1/2.\\
    \end{cases}
\end{equation}
Hence, from now on, $(m,v,w)$ will denote an arbitrary vector in $\mathcal{M}$ (in particular, it is not related to $(\mu, \lambda,p)$). Remind the function $\kappa$ defined in Equation \ref{eq:definition:kappa} and let us define $d^{(+)}, d^{(-)} : (0,1)\times (0,\infty) \to \mathbb{R}$ by, for $a\in \{-,+\}$,
\begin{equation}\label{eq:d(a)}
    d^{(a)}(m,w) = \begin{cases}
        \frac{4r_+r_- + a \sqrt{(4r_+r_-)^2-4r_+r_-+\kappa(m,w)}}{4r_+r_- - \kappa(m,w)}, & \ \text{if} \ \kappa(m,w) \neq 4r_+r_-;\\
        (8r_+r_-)^{-1} , & \ \text{if} \ \kappa(m,w) = 4r_+r_-.\\
    \end{cases}
\end{equation}
If $(\mu,\lambda,p)$ and $(m,v,w)$ are related through \eqref{eq:definition:m:v:w} then the value $d^{(a)}(m,w)$ is a candidate to be equal to $D(\lambda,p)$. The values $d^{(+)}(m,w)$ and $d^{(-)}(m,w)$ are the roots of the quadratic equation
\begin{equation}\label{eq:quadratic:equation:d}
    [4r_+r_- - \kappa(m,w)] u^2 - (8r_+r_-) u + 1 = 0.
\end{equation}
On the one hand, $d^{(-)}$ is positive and one can check that it is $C^1$ (even when $\kappa(m,w) \to 4r_+r_-$). On the other hand, $d^{(+)}$ may be negative and goes to infinity when $\kappa(m,w) \to 4r_+r_-$. 
In view of this remark and Proposition \ref{prop:preliminary:D:Psi}, $d^{(-)}(m,w)$ is expected to be the good candidate to be equal to $D(\lambda,p)$. This is true most of the time but not for the whole range of parameters (see Proposition \ref{prop:inversion} below).

From now on, let $a\in\{-,+\}$. Let us define $\phi_1^{(a)} : (0,1)\times (0,\infty) \to (0,\infty)$ by
\begin{equation}
    \phi_1^{(a)}(m,w) = \begin{cases}
        w[m (1 - m)]^{-1} - 1, & \ \text{if} \ r_+=1/2;\\
        [1 - d^{(a)}(m,w)]^{2}(r_+ - r_-)^{-2} , & \ \text{else},\\
    \end{cases}
\end{equation}
and $\phi_2^{(a)} : \mathcal{M} \to (1,\infty)$ by 
\begin{equation}
    \phi_2^{(a)}(m,v,w) = 1 + \frac{v}{[(m - r_-)^2 + r_+r_-] \phi_1^{(a)}(m,w)}.
\end{equation}
If $(\mu,\lambda,p)$ and $(m,v,w)$ are related through \eqref{eq:definition:m:v:w} then $\phi_1^{(a)}(m,w)$ (\emph{respectively} $\phi_2^{(a)}(m,v,w)$) are two candidates to be equal to $(1-\lambda)^2p^2$ (\emph{resp.} $p^{-1}$).
Then, for $k=1,2,3$, let $\Phi_k^{(a)} : \mathcal{M} \to \mathbb{R}$ be defined by,
\begin{equation}\label{eq:definition:Phi}
    \begin{cases}
        \Phi_1^{(a)}(m,v,w) = m\left( 1 - (r_+ - r_-)\sqrt{\phi_1^{(a)}(m,w)} \right) - r_-\sqrt{\phi_1^{(a)}(m,w)},\\
        \Phi_2^{(a)}(m,v,w) = 1 - \phi_2^{(a)}(m,v,w)\sqrt{\phi_1^{(a)}(m,w)},\\
        \Phi_3^{(a)}(m,v,w) = (\phi_2^{(a)}(m,v,w))^{-1}.
    \end{cases}
\end{equation}
Finally, let $\Phi^{(a)} : \mathcal{M}  \to \mathbb{R}^3$ be defined by the three coordinate functions above. Let us remark that all the functions involved in the definition of $\Phi^{(a)}$ are obviously $C^\infty$ except $d^{(a)}$. Nevertheless, it is easy to check that $d^{(-)}$ is (at least) $C^1$ when $\kappa(m,w) \to 4r_+r_-$. In turn, $\Phi^{(-)}$ is regular.
\begin{proposition}\label{prop:Phi:smooth}
    The function $\Phi^{(-)}$ is $C^1$.
\end{proposition}

Finally, the functions $\Phi^{(+)}$ and $\Phi^{(-)}$ are related with $\Psi$ in the following sense.
\begin{proposition}\label{prop:inversion}
    Whatever the value of $r_+$, the following results hold.
    \begin{enumerate}
        \item For all $(\mu, \lambda, p)\in \Lambda$, $(\mu, \lambda, p) \in \{ \Phi^{(+)}\circ \Psi(\mu, \lambda, p), \Phi^{(-)}\circ \Psi(\mu, \lambda, p)\}$,
        \item moreover, if $r_+ \geq 1/2$ or 
        $$
        \kappa(\Psi_1(\mu, \lambda, p), \Psi_3(\mu, \lambda, p)) \geq 4r_+r_-,
        $$
        then $(\mu, \lambda, p) = \Phi^{(-)}\circ \Psi(\mu, \lambda, p)$.
        \item Let $(m,v,w)\in \mathcal{M}$ and $a\in \{-,+\}$. If $\Phi^{(a)}(m,v,w) \in \Lambda$ and 
        $$
        r_+=1/2 \text{ or } \operatorname{sgn}(1 - d^{(a)}(m,w)) = \operatorname{sgn}(r_+ - r_-),
        $$
        then $\Psi\circ\Phi^{(a)}(m,v,w) = (m,v,w)$.
    \end{enumerate}
\end{proposition}
\begin{remark}
    Remind that $d^{(a)}(m,w)$ is a candidate for $D(\lambda,p)$. In that regard, the sign condition appearing in Item 2 above is consistent with Proposition \ref{prop:preliminary:D:Psi}.
\end{remark}
\begin{proof}\ 
\paragraph*{Proof of 1}
Let $(\mu, \lambda, p)\in \Lambda$. In the following, we shorten the notation $\Psi_k = \Psi_k(\mu, \lambda, p)$. Finally, we denote $\Psi = (\Psi_1, \Psi_2, \Psi_3) \in \mathbb{R}^3$. 

Let us first consider the case $r_+ = 1/2$. In that case, the $\phi$ and $\Phi$ functions do not depend on the superscript $a\in \{-,+\}$ and we omit it in the following. Furthermore, $D(\lambda,p)=1$ and $\Psi$ reduces to
\begin{equation*}
    \begin{cases}
        \Psi_1 = \mu + (1 - \lambda)p/2\\
        \Psi_2 = (1-\lambda)^2 p(1-p)[(\Psi_1 - 1/2)^2 + 1/4]\\
        \Psi_3 = \Psi_1[1 - \Psi_1] [1 + (1-\lambda)^2p^2].
    \end{cases}
\end{equation*}
Hence, it is easy to check that $\phi_1(\Psi_1, \Psi_3) = (1-\lambda)^2p^2$ and 
\begin{equation*}
    \phi_2(\Psi) = 1 + \frac{(1-\lambda)^2 p(1-p)[(\Psi_1 - 1/2)^2 + 1/4]}{[(\Psi_1 - 1/2)^2 + 1/4] (1-\lambda)^2p^2} = 1 + \frac{1-p}{p} = \frac{1}{p}.
\end{equation*}
Then, we conclude by
\begin{equation*}
    \begin{cases}
        \Phi_1^{(a)}(\Psi) = \Psi_1 - (1-\lambda)p/2 = \mu,\\
        \Phi_2^{(a)}(\Psi) = 1 - \frac{1}{p}(1-\lambda)p = \lambda,\\
        \Phi_3^{(a)}(\Psi) = (1/p)^{-1} = p.
    \end{cases}
\end{equation*}

Let us then consider the case $r_+ \neq 1/2$. First, we remark that
\begin{equation*}
    \kappa(\Psi_1,\Psi_3) = (r_+-r_-)^2\frac{1+4(1-\lambda)^2p^2r_+r_-}{D(\lambda,p)^2}.
\end{equation*}
Then, substituting $(1-\lambda)p=(1-D(\lambda,p))/(r_+-r_-)$ into the equation aboves gives that
$$
\kappa(\Psi_1,\Psi_3) = \frac{(r_+-r_-)^2+4r_+r_-(1-D(\lambda,p))^2}{D(\lambda,p)^2}.
$$
Then, using the fact that $(r_+-r_-)^2 = (r_+ + r_-)^2 - 4r_+r_- = 1 - 4r_+r_-$, it follows that
\begin{equation*}
\kappa(\Psi_1,\Psi_3) D(\lambda,p)^2 = 1 +4r_+r_- [D(\lambda,p)^2-2D(\lambda,p)],
\end{equation*}
which means that $D(\lambda,p)$ is a solution of the quadratic equation
\begin{equation*}
    [4r_+r_- - \kappa(\Psi_1,\Psi_3)] X^2 - (8r_+r_-) X + 1 = 0.
\end{equation*}
By definition of $d^{(+)}$ and $d^{(-)}$, we necessarily have $D(\lambda,p) \in \{d^{(+)}(\Psi_1,\Psi_3), d^{(-)}(\Psi_1,\Psi_3)\}$.
In turn, since $(1 - D(\lambda,p))^2(r_+ - r_-)^{-2} = (1-\lambda)^2p^2$, we have 
$$
(1-\lambda)^2p^2 \in \{\phi_1^{(+)}(\Psi_1,\Psi_2), \phi_1^{(-)}(\Psi_1,\Psi_2)\},
$$
and the proof is concluded in the same manner as the case $r_+ = 1/2$.

\paragraph*{Proof of 2}
Let us first remark that, if $r_+ = 1/2$ or $\kappa(\Psi_1, \Psi_3) = 4r_+r_-$, then $\Phi^{(+)}=\Phi^{(-)}$ and the result is trivial. 

On the one hand, if $\kappa(\Psi_1, \Psi_3) > 4r_+r_-$ then the root $d^{(+)}(\Psi_1,\Psi_3)$ is negative which implies that $D(\lambda,p) = d^{(-)}(\Psi_1,\Psi_3)$. On the other hand, if $r_+ > 1/2$ and $\kappa(\Psi_1, \Psi_3) < 4r_+r_-$ then
\begin{equation*}
    d^{(+)}(\Psi_1,\Psi_3) \geq \frac{4r_+r_-}{4r_+r_- - \kappa(\Psi_1, \Psi_3)} > 1
\end{equation*}
Yet, we know that $D(\lambda,p)<1$ by Proposition \ref{prop:preliminary:D:Psi} which implies that $D(\lambda,p) = d^{(-)}(\Psi_1,\Psi_3)$.

As a summary, we have $D(\lambda,p) = d^{(-)}(\Psi_1,\Psi_3)$ in any case. In turn, it implies that $(1-\lambda)^2p^2 = \phi_1^{(-)}(\Psi_1,\Psi_3)$ and the proof is concluded as above.

\paragraph*{Proof of 3}
Let $(m,v,w)\in \mathcal{M}$, and $a\in \{-,+\}$ such that $\Phi^{(a)}(m,v,w) \in \Lambda$. In the following, we shorten the notation $d^{(a)} = d^{(a)}(m,w)$, $\phi_1^{(a)} = \phi_1^{(a)}(m,w)$, $\phi_2^{(a)} = \phi_2^{(a)}(m,v,w)$ and $\Phi_k^{(a)} = \Phi_k^{(a)}(m,v,w)$ for $k=1,2,3$. Finally, we denote $\Phi^{(a)} = (\Phi_1^{(a)}, \Phi_2^{(a)}, \Phi_3^{(a)}) \in \Lambda$. 

Without any condition, it is easy to check that $D(\Phi_2^{(a)}, \Phi_3^{(a)}) = 1 - (r_+ - r_-)\sqrt{\phi_1^{(a)}}$. In particular, it gives
\begin{eqnarray*}
    \Psi_1(\Phi^{(a)}) & = & \frac{\Phi_1^{(a)} + (1 - \Phi_2^{(a)})\Phi_3^{(a)}r_-}{ 1 - (r_+ - r_-)\sqrt{\phi_1^{(a)}} }\\
    &=& \frac{m\left[ 1 - (r_+ - r_-)\sqrt{\phi_1^{(a)}} \right] - r_-\sqrt{\phi_1^{(a)}} + r_-\sqrt{\phi_1^{(a)}}}{ 1 - (r_+ - r_-)\sqrt{\phi_1^{(a)}} } = m,
\end{eqnarray*}
and
\begin{eqnarray*}
    \Psi_2(\Phi^{(a)}) & = & (1 - \Phi_2^{(a)})^{2} \Phi_3^{(a)} (1 - \Phi_3^{(a)}) [(m - r_-)^2 + r_+r_-]\\
    &=& \phi_1^{(a)} [\phi_2^{(a)} - 1] [(m - r_-)^2 + r_+r_-] = v.
\end{eqnarray*}
For the last coordinate, we use the condition stated in the Proposition. 

Let us first consider the case $r_+=1/2$. In that case, the $\phi$ and $\Phi$ functions do not depend on the superscript $a\in \{-,+\}$ and we omit it in the following. First, remark that $D(\Phi_2^{(a)}, \Phi_3^{(a)}) = 1 - (r_+ - r_-)\sqrt{\phi_1^{(a)}} = 1$ in that case, so that
\begin{eqnarray*}
    \Psi_3(\Phi^{(a)}) & = & m(1-m) [1 + (1-\Phi_2^{(a)})^2 (\Phi_3^{(a)})^2] \\
    &=& m(1-m) \left[1 + \frac{w}{m(1-m)} - 1\right] = w.
\end{eqnarray*} 

Let us then consider the case $r_+ \neq 1/2$ and assume the sign condition: $\operatorname{sgn}(1 - d^{(a)}) = \operatorname{sgn}(r_+ - r_-)$. This condition implies that $D(\Phi_2^{(a)}, \Phi_3^{(a)}) = 1 - (r_+ - r_-)\sqrt{\phi_1^{(a)}} = d^{(a)}$. Hence,
\begin{eqnarray*}
    \Psi_3(\Phi^{(a)}) & = & m(1-m) [1 + 4r_+r_- (1-\Phi_2^{(a)})^2 (\Phi_3^{(a)})^2] / (d^{(a)})^2\\
    &=& m(1-m) \left[1 + 4r_+r_- \frac{(1-d^{(a)})^2}{(r_+ - r_-)^2}\right] / (d^{(a)})^2.
\end{eqnarray*} 
Yet, 
$$
4r_+r_- (1-d^{(a)})^2 = 4r_+r_- - 1 + (1 - 8r_+r_-d^{(a)} +4r_+r_-(d^{(a)})^2) = 4r_+r_- - 1 + \kappa(m,w) (d^{(a)})^2
$$ 
because $d^{(a)}$ solves the quadratic equation \eqref{eq:quadratic:equation:d}. Finally, using the definition of $\kappa$ and the fact that $1 - 4r_+r_- = (r_+ - r_-)^2$ we get that $\Psi_3(\Phi^{(a)}) = w$ which concludes the proof.
\end{proof}

\section{Auxiliary results}
\label{app:auxiliary}

\changes{
Here are two technical lemmas used throughout the proof of Proposition \ref{prop:control:N:infty}. 
}

\begin{lemma}\label{lem:technical:V:v}
    Let $(V^N)_N, (v^N)_N$ be two sequences of random vectors such that for all $N>0$, $V^N, v^N\in \mathbb{R}^N$. Assume that there exists a constant $K$ such that, for all $N>0$, $v^N$ can be written as $v^N = v_1^N + v_2^N + v_3^N$, where $v_1^N, v_2^N, v_3^N\in \mathbb{R}^N$, and
    \begin{enumerate}[(i)]
        \item $\E\left[ \|V^N\|_2^4 \right] + \|V^N\|_\infty^2 \leq K$ almost surely,
        \item $\E\left[ \|v_1^N\|_1^2 + \|v_2^N\|_2^4 \right]^2\leq K$,
        \item $\E\left[ \langle V^N, v_3^N \rangle^2 \right] \leq K$.
    \end{enumerate}
    Then, for all $N>0$,
    \begin{equation*}
        \E\left[ \langle V^N, v^N \rangle^2 \right] \leq 3 (K^2 + 2K).
    \end{equation*}
\end{lemma}
\begin{proof}
    First, by Holder inequality,
    \begin{equation*}
        \E\left[ \langle V^N, v^N_1 \rangle^2 \right] \leq \E\left[ \| V^N \|_\infty^2 \|v_1^N\|_1^2 \right] \leq K \E\left[ \|v_1^N\|_1^2 \right] \leq K^2.
    \end{equation*}
    Secondly, using Cauchy Schwarz inequality twice, we have
    \begin{equation*}
        \E\left[ \langle V^N, v^N_2 \rangle^2 \right] \leq \E\left[ \| V^N \|_2^2 \|v_2^N\|_2^2 \right] \leq \left( \E\left[ \|V^N\|_2^4 \right] \E\left[ \|v_2^N\|_2^4 \right] \right)^{1/2} \leq K.
    \end{equation*}
    And, $\E\left[ \langle V^N, v_3^N \rangle^2 \right] \leq K$ by assumption. Finally, we conclude by combining those three inequalities thanks to the  convexity of the square function.
\end{proof}

\begin{lemma}\label{lem:technical:u:v}
    Let $a\in \{-,+\}$. Let $(u^N)_N, (v^N)_N$ be two sequences of random vectors such that for all $N>0$, $u^N, v^N\in \mathbb{R}^N$ are supported in $\mathcal{P}_a$. Assume that there exist two random variables $u_3^\infty, v_3^\infty\in \mathbb{R}$ and a constant $K$ such that, for all $N>0$, $u^N = u_1^N + u_2^N + u_3^N 1_{\mathcal{P}_a}$ and $v^N = v_1^N + v_2^N + v_3^N 1_{\mathcal{P}_a}$, where $u_1^N, u_2^N, v_1^N, v_2^N\in \mathbb{R}^N$ are random vectors supported in $\mathcal{P}_a$ and $u_3^N, v_3^N\in \mathbb{R}$ are random variables, and
    \begin{itemize}
        \item $\E\left[ \|u_1^N\|_1^2 + \|u_2^N\|_2^4 + N^2 (\overline{u_2^N})^2 \right] + \|u_1^N + u_2^N\|_\infty + |u_3^N| + |u_3^\infty| \leq K$ almost surely,
        \item $\E\left[ \|v_1^N\|_1^2 + \|v_2^N\|_2^4 + N^2 (\overline{v_2^N})^2 \right] + \|v_1^N + v_2^N\|_\infty + |v_3^N| + |v_3^\infty| \leq K$ almost surely,
        \item $\E\left[ |u_3^N - u_3^\infty|^2 + |v_3^N - v_3^\infty|^2  \right] \leq K N^{-2}$.
    \end{itemize}
    Then, there exists another constant $K$ (independent of $N$) such that, for all $N>0$,
    \begin{equation*}
        \E\left[ \left| \frac{1}{N} \langle u^N, v^N \rangle - r_a v_3^\infty u_3^\infty \right|^2 \right] \leq K N^{-2}.
    \end{equation*}
\end{lemma}
\begin{proof}
Throughout the proof, the constant $K$ may change from line to line and even within the same line. Since $u^N$ and $v^N$ are decomposed into three vectors each, there are naturally nine contributions to the scalar product $\langle u^N, v^N \rangle$. 
    
    Let us first prove that the limit term $ r_a v_3^\infty u_3^\infty $ comes from the contribution of the constant parts, i.e. $N^{-1} \langle u^N_3 1_{\mathcal{P}_a}, v^N_3 1_{\mathcal{P}_a} \rangle = r_a^N u_3^N v_3^N$. Indeed, 
    \begin{eqnarray*}
        \left| r_a^N u_3^N v_3^N - r_a v_3^\infty u_3^\infty \right| 
        &\leq& \left| r_a^N u_3^N \right| \left| v_3^N -  v_3^\infty \right| + \left| r_a^N v_3^\infty \right| \left| u_3^N - u_3^\infty  \right| + \left| u_3^\infty v_3^\infty \right| \left| r_a^N - r_a \right|\\
        &\leq& K\left( \left| u_3^N - u_3^\infty  \right| + \left| v_3^N -  v_3^\infty \right| \right) + K^2 \left| r_a^N - r_a \right|,
    \end{eqnarray*}
    so that $\E\left[ \left| N^{-1} \langle u^N_3 1_{\mathcal{P}_a}, v^N_3 1_{\mathcal{P}_a} \rangle - r_a v_3^\infty u_3^\infty \right|^2 \right] \leq KN^{-2}$ thanks to the assumptions on $u_3^N$, $v_3^N$ and the sequence of fractions $r_a^N$.

    Then, let us prove that the other contributions to the scalar product are negligible. We have,
    \begin{equation*}
        \E\left[  \left< u_1^N, v^N \right> ^2 \right] \leq \E\left[ \|u_1^N\|_1^2 \|v^N\|_\infty^2 \right] \leq K \E\left[ \|u_1^N\|_1^2 \right] \leq K,
    \end{equation*}
    and similarly $\E\left[  \left< u^N, v_1^N \right> ^2 \right] \leq K$. Applying Cauchy-Schwarz inequality twice, we have
    \begin{equation*}
        \E\left[ \left< u_2^N, v_2^N \right>^2 \right] \leq \E\left[  \|u_2^N\|_2^2 \|v_2^N\|_2^2  \right] \leq \left(  \E\left[ \|u_2^N\|_2^4 \right] \E\left[ \|v_2^N\|_2^4 \right] \right)^{1/2} \leq K.
    \end{equation*}
    Finally, 
    \begin{equation*}
        \E\left[ \left< u_2^N, v_3^N 1_{\mathcal{P}_a} \right>^2 \right] \leq K \E\left[ \left< u_2^N, 1_{\mathcal{P}_a} \right>^2 \right] = K \E\left[ N^2 (\overline{u_2^N})^2 \right] \leq K,
    \end{equation*}
    and similarly $\E\left[ \left< u_3^N 1_{\mathcal{P}_a}, v_2^N  \right>^2 \right] \leq K$. Finally, we conclude by combining all those inequalities thanks to the  convexity of the square function.
\end{proof}

\section{A note on lower bounds}
\label{app:lower:bound}

The goal of this section is to discuss the optimality of our estimation rate. To that end,  we analyze two simple and related statistical settings (both inspired by the one considered in \cite{delattre2016statistical}) where all the computations can be done more transparently. The first statistical setting is discussed in Section \ref{subsec:binomial:mixture:1} and the second one in Section \ref{subsec:binomial:mixture:2}. In the end of that subsection, we also present some concluding remarks connecting the results proved.  
Finally, in Section \ref{app:lower:bound:auxiliary}, we state and prove a Gaussian approximation used in Section \ref{subsec:binomial:mixture:1}.

\smallskip
Throughout the section, we write $\theta$ to denote a random variable distributed as $\text{Bin}(N,p)$ where $0<p_{min}\leq p\leq p_{max}<1$ is an unknown parameter. Let $\kappa\in (0,p_{max}/2)$ be a known value and define
$\gamma(p)=\kappa/p$ for all $p_{min}\leq p\leq p_{max}$.
Observe that $\gamma(p)<1/2$ for all values of $p$. In what follows, we denote $m=1/2+\kappa$.

\subsection{Statistical setting 1 and problem formulation} 
\label{subsec:binomial:mixture:1}
Let $\theta\sim\text{Bin}(N,p)$ and consider a 
discrete random variable $B$ taking values in $\{0,\ldots, T\}$ such that $B|\theta\sim \text{Bin}(T,1/2+\gamma(p)\theta/N)$.
 Note $1/2+\gamma(p) \theta/N<1 $ for all
$p$ and all realizations of $\theta$, since $\gamma(p)<1/2$.
In particular, the conditional distribution of $B|\theta$ is well-defined no matter the realization of $\theta$. 
%By our assumption on the product $\gamma p$, the value $m=m(\gamma,p)$ is know for all values of $p$ and $\gamma$. 

Suppose we observe $N$ independent copies $B_1,\ldots, B_N$ of the random variable $B$ as defined above. Given these observations, we want to find an unbiased estimator for the parameter $1/p$ whose variance is of order $(N^{1/2}/T+1/N^{1/2})^2.$

In what follows, we write $\P_{p}$ to denote the probability distribution of $B_1,\ldots, B_N$ associated to the choice of $p$. The expectation and variance computed with respect to $\P_{p}$ are denoted $\E_{p}$ and $\text{Var}_{p}$ respectively.
One can easily check that, for all values of $p$, the mean of each random variable $B_i$ under $\P_{p}$ is a known quantity: $\E_{p}[B_i]=\E_{p}[B]=Tm.$

\begin{remark}
\label{rmk:stat_set_1}
Note that the goal is to determine an unbiased estimator for the parameter $1/p$ and not for the parameter $p$ itself. In Section \ref{subsec:binomial:mixture:2}, we explain why the parameter $1/p$ is a ``natural'' parameter to be estimated in this setting.  

Note also that the random variables $B_1, \ldots, B_N$ can be seen as an oversimplified version of the random variables $\sum_{t=1}^TX_{1,t}, \ldots, \sum_{t=1}^TX_{N,t}$ 
in which we drop the temporal and spatial dependence between the random variables $X_{i,t}$ given the realization of the environment and consider $\mu=1/2$ and ${\cal P_-}=\emptyset$ (no inhibition).
The parameter $\gamma$ plays the role of the parameter $(1-\lambda)$ in the original model. In particular, it is implicitly assumed that the baseline parameter $\mu$ is known and equals to $1/2$, making the estimation problem even simpler.     
\end{remark}

Recall that $m=1/2+\kappa$, where $\kappa$ is a known quantity. Given the random variables $B_1,\ldots, B_N$, define 
$$
\hat{V}=\frac{1}{N}\sum_{i=1}^N (B_i-Tm)^2,
$$
and consider the following estimator for $1/p$:
\begin{equation}
\label{def:est:1overp}
\hat{S}=\frac{N}{T(T-1)\kappa^2}(\hat{V}-Tm(1-m))+1.
\end{equation}

%\comJulien{The estimator $\hat{S}$ is natural and follows from the method of moments. We were not able to write it as a function of a complete statistics which would in turn prove that $\hat{S}$ is the unbiased estimator of $1/p$ with least variance. However, we expect that its variance gives the optimal rate of convergence.}

The estimator $\hat{S}$ is unbiased (see the proposition below) and follows from the method of moments.
We expect that its variance gives the optimal rate of convergence within the class of unbiased estimators of $1/p$.
However, we do not have a proof of such result at the moment.
%{\color{magenta} JC: I prefer your version so I commented mine. Maybe we could add in parenthesis that we tried the "sufficient complete" and the Cramer-Rao approaches without success ?}
%{\bf EVA : we can say it but not here, maybe in letter to referees? }
%{\color{magenta} JC: OK}
%We were not able to write it as a function of a complete statistics which would in turn prove that $\hat{S}$ is the unbiased estimator of $1/p$ with least variance}
The goal of this subsection is to show the following result.

\begin{proposition}
\label{thm:optimality_unbiased_class}
For the estimator $\hat{S}$ defined in \eqref{def:est:1overp} the following results hold:    
\begin{enumerate}
\item $\hat{S}$ is an unbiased estimator of $1/p$: for all $p$,
$$
\E_{p}[\hat{S}]=1/p.
$$
    \item For all $p$,
$$    
\text{Var}_{p}\left(\hat{S}\right)=\frac{2}{\kappa^4}\left(\frac{m(1-m)N^{1/2}}{T}+\frac{\gamma^2p(1-p)}{N^{1/2}}\right)^2(1+o(1)).
$$
\end{enumerate}

\end{proposition}

\begin{proof}

We start proving Item 1. First, use the decomposition 
$$\text{Var}_{p}(B)=\E_{p}[\text{Var}_{p}(B|\theta)]+\text{Var}_{p}[\E_{p}(B|\theta)]
$$
to check that 
\begin{equation}
\label{eq:var:B}
\text{Var}_{p}(B)=Tm(1-m)+\frac{T(T-1)}{N}\gamma^2p(1-p).
\end{equation}
Then, to conclude the proof, combine the above identity with the fact that 
$$
\E_{p}(\hat{S})=\frac{N}{T(T-1)(\gamma p)^2}(\text{Var}_{p}(B)-Tm(1-m))+1.
$$
Next, we prove Item 2. First, observe that
\begin{equation}
\label{eq:var:hatS}
\text{Var}_{p}\left(\hat{S}\right)=\frac{N}{T^2(T-1)^2(\gamma p)^4}\text{Var}_{p}\left((B-Tm)^2\right) .
\end{equation}
Next, write $\text{Var}_{p}\left((B-Tm)^2\right)=\E_{p}[(B-Tm)^4]-\text{Var}^2_{p}\left(B\right)$ and use \eqref{eq:var:B} to show that
\begin{equation}
\label{eq:rescaled:var:B}
\frac{N}{T^2(T-1)^2}\text{Var}^2_{p}\left(B\right)=\left(\frac{N^{1/2}}{(T-1)}m(1-m)+\frac{\gamma^2p(1-p)}{N^{1/2}}\right)^2.
\end{equation}
To conclude the proof, we approximate the term $\E_{p}[(B-Tm)^4]$. To that end, we use the Gaussian approximation provided in Lemma \ref{Lem:strong_coupling} below. Using this approximation, we can write $B/T-m=G+\epsilon$ so that
\begin{equation}
\label{eq:4thmomentB}
\E_{p}[(B-Tm)^4]=T^4\left(\E_{p}[G^4]+\sum_{k=1}^4 {4\choose k} \E_{p}[G^{4-k}\epsilon^k] \right).    
\end{equation}
Now, since $G$ is a centered Gaussian random variable with variance $m(1-m)T^{-1}+\gamma^2 p(1-p)N^{-1}$, we obtain that  
\begin{equation}
\label{eq:4thmomentG}
\E_{p}[G^4]=3\left(\frac{m(1-m)}{T}+\frac{\gamma^2p(1-p)}{N}\right)^2.
\end{equation}
Moreover, using Cauchy-Schwarz inequality, the estimate \eqref{uniformly_bounded_second_moment_of_errors}, and the fact that G is a centered Gaussian random variable, one can show that there exists a constant $C_k$ depending only on $k$ such that
\begin{multline*}
\E_{p}[G^{4-k}\epsilon^k]\leq C_k \left(\frac{m(1-m)}{T}+\frac{\gamma^2p(1-p)}{N}\right)^{(4-k)}\times \\ 
\left(\frac{1}{T^{k/2}N^{k/4}}+\frac{\log^k(T)}{T^k}+\frac{\log^k(N)}{N^k}\right).
\end{multline*}
Combining the above inequality with Jensen inequality, we then conclude that
\begin{multline*}
N\E_{p}[G^{4-k}\epsilon^k]\leq \tilde{C}_k\left(\frac{N^{(4-k)/4}}{T^{4-k/2}}+\frac{N\log^k(T)}{T^4}+\frac{\log^k(N)}{N^{(k-1)}T^{(4-k)}}+\frac{1}{T^{k/2}N^{3-3k/4}}\right. \\ \left. 
+\frac{\log^k(T)}{N^{(3-k)}T^{k}}+\frac{\log^k(N)}{N^{3}}\right),    
\end{multline*}
where $\tilde{C}_k$ is a constant depending only on $k$.
Since 
$\frac{N^{(4-k)/4}}{T^{4-k/2}}\frac{T^2}{N}=\frac{1}{T^{2-k/2}N^{k/4}}\to 0$ and
$\frac{N\log^k(T)}{T^{4}}\frac{T^2}{N}=\frac{\log^k(T)}{T^{2}}\to 0$ as $N,T\to \infty$, 
one can deduce from the above inequality that
\begin{equation}
\label{eq:estimates:covG_error}
\max_{k=1,\ldots, 4}N\E_{p}[G^{4-k}\epsilon^k]=o\left(\left(\frac{m(1-m)N^{1/2}}{T}+\frac{\gamma^2p(1-p)}{N^{1/2}}\right)^2\right) .
\end{equation}
Combing \eqref{eq:var:hatS}, \eqref{eq:rescaled:var:B}, \eqref{eq:4thmomentB}, \eqref{eq:4thmomentG} and \eqref{eq:estimates:covG_error}, the result follows.

\end{proof}

\subsection{Statistical setting 2 and problem formulation}
\label{subsec:binomial:mixture:2}
As in the previous subsection, let $\theta\sim\text{Bin}(N,p)$.
Here, we assume that the random variable $B$ takes values in $\{0,1,2\}$ and is such that $B|\theta\sim \text{Bin}(2,1/2+\gamma(p)\theta/N)$.

%We suppose that we observe $N$ independent copies $B_1,\ldots, B_N$ of the random variable $B$ as defined above. Besides, we suppose we have access to an oracle that reports independently of the observations $B_1,\ldots, B_N$ the true value of the product $p\gamma$. Specifically, we suppose that we also observe a random variable $O$ such that $O=p\gamma$ almost surely under $\P_{p,\gamma}$. Therefore, given the sample $B_1,\ldots, B_N$ and the oracle $O$, we want to find the best estimator for the parameter $1/p$ in a suitable class of unbiased estimators.

%One can easily check that, for all values of $p$, the mean of the random variable $B$ under $\P_{p}$ is a known quantity: $\E_{p}[B]=1/2+\gamma p=1/2+\kappa:=m.$

%By our assumption on the product $\gamma p$, the value $m=m(\gamma,p)$ is know for all values of $p$ and $\gamma$. 

Suppose that $ T$ is even and that we observe $n:=(NT)/2$ independent copies $B_1,\ldots, B_n$ of the random variable $B$ as defined above. We want to show that the variance of any
unbiased estimator of the parameter $1/p$ is larger than $K(N/T)$ for some positive universal constant $K$. Hence,   
if $N/T$ remains bounded away from $0$, then no unbiased estimator of the parameter $1/p$ is consistent.   

With a slight abuse of notation, we also write $\P_{p}$ to denote the probability distribution of $B_1,\ldots, B_n$ associated to the choice of $p$. The expectation and variance computed with respect to $\P_{p}$ are denoted $\E_{p}$ and $\text{Var}_{p}$ respectively. In this setting, observe that $\E_{p}[B_i]=2(1/2+\kappa)=2m$ for all values of $p$.

\begin{remark}
\label{rmk:stat:set:2}
In this setting, the random variables $B_1, \ldots, B_n$ are distributed as the random variables $(X_{1,1}+X_{1,2}), (X_{1,3}+X_{1,4}), \ldots,  (X_{1,T-1}+X_{1,T}),\ldots, (X_{N,1}+X_{N,2}), (X_{N,3}+X_{N,4}), \ldots, (X_{N,T-1}+X_{N,T})$ 
under the assumptions that: 1) $\mu=1/2$ and ${\cal P_-}=\emptyset$ (no inhibition), 2) for all $1\leq i\leq N$ and $1\leq t\leq T/2$, the random variables $X_{i,2t-1}$ and $X_{i,2t}$ are conditionally independent given the realization of the environment and 3) all the other dependencies are dropped.
Note that the random variables $X_{i,t}$ are less dependent under this set of assumptions than under the assumptions made in the statistical setting 1. 
\end{remark}

While in our original problem, the ``natural'' parameter to estimate is $p$, in this statistical setting the ``natural'' parameter is $1/p$ (as well as in the statistical setting 1). We can see this by checking that the probability mass function $f_{p}$ of each random variable $B_i$ under $\P_{p}$ can be written as 
\begin{equation*}
f_{p}(b)={2\choose b}\left[m^b(1-m)^{(2-b)}+\frac{\kappa^2}{N}\left(\frac{1}{p}-1\right)s(b)\right],  \ b\in\{0,1,2\}, 
\end{equation*}
where $s(b)=1$ if $b\in\{0,2\}$ and $s(b)=-1$ if $b=1$. The expression above for $f_p(b)$ follows from the fact that
$$
f_{p}(b)=\E_p\left[{2\choose b}(1/2+\gamma(p)\theta/N)^b(1/2-\gamma(p)\theta/N)^{2-b}\right]
$$
and simple calculations involving binomial random variables.

To alleviate the notation, we parametrize the model by $\eta=1/p$ in what follows. In particular, we write $\P_{\eta} = \P_{1/p}$, $\E_{\eta} = \E_{1/p}$ and $\text{Var}_{\eta} = \text{Var}_{1/p}$. Therefore, we assume that we observe $n$ independent copies $B_1,\ldots, B_n$ of a random variable $B$ taking values in $\{0,1,2\}$ having as probability mass function $g_{\eta}(b)=f_{1/\eta}(b)$. 
We want to show that no unbiased estimator of $\eta$ based on these observations is consistent whenever $N/T$ remains bounded away from $0.$ This is the content of the proposition below. 
%The variance computed with respect to $\P_{\eta}$ is denoted  

\begin{proposition}
\label{prop:lower:bound:stat:set2}
For all $N$ large enough,  we have for any unbiased estimator $\hat{\eta}(B_1,\ldots, B_n)$ of $\eta$ that 
$$
\text{Var}_{\eta}\left(\hat{\eta}(B_1,\ldots, B_n)\right)\geq \frac{NC^2}{T\kappa^4},
$$
where $C:= \min\{m^2, m(1-m), (1-m)^2\}>0$.  
In particular, if $N/T$ remains bounded away from $0$, then no unbiased estimator of $\eta$ is consistent.
\end{proposition}
\begin{proof}
First, verify that the Fisher information of a single observation is given by
$$
I_1(\eta):=-\E_{\eta}\left[\frac{\partial^2\log g_{\eta}(B)}{\partial \eta^2}\right]=\frac{\kappa^4}{N^2}\E_{\eta}\left[\frac{1}{(m^B(1-m)^{(2-B)}+\frac{\kappa^2}{N}(\eta-1)s(B))^2}\right].
$$
Then, use Corollary 5.9 of \cite{lehmann1998theory} to deduce that the Fisher information of $n=NT/2$ i.i.d observations is given by
$$
I_n(\eta)=nI_1(\eta)=\frac{T\kappa^4}{2N}\E_{\eta}\left[\frac{1}{(m^B(1-m)^{(2-B)}+\frac{\kappa^2}{N}(\eta-1)s(B))^2}\right].
$$
Next, observe that
\begin{multline*} (m^B(1-m)^{(2-B)}+\frac{\kappa^2}{N}(\eta-1)s(B))^2\\
\geq m^B(1-m)^{2- B}\left(m^B(1-m)^{2- B} + 2\kappa^2 (\eta-1)/N s(B) \right) 
\end{multline*}
and that $(\eta-1)/N s(B) \geq 0$ for $ B = 0, 2.$ Combining these inequalities with the definition of $C= \min\{m^2, m(1-m), (1-m)^2\}$, we deduce that 
$$
(m^B(1-m)^{(2-B)}+\frac{\kappa^2}{N}(\eta-1)s(B))^2\geq C\left(C-2\kappa^2 (\eta-1)/N\right),
$$
%{\color{blue} (J: I don't see how you get rid of the powers in $(m^B(1-m)^{(2-B)}$. Maybe it is simpler to clearly state from the beginning that $m$ is fixed in $(0,1)$ so that $C = \min\{m^2, m(1-m), (1-m)^2\} >0$ is the "constant" that should appear on the right hand side above ?)}
from which 
it follows that for $N\geq 4(\eta-1)\kappa^2/C$, 
$$
(m^B(1-m)^{(2-B)}+\frac{\kappa^2}{N}(\eta-1)s(B))^2\geq C^2/2 ,
$$
which, in turn, implies that  
$$
I_n(\eta)\leq \frac{T\kappa^4}{NC^2},
$$
as long as $N\geq 4(\eta-1)\kappa^2/C$. By Assumption 1 (and restricting the range of $p$ if necessary), we have that $(\eta-1)\leq M$ for some suitable constant $M$ so that for all $N\geq 4M\kappa^2/C$, we have that  
$$
I_n(\eta)\leq \frac{T\kappa^4}{NC^2}.
$$
Therefore, by Cramér-Rao lower bound (see Theorem 5.10 of \cite{lehmann1998theory}), it then follows that for all unbiased estimators $\hat{\eta}(B_1,\ldots, B_n)$ of $\eta$, 
$$
\text{Var}_{\eta}\left(\hat{\eta}(B_1,\ldots, B_n)\right)\geq (I_n(\eta))^{-1}=\frac{NC^2}{T\kappa^4},
$$
and the result follows.
\end{proof}

{\it Concluding remarks.}
Both statistical settings 1 and 2 consider a statistical model with a hidden layer described by the random variable $\theta$, similar as in the original model.
As indicated in Remarks \ref{rmk:stat_set_1} and \ref{rmk:stat:set:2}, the laws of the statistical models considered in these settings are related to the law of the original model under different sets of assumptions dropping many dependencies. Moreover, more dependencies of the original model are preserved in setting 1. 

Both Propositions \ref{thm:optimality_unbiased_class} and  \ref{prop:lower:bound:stat:set2} highlight the interplay between the dimension of the original model (number $N$ of observed nodes) and the sample size (the number $T$ of time units observed per node). In the regime $N^2\gg T$ (which implies that $N/T\gg (N^{1/2}/T+1/N^{1/2})^2$), the results of Proposition  \ref{thm:optimality_unbiased_class} and  \ref{prop:lower:bound:stat:set2} combined indicate that  estimating the model parameter becomes easier as more dependencies of the original model are preserved.
This partially explains why the lower bound in Proposition \ref{prop:lower:bound:stat:set2} suggests that the condition $T\gg N$ is necessary for our original estimation problem, while Item 2 of Proposition \ref{thm:optimality_unbiased_class} indicates that the weaker condition $T\gg N^{1/2}$ might suffice for consistently estimating the model parameters. 
Proving such a result in the original model remains an interesting open question.

\subsection{Auxiliary results}
\label{app:lower:bound:auxiliary}

In this section, we state and prove the Gaussian approximation used to prove Item 2 of Proposition \ref{thm:optimality_unbiased_class}. This result follows almost immediately from the KMT approximation theorem (see for example \cite{komlos1975approximation}). 

\begin{lemma}
\label{Lem:strong_coupling}
Let $B$ be a discrete random variable taking values in $\{0,\ldots, T\}$ such that $B|\theta\sim \text{Bin}(T,1/2+\gamma\theta/N)$ where $\theta\sim\text{Bin}(N,p)$, $0<p_{min}\leq p\leq p_{max}<1$ and $0<\gamma_{min}\leq \gamma\leq \gamma_{max}<1/2$.
Denote $m=1/2+\gamma p$. There exists a probability space $(\Omega, {\cal F},\P)$ on which almost surely 
\begin{equation*}
\frac{B}{T}-m=G+\epsilon,    
\end{equation*}
where $G$ is a centered Gaussian random variable with variance $m(1-m)T^{-1}+\gamma^2p(1-p)N^{-1}$ and the error term $\epsilon$ satisfies for all integer $\alpha\geq 1$,
\begin{equation}
\label{uniformly_bounded_second_moment_of_errors}
 \E[|\epsilon|^{\alpha}]\leq K_{\alpha}\left(\frac{1}{T^{\alpha/2}N^{\alpha/4}}+\left(\frac{\log(T)}{T}\right)^{\alpha}+\left(\frac{\log(N)}{N}\right)^{\alpha}\right), 
 \end{equation}
for some positive constant $K_{\alpha}$ depending only $\alpha$.   
\end{lemma}
  
\smallskip

\begin{proof}
Throughout the proof, we assume we work on a probability space $(\Omega, {\cal F},\P)$ rich enough so that all the processes below are well-defined. 

Let $(U_{t})_{1\leq t\leq T}$ and $(V_{j})_{j\in [N]}$ be two independent collections of i.i.d random variables uniformly distributed on $(0,1)$. For each $u\in [0,1]$, define
$$
F_{T}(u)=\frac{1}{T}\sum_{t=1}^T1_{U_{t}\leq u} \ \text{and} \ F_{N}(u)=\frac{1}{N}\sum_{j=1}^N 1_{V_{j}\leq u},
$$
and 
$$
\alpha_{T}(u)=\sqrt{T}(F_{T}(u)-u) \ \text{and} \ \alpha_{N}(u)=\sqrt{N}(F_{N}(u)-u).
$$
With this notation, one can check that
\begin{equation}
\label{proof_lower_bound_eq_1}
\frac{B}{T}=m+\frac{1}{\sqrt{T}}\alpha_{T}(1/2+\gamma\theta/N)+\frac{\gamma}{\sqrt{N}}\alpha_{N}(p).    
\end{equation}
Next, consider two independent Brownian Bridges $(B_{T}(t))_{0\leq u\leq 1}$ and $(B_{N}(t))_{0\leq u\leq 1}$ and define for each $u\in [0,1]$,
\begin{equation}
\label{proof_lower_bound_eq_2}
\epsilon_{T}(u)=\alpha_{T}(u)-B_{T}(u) \ \text{and} \ \epsilon_{N}(u)=\alpha_{N}(u)-B_{N}(u).
\end{equation}
We also suppose that the first Brownian Bridge $(B_{T}(t))_{0\leq t\leq 1}$ is independent of $\theta$ as well.
Furthermore, we follow the construction used in the KMT approximation (see \cite{komlos1975approximation}) in such a way that we can also assume that for $z>0$,
\begin{equation}
\label{proof_lower_bound_eq_3}
\P\left(\sup_{0\leq u\leq 1}|\epsilon_{T}(u)|\geq \frac{1}{T^{1/2}}(a\log(T)+z)\right)\leq b e^{-cz},
\end{equation}
and 
$$
\P\left(\sup_{0\leq u\leq 1}|\epsilon_{N}(u)|\geq \frac{1}{N^{1/2}}(a\log(N)+z)\right)\leq b e^{-cz},
$$
for some positive universal constants $a, b$, and $c$.
Using that $B_{N}(u)\sim N(0,u(1-u))$ for each $u\in [0,1]$, we immediately see that
$$
\frac{\gamma}{\sqrt{N}}B_{N}(p)\sim N\left(0,\frac{\gamma^2p(1-p)}{N}\right). 
$$
By similar arguments, we also have that 
$\frac{1}{\sqrt{T}}B_{T}(m)\sim N(0,\frac{m(1-m)}{T})$, so that it follows from the independence between $B_{T}(p)$ and $B_{N}(p)$ that
\begin{equation}
\label{proof_lower_bound_eq_4}
\frac{1}{\sqrt{T}}B_{T}(m)+\frac{\gamma}{\sqrt{N}}B_{N}(p)\sim N\left(0,\frac{m(1-m)}{T}+\frac{\gamma^2p(1-p)}{N}\right).
\end{equation}
Combining identities \eqref{proof_lower_bound_eq_1}, \eqref{proof_lower_bound_eq_2} and \eqref{proof_lower_bound_eq_4}, we then deduce that
\begin{equation*}
\frac{B}{T}-m=G+\epsilon,
\end{equation*}
 where 
 $$
G=\frac{1}{\sqrt{T}}B_{T}(m)+\frac{\gamma}{\sqrt{N}}B_{N}(p)
 $$
is a centered Gaussian random variable with variance $m(1-m)T^{-1}+\gamma^2p(1-p)N^{-1}$ and 
\begin{equation}
\label{proof_lower_bound_eq_5}
\epsilon=\frac{1}{\sqrt{T}}\left((B_{T}(1/2+\gamma\theta/N)-B_{T}(m))+\epsilon_{T}(1/2+\gamma\theta/N)\right)
+\frac{\gamma}{\sqrt{N}}\epsilon_{N}(p).   
\end{equation}

Next, we will prove that the error term $\epsilon$ defined in  \eqref{proof_lower_bound_eq_5} satisfies the bound stated in \eqref{uniformly_bounded_second_moment_of_errors}. By combining the tail integral formula for the moments with \eqref{proof_lower_bound_eq_3}, one can show that for any integer $\alpha\geq 1$  
\begin{equation}
\label{proof_lower_bound_eq_6}
\E\left[\left(\sup_{0\leq u\leq 1}|\epsilon_{T}(u)|\right)^{\alpha}\right]\leq K\frac{\log^{\alpha}(T)}{T^{\alpha/2}},
\end{equation}
for some constant $K$ depending only on $a,b, c$ and $\alpha$.
Similarly, one can also show that for any integer $\alpha\geq 1$  
\begin{equation}
\label{proof_lower_bound_eq_7}
\E\left[\left(\sup_{0\leq u\leq 1}|\epsilon_{N}(u)|\right)^{\alpha}\right]\leq K\frac{\log^{\alpha}(N)}{N^{\alpha/2}},
\end{equation}
for some constant $K$ depending only on $a,b, c$ and $\alpha$.
Moreover, by the independence between $(B_{T}(u))_{0\leq u\leq 1}$ and $\theta$, we can write
$$
\E[(B_{i,T}(1/2+\gamma\theta/N)-B_{i,T}(m))^{\alpha}]=\E[\varphi_{\alpha}(1/2+\gamma\theta/N)],
$$
where $\varphi_{\alpha}(q)=\E[(B_{T}(q)-B_{T}(m))^{\alpha}]$, $q\in [0,1]$.
Next, using properties of a Brownian Bridge, we can check that 
$$\varphi_{\alpha}(q)\leq 2^{\alpha-1}|q-m|^{\alpha/2}\E\left[|N(0,1)|^{\alpha}\right](1-|q-m|^{\alpha/2})$$ 
so that denoting $K_{\alpha}=2^{\alpha-1}\E\left[|N(0,1)|^{\alpha}\right]$, we obtain that
$$
\E[(B_{T}(1/2+\gamma\theta/N)-B_{T}(m))^{\alpha}]\leq K_{\alpha}\gamma\E[|\theta/N-p|^{\alpha/2}],
$$
where we have also used that $0\leq |q-m|\leq 1$ which implies that $(1-|q-m|^{\alpha/2})\leq 1$.
 
By applying Lemma F.7 of the main text, we then deduce that
\begin{equation}
\label{proof_lower_bound_eq_8}
\E[(B_{T}(1/2+\gamma\theta/N)-B_{T}(m))^{\alpha}]\leq KN^{-\alpha/4}
\end{equation}
for $K=K_{\alpha}\gamma_{max}^{\alpha}C_{\alpha}$, where $C_{\alpha}$ is a constant depending only on  $\alpha$.
Combining \eqref{proof_lower_bound_eq_6}, \eqref{proof_lower_bound_eq_7}, \eqref{proof_lower_bound_eq_8} with Jensen inequality, we then obtain that
\begin{equation}
\label{proof_lower_bound_eq_9}
\E[|\varepsilon|^{\alpha}]\leq K\left(\frac{1}{T^{\alpha/2}N^{\alpha/4}}+\left(\frac{\log(T)}{T}\right)^{\alpha}+\left(\frac{\log(N)}{N}\right)^{\alpha}\right),
\end{equation}
for some constant $K$ depending only on $a, b, c, \gamma_{max}$ and $\alpha$. 

\end{proof}

\end{appendix}

\end{document}